\documentclass[11pt]{book}

\usepackage{amsmath,amsfonts,amsthm,amscd,amssymb,graphicx}

\usepackage{hyperref}

\usepackage{subfigure}
 
 \usepackage{xcolor}

\newtheorem{theorem}{Theorem}[section]
\newtheorem{theo}{Theorem}[section]
\newtheorem{lemma}[theorem]{Lemma}
\newtheorem{lem}[theorem]{Lemma}
\newtheorem{proposition}[theorem]{Proposition}
\newtheorem{prop}[theorem]{Proposition}

\newtheorem{remark}[theorem]{Remark}

\def\eps{\varepsilon }

\newcommand{\RR}{\mathbb{R}}
\newcommand{\cO}{\mathcal{O}}

\newcommand{\CC}{\mathbb{C}}

\newcommand{\NN}{{\mathbb N}}
\newcommand{\ZZ}{{\mathbb Z}}
\newcommand{\TT}{{\mathbb T}}

\newcommand{\Range}{{\rm range}  \,}

\def\beq{\begin{equation}}
\def\eeq{\end{equation}}
\def\bb1{{1\!\!1}}
\def\I{\mbox{Im }}

\def\cB{\mathcal{B}}

\newcommand{\pa}{{\partial}}

\def\bl{\mathrm{bl}}

\def\cA{\mathcal{A}}

\def\rit{{\Bbb R}}
\def\cit{{\Bbb C}}

\def\zit{{\Bbb Z}}

\def\eps{\varepsilon}

\begin{document}

\title{Stability of shear flows near a boundary}

\author{Emmanuel Grenier\footnotemark[1]
 \and Toan T. Nguyen\footnotemark[2]}

\date\today

\maketitle 

\renewcommand{\thefootnote}{\fnsymbol{footnote}}

\footnotetext[1]{Equipe Projet Inria NUMED,
 INRIA Rh\^one Alpes, Unit\'e de Math\'ematiques Pures et Appliqu\'ees., 
 UMR 5669, CNRS et \'Ecole Normale Sup\'erieure de Lyon,
               46, all\'ee d'Italie, 69364 Lyon Cedex 07, France. Email: egrenier@umpa.ens-lyon.fr}

\footnotetext[2]{Department of Mathematics, Penn State University, State College, PA 16803. Email: nguyen@math.psu.edu. TN is partly supported by the NSF under grant DMS-1764119 and an AMS Centennial fellowship.}

\tableofcontents


\chapter{Introduction}\label{chapter-intro}


\section{Introduction}


This book is devoted to the study of the linear and nonlinear stability of shear flows for Navier Stokes equations
for incompressible fluids with Dirichlet boundary conditions in the case of small viscosity.

More precisely, we shall consider the classical Navier-Stokes equations for an incompressible
fluid in a  spatial domain $\Omega\subset \RR^d$, with $d \ge2$, 
\beq \label{NavierStokes1}
\partial_t u +  u\cdot\nabla u - \nu \Delta u + \nabla p = 0,
\eeq
\beq \label{NavierStokes2}
\nabla\cdot u = 0,
\eeq
together with the initial  condition
\beq \label{NavierStokes3}
u(0,\cdot)  = u_0,
\eeq
and the classical no-slip boundary condition \beq \label{NavierStokes4}
u = 0 \quad \hbox{on} \quad \partial \Omega.
\eeq
In these equations, $u(t,x)$ is the velocity of the fluid, $p(t,x)$ the pressure, $\nu > 0$
the viscosity and $u_0$ an initial velocity field. The Dirichlet boundary condition (\ref{NavierStokes4}) 
expresses the fact that the fluid "sticks" to the boundary, and hence its velocity vanishes on $\partial \Omega$.

\medskip

In this book we will focus on the study of the stability of particular solutions of Navier Stokes, known as
"shear flows" solutions. These solutions are of the form
\beq \label{shear}
U_{shear}(t,x,y,z) = \left( \begin{array}{c}
U_s(t,z) \cr
0 \cr
\end{array} \right),
\eeq
where $U_s$ satisfies the classical heat equation
\beq \label{heatUs}
\partial_t U_s - \nu \partial_z^2 U_s = 0 .
\eeq
Note that the solution of this equation is well defined and smooth for any positive time. The shear
flow changes of time scales of order $O(\nu^{-1})$.

We will also consider the case of a time independent shear flow $U_{shear}(z)$, up to an additional forcing term
in the right hand side of (\ref{NavierStokes1}), of the form
\beq \label{shearforce}
F_{shear}(t,x,y,z) = \left( \begin{array}{c}
- \nu \partial_z^2 U_s \cr
0 \cr
\end{array} \right).
\eeq
Note that the forcing term is small, of order $O(\nu)$. 

\medskip

We are interested in the stability of these shear flows for small viscosity $\nu$. 
As $\nu \to 0$, formally,
the limit is the classical Euler equations for an incompressible ideal fluid, namely
\beq \label{Euler1}
\partial_t u + u\cdot \nabla u + \nabla p = 0,
\eeq
\beq \label{Euler2}
\nabla\cdot u = 0,
\eeq
together with the initial condition
\beq \label{Euler3}
u(0,\cdot)  = u_0,
\eeq
and the boundary condition 
\beq \label{Euler4}
u \cdot n= 0 \quad \hbox{on} \quad \partial \Omega,
\eeq
where $n(x)$ denotes the unit vector normal to the boundary $\partial \Omega$.

\medskip

Let us now describe the general landscape in an informal way. Two cases arise.

\begin{itemize}

\item If the shear flow $U_{shear}(0,x,y,z)$ is spectrally unstable for Euler equation, then we
can expect that it is also unstable for Navier Stokes equations provided $\nu$ is small enough.
We will prove in this book that it is indeed the case: if $\nu$ is small enough, the shear flow is linearly and
nonlinearly unstable for Navier Stokes equations.

\item If the shear flow $U_{shear}(0,x,y,z)$ is spectrally stable for Euler equation, then we could expect that it would
be also stable for Navier Stokes equations. This turns out not to be the case, which is somehow counter intuitive.
We will prove in this book that such profiles are linearly unstable for Navier Stokes equations, and
give preliminary nonlinear instability results.

\end{itemize}


\section{Main results}


Let us now state the two main results of the book, in an informal way.

\medskip

{\bf Result $1$}

\medskip

If the shear layer profile $U_s$ is spectrally unstable for Euler equations, then it is also spectrally unstable for
Navier Stokes equations provided the viscosity is small enough. 
It is also nonlinearly unstable in the following sense. For arbitrarily large
$N$ and $s$ we can find a perturbation of order $\nu^N$ in Sobolev space $H^s$, such that at
a later time $T_\nu$ of order $\log \nu^{-1}$, the perturbation reaches a size $O(1)$ in $L^\infty$ and $L^2$.

\medskip

{\bf Result $2$}

\medskip

If the shear layer profile $U_s$ is spectrally stable for Euler equations, then it is spectrally
{\it unstable} for Navier Stokes equations provided the viscosity is small enough. The corresponding
eigenvalue has a very small real part of order $O(\nu^{1/2})$. It is also "weakly" nonlinearly
unstable in the following sense. For arbitrarily large
$N$ and $s$ we can find a perturbation of order $\nu^N$ in Sobolev space $H^s$, such that at
a later time $T_\nu$ of order $\nu^{-1/2} \log \nu^{-1}$, 
the perturbation reaches a size $O(\nu^{1/4})$ in $L^\infty$ and $L^2$.

\medskip


\section{Physical introduction}


The question of the linear and nonlinear instability of shear flows is one of the most classical questions
of fluid mechanics. Its study goes back to Rayleigh at the end of the nineteenth century, and
expanded through the beginning of the twentieth century, thanks to Orr, Sommerfeld, C.C. Lin, Tollmien, Schlichting,...

\medskip

Physically, the inviscid limit is linked to the high Reynolds number limit. 
The Reynolds number is a non-dimensional number, defined by 
\begin{equation}\label{def-Re}
Re = {U L \over \nu}
\end{equation}
where $U$ is a typical velocity of the flow, $L$ a typical length and $\nu$ the viscosity. 
Clearly, for each fixed $U$ and $L$, $\nu \to 0$ if and only if $Re \to \infty$. 
Throughout out the book, small viscosity or high Reynolds number is used interchangeably. 

In his seminal experiments in $1883$,  Reynolds first pointed out that flows at a high Reynolds number experience turbulence.
 In other words, well-organized flows can become chaotic under infinitesimal disturbances when the Reynolds number exceeds 
 a critical number. 
 
 Let us give a simple example, and consider a time independent solution $u_0$ of Navier Stokes equations with a forcing
term $f_0$. Let us prove that $u_0$ is stable provided the Reynolds number is small enough.
For sake of simplicity, let us assume that $\Omega$ is a bounded domain. Let $u$ be any time-dependent solution of
Navier Stokes equations with the same forcing term $f$. Then, the perturbation $v = u - u_0$ satisfies
$$
\partial_t v + (u_0 + v)\cdot \nabla v + v\cdot \nabla u_0  - \nu \Delta v 
+ \nabla q = 0,
$$
$$
\nabla\cdot v = 0
$$
with the zero boundary condition $v = 0$ on $\partial \Omega$. The  energy induced by the perturbation satisfies  
$$
\frac{1}{2}\frac{d}{dt} \int_\Omega |v|^2 \; dx  + \nu \int_\Omega | \nabla v |^2\; dx =   - \int_\Omega (v \cdot \nabla) u_0\cdot v \; dx .
$$
Clearly, the term on the right is responsible for any possible energy production
and is bounded by $\| v \|_{L^2}^2 \| \nabla u_0 \|_{L^\infty}$.
The Poincar\'e inequality gives 
$$
\| v \|^2_{L^2} \le C(\Omega) \| \nabla v \|^2_{L^2},
$$ 
for some constant $C(\Omega)$ that depends on $\Omega$. This yields 
$$
\frac{1}{2}\frac{d}{dt} \| v\|_{L^2}^2  \le  \Bigl(C(\Omega) \| \nabla u_0 \|_{L^\infty} - \nu   \Bigr)
\| v \|_{L^2}^2 .
$$
Hence, if
$$
\nu  \ge  C(\Omega) \| \nabla u_0 \|_{L^\infty} 
$$
then 
$$
\frac{d}{dt} \| v \|_{L^2}^2 \le 0.
$$
There is no energy production for any perturbation of the steady solution. This implies that the steady solution is nonlinearly stable.
By the dimensional analysis, it follows that $C(\Omega)$ behaves like $C L^2$ and $ \| \nabla u_0 \|_{L^\infty}$ like
$\| u_0 \|_{L^\infty} / L$, where $L$ is a typical length of $\Omega$.
Thus, the stationary solution $u_0$ is stable, provided that the ratio 
$$
Re = {\| u_0 \|_{L^\infty} L \over \nu } 
$$
is sufficiently small, or equivalently, the corresponding Reynolds number is small. 

\medskip

On the contrary, if the Reynolds number is large enough, the steady solution $u_0$ would become 
unstable. This second assertion is much more delicate to prove, and has been the subject
of many studies since the $19^{th}$ century, from both a mathematical and a physical
point of view. Many works focus on the particular case of a shear layer profile, with the fluid domain $\Omega$
 being the half plane $\RR_+^2$ or half space $\RR_+^3$, and the flow of the form 
$$
u_0 = \left( \begin{array}{c} U(z) \cr 0 \cr 0 \cr \end{array} \right) , \qquad z\ge 0,
$$
where $U(0) = 0$ and $\lim_{z \to + \infty} U(z)$ exists and is finite.
Let us formalize our claim. 

\medskip

{\it Claim: Any shear layer profile is linearly and nonlinearly unstable if the Reynolds number is large enough.}

\medskip

The rigorous study of this claim turns out to be very delicate.
The natural idea is to study the stability of a shear profile for the limiting system, namely
Euler equations ($Re = \infty$), and then to make a perturbative argument to deduce the stability or instability in the high Reynolds number regime. 
This perturbation however turns out to be very singular and subtle. In fact, two cases arise:

\medskip

{\it Case 1: Instability for Euler equations:} there are shear layer profiles that are unstable to Euler equations.
When viscosity is added, but remains sufficiently small, it is then conceivable that the shear layer remains 
unstable for Navier Stokes equations. This is indeed the case, but requires a delicate analysis to prove it rigorously. 
In this book, we shall provide a complete nonlinear proof of this instability, arising from that of Euler equations. 

\medskip

{\it Case 2: Stability for Euler equations:} there are shear layer profiles that are stable for Euler equations. 
For instance, all the profiles that do not have an inflection point are stable, thanks to the classical Rayleigh's stability condition. 
In this case, a naive idea is that viscosity should have an overall stabilizing effect: for a small viscosity, the shear layer
would be stable for Navier Stokes equations. Strikingly, this appears to be false, which is somehow
paradoxical. A small viscosity does destabilize the flow.

\medskip

That is, {\em all shear profiles are unstable for large Reynolds numbers.} Physics textbooks claim that
there are lower and upper marginal stability branches $\alpha_\mathrm{low}(Re), \alpha_\mathrm{up}(Re)$, 
depending on the Reynolds number, so that whenever the horizontal wave number $\alpha$ 
of a perturbation belongs to $[\alpha_\mathrm{low}(Re),\alpha_\mathrm{up}(Re)]$, 
the linearized Navier-Stokes equations about the shear profile is spectrally and linearly unstable, that is, admits 
an exponentially growing solution.

\begin{figure}[t]
\centering
\includegraphics[scale=.4]{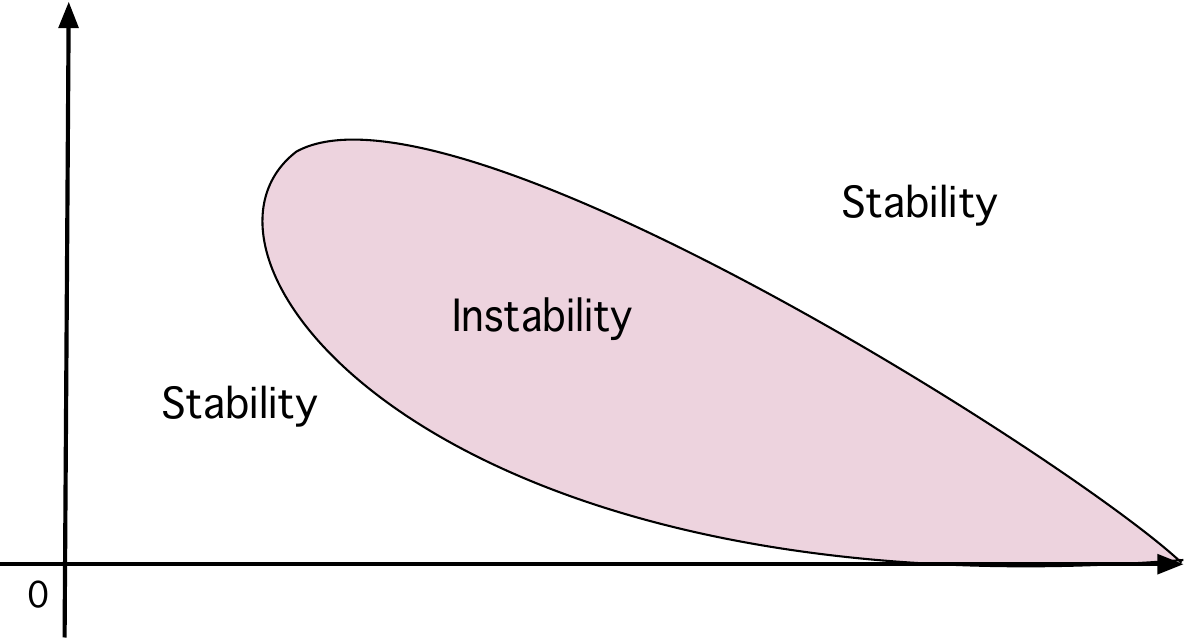}
\put(-20,1){$Re^{1/5}$}
\put(-210,117){$\alpha^2$}
\put(-213,30){$\alpha_\mathrm{low}\approx Re^{-1/4}$}
\put(-75,70){$\alpha_\mathrm{up}\approx Re^{-1/10}$}
\caption{\em Instability in the stable case.}
\label{fig-shear2}
\end{figure}

The asymptotic behavior of these branches $\alpha_\mathrm{low}$ and $\alpha_\mathrm{up}$ depends on the profile:

\begin{itemize}

\item for plane Poiseuille flow in a channel: $U(z) = 1- z^2 $ for $-1 < z < 1$, 
\begin{equation}\label{ranges-alpha0}\alpha_\mathrm{low}(R) = A_{1c}R^{-1/7}\qquad  \mbox{and}\qquad  \alpha_\mathrm{up}(R)= A_{2c} R^{-1/11}\end{equation}

\item for boundary layer profiles,
\begin{equation}\label{ranges-alphabl}\alpha_\mathrm{low}(R) = A_{1c}R^{-1/4}\qquad  \mbox{and}\qquad  \alpha_\mathrm{up}(R)= A_{2c} R^{-1/6}\end{equation}   

\item for Blasius (a particular boundary layer) profile (see figure \ref{fig-shear2}),
\begin{equation}\label{ranges-alpha}\alpha_\mathrm{low}(R) = A_{1c}R^{-1/4}\qquad  \mbox{and}\qquad  \alpha_\mathrm{up}(R)= A_{2c} R^{-1/10}.\end{equation} 

\end{itemize}

These expressions have been compared with modern numerical computations
and also with real  experiments, showing a very good agreement. 

\medskip

 Heisenberg, then Tollmien and C. C. Lin \cite{LinBook} were among the first physicists to use asymptotic expansions to study the instability
 (see also Drazin and Reid \cite{Reid} for a complete account of the physical literature on the subject). Their formal analysis has been compared with modern numerical computations
and also with experiments, showing a very good agreement; see \cite[Figure 5.5]{Reid} or Figure \ref{fig-shear2} for a sketch of the marginal stability curves. Not until recently, the complete mathematical proof of the linear stability theory was given \cite{GGN1,GGN3,GGN2}.


\section{Rayleigh and Orr Sommerfeld equations}


This section is devoted to the introduction of Rayleigh and Orr Sommerfeld equations, which
are reformulations of linearized Euler and Navier Stokes equations near a steady shear flow
$$
 U_{shear}: = \begin{pmatrix}U_s(z) \\ 0\end{pmatrix}.
$$
The first observation is that if a shear profile $U(z)$ is unstable in three dimensions, then it is also unstable
in two dimension. Physically this is known as Squire's theorem. We therefore just need to
focus on the two dimensional case. 

The 2D Euler equations linearized around a shear flow read
\begin{equation}\label{lin-shear}
\begin{aligned}
 v_t +  U_{shear} \cdot \nabla  v +  v\cdot \nabla  U_{shear} + \nabla q &= 0 \\
 \nabla \cdot  v & =0, \end{aligned}\end{equation}   
with $v = 0$ on the boundary $z=0$. Let 
$$
\omega = \nabla \times  v = \partial_x v - \partial_z u
$$
be the vorticity. Then, $\omega$ solves 
$$ 
(\partial_t + U_s \partial_x)\omega + v U''  = 0.
$$
We write this equation in term of the stream function defined through 
$$
u = \partial_z \psi \quad \hbox{and} \quad  v = - \partial_x \psi.
$$
 We note that 
 $$
 \omega = \partial_z u - \partial_x v = \Delta \psi.
 $$
  The vorticity equation then reads
$$  
(\partial_t + U_s \partial_x) \Delta \psi - U_s'' \partial_x \psi   = 0.
$$
We then take the Fourier transform in the tangential variable $x$ and the Fourier-Laplace
transform in time and introduce
 $$
 \psi = e^{i\alpha (x-c t)} \phi(z),
 $$
  in which $\alpha$ is a real-valued positive wave number (Fourier variable in $x$),
  $$
  \lambda = -i\alpha c
  $$
  is the Laplace variable in time (and so, $c$ is a complex number), 
  and $\phi(z)$ is a complex-valued function.  The link between $\lambda$ and $c$ is traditional. 

Putting $\psi = e^{i\alpha (x-c t)} \phi(z)$ into the vorticity equation yields the well-known Rayleigh equations, which are just linearized Euler equations in the stream function formulation, and after 
an horizontal Fourier transform and a time Fourier Laplace transform
\begin{equation}\label{Rayleigh}
\left \{ 
\begin{aligned}
(U_s-c) (\partial_z^2 - \alpha^2) \phi - U_s'' \phi &= 0
\\
\phi_{\vert_{z=0}}  = 0, \qquad \lim_{z\to \infty} \phi(z) &=0.
\end{aligned}
\right.
\end{equation} 
Here, the boundary conditions is exactly:  $v =0$ on the boundary. 

Note that in three dimensional space, we have to take Fourier transform in two horizontal
variables instead of only one. Up to a change in horizontal variables, the analysis is similar.
This remark is known as Squire's theorem in the physical literature.
  
Its important to note that, up to a multiplication by $- i \alpha$, the spectrum of linearized Euler
equations and Rayleigh equations are identical. Rayleigh is much easier to deal with than Euler
linearized equation since the divergence free condition is built in.

Now if we consider (\ref{lin-shear}) with a source term 
$$
f(t,x,y) = e^{i \alpha (x - c t) } F(z),
$$
where $F(z)$ is a two components vector,
we are led to 
\beq \label{Rayleigh3}
(U_s - c) (\partial_z^2 - \alpha^2) \phi - U_s'' \phi = {S \over i \alpha}
\eeq
where 
$$
S = \partial_z F_1 - i \alpha F_2 .
$$
Note that (\ref{Rayleigh3}) is the resolvent equation of linearized Euler equations with forcing term $f$.

Rayleigh equation is a second order differential equation, with boundary conditions at $z = 0$ and at infinity.
It is singular when $U_s - c$ vanishes, or at least is very small. This remark will play a central part
in the study of Euler stable profiles.

\medskip

If we consider Navier Stokes equations instead of Euler equations, we get the so called
Orr Sommerfeld equations, namely
\begin{equation}\label{OS1}
\left \{ 
\begin{aligned}
(U-c) (\partial_z^2 - \alpha^2) \phi - U'' \phi &= \epsilon (\partial_z^2 - \alpha^2)^2 \phi , \qquad \epsilon = \frac{\nu}{i\alpha} 
\\
\phi_{\vert_{z=0}} = \phi'_{\vert_{z=0}} &= 0, \qquad \lim_{z\to \infty} \phi(z) =0.
\end{aligned}
\right.
\end{equation} 
These are fourth order ordinary differential equations. As $\nu$ goes to $0$, the "viscous" term
$\epsilon (\partial_z^2 - \alpha^2)^2 \phi$ disappears, and this fourth order equation degenerates in
a second order equation, namely Rayleigh equations.
This singular limit is classical if Rayleigh is not singular, namely if $U_s - c$ is bounded away from $0$.
It Rayleigh is itself singular, namely if $U_s - c$ vanishes at some point $z_c$, then near $z_c$ the
situation is dramatic, since both the fourth order and the second order terms vanish at $z_c$. The Orr
Sommerfeld equation then reduces to $- U''\phi =0$, which is a very severe degeneracy.
Such a $z_c$ is called a critical layer. The study of Orr Sommerfeld in the critical layer involves Airy functions.


\section{Link with Prandtl equations}


The question of the stability of shear flows is highly connected to the so called Prandtl equation.
More precisely, the study of the inviscid limit with the Dirichlet boundary condition is a long standing problem,
with a rich physical and mathematical history. The leading question is the convergence of Navier-Stokes solutions in 
energy norm in the inviscid limit: 

\medskip

{\it For $\nu >0$, let $u^{\nu}$ be a sequence of (smooth) solutions of
the Navier-Stokes problem (\ref{NavierStokes1})-(\ref{NavierStokes4})
on a given interval $[0,T]$. Does
$u^\nu$ converge in $L^\infty([0,T],L^2(\Omega))$, as $\nu\to 0$, to a vector field $u$, which is a solution of the Euler problem (\ref{Euler1})-(\ref{Euler4}) ? }

\medskip

The main problem in this inviscid limit is the change in the boundary condition. As the viscosity 
vanishes, the Laplacian term $- \nu \Delta u$ disappears, leading to a change in the number
of boundary conditions. Instead of $u = 0$, in the limit we only have $u\cdot n = 0$ on the boundary. 
In particular, the tangential velocity may be non zero in the limit. 
This change in the number of boundary conditions leads to a boundary layer type behavior
near $\partial \Omega$: the velocity $u$ will rapidly change near the boundary in order to``recover", 
or rather ``correct'', the missing boundary conditions. 
It turns out that the Dirichlet boundary condition \eqref{NavierStokes4}, 
which is the main focus of this book, is the most difficult boundary condition to study the inviscid limit problem. 
There are other interesting physical boundary conditions such as Navier slip boundary conditions. 
These latter conditions are in fact easier to handle in the inviscid limit,
 for the reason that they yield certain control on the vorticity near the boundary, see for instance 
 \cite{IftimieSueur, MasmoudiRousset1}.

\medskip
\medskip

Up to now, this question appears to be out of reach in the case when $\Omega$ has a boundary. Its answer is probably linked to some extent to the understanding of wall turbulence, a mathematically widely open subject. There is a beautiful criterium by Kato \cite{Kato84} which asserts that the convergence fails in the inviscid limit if and only if the dissipation of energy in a strip of size $\nu$ near the boundary does not go to zero. For similar conditional results, see \cite{BardosTiti, ConstVV18, Kelliher17}, and the references therein.  Instead of $L^2$, we will restrict ourself to $L^\infty$ norms and focus on the following question:

\medskip

{\it Can we describe the limiting behavior of $u^\nu$ in $L^\infty([0,T],\Omega)$ in the inviscid limit ? }

\medskip

This question is stronger, since it requires to describe the behavior of $u^\nu$ very close to the wall, which may lead to a cascade of boundary layers of thinner and thinner sizes starting from the classical Prandtl layer of size $\sqrt \nu$. In this book we will give preliminary results in that direction.

We end the introduction with a brief mathematical bibliography on the study of Prandtl boundary layers. The existence and uniqueness of solutions to the Prandtl equations 
have been constructed for monotonic data by Oleinik  \cite{Ole} in the sixties. There are also recent reconstructions \cite{Alex, MW} of Oleinik's solutions via a more direct energy method. For data with analytic or Gevrey regularity, the well-posedness of the Prandtl equations is established in \cite{SammartinoCaflisch1, GVMasmoudi15, DGV19}, among others. In the case of non-monotonic data with Sobolev regularity, the Prandtl boundary layer equations are known to be ill-posed (\cite{GVDormy,GVN1,GN1}). 

On the other hand, the justification of Prandtl's boundary layer Ansatz for the behavior of solutions to the Navier Stokes equations has been justified for analytic data in a pioneered work by Caflisch and Sammartino \cite{SammartinoCaflisch1,SammartinoCaflisch2}. See also \cite{Mae,2N}. The stability of shear flows under perturbations with Gevrey regularity is recently proved in \cite{GVM}.

However these positive results hide a strong instability occurring at high spatial frequencies. For some profiles, instabilities
with horizontal wave numbers of order $\nu^{-1/2}$ grow like $\exp(C t / \sqrt{\nu})$. Within an analytic framework, these
instabilities are initially of order $\exp(-D / \sqrt{\nu})$ and grow like $\exp( (Ct - D) /\sqrt{\nu})$. They remain negligible in bounded
time (as long as $t < D/ 2 C$ for instance).

Within Sobolev spaces, these instabilities are predominant. Indeed, the first author proved in \cite{Grenier00CPAM} that Prandtl's asymptotic expansion for Sobolev data near unstable profiles is false,  up to a remainder of order $\nu^{1/4}$ in $L^\infty$ norm. In this book, we shall present our recent instability results for data with Sobolev regularity, which in particular prove that the Prandtl's boundary layer Ansatz is false and the boundary layer asymptotic expansions for stable monotone profiles are invalid.


\section{Structure of the book}


The aim of this book is to provide a comprehensive presentation to recent advances on boundary layers stability. 
It targets graduate students 
and researchers in mathematical fluid dynamics and only assumes that the readers have a basic knowledge on ordinary differential equations
and complex analysis. No prerequisites are required in fluid mechanics, excepted a basic knowledge on Navier Stokes and Euler equations,
including Leray's theorem.

\medskip

Part I is devoted to the presentation of classical results and methods:
Green functions techniques, resolvent techniques, analytic functions.
Part II focuses on the linear analysis, first of Rayleigh equations, then of Orr Sommerfeld equations. This enables the construction
of Green functions for Orr Sommerfeld, and then the construction of the resolvent of linearized Navier Stokes equations.
Part III details
the construction of approximate solutions for the complete nonlinear problem and nonlinear instability results.


\part{Preliminaries}



\chapter{Estimates using resolvent}\label{chapter-linear}

The aim of this chapter is to give a short introduction to the use of the classical Laplace transform  
to construct solutions for the following simple equations: 
linear systems in finite dimension space,  heat equation and second order linear parabolic 
equations. These techniques will be used for Rayleigh and
Orr Sommerfeld equations in forthcoming chapters.


\section{Finite dimension case}


Let us consider in this section the following simple linear ordinary differential system of $N$ equations
\beq \label{matrice1}
\dot x = A x,
\eeq
with initial condition 
\beq \label{matrice2}
x(0) = x_0.
\eeq
Here, $A$ is a given fixed $N \times N$ constant matrix, $x(t)$ is a vector function in $\RR^N$,
and $x_0$ a given vector.
The solution to \eqref{matrice1}-\eqref{matrice2} is simply given by 
$$
x(t) = \exp( t A) x_0 .
$$
In the case when $A$ is diagonalizable, we can write 
$$
A = P D P^{-1}
$$
where $D$ is diagonal, with coefficients
$(\lambda_i)_{1 \le i \le N}$, and $P$ is a change of basis. Then, there holds 
\beq \label{matrice2a}
x(t) = P \exp(t D) P^{-1} x_0,
\eeq
in which $ \exp(t D) $ is diagonal with coefficients $(e^{\lambda_i t})_{1\le i\le N}$. 

In the case when $A$ is no longer diagonalizable, then polynomials in time will appear in front of the exponential term. We will not
detail this case here.
Of course, this approach is limited to matrices and finite dimensional systems. In order to deal
with partial differential equations, we will have to develop another approach, based on the Laplace / Fourier
transform.

To be more specific, let us define 
$$
y(t) = x(t)  1_{t \ge 0}.
$$
Then $y(t)$ satisfies in the sense of distributions
$$
\dot y = A y + x_0 \delta_0
$$
where $\delta_0$ is the Dirac mass centered at $t = 0$. Note that $y$ may have an exponential growth
in large times. To turn this potential growth into a decay we multiply $y$ by $e^{ - M t}$, 
for some sufficiently large constant $M>0$, and introduce
$$
z(t) = y(t) e^{- M t} = x(t) e^{- M t} 1_{t \ge 0} .
$$
Then 
\begin{equation}\label{intro-zODE}
\dot z(t) = (A - M \mathrm{Id}) z(t) + x_0 \delta_0 ,
\end{equation}
where $\mathrm{Id}$ denotes the identity matrix. Occasionally, we simply drop the notation $\mathrm{Id}$. 
Now $z(t)$ has an exponential decay as $t$ goes to $+ \infty$, and vanishes for negative $t$.

To solve \eqref{intro-zODE}, we may introduce its Fourier
transform $\hat z(\xi)$, yielding
$$
i \xi \hat z(\xi) = (A - M \mathrm{Id}) \hat z(\xi) + x_0.
$$
Therefore,
$$
\hat z(\xi) = - ( A - M \mathrm{Id} - i \xi)^{-1} x_0 .
$$
Now inverting the Fourier transform gives
$$
z(t) =-  {1 \over 2 \pi} \int_\RR e^{i \xi t} (A - M \mathrm{Id} - i \xi )^{-1} x_0  d\xi ,
$$
which solves \eqref{intro-zODE}. 
Using analyticity in $\xi$, we can shift the contour integration, yielding 
$$
x(t) =-  {1 \over 2 \pi} \int_{\Im \xi = i M} e^{i t \xi}   (A - i \xi )^{-1} x_0  d\xi .
$$
Following the traditional notation, we introduce the Laplace transform variable $\lambda = i \xi$, and thus 
\begin{equation}\label{soln-ODE-Laplace}
x(t) = {1 \over 2 i \pi} \int_{\Gamma} e^{ \lambda t} (\lambda  - A)^{-1} x_0 d\lambda 
\end{equation}
where $\Gamma$ is a contour "on the right" of the spectrum of $A$, as depicted in Figure \ref{fig-contourODE1}. 
This defines a solution to the ODE problem \eqref{matrice1}-\eqref{matrice2}.

Let us introduce the resolvent operator 
\beq \label{matrice3}
R(\lambda) = (\lambda - A)^{-1} .
\eeq
Then, the solution $x(t)$ is represented by 
\beq \label{resolvantgeneral}
x(t) = {1 \over 2 i \pi} \int_{\Gamma} e^{\lambda t}R(\lambda) x_0 d \lambda ,
\eeq
in which the function $R(\lambda) x_0$ is simply the solution $\phi$ of the resolvent equation
$$
\lambda \phi = A \phi + x_0
$$
(which is of course the Laplace transform of (\ref{matrice1}) with source term $x_0$).

\begin{figure}[t]
\centering
\includegraphics[scale=.5]{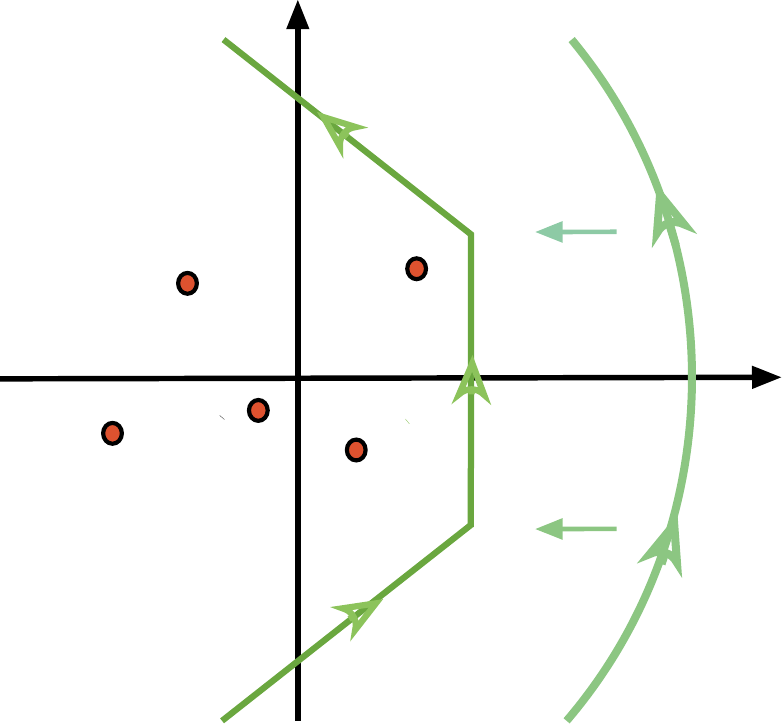}
\put(-35,1){$\Gamma$}
\put(-100,10){$\Gamma_1$}
\put(-70,70){$\Gamma_2$}
\put(-100,150){$\Gamma_3$}
\caption{\em Illustrated are possible point spectrum of $A$ and a contour decomposition of $\Gamma$ lying on the right of the spectrum.}
\label{fig-contourODE1}
\end{figure}

Now, in order to obtain bounds on the solution $x(t)$ in term of time growth, we may decompose the contour of integration $\Gamma$ to the left, as long as $\Gamma$ does not cross the spectrum of $A$; see
Figure \ref{fig-contourODE1}. For instance, we can deform $\Gamma$ as follows:  
$$\Gamma = \Gamma_1 \cup \Gamma_2 \cup \Gamma_3$$
where $\Gamma_2 = P + i [-B,B]$,
$\Gamma_1 = (1+i) \rit_- + P - i B$,
and $\Gamma_3 = (-1 + i) \rit_+ + P + iB$, for some constant $P$ and sufficiently large constant $B$ so that the spectrum of $A$ remains on the left. 

The integral on $\Gamma_2$ is simply bounded by $Ce^{P t}$.
The integrals on $\Gamma_1$ and $\Gamma_3$ are bounded by
$$
Ce^{P t} \int_{\rit_+} {e^{- s t} \over 1+ s} ds \le C' e^{P t}.
$$
Hence, for any constant $P$ that is greater than the maximum of real part of the spectrum of $A$, there exists a constant $C_P$ so that  
$$
\| x(t) \| \le C_Pe^{P t} , \qquad \forall~ t\ge 0. 
$$
To get optimal bounds, and even an explicit expression for the solution, we need to detail the behavior of the resolvent operator $R(\lambda)$
$$
R(\lambda) := \frac{1}{\det(\lambda  - A) }{\mathrm{com}( \lambda - A)^{tr} }
$$
in which $\mathrm{com}(M)$ denotes the comatrix of a matrix $M$. Note that $R(\lambda)$ is an holomorphic function, 
singular on $Sp(A)$, and
which has a Laurent's asymptotic expansion at $\lambda \in Sp(A)$. By the residue theorem, the integral
(\ref{resolvantgeneral}) is the sum of the residues at the poles of $R(\lambda)$. These residues
are sum of products of projectors with polynomials  in time (of degree the multiplicity
of the eigenvalue minus one).
We thereby recover (\ref{matrice2a}) in the case of simple eigenvalues.


\section{Heat equation}


%
%
%
%

Let us now detail how to solve the classical heat equation
\beq \label{heatc1}
\partial_t u - \nu \partial_z^2 u = 0,
\eeq
with initial condition $u_0(z)$ on the whole line $z \in \rit$, using a similar approach. 
Using the Laplace transform approach, we have 
\beq \label{heatc}
u(t) = {1 \over 2 i\pi} \int_{\Gamma} e^{\lambda t} u_\lambda d\lambda
\eeq
in which $\Gamma$ is a contour of integration lying on the right of the spectrum of $\nu \partial_z^2$, say in $L^2$, and $u_\lambda$ is the solution of the resolvent equation
\beq \label{heatc2}
\lambda u_\lambda - \nu \partial_z^2 u_\lambda = u_0 .
\eeq
Next, we solve (\ref{heatc2}) via the Green function approach. To this end, we first construct increasing and decreasing 
solutions $\psi_{\pm}$ of the homogenous equation
\beq \label{heatc3}
\lambda \psi_{\pm} = \nu \partial_z^2 \psi_{\pm},
\eeq
which are
\beq \label{heatc4}
\psi_{\pm}(z) = e^{\mp z \sqrt{\lambda / \nu}} .
\eeq
Note that $\psi_\pm$ is holomorphic in $z \in \cit$, and locally in $\lambda \ne 0$, with a "branching point" at $\lambda = 0$.
However if the contour $\Gamma$ does not cross $\rit_-$, we can choose $\sqrt{\lambda / \nu}$ (and hence $\psi_\pm$) in an holomorphic
way such that its real part is strictly positive. 

On $\rit$, the Green function of (\ref{heatc2}) is simply
$$
G_\lambda(x,z) = - { 1\over 2 \sqrt{\lambda \nu}} e^{- |x-z|\sqrt{\lambda / \nu}} ,
$$
and hence, the solution of (\ref{heatc2}) is given, for real values of $z$, by
\beq \label{heatc5}
u_\lambda(z) = \int_\rit G_\lambda(x,z) u_0(x) dx .
\eeq
Now to explicitly compute the value of $u(t)$ via \eqref{heatc}, we choose a contour $\Gamma$ which passes on the right of $0$ and is asymptotic to the half lines 
$L_\pm = \rit_+ e^{\pm 2 i \theta}$ for some $\pi / 2 < 2 \theta < \pi$.
In this case, the integral of (\ref{heatc}) converges and gives the solution of (\ref{heatc1}).


Alternatively, we introduce the temporal Green function $G(t,x,z)$, defined by 
\begin{equation}\label{def-tempG-heat} 
G(t,x,z) = \frac{1}{2i \pi} \int_\Gamma e^{\lambda t} G_\lambda(x,z) \; d\lambda,
\end{equation}
and derive its pointwise bounds. The solution to the heat equation is then the convolution of $G(t,x,z)$
 with the initial data $u_0(z)$. To estimate the integral \eqref{def-tempG-heat}, we bound the integrand 
by
$$ |e^{\lambda t} e^{-|x-z|\sqrt{\lambda/\nu} } | = e^{\Re \lambda t - |x-z|\Re \sqrt{\lambda /\nu}} .$$
As the imaginary part of $\sqrt{\lambda/\nu}$ plays no role in the modulus estimate, we parametrize $\Gamma$ through
the imaginary part of $\sqrt{\lambda/\nu}$,  introduce
$$
\sqrt{\lambda/\nu} = a + ik,
$$
and parametrize $\Gamma$ through $k$.
Note that $e^{\Re \lambda t - |x-z|\Re \sqrt{\lambda /\nu}}$ is minimal (for each fixed $k$) when 
$$
a = \frac{|x-z|}{2\nu t}.
$$
We therefore fix $a$ to this value. We then have
$$
{\lambda \over \nu} = (a + i k)^2 ,
$$
which leads to the following choice of $\Gamma$
$$
\Gamma: = \Big\{ \lambda  =  a^2 \nu - k^2 \nu + 2i a k \nu, \quad k\in \RR\Big\}.
$$
Then, we compute 
$$
\begin{aligned}
|G(t,x,z)| 
&\le \frac{1}{4\pi \sqrt\nu} \int_\Gamma e^{\Re \lambda t - |x-z|\Re \sqrt{\lambda /\nu}} \;  \frac{|d\lambda|}{|\sqrt{\lambda}|}
\\&\le \frac{1}{2\pi} \int_\RR e^{a^2 \nu t - |x-z| a} e^{-k^2 \nu t}\;  dk 
\\&\le \frac{1}{\sqrt{4\pi \nu t}} e^{-\frac{|x-z|^2}{4\nu t}}. 
\end{aligned}
$$
The temporal Green function of the heat problem \eqref{heatc1} behaves exactly as the classical Gaussian kernel. 

%
%
%
%
%


\section{Second order linear parabolic equations}


Let us turn to systems of the form
\beq \label{para1}
\partial_t u - \nu \Delta u = A(x) u
\eeq
where $A(x)$ is a given smooth matrix, with initial data $u_0$.
Formula (\ref{resolvantgeneral}) can be extended to this case and gives
\beq \label{para2}
u(t) = \int_\Gamma e^{\tau t} R(\tau) u_0 \, d\tau
\eeq
where 
\beq \label{para3}
R(\tau) = (i \tau \mathrm{Id} - \nu \Delta - A)^{-1}
\eeq
is the resolvent, and where $\Gamma$ is a contour on the right of the singularities
of $R$.
To bound $u(t)$ we need to bound $R(\tau)$, namely to solve
\beq \label{para4}
i \tau v - \nu \Delta v - A(x) v = u_0(x),
\eeq
which is a system of linear ordinary  differential equations.


\subsubsection*{Green function}


To solve (\ref{para4}) we introduce the Green function $G(x,y) u$,
solution of
\beq \label{para5}
i \tau G - \nu \Delta G - A(x) G = \delta_y u,
\eeq
where $u$ is a fixed vector. Note that we skipped $\tau$ is the notation of the Green function.
 The solution of (\ref{para4}) is then given by 
\beq \label{para6}
v(x) = \int_\rit G(x,y) u_0(y) dy .
\eeq
Next to define the Green function $G$ we introduce 
 $\psi_j^-$ and $\psi_j^+$
with $1 \le j \le N$,  independent solutions of 
$$
i \tau v - \nu \Delta v - A(x) v = 0
$$
which go to $0$ as $x \to - \infty$ (for $\psi_j^-$) or  $x \to + \infty$
(for $\psi_j^+$).
Then
$$
G(x,y) u = \sum_{1 \le j \le N} \alpha_j(y) \psi_j^-(x)
$$
for $x < y$ and for some constants $\alpha_j(y)$ depending on $u$ and
$$
G(x,y) u = \sum_{1 \le j \le N} \beta_j(y) \psi_j^+(x)
$$
for $x > y$ and for some constant $\beta_j(y)$.
Note that $G(x,y) u$ is continuous at $y$ and that its derivative has a jump $u$ at $y$.
Hence
$$
 \sum_{1 \le j \le N} \alpha_j \psi_j^+ =  \sum_{1 \le j \le N} \beta_j \psi_j^-
 $$
and
$$
 \sum_{1 \le j \le N} \alpha_j \partial_x \psi_j^+ =  \sum_{1 \le j \le N} \beta_j \partial_x \psi_j^-
 - \nu^{-1} u.
 $$
Let
$$
M = \left( \begin{array}{cc}
\psi_j^+ & \psi_j^- \cr
\partial_x \psi_j^+ & \partial_x \psi_j^- \cr
\end{array} \right) .
$$
Note that $M$ is simply the jacobian matrix of the independent solutions $\psi_j^+$, $\psi_j^-$,
which depends on $y$ and $\tau$.
Then
\beq \label{para12}
\left( \begin{array}{c} \alpha_j \cr - \beta_j \end{array} \right) 
= M^{-1} 
\left( \begin{array}{c} 0 \cr - \nu^{-1} u \end{array} \right)  .
\eeq
We  introduce the Evans function 
\beq \label{para13}
{\cal E} = \| \nu^{-1}M^{-1} \|.
\eeq
Then if $| \tau | \le \alpha$ we get 
\beq \label{para14}
\sum_{1 \le j \le N} | \alpha_j | + | \beta_j |  \le  {\cal E}^{-1} | u | 
\eeq
for some constant $C$.


\subsubsection*{Analyticity}


Let us study $\psi_j^{\pm}$.
Let us assume that $A(x)$ is holomorphic in the strip $S = \{ | \Im x | \le \sigma_0\} $ for some positive $\sigma_0$.
For bounded $\tau$, the functions $\psi_j^{\pm}$ are solutions of a holomorphic equation on $S$. As a consequence they
are defined and holomorphic on $S$. It is also possible to show that $\psi_j^{\pm}(z)$ go to zero as $\Re z$ go to
$\pm \infty$.

For large $\tau$,  $\psi_j^{\pm}$  behave like the solutions $\psi$ of
$$
\tau \psi - \nu \partial_x^2 \psi = 0,
$$
namely like
$$
\psi_\pm(z) \sim {1 \over 2} \sqrt{\nu \over \tau} e^{- \pm z \sqrt{\tau / \nu}} ,
$$
which is similar to the heat equation case.
The bounds on the Green function for large $\tau$ are therefore the same as the bounds for the heat equation.


\subsubsection*{Bounds on the resolvant}


To bound the resolvant $R(\tau)$ we need
\begin{itemize}
\item a bound for bounded $\tau$, given by (\ref{para14})
\item another bound for unbounded $\tau$, of the form $\tau = A + e^{i \theta} t$
with $t > 0$, for $t$ large. In this case $A$ can be treated as a small perturbation of
$i \tau \mathrm{Id} - \nu \Delta$, which lead to bounds similar to the case of the heat equation.

\end{itemize}
Combining these two estimates we can then bound $R(\tau)$ and thus $u(t)$.
We will not detail this point here.

In (\ref{para2}) the contour $\Gamma$ can be moved to another contour $\Gamma'$ on its
left, leaving some eigenvalues $\lambda_i$ with $1 \le i \le p$ on its right.
Then, if all the eigenvalues are simple,
\beq \label{para2bis}
u(t) = \int_{\Gamma'} e^{\tau t} R(\tau) u_0 d\tau
+ \sum_i P_i u_0 e^{\lambda_i t},
\eeq
where $P_i$ is the projector on the eigenspace $E_i$ associated with the eigenvalue $\lambda_i$.

By carefully choosing  the contour $\Gamma'$ we can get an asymptotic expansion on $u(t)$.
In particular if there exists some unstable eigenvalue $\lambda_i$ then $u(t)$ generically
behaves like $\exp(\lambda_i t)$. Note that the eigenvalues are simply the singularities of
${\cal E}$, namely the zeros of the Evans function $\det(M)$, which is explicit in the functions 
$\psi_j^{\pm}$.

\chapter{From linear to nonlinear instability }\label{chapter-ode}

 This chapter is devoted to simple remarks on ordinary differential equations in order to clarify basic ideas on the link between
linear and nonlinear instabilities.


\section{From linear to nonlinear instability}


Let us recall in this section how we can prove nonlinear instability, assuming linear instability in the case of ordinary differential 
equations. Let us consider the following system
\beq \label{ooode}
\partial_t \phi = A \phi + Q(\phi,\phi)
\eeq
where $\phi$ is a vector, $A$ a matrix and $Q$ a quadratic term. Let us assume that $A$ is spectrally unstable, namely
that there exists an eigenvalue $\lambda$ and a corresponding eigenvector $v_0$ with $\Re \lambda > 0$. For simplicity we
assume that $\lambda$ is a single eigenvalue. We want to prove nonlinear instability, and more precisely

\begin{theo}
There exists a constant $\sigma_0$ such that for any arbitrarly small $\eps > 0$, there exists a solution $\phi$
to (\ref{ooode}) satisfying the following properties:
$$
| \phi(0) | \le \eps,
$$
$$
| \phi(T_\eps) | \ge \sigma_0
$$
for some positive time $T_\eps$. 
\end{theo}

The proof of this nonlinear instability result relies on two steps: first the construction
of an accurate approximate solution, then an estimate between the approximate and the true solution.


\subsection{Construction of an approximate solution}


We look for an approximate solution of the form
$$
\phi_{app} = \sum_{i=1}^N \phi_i
$$
where $N$ will be chosen later. We start with 
$$
\phi_1 = \eps v_0 e^{\lambda t} ,
$$
where $\eps$ is a small parameter. Note that $\phi_0$ is bounded by
$$
| \phi_1 | \le C \eps e^{\Re \lambda t} .
$$
For $i \ge 2$, $\phi_i$ is constructed by iteration through
$$
\partial_t \phi_i = A \phi_i + \sum_{j+k = i} Q(\phi_j,\phi_{k}),
$$
with $\phi_i(0) = 0$ (where $\phi_j =0$ for $j\le 0$).
We will prove by iteration that
\beq \label{oodeiter}
| \phi_j | \le C_j \eps^j e^{j \Re \lambda t} .
\eeq
This bound is true for $j = 1$. If it is true for $j < i$ then
$$
\phi_i(t) = \sum_{j+k = i} \int_0^t e^{A (t - \tau)} Q(\phi_j,\phi_{k}) d \tau,
$$
which leads to
$$
| \phi_i(t) | \le C_i \eps^i \sum_{j+k = i} \int_0^t  e^{\Re \lambda (t - \tau)}  e^{i \Re \lambda \tau} d\tau
$$
which gives (\ref{oodeiter}).

Now, $\phi_{app}$ is an approximate solution in the sense that 
$$
\partial_t \phi_{app} = A \phi_{app} + Q(\phi_{app},\phi_{app}) + R_{app},
$$
where $R_{app}$ is an error term. A short computation shows that
$$
| R_{app} | \le C_{N+1} \eps^{N+1} e^{(N+1) \Re \lambda t} .
$$


\subsection{Stability}


The next step is to use a stability estimate. Let $\phi$ be the solution of (\ref{ooode}) with initial data $\phi_0(0)$.
Let 
$$
\theta = \phi - \phi_{app}
$$
be the difference between the "true" solution and the approximate one. We have
$$
\partial_t \theta = A \theta + Q(\phi_{app},\theta) + Q(\theta,\phi_{app}) + Q(\theta,\theta) - R_{app}.
$$
Let us work on the largest time $T_0$ such that on $0 \le t \le T_0$, $| \phi_{app}(t) | \le 1$ and
$| \theta(t) | \le 1$. Let $0 \le t \le T_0$. We take the scalar product of the previous equation with $\theta$. This leads to
$$
\partial_t | \theta |^2 \le C_0 | \theta |^2 + | R_{app} |^2 
$$
since $(A \theta,\theta)$, $(Q(\phi_{app},\theta),\theta)$, $(Q(\theta,\phi_{app}),\theta)$, $(Q(\theta,\theta),\theta)$
are all bounded by $C_0 | \theta |^2$ for some constant $C_0$. Therefore
$$
| \theta |^2(t) \le \eps^{2(N+1)} \int_0^t e^{C_0 (t - \tau)} e^{2 (N+1) \Re \lambda \tau} d\tau
\le C \eps^{2 (N+1)} e^{2 (N+1) \Re \lambda t},
$$
provided 
$$
2 (N+1) \Re \lambda > C_0,
$$
namely, provided $N$ is large enough.

We now define
$$
T_1 = - {\log \eps \over \Re \lambda } - \sigma
$$
where $\sigma$ will be chosen later. Before $T_1$, all the $\phi_i$ are bounded 
A direct computation shows that provided $\sigma$ is large enough, $T_1 < T_0$. 
Then at $T_1$, 
$$
| \phi(T_1) | \ge | \phi_{app}(T_1) | - | \theta(T_1) | 
\ge e^{- \sigma} - \sum_{i=1}^{N+1} C_i e^{-i \sigma} 
\ge {e^{- \sigma} \over 2} 
$$
provided $\sigma$ is small enough. This ends the proof of the Theorem, chosing $\sigma_0 = e^{- \sigma} / 2$.


\subsection{Link with Navier Stokes equations}


The previous theorem is a toy model for the stability of viscous boundary layers. In the case of Navier Stokes equations, $Q$ is the nonlinear
transport term and $A$ the linearized Navier Stokes equations near a boundary layer profile.
The first difficulty is to prove the existence of an unstable mode for linearized Navier Stokes equation, 
which is the equivalent of the construction of $\lambda$ and $v_0$. This is obtained through a detailed spectral 
analysis of  Orr Sommerfeld equations.
 The second difficulty is to have good estimates on $e^{A t}$, which is done through the construction of the corresponding
 Green function. The last step, namely the energy estimate, is the standard $L^2$ estimate on Navier Stokes.
 
 However, because of large gradients in the sublayer of the instability, the construction ends 
 before the instability reaches $O(1)$. This approach stops when the instability reaches a size $O(\nu^{1/4})$,
 as is detailed in \cite{Grenier00CPAM}. To go up to $O(1)$ we will have to design a new method, developed
 in the next chapter.


\section{Small unstable eigenvalues}


Let us recall that we face two different cases. A first case arises when the shear layer profile is unstable for Rayleigh equations. 
In this case the largest unstable eigenvalue for Navier Stokes equation is of order $O(1)$. In the second case, 
when the shear layer
is stable for Rayleigh equation, $\Re \lambda$ is small, of order $\nu^{1/4}$. This latest case can not be handled by
the techniques developed in the previous section. We will discuss it, always on a toy model case, in this section.

Let us focus in this section on systems of the form
\beq \label{oode}
\partial_t \phi = A_\eps \phi + Q(\phi,\phi)
\eeq
where $\phi$ is a vector, $A_\eps$ is a diagonal matrix and depends on a small parameter $\eps > 0$, and $Q$ is quadratic.
We assume that $A_\eps$ has only one eigenvalue with positive real part, denoted by $\lambda(\eps)$. We assume
that 
$$
\Re \lambda(\eps) = \eps,
$$
which is always possible, up to a renumbering of $\eps$. Let $v_0(\eps)$ be the corresponding eigenvalue. Then
$$
\phi_1(t) = \eps^N v_0(\eps) e^{\lambda(\eps) t}
$$ 
is an exponential growing solution to the linearized equation, which initially is of order $O(\eps^N)$.
Let us construct an approximate solution of (\ref{oode}). The second term in the expansion is defined by
$$
\partial_t \phi_2 = A_\eps \phi_2 + Q(\eps^N v_0 e^{ \lambda(\eps) t},\eps^N v_0 e^{\lambda(\eps) t}),
$$ 
which leads to
$$
\phi_2(t) = \int_0^t e^{A_\eps (t - \tau)} Q \Bigl( \eps^N v_0 e^{\lambda(\eps) \tau},  \eps^N v_0 e^{\lambda(\eps) \tau} \Bigr) d\tau,
$$
assuming $\phi_2(0) = 0$. Note that
$$
| Q \Bigl( \eps^N v_0 e^{\lambda(\eps) \tau},\eps^N v_0 e^{\lambda(\eps) t} \Bigr)  | \le C \eps^{2N} e^{2 \Re \lambda \tau},
$$
hence
$$
| \phi_2(t) | \le C \eps^{2N} \int_0^t  e^{\Re \lambda (t - \tau) } e^{2 \Re \lambda \tau} d\tau \le C 
{C \eps^{2 N} \over \Re \lambda } e^{2 \Re \lambda t}.
$$
Therefore $\phi_2(t)$ is of size $\eps^{2N-1} e^{2 \Re \lambda t}$, namely much larger than in the previous section, where it was of order
$\eps^{2N} e^{2 \Re \lambda t}$. Iterating the process we 
can construct an approximate solution of the form 
$$
\phi \sim \sum_n \phi_n,
$$
where
$$
| \phi_n | \le {C_n \eps^{n(N-1) +1}} e^{n \Re \lambda t} .
$$
The first two terms of this series are of the same order when $e^{\Re \lambda t}$ is of order $\eps^{1-N}$. 
Then all the terms of the series are of the same order, namely of order $O(\eps)$. 
It is therefore not possible to construct a solution of (\ref{oode}) larger than $O(\eps)$ with this method.

\medskip

To understand this point, let us take the simplest possible example, namely the scalar example where 
$A_\eps = \eps$ and $Q(u) = \alpha u^2$.
We assume that $\phi(0) > 0$.
Then (\ref{oode}) is simply
\beq \label{oode2}
\partial_t \phi = \eps \phi + \alpha \phi^2 
\eeq
for some constant $\alpha$. The qualitative evolution of $\phi$ depends on the sign of $\alpha$. If $\alpha < 0$, then $\phi(t)$ converges
to $\alpha^{-1} \eps$ as $t$ goes to infinity. If on the contrary $\alpha > 0$ then $\phi(t)$ blows up in finite time. Qualitatively,
$\phi(t)$ increases exponentially until it reaches $O(\eps)$ where the quadratic term becomes important and speeds up the
growth.

\medskip

Therefore in the case of small eigenvalues, the nonlinear term plays a crucial role. 
The situation is in fact very close to a bifurcation.

Let us go back to Navier Stokes equations. When the shear flow is linearly stable for Euler equation, it will be shown that it is linearly
unstable for Navier Stokes equations, but with a very slow instability, of order $O(\nu^{1/4})$, which plays the
role of the $\eps$ of this paragraph. We then suffer from similar limitations: the instability can only be proven
up to a size $O(\nu^{1/4})$ in $L^\infty$ norm.

\chapter{Analyticity and generator functions}\label{chapter-tech}

This chapter is devoted to the definition of various spaces of analytic functions.


\section{Analyticity \label{defAeta}}


As a warm up we first describe some very classical analytic spaces.


\subsection{On the real line}


Let us first define various spaces of analytic functions of the real line $\rit$. 
The first idea is to consider the extension of analytic functions to the complex plane. Namely, for $\rho > 0$, 
we define the complex strip
$$
\Gamma_\rho = \Big \{z \in \CC, \quad | \Im z | \le \rho \Big \} ,
$$
and consider the space of functions $f$ which are holomorphic on this strip. We then define the analytic spaces ${\cal A}^\rho$ by: $f \in {\cal A}^\rho$ if and only if there exists a bounded
holomorphic function $g$ defined on $\Gamma_{\rho}$ such that $f(z) = g(z)$ for all real $z\in \rit$.
The norm on ${\cal A}^\rho$ is defined by
 $$
\| f \|_{\rho} = \sup_{z \in \Gamma_\rho} | g(z) | e^{\beta |\Re z|}  
 $$
for some fixed nonnegative number $\beta$. The exponential weight allows to control the behavior of $f$ at infinity. 
 By a slight abuse of notation, we denote by $f$ the extension of $f(z)$ to $\Gamma_\rho$. 
 
 \begin{proposition} \label{propanalytic1}
For every functions $f$ and $g$ in ${\cal A}^\rho$ there holds
$$
\| f g \|_\rho \le \| f \|_\rho \| g \|_\rho .
$$
 \end{proposition}
 
\begin{proof} 
Let $f$ and $g$ be the extensions to the complex strip $\Gamma_\rho$. Then
$$
| f(z) g(z) | e^{\beta |\Re z|} \le | f(z) | e^{\beta |\Re z|} | g(z) | e^{\beta |\Re z|}
$$
since $\beta \ge 0$. We then can the supremum norm of the right hand side, and then of the left hand side.
\end{proof}

One of the main advantages of analytic functions is that the $L^\infty$ norm
of their derivatives can be controlled
by the $L^\infty$ norm of the function. Namely we have
\begin{proposition} \label{propanalytic2}
There exists a constant $C > 0$ such that
for any function $f \in {\cal A}^\rho$ and for any $0 < \rho' < \rho$, we have
$$
\| \partial_z f \|_{\rho'} \le {C \over \rho - \rho'} \| f \|_{\rho} .
$$
 \end{proposition}

\begin{proof}
The proof is based on the use of Cauchy's formula
$$
\partial_z f(z) = {1 \over 2 i \pi} \int_{C(z,R)} {f(z') \over (z' - z)^2} dz'
$$
where $C(z,R)$ is the circle, centered at $z$ and of radius $R$. We can take $R$ sufficiently small so that $C(z,R) \in \Gamma_\rho$, for $z \in \Gamma_\rho$. Hence
$$
| \partial_z f(z) | \le {\| f \|_\rho e^{- \beta |\Re z|} \over d(z,\partial \Gamma_{\rho})},
$$
where $d(z,\partial \Gamma_{\rho})$ is the distance from $z$ to the boundary of $\Gamma_{\rho}$. Now if $z \in \Gamma_{\rho'}$, 
$$
d(z,\partial \Gamma_{\rho,r}) \ge C (\rho - \rho'),
$$
which ends the proof.
\end{proof}


\subsection{On the half line}


We next study the case of the half real line $\rit_+$. Let  $\sigma > 0$ and let $r > 0$. Let us introduce the "pencil" type domain
 $$
 \Gamma_{\sigma,r} = \Bigl\{ z, | \Im z | \le \min( \sigma \Re z, \sigma r), \Re z \ge 0 \Bigl\} .
 $$
 In the sequel, $r$ will be fixed and we will omit it in the various notations.
Let $f$ be a smooth function of real $z \ge 0$. As in the previous case, we define the analytic spaces ${\cal B}^\sigma$ by: $f \in {\cal B}^\sigma$ if and only if there exists a bounded
holomorphic function $g$ defined on $\Gamma_{\sigma,r}$ such that $f(z) = g(z)$ for all real $z\ge 0$.
The norm on ${\cal B}^\sigma$ is defined by
 $$
\| f \|_\sigma = \sup_{z \in \Gamma_{\sigma,r}} | g(z) | e^{\beta \Re z} 
 $$
for some fixed nonnegative constant $\beta$. By a slight abuse of notation, we denote by $f$ the extension of $f(x)$ to $\Gamma_{\sigma,r}$. As in Proposition \ref{propanalytic1}, we have 
$$
\| f g \|_\sigma \le \| f \|_\sigma \| g \|_\sigma ,
$$
For every functions $f$ and $g$ in ${\cal B}^\sigma$. In addition, we have 

\begin{proposition} \label{propanalytic3}
Let
$$
\varphi(z) = {z \over 1 + z} .
$$ 
There exists a constant $C > 0$ such that
for any function $f \in {\cal B}^\sigma$, and for any $0 < \sigma' < \sigma$, we have
$\varphi(z) \partial_z f \in {\cal B}^{\sigma'}$ and
\beq \label{derivee2ana}
\| \varphi(z)  \partial_z f \|_{\sigma'} \le {C \over \sigma - \sigma'} \| f \|_{\sigma} .
\eeq
 \end{proposition}

\begin{proof}
The proof is similar to that of Proposition \ref{propanalytic2}, upon noting that for $z \in \Gamma_{\sigma',r}$, we have 
$$
d(z,\partial \Gamma_{\sigma,r}) \ge C (\sigma - \sigma') |\varphi(z)|,
$$
which ends the proof.
\end{proof}


\subsection{In two space dimensions}


Let us now define analytic spaces for functions which are defined on $\rit \times \rit_+$. We define the domain
$\Gamma_{\rho,\sigma,r}$ by
$$
\Gamma_{\rho,\sigma,r} = \Bigl \{ (x,y), \quad  | \Im x | \le \rho, \quad | \Im y | \le \min(\sigma \Re y,\sigma r), \Re y \ge 0 \Bigr\} .
$$
We then define ${\cal A}^{\rho,\sigma}$ as the space of holomorphic functions, defined on $\Gamma_{\rho,\sigma,r}$, together with the norm
$$
\| f \|_{\rho,\sigma} = \sup_{z \in \Gamma_{\rho,\sigma,r}} | f(z) | e^{\beta \Re y}
$$
where $\beta > 0$ is a fixed constant. Note that $\rho$ measures the regularity in the $x$ variable, and $\sigma$
in the $y$ variable.
Propositions similar to \ref{propanalytic1} and \ref{propanalytic2} hold true, namely
$$
\| f g \|_{\rho,\sigma} \le \| f \|_{\rho,\sigma}\| g \|_{\rho,\sigma},
$$
$$
\| \partial_x f \|_{\rho',\sigma} \le {C \over \rho - \rho'} \| f \|_{\rho,\sigma} , \qquad \| \varphi(z) \partial_z f \|_{\rho,\sigma'} \le {C \over \sigma - \sigma'} \| f \|_{\rho,\sigma}
.
$$


\section{Generator functions}


We now introduce other classes of analytic functions, which turn be more adapted to Navier Stokes equations.


\subsection{In one space dimension}


Another way to say that a function $f$ is holomorphic, is to say that its successive derivatives $\partial_z^n f$
grow at most like $n ! M^n$ for some constant $M$. This leads to the following definitions.

Let $\| . \|_n$ be a family of norms, which satisfies, for any $0 \le p \le n$
$$
\| f g \|_n \le C_0 \| f\|_p \| g \|_{n-p} 
$$
for a given constant $C_0$ independent on $f$, $g$, $n$ and $p$. In the sequel we will take the supremum norm or a weighted norm,
specially designed for boundary layers.

Let $f(z)$ be a given function of a real variable $z \in \rit$ or $z \in \rit_+$. We define its generator function $Gen(f)$ by
$$
Gen(f)(z) = \sum_{n \ge 0}  \| \partial_z^n f \|_n {z^n \over n !} . 
$$
Note that $Gen(f)$ is defined on an interval of the form $(-R,R)$, with $R > 0$ if $f$ is analytic.

Note that $Gen(f)$ is a non negative, nondecreasing, convex function. All its derivatives, at any order, are non negative,
nondecreasing and convex.

\begin{prop}
Let $f$ and $g$ be two given functions, and let $N > 0$. Then
\beq \label{Gen1}
Gen(fg) \le C_0 Gen(f) Gen(g).
\eeq
\end{prop}

\begin{proof}
We have
$$
\partial_z^n (fg) = \sum_{0 \le k \le n} {n ! \over k ! (n-k)! } \partial_z^k f \partial_z^{n-k} g,
$$
hence
$$
\| \partial_z^n (fg) \|_n  \le C_0 \sum_{0 \le k \le n} {n ! \over k ! (n-k)! } \| \partial_z^k f \|_k 
\| \partial_z^{n-k} g \|_{n-k} .
$$
Therefore
$$
\| \partial_z^n (fg) \|_n {z^n \over n!} \le
C_0 \sum_{0 \le k \le n}  \| \partial_z^k f \|_k {z^k \over k !}
\| \partial_z^{n-k} g \|_{n-k} {z^{n-k} \over (n-k) !}.
$$
Summing with respect to $n$
we get (\ref{Gen1}).
\end{proof}


\subsection{On the half plane}


In two space dimensions, we will 
use the Fourier transform in the $x$ variable. Let $\alpha$ be the dual variable. For any smooth function $f(x,y)$
we introduce $f_\alpha(y)$, its Fourier transform in $x$, which satisfies
$$
f(x,y) = \sum_\alpha e^{i \alpha x} f_\alpha(y) .
$$
We then introduce the generator function
$$
Gen(f)(z_1,z_2) = \sum_{k,\alpha}  e^{| \alpha | z_1} \|  \partial_y^k f_\alpha \|_k  {z_2 ^k \over k !} .
$$ 
Generator functions have very nice properties, since the generator of a product is dominated by the product
of the generators, and the generator of the  $x$ derivative is bounded by the $x$ derivative of the generator,
two strong and very useful properties!

\begin{prop} \label{productGen}
For any functions $f$ and $g$, there hold 
$$
Gen(fg) \le C_0 Gen(f) Gen(g), \qquad Gen(\partial_x f) \le \partial_{z_1} Gen(f).
$$
\end{prop}

\begin{proof}
We have
$$
\begin{aligned}
Gen(fg) &= \sum_{k,\alpha, k' \le k, \alpha'} 
e^{| \alpha | z_1} \| \partial_y^{k'} f_\alpha' \partial_y^{k - k'} f_{\alpha - \alpha'} \|_k {z_2^k \over k' ! (k - k') ! }
\\
&\le C_0  \sum_{k,\alpha, k' \le k, \alpha'} e^{| \alpha' | z_1}  \| \partial_y^{k'} f_\alpha'  \|_{k'}
e^{ | \alpha - \alpha'| z_1 } \| \partial_y^{k - k'} f_{\alpha - \alpha'} \|_{k - k'} {z_2^k \over k' ! (k - k') ! }
\\&
\le C_0 Gen(f) Gen(g).
\end{aligned}
$$
The equality involving $Gen(\partial_x f)$ is straightforward. Note that we do not have a similar expression
for $Gen(\partial_y f)$ since $\partial_y f$ may become large near $0$. 
\end{proof}


\subsection{Boundary layer norms}


We will now focus on the case where there is a boundary layer in the second variable $y$. More precisely, if we assume that 
$f$ has a boundary layer of size $\delta$ near $y = 0$ then we expect, for any $k$ and $\ell$, 
$\partial_x^k f$ and $y^\ell \partial_y^\ell f$ are bounded. For this reason, 
we introduce analytic boundary layer norms: for $z_1,z_2\ge 0$ we define
\beq \label{Genbl}
\begin{aligned}
Gen_0(f)(z_1,z_2) &= \sum_{\alpha \in \ZZ} \sum_{\ell \ge 0}  
e^{z_1 |\alpha|}  \|  \partial_y^\ell f_\alpha \|_{\ell,0}{z_2^\ell \over \ell ! } ,
\\Gen_\delta(f)(z_1,z_2) &= \sum_{\alpha \in \ZZ}\sum_{\ell \ge 0}  
e^{z_1 |\alpha|} \| \partial_y^\ell f_\alpha \|_{\ell,\delta} {z_2^\ell \over \ell ! } ,
\end{aligned}\eeq
in which $f_\alpha(y)$ denotes the Fourier transform of $f(x,y)$ with respect to the $x$ variable. In these sums, 
$$
\begin{aligned}
\| f_\alpha \|_{\ell,0} &= \sup_{y} \varphi(y)^\ell| f_\alpha(y) | ,
\\\| f_\alpha \|_{\ell,\delta} &= \sup_{y} \varphi(y)^{\ell}| f_\alpha(y) | \Bigl( \delta^{- 1} e^{-y / \delta} + 1 \Bigr)^{-1} ,
\end{aligned}$$
where 
$$
\varphi(y) = \frac{y}{1+y}
$$ 
and where the boundary layer thickness $\delta$ is equal to
$$
\delta = \gamma_0 \nu^{1/4}
$$
for some sufficiently large $\gamma_0>0$. 

In practice, these generator functions $Gen_0(\cdot)$ and $Gen_\delta(\cdot)$ will respectively control the velocity and the vorticity of the solutions
of Navier Stokes equations, and the constant $\gamma_0 $ will be chosen so that $\gamma_0^{-1} \le \sqrt{\Re \lambda_0/2}$, 
where $\lambda_0$ is the maximal unstable eigenvalue of the linearized Euler equations around $U$.  

For convenience, we introduce the following generator functions of one-dimensional functions $f = f(y)$:
\beq \label{Genbl-a}
\begin{aligned}
Gen_{0,\alpha}(f)(z_2) &= \sum_{\ell \ge 0}  \|  \partial_y^\ell f \|_{\ell,0}{z_2^\ell \over \ell ! } ,
\\Gen_{\delta,\alpha}(f)(z_2) &= \sum_{\ell \ge 0}   \| \partial_y^\ell f \|_{\ell,\delta} {z_2^\ell \over \ell ! } .
\end{aligned}\eeq
Of course, it follows that 
$$
Gen_0(f) = \sum_{\alpha\in \ZZ} e^{z_1|\alpha|} Gen_{0,\alpha} (f_\alpha)
$$ 
for functions of two variables $f = f(x,y)$, and similarly for $Gen_\delta$.

We note that $Gen_0$, $Gen_\delta$ and all their derivatives are non negative for positive $z_1$ and $z_2$. It follows easily that for any $\ell, \ell'\ge 0$, we have
\begin{equation}\label{alg}
\begin{aligned}
\| f\|_{\ell, \delta} \le \|f\|_{\ell, 0},& \qquad \| f\|_{\ell + 1, \delta} \le \|f\|_{\ell, \delta}, \\
\| fg\|_{\ell,\delta} &\le \| f \|_{\ell',0}\|g\|_{\ell-\ell',\delta }.
\end{aligned}\end{equation}
Next, we have the following Proposition 

\begin{prop}\label{prop-Gen}
Let $f$ and $g$ be two functions. For non negative $z_1$ and $z_2$,
there hold
$$
Gen_\delta(f g) \le  Gen_0(f) Gen_\delta(g) ,
$$
$$
Gen_\delta(\partial_x f) = \partial_{z_1} Gen_\delta(f), 
\qquad
Gen_\delta(\partial_x^2 f) = \partial_{z_1}^2 Gen_\delta(f), 
$$
$$Gen_\delta(\varphi \partial_y f) \le C_0 \partial_{z_2} Gen_\delta(f),$$
for some universal constant $C_0$, provided $| z_2 |$ is small enough. 
\end{prop}
\begin{proof} First, note that 
$$ (fg)_\alpha = \sum_{\alpha'\in \ZZ} f_{\alpha'} g_{\alpha - \alpha'} ,$$
and
$$
 \partial_y^\beta (f g)_\alpha =
\sum_{\alpha'\in \ZZ} \sum_{0 \le \beta' \le \beta} 
{\beta ! \over \beta' ! (\beta - \beta') !}  \partial_y^{\beta'} f_{\alpha'} 
\partial_y^{\beta - \beta'} g_{\alpha-\alpha'} .
$$
Thus, 
$$
\begin{aligned}
&Gen_\delta(fg)(z_1,z_2) 
\\&= \sum_{\alpha \in \ZZ}\sum_{\beta \ge 0} e^{z_1|\alpha|}\| \partial_y^\beta (fg)_\alpha \|_{\beta,\delta} {z_2^\beta \over \beta ! } 
\\
&\le \sum_{\alpha \in \ZZ}\sum_{\beta \ge 0} \sum_{\alpha'\in \ZZ}\sum_{0 \le \beta' \le \beta} e^{z_1|\alpha|} \|\partial_y^{\beta'} f_{\alpha'}\|_{\beta',0}
\| \partial_y^{\beta - \beta'} g_{\alpha-\alpha'} \|_{\beta-\beta',\delta}  
{z_2^\beta \over \beta' ! (\beta - \beta') !}  
\\
&\le \sum_{\alpha,\alpha'\in \ZZ}\sum_{ \beta \ge 0} \sum_{\beta \ge \beta' } e^{z_1|\alpha'|} e^{z_1|\alpha-\alpha'|} \|\partial_y^{\beta'} f_{\alpha'}\|_{\beta',0}
\| \partial_y^{\beta - \beta'} g_{\alpha - \alpha'} \|_{\beta-\beta',\delta} 
{z_2^{\beta'}z_2^{\beta - \beta'} \over \beta' ! (\beta - \beta') !}  
\\
&\le Gen_0(f)(z_1,z_2) Gen_\delta (g)(z_1,z_2).
\end{aligned}$$
Next, we write 
$$
\begin{aligned} Gen_\delta(\partial_xf) &= \sum_{\alpha\in \ZZ }\sum_{ \ell \ge 0} e^{z_1|\alpha|}\| \alpha \partial_y^\ell f_\alpha \|_{\ell,\delta} {z_2^\ell \over \ell ! } 
\\&= \partial_{z_1}
\sum_{\alpha\in \ZZ }\sum_{ \ell \ge 0} e^{z_1|\alpha|}\|  \partial_y^\ell f_\alpha \|_{\ell,\delta} {z_2^\ell \over \ell ! } 
= 
\partial_{z_1} Gen_\delta (f),
\end{aligned} $$
and similarly for $Gen_\delta(\partial_x^2 f)$.
Finally, we compute 
$$
\begin{aligned} Gen_\delta(\varphi \partial_yf) 
&= \sum_{\alpha\in \ZZ }\sum_{ \ell \ge 0} e^{z_1|\alpha|}\|  \partial_y^\ell (\varphi \partial_yf_\alpha) \|_{\ell ,\delta} {z_2^\ell \over \ell ! } 
\\
&\le \sum_{\alpha\in \ZZ } \sum_{ \ell \ge 0} 
\sum_{0\le \ell'\le \ell}e^{z_1|\alpha|}\|  \partial_y^{\ell'} \varphi \partial^{\ell - \ell' + 1}_yf_\alpha \|_{\ell ,\delta} {z_2^\ell \over \ell' ! (\ell - \ell')! } 
\\
&\le \Bigl( 1 + \sum_{\ell'\ge 0} \|\partial_y^{\ell'} \varphi\|_{0,0} {z_2^{\ell'} \over \ell' ! }\Bigr)  
 \sum_{\alpha\in \ZZ }\sum_{ \ell -\ell'\ge 0} 
  e^{z_1|\alpha|}\|  \partial^{\ell - \ell' + 1}_yf_\alpha \|_{\ell - \ell'+1,\delta} {z_2^{\ell -\ell'}\over (\ell - \ell')! } 
\\&\le C_0 \partial_{z_2} Gen_\delta (f),
\end{aligned} $$
where we distinguished the cases $\ell' = 0$ and $\ell' > 0$.
As $\varphi$ is analytic,
$\sum_{\ell'\ge 0} \|\partial_y^{\ell'} \varphi\|_{0,0} {z_2^{\ell'} / \ell' ! }$ converges provided $z_2$ is small enough.
The Proposition follows. 
\end{proof}


\section{Laplace equation}


In this section, we study the Laplace equation and the generator functions of solutions to the Laplace equation. 
In the latter chapters, we shall apply a similar analysis to the more complex Orr Sommerfeld equations, 
which is the resolvent equation for linearized Navier-Stokes equations around a boundary layer profile. 


\subsection{In one space dimension}


As an exercice we now investigate the analytic regularity of the solution of the classical Laplace equation with damping.
Let $\alpha > 0$ be fixed and let us study the following equation
\beq \label{Lap11}
\partial_y^2 \phi - \alpha^2 \phi = f
\eeq
on the half line $z \ge 0$, with boundary condition 
\beq \label{Lap21}
\phi(0) = 0 .
\eeq
The Green function of $\partial_z^2 - \alpha^2$ is
$$
G(x,z) = - {1 \over  2 \alpha } \Bigl( e^{- \alpha | z - x |} - e^{- \alpha | z + x |} \Bigr) 
=  \left\{ \begin{aligned}  & G_-(x,z)  \qquad 0\le x\le z \\
& G_+(x,z), \qquad 0\le z\le x.
 \end{aligned}\right. $$
 Or, equivalently, using 
 $$
 \sinh(x) = \frac{1}{2} (e^x - e^{-x}),
 $$ 
  we have
 $$
 G_-(x,z) = - {1 \over \alpha} e^{- \alpha z} \sinh(\alpha x),
\qquad  G_+(x,z) = - {1 \over \alpha} e^{- \alpha x} \sinh(\alpha z) .
 $$
For real values of $z$, the solution $\phi$ of (\ref{Lap11}) is explicitly given by
\begin{equation}\label{laplacephi11}
\phi(z) =  \int_0^z G_- (x,z) f(x) \; dx + \int_z^\infty G_+(x,z) f(x)\; dx.
\end{equation}
This equality defines the solution for real values of $z$. It can be extended to any complex value of $z$ provided
we replace the integral from $0$ to $z$ by the integral over the segment $[0,z]$ and the integral
from $z$ to $+ \infty$ by the integral on the half line
$z + \rit_+$.

 We will prove the following classical proposition, which asserts that the inversion of the Laplace
 operator leads to a gain of two derivatives. We recall that 
 $$
 \|f \|_{0,0} = \sup_{y \ge 0}|f(y) |.
 $$ 
Let us first recall the following classical result:

\begin{prop}\label{proplaplace1} ($L^\infty$ bounds). \\
Let $\phi$ solve the one-dimensional Laplace problem \eqref{Lap11}, with Dirichlet boundary condition. There holds
\beq \label{Lap31}
\alpha^2 \| \phi \|_{0,0} + | \alpha | \,  \| \partial_y \phi \|_{0,0}
+ \| \partial_y^2 \phi \|_{0,0}  \le C \| f \|_{0,0},
\eeq
where the constant $C$ is independent of the integer $\alpha \ne 0$.
\end{prop}
\begin{proof} 
We will only consider the case $\alpha > 0$, the opposite case being similar.
The Green function of $\partial_y^2 - \alpha^2$ is
$$
G(x,y) =  -{1 \over 2\alpha } \Bigl( e^{- \alpha | x-y |} - e^{- \alpha | x+y |} \Bigr) 
$$
and its absolute value is bounded by $\alpha^{-1} e^{-\alpha|x-y|}$. The solution $\phi$ of (\ref{Lap11}) is explicitly given by
\begin{equation}\label{laplacephi1}
\phi(y) =\int_0^\infty G(x,y) f(x) dx .
\end{equation}
A direct bound leads to
$$
|\phi(y)|\le \alpha^{-1}\|f\|_{0,0}\int_0^\infty e^{-\alpha |x-y|}\; dx \le  C \alpha^{-2} \| f \|_{0,0}
$$
in which the extra $\alpha^{-1}$ factor  is due to the $x$-integration. 
Splitting the integral formula (\ref{laplacephi11}) in $x < y$ and $x > y$ and differentiating it, we get
$$
 \|\partial_y \phi \|_{0,0} \le C \alpha^{-1} \| f \|_{0,0} .
$$
We then use the equation to bound $\partial_y^2 \phi$, which ends the proof of (\ref{Lap31}).
\end{proof}

Next, in the case when $f$ has a boundary layer behavior, we obtain the following result:
 
\begin{prop} \label{proplaplace3a} (Boundary layers norms)\\
Let $\phi$ solve the one-dimensional Laplacian problem \eqref{Lap11} with Dirichlet boundary condition. 
Provided 
\beq \label{condLap4}
| \delta \alpha^2 | \le 1
\eeq
there holds
\beq \label{Lap4}
 \| \nabla_\alpha \phi \|_{0,0}   \le C \| f \|_{0,\delta}
\eeq
and
\beq \label{Lap4bis1}
| \alpha |^2 \, \| \phi \|_{0,0} 
+  \| \partial_y^2 \phi \|_{0,\delta}   \le C \| f \|_{0,\delta}
\eeq
where the constant $C$ is independent of the integer $\alpha$.
\end{prop}
Note that in the case of boundary layer norms, we only gain "one" derivative in supremum norm, but the usual two derivatives
in boundary layer norm. 
\begin{proof}
Using (\ref{laplacephi11}), we estimate
$$
| \phi (y) | \le \alpha^{-1} \| f \|_{0,\delta} \int_0^\infty
e^{- \alpha |   y -   x |} 
\Bigl( 1 + \delta^{-1} e^{-x/\delta}\Bigr) d  x
$$
$$
\le  \alpha^{-1} \| f \|_{0,\delta} 
\Bigl( \alpha^{-1} + \delta^{-1} \int_0^\infty e^{-x/\delta}d   x \Bigr) 
$$
which yields the claimed bound for $\alpha \phi$. The bound on $\partial_y \phi$ is obtained by differentiating  (\ref{laplacephi11}).

Let us turn to (\ref{Lap4bis1}). Note that  $| \partial_x G(x,y) |\le 1$. As
$G(0,y) = 0$ this gives $|G(x,y)| \le | x |$. 
Therefore
$$
| G(x,y) | \le \min(\alpha^{-1}e^{-\alpha | x - y|}, | x|) ,
$$
and hence
$$
| \phi(y) | \le \| f \|_{0,\delta} \int_0^\infty \min(| x | , \alpha^{-1} e^{- \alpha  |x - y|} ) \Bigl( \delta^{-1} e^{- x / \delta} + 1 \Bigr) dx
$$
$$
\le  C \| f \|_{0, \delta} \Bigl(\delta +  \alpha^{-2} \Bigr)
$$
which gives the desired bound when $| \delta \alpha^2 | \le 1$. We then use the equation to get the bound on 
$\| \partial^2_y \phi \|_{0,\delta}$.
\end{proof}



\subsection{Laplace equation and generator functions}


In this section, we will study the generator functions of solutions to the Laplace equation 
$\Delta\phi = \omega$. In the sequel, it is important to keep in mind that, in the application to Prandtl boundary layer stability,
$\omega$ will  have a boundary layer behavior, namely will behave like
$\delta^{-1} e^{- C y / \delta}$, whereas the stream function $\phi$ will be bounded in the limit. 

\begin{proposition}\label{prop-elliptic}  
Let 
$$
\Delta_\alpha \phi_\alpha = \omega_\alpha
$$ 
on $\RR_+$ 
with the Dirichlet boundary condition ${\phi_\alpha}_{\vert_{y=0}} = 0$. For $| \delta \alpha^2 | \le 1$,
 there are positive constants $C_0,\theta_0$ so that 
\beq \label{prop-e}
Gen_{\delta,\alpha}(\nabla_\alpha^2 \phi_\alpha) 
+ Gen_{0,\alpha}(\nabla\phi_\alpha) \le C_0 Gen_{\delta,\alpha}(\omega_\alpha),
\eeq
for all $z_2$ so that $| z_2 | \le \theta_0$. \\
Moreover if $\phi = \Delta^{-1} \omega$ and if $\omega_\alpha = 0$
 for all $\alpha$ such that $| \delta \alpha^2 | \ge 1$, then
\beq \label{prop-e1}
Gen_0(\nabla \phi) \le C Gen_\delta(\omega),
\eeq
\beq \label{prop-e2}
\partial_{z_1} Gen_0(\nabla \phi) \le C \partial_{z_1} Gen_\delta(\omega),
\eeq
\beq \label{prop-e3}
\partial_{z_2} Gen_0(\nabla \phi) \le C \partial_{z_2} Gen_\delta(\omega) + Gen_\delta(\omega).
\eeq
\end{proposition}

\begin{proof} 
For $n\ge 1$, from the elliptic equation $\Delta_\alpha \phi_\alpha = \omega_\alpha$, we compute  
$$
 \Delta_\alpha ( \varphi^n \partial_y^n \phi_\alpha)  
= \varphi^n \partial_y^n \omega_\alpha + 2 \partial_y (\varphi^n) \partial_y^{n+1} \phi_\alpha
 + \partial_y^2 (\varphi^n) \partial_y^n \phi_\alpha 
.$$
Note that
$$
\partial_y (\varphi^n) \partial_y^{n+1} \phi_\alpha
= n \varphi'  \varphi^{n-1}  \partial_y^{n+1} \phi_\alpha
,$$
and hence the $\| . \|_{0,\delta}$ norm of this term is bounded by $n \| \varphi^{n-1} \partial_y^{n+1} \phi_\alpha \|_{0,\delta}$.
Moreover, 
$$
 \partial_y^2 (\varphi^n) \partial_y^n \phi_\alpha = 
 \Bigl( n (n-1) \varphi'^2 \varphi^{n-2}  + n \varphi'' \varphi^{n-1} \Bigr) \partial_y^n \phi_\alpha
 $$
whose $\| . \|_{0,\delta}$ norm is bounded by $n (n-1) \| \varphi^{n-2} \partial_y^{n} \phi_\alpha \|_{0,\delta}$.
 Using Proposition \ref{proplaplace3a}, we get
 $$
\begin{aligned}
&|\alpha|^2 \| \varphi^n \partial_y^n \phi_\alpha\|_{0,0} +  \| \partial_y^2 (\varphi^n \partial_y^n \phi_\alpha) \|_{0,\delta}  
 + \| \nabla_\alpha (\varphi^n \partial_y^n \phi_\alpha)  \|_{0,0}
\\& \le C \| \varphi^n \partial_y^n \omega_\alpha \|_{0,\delta} 
 + C n \| \varphi^{n-1} \partial_y^{n+1} \phi_\alpha \|_{0,\delta}
 + C n (n-1) \| \varphi^{n-2} \partial_y^{n} \phi_\alpha \|_{0,\delta}.
\end{aligned} $$
 Expanding the left hand side, we get
 \beq \label{expand}
\begin{aligned}
&
|\alpha|^2 \|  \varphi^n \partial_y^n \phi_\alpha\|_{0,0} +  \| \varphi^n \partial_y^{n+2} \phi_\alpha \|_{0,\delta}  
 + \| \varphi^n \partial_y^n \nabla_\alpha \phi_\alpha  \|_{0,0}
\\& \le C_0 \| \varphi^n \partial_y^n \omega_\alpha \|_{0,\delta} 
 + C_0n \| \varphi^{n-1} \partial_y^{n+1} \phi_\alpha \|_{0,\delta}
 \\& + C_0n (n-1) \| \varphi^{n-2} \partial_y^{n} \phi_\alpha \|_{0,\delta} 
 + C_0 n \| \varphi^{n-1} \partial_y^n \phi \|_{0,0}.
\end{aligned} 
\eeq
 Let
 $$
 A_n = 
 |\alpha|^2 \|  \varphi^n \partial_y^n \phi_\alpha\|_{0,0} +  \| \varphi^n \partial_y^{n+2} \phi_\alpha \|_{0,\delta}  
 + \| \varphi^n \partial_y^n \nabla_\alpha \phi_\alpha  \|_{0,0}
. $$
Multiplying by $z_2^n / n!$ and summing over $n$, we get
$$
\begin{aligned}\sum_{n \ge 0}A_n {z_2^n \over n ! }
&\le C_0\sum_{n \ge 0}\| \varphi^n \partial_y^{n} \omega_\alpha \|_{0,\delta} {z_2^n \over n ! }
+ C_0 \sum_{n \ge 1}  A_{n-1} {z_2^n \over (n-1) ! }
\\&\quad + C_0 \sum_{n \ge 2}  A_{n-2} {z_2^n \over (n-2) ! }
\\ & \le C_0\sum_{n \ge 0}\| \varphi^n \partial_y^{n} \omega_\alpha \|_{0,\delta} {z_2^n \over n ! }
+ C_0 (z + z^2) \sum_{n \ge 0}A_n {z_2^n \over n ! },
\end{aligned}$$
hence
$$
\sum_{n \ge 0}A_n {z_2^n \over n ! } \le C_0' \sum_{n \ge 0}\| \varphi^n \partial_y^{n} \omega_\alpha \|_{0,\delta} {z_2^n \over n ! }
$$
provided $| z_2 |$ is small enough, which ends the proof of (\ref{prop-e}).

Next (\ref{prop-e1}) is a direct consequence of (\ref{prop-e}), just summing in $\alpha$. 
If we multiply (\ref{prop-e}) by $| \alpha |$ before summing it, this gives (\ref{prop-e2}).
Now we multiply (\ref{expand}) by $z_2^{n-1} / (n-1)!$ instead of $z_2^n / n!$. This gives
$$
\begin{aligned}\sum_{n \ge 1}A_n {z_2^{n-1} \over (n-1) ! }
&\le C_0\sum_{n \ge 1}\| \varphi^n \partial_y^{n} \omega_\alpha \|_{0,\delta} {z_2^{n-1} \over (n-1) ! }
+ C_0 \sum_{n \ge 1}  A_{n-1} n {z_2^{n-1} \over (n-1) ! }
\\&\quad + C_0 \sum_{n \ge 2}  n (n-1) A_{n-2} {z_2^{n-1} \over (n-1) ! }.
\end{aligned}$$
The terms in the right hand side may be absorbed by the left hand side provided $z_2$ is small enough, 
except
$C_0 A_0$, which is bounded by $Gen_{\delta,\alpha} (\omega_\alpha)$. This ends the proof of the Proposition.
\end{proof}


\section{Divergence free vector fields}



\subsection{Generator function and divergence free condition}


Note that for any functions $u$ and $g$, Proposition \ref{prop-Gen} yields 
\begin{equation}\label{bd-udxg}
Gen_\delta(u \partial_x g) \le Gen_0(u) \partial_{z_1} Gen_\delta(g) .
\end{equation}
This is not true for $Gen_\delta(v \partial_y g)$, due to the boundary layer weight. We will investigate $Gen_\delta(v \partial_y g)$ 
when $(u,v)$ satisfies the divergence free condition, namely 
$$
\partial_x u + \partial_y v = 0 .
$$
Precisely, we will prove the following Proposition. 
\begin{prop}\label{prop-Gendy} For $| z_2 | \le 1$, there holds
$$
Gen_\delta(v \partial_y g) \le C \Bigl(Gen_0(v) + \partial_{z_1} Gen_0(u) \Bigr) \partial_{z_2} Gen_\delta(g) .
$$
\end{prop}
Note that we "loose" one derivative: our bound involves $\partial_x u$. 
\begin{proof}
We compute 
$$
Gen_\delta(v \partial_y g) = \sum_{\alpha\in\ZZ}\sum_{\beta\ge 0} e^{z_1|\alpha|}  
\| \partial_y^\beta ( v \partial_y g)_\alpha \|_{\beta,\delta} 
{z_2^\beta \over \beta ! } ,
$$
in which $$
\partial_y^\beta  ( v \partial_y g)_\alpha
= \sum_{\alpha'\in\ZZ} \sum_{0 \le \beta' \le \beta} 
{\beta ! \over \beta'! (\beta - \beta')!} \partial_y^{\beta'} v_{\alpha'}
 \partial_y^{\beta - \beta' + 1} g_{\alpha-\alpha'}.
$$
For $\beta'>0$, using the divergence-free condition $\partial_yv_\alpha = -i\alpha u_\alpha $, we estimate  
$$\begin{aligned}
\|
\partial_y^{\beta'} v_{\alpha'}\partial_y^{\beta - \beta' + 1} g_{\alpha-\alpha'}\|_{\beta,\delta}
\le \| \alpha' \partial_y^{\beta'-1} u_{\alpha'} \|_{\beta'-1,0}
\| \partial_y^{\beta - \beta' + 1} g_{\alpha-\alpha'}\|_{\beta - \beta' + 1,\delta} .
 \end{aligned}$$
On the other hand, for $\beta' =0$, we estimate 
$$\begin{aligned}
\|
v_{\alpha'}
\partial_y^{\beta  + 1} g_{\alpha-\alpha'}\|_{\beta,\delta}
\le 
\| \varphi^{-1} v_{\alpha'} \|_{0,0}
\|\partial_y^{\beta + 1} g_{\alpha-\alpha'}\|_{\beta  + 1,\delta}
. \end{aligned}$$
We note that for $y\ge 1$, $\varphi(y) \ge 1/2$ and hence 
$$
\| \chi_{\{y\ge 1\}}\varphi^{-1} v_{\alpha'}\|_{0,0} \le 2 \|v_{\alpha'} \|_{0,0}. 
$$ 
When $y\le 1$, using again the divergence-free condition, 
we write $$
v_{\alpha'}(y) = - i\alpha' \int_0^y  u_{\alpha'}(y') dy'  = - i\alpha'  y \int_0^1 u_{\alpha'}(x,\theta y) d\theta.
$$
Therefore, 
$$
\varphi(y)^{-1}|v_{\alpha'} (y)| \le \sup_y |\alpha' u_{\alpha'} (y)|
$$
for $y\le 1$. This proves that 
$$
 \| \varphi^{-1} v_{\alpha'} \|_{0,0} \le 2 \| v_{\alpha'} \|_{0,0} + \| \alpha' u_{\alpha'} \|_{0,0}.
 $$
Combining these inequalities for any $\alpha\in \ZZ$ and $\beta\ge 0$, we obtain 
$$
\begin{aligned}
& \| \partial_y^\beta ( v \partial_y g)_\alpha \|_{\beta,\delta} 
\le \sum_{\alpha'\in \ZZ}
(2 \| v_{\alpha'} \|_{0,0} + \| \alpha' u_{\alpha'} \|_{0,0} )\| \partial_y^{\beta +1} g_{\alpha-\alpha'} \|_{\beta+1,\delta}  
\\&\quad + \sum_{\alpha'\in \ZZ} \sum_{1 \le \beta' \le \beta} \| \alpha'\partial_y^{\beta'-1} u_{\alpha'} \|_{\beta'-1,0}
\| \partial_y^{\beta - \beta' + 1} g_{\alpha-\alpha'}\|_{\beta - \beta' + 1,\delta} {\beta ! \over \beta' ! (\beta - \beta')!}
.\end{aligned}$$
It remains to multiply by $e^{z_1 | \alpha|} z_2^\beta / \beta !$ and to sum all the terms over $\alpha$, $\alpha'$, $\beta$ and
$\beta'$. The second term in the right hand side is bounded by the product of
$$
\sum_\alpha \sum_\beta e^{| \alpha | z_1}  \| \alpha \partial_y^\beta u_{\alpha} \|_{\beta,0} {z_2^{\beta + 1} \over (\beta +1) !},
$$
which is bounded by $Gen_0(\partial_x u)$ provided $| z_2  | \le 1$
and of
$$
\sum_\alpha \sum_\beta e^{| \alpha | z_1} \| \partial_y^{\beta+1} g_\alpha \|_{\beta +1} {z_2^\beta \over \beta !}
,$$
which equals $\partial_{z_2} Gen_\delta(g)$. The first term is similar, which ends the proof.
\end{proof}


\subsection{Bilinear estimates}


Let us now  bound derivatives of the transport term $u \partial_x g + v \partial_y g$.
\begin{prop}\label{prop-Gendy} 
Let
$$
{\cal A} = \Bigl(Id + \partial_{z_1} + \partial_{z_2}\Big) Gen_\delta
$$
and $$
{\cal B} = Gen_0(u) + Gen_0(v) + \partial_{z_1} Gen_0(u) + {\cal A}(g).
$$
Then
$$
{\cal A} (u \partial_x g + v \partial_y g) \le C {\cal B} \partial_{z_1} {\cal B} + C {\cal B} \partial_{z_2} {\cal B} .
$$
\end{prop}
Note that all the terms in ${\cal A}$ are non negative, since all the derivatives of generator functions are non negative.
\begin{proof}
Let us successively bound all the terms appearing in ${\cal A} (u \partial_x g + v \partial_y g)$. First, $Gen_\delta(u \partial_x g + v \partial_y g)$
has been bounded in \eqref{bd-udxg} and in the previous proposition. Next we compute 
\beq \label{Gendx}
\begin{aligned}
\partial_{z_1} Gen_\delta (u \partial_x g) 
&= Gen_\delta(\partial_x (u \partial_x g)) 
= Gen_\delta(\partial_x u \partial_x g + u \partial_x^2 g) 
\\&\le \partial_{z_1} Gen_0(u) \partial_{z_1} Gen_\delta(g) 
+ Gen_0(u) \partial_{z_1}^2 Gen_\delta(g).
\end{aligned}\eeq
Moreover, using Proposition \ref{prop-Gendy}, 
$$
\begin{aligned}
\partial_{z_1} Gen_\delta (v \partial_y g) 
&= Gen_\delta( \partial_x( v \partial_y g))
= Gen_\delta( \partial_x v \partial_y g + v \partial_y \partial_x g) 
\\
&\le C (\partial_{z_1} Gen_0(v) + \partial_{z_1}^2 Gen_0(u)) \partial_{z_2} Gen_\delta(g)
\\&\quad+ C (Gen_0(v) + \partial_{z_1} Gen_0(u)) \partial_{z_2} \partial_{z_1} Gen_\delta(g) .
\end{aligned}$$
Let us now bound the term $\partial_{z_2} Gen_\delta(v \partial_y g)$.
Precisely, we have to bound 
$$
\begin{aligned}
&{z_2^n \over n !} \| \varphi^{n+1} \partial_y^{n+1} (v_{\alpha'} \partial_y g_{\alpha - \alpha'})  \|_{0,\delta}
= {z_2^n \over n !} \| \varphi^{n+1} 
 \partial_y^n (\partial_y v_{\alpha'} \partial_y g_{\alpha - \alpha'}
 + v_{\alpha'} \partial_y^2 g_{\alpha - \alpha'})   \|_{0,\delta} 
\\&\le \sum_{0 \le k \le n} {z_2^n \over k ! (n-k) !}
\| \varphi^{n+1} \partial_y^{k+1} v_{\alpha'} \partial_y^{n+1-k} g_{\alpha - \alpha'} 
+ \varphi^{n+1} \partial_y^k v_{\alpha'} \partial_y^{n+2-k} g_{\alpha - \alpha'} \|_{0,\delta}.
\end{aligned}$$
Let us split this sum in two. The first sum equals, using the divergence free condition,
$$
\begin{aligned}
&\sum_{0 \le k \le n} {z_2^n \over k ! (n-k) !} \| \varphi^k \partial_y^k \partial_x u_{\alpha'} \, \,
\varphi^{n+1-k} \partial_y^{n+1-k} g_{\alpha - \alpha'}  \|_{0,\delta}
\\&\le \sum_{0 \le k \le n} {z_2^k \over k ! } \| \varphi^k \partial_y^k \partial_x u_{\alpha'} \|_{0,0}
{z_2^{n-k} \over (n-k)! } \|\varphi^{n+1-k} \partial_y^{n+1-k} g_{\alpha - \alpha'}  \|_{0,\delta}.
\end{aligned}$$
Multiplying by $e^{| \alpha | z_1}$ and summing over $\alpha$ and $\alpha'$, the sum  is bounded by
$$
Gen_0(\partial_x u) \partial_{z_2} Gen_\delta(g) = \partial_{z_1} Gen_0(u) \partial_{z_2} Gen_\delta(g).
$$
On the other hand, the second sum equals to 
\begin{equation}\label{sum-0}
\sum_{0 \le k \le n} {z_2^n \over k ! (n-k) !} \| \varphi^{n+1} \partial_y^k v_{\alpha'} \, \,
 \partial_y^{n+2-k} g_{\alpha - \alpha'}  \|_{0,\delta}.
\end{equation}
We follow the proof of the previous Proposition. First, for $k > 0$, this sum equals to
$$
\sum_{1 \le k \le n} {z_2^n \over k ! (n-k) !} \| \varphi^{k-1} \partial_y^{k-1} \partial_x u_{\alpha'} \, \,
 \varphi^{n+2 - k}\partial_y^{n+2-k} g_{\alpha - \alpha'}  \|_{0,\delta}.
$$
Multiplying by $e^{| \alpha | z_1}$, the corresponding sum is bounded by
$$
\partial_{z_1} Gen_0(u) \partial_{z_2}^2 Gen_\delta(g),
$$
provided that $| z_2 | \le 1$. It remains to bound the term $k=0$ in \eqref{sum-0}:
$$
{z_2^n \over n !} \| \varphi^{n+1}  v_{\alpha'} \partial_y^{n+2} g_{\alpha - \alpha'}  \|_{0,\delta}
\le \Bigl( 2 \| v_{\alpha'} \|_{0,0} + \| \alpha' u_{\alpha'} \|_{0,0} \Bigr) 
{z_2^n \over n!} \| \varphi^{n+2} \partial_y^{n+2} g_{\alpha - \alpha'} \|_{0,\delta}.
 $$
Multiplying by $e^{| \alpha | z_1}$, the corresponding sum is bounded by
$$
(Gen_0(v) + \partial_{z_1} Gen_0(u) ) \partial_{z_2}^2 Gen_\delta(g) .
$$
This leads to 
$$
\begin{aligned}
\partial_{z_2} Gen_\delta (v \partial_y g) &\le \partial_{z_1} Gen_0(u) \partial_{z_2} Gen_\delta(g) 
\\&\quad + C_0 (Gen_0(v) + \partial_{z_1} Gen_0(u)) \partial_{z_2}^2 Gen_\delta(g). 
\end{aligned}$$
The bound on $\partial_{z_2} Gen_\delta (u \partial_x g)$ is similar
which ends the proof of this Proposition.
\end{proof}


\section{Applications of generator functions to instability}


In this section, we introduce a framework to use the notion of generator functions in proving instability results. 
We first apply it to a toy model equation and then carry out the analysis for Euler equations. Later on, 
we shall use this approach to prove the instability of boundary layers.

\subsection{A toy model equation}


We will now use generator functions to prove an instability result for a toy model equation. More precisely, let us consider the following classical
Hopf equation
\beq \label{Hopf1}
\partial_t u + u  \partial_z u = \alpha u
\eeq
in the periodic setting, where $\alpha > 0$, with initial data $u(0,z) = u_1(z)$.
Assume that $u_1$ is analytic, its generator $Gen(u_1)$ is defined on some interval $[-z_0,z_0]$.
We will look for an instability solution of the form
\beq \label{Hopf2}
u = \sum_{n \ge 1} e^{n \alpha t} u_n .
\eeq
Putting (\ref{Hopf2}) into (\ref{Hopf1}), we get the recurrence relation
$$
(n - 1) \alpha u_n = - \sum_{1 \le k \le n-1} u_k  \partial_z u_{n-k} .
$$
We will prove that the series (\ref{Hopf2}) is convergent provided $e^{\alpha t}$ is small enough. The instability of \eqref{Hopf1} thus follows from the growth $e^{\alpha t}$ encoded in the series  (\ref{Hopf2}). Precisely, we have 
 
\begin{theo}
There exists a positive $T_0$ such that the series (\ref{Hopf2}) converges for every $t \le T_0$.
\end{theo}
\begin{proof}
For $n \ge 2$, we compute 
$$
 (n-1) \alpha Gen(u_n) \le \sum_{1 \le k \le n-1} Gen(u_k) \partial_z Gen(u_{n-k}) .
$$
Therefore, for $n \ge 2$, multiplying by $t^{n-2}$, we get 
$$
 (n-1) \alpha Gen(u_n) t^{n-2} \le \sum_{1 \le k \le n-1} Gen(u_k) t^{k-1} \partial_z Gen(u_{n-k}) t^{n-k-1} .
$$
Summing over $n$, we thus obtain
$$
\alpha \sum_{n \ge 2} (n-1) Gen(u_n) t^{n-2} \le \Bigl( \sum_{k \ge 1} Gen(u_k) t^{k-1} \Bigr)
\partial_z \Bigl( \sum_{k \ge 1} Gen(u_k) t^{k-1} \Bigr).
$$
Let 
$$
G(t,z) =  \sum_{k \ge 1} Gen(u_k) t^{k-1}.
$$
Then
\beq \label{eqGenerator}
\alpha \partial_t G - G  \partial_z G \le 0.
\eeq
Therefore, $G$ satisfies an Hopf-type differential inequality. 
 Note that $G$ is convex, and more precisely, all its derivatives, at all orders, are non negative for $t > 0$.
We would like to deduce from this inequality that $G$ is defined on some non vanishing
time interval. However two difficulties arise. First we do not know whether $G$ converges, and second as $G$ goes rapidly to 
$+ \infty$, the existence time for the Hopf differential equation could be zero, since the characteristics coming from the right can reach $x = 0$ in a
vanishing time.

To fix the first point, we first truncate $Gen$  and define $G_N$ to be
$$
G_N(t,z) = \sum_{k \le N} Gen_{N-k}(u_k)(z) t^{k-1} 
$$
and derive a similar Hopf inequality for $G_N$. For $k = 1,\cdots,n-1$, we note that $N-k \ge N-n$ and $N-n+k \ge N-n +1$. Hence, 
$$
 (n-1) \alpha Gen_{N-n}(u_n) \le \sum_{1 \le k \le n-1} Gen_{N-k}(u_k) \partial_z Gen_{N-n+k}(u_{n-k}) .
$$
Summing over $n$ we get
$$
\alpha \sum_{n \ge 2} (n-1) Gen_{N-n}(u_n) t^{n-2} \le \Bigl( \sum_k Gen_{N-k}(u_k) t^{k-1} \Bigr)
\partial_z \Bigl( \sum_k Gen_{N-k}(u_k) t^{k-1} \Bigr),
$$
which yields the Hopf differential inequality 
\beq \label{eqGenerator1}
\alpha \partial_t G_N - G_N  \partial_z G_N \le 0.
\eeq
Here, we know that $G_N(t,z)$ is defined for every $t$ and every $z$. Hence (\ref{eqGenerator1}) holds true for any $t \ge 0$ and
any $z \ge 0$.

It remains to fix the second issue, namely to transform the inwards characteristics into outwards ones. For this we introduce
$$
H_N(t,z) = G_N(t, \phi(t) z)
, \qquad H(t,z) = G(t,\phi(t) z) .
$$
A direct computation yields 
\beq \label{equationHN}
\begin{aligned}
\alpha \partial_t H_N  &= \alpha \partial_t G_N + \alpha z \phi'(t) \partial_z G_N 
\\&\le \Bigl( \alpha z \phi'(t) + G_N \Bigr) \partial_z G_N 
\\
&\le \Bigl(H_N + \alpha z \phi'(t) \Bigr) \phi^{-1}(t) \partial_z H_N
\end{aligned}
\eeq
We will choose $\phi$ such that $\phi(0) = 1$. Then
$$
H_N(0,z) \le H(0,z) = G(0,z) = Gen(u_1)(z).
$$ 
As $u_1$ is holomorphic, $Gen(u_1)$ is defined near $0$ on $[0,\eta_0]$ for some positive $\eta_0$.
Let 
$$
M_0 = \sup_{0 \le z \le \eta_0} Gen(u_1) (z) .
$$
We will study (\ref{equationHN}) on $[0,\eta_0]$, and focus on an interval $[0,T_N]$ such that 
$$
\sup_{0 \le z \le \eta_0, 0 \le t \le T_N} G_N(t,z) \le 2 M_0.
$$
We define $\phi(t)$ such that
$$
2 M_0 + \alpha \eta_0 \phi'(t) = - M_0,
$$
or equivalently, 
$$
\phi(t) = 1 - {3 M_0 \over \alpha \eta_0} t .
$$
We also take $ 
T \le {\alpha \eta_0 / 6 M_0}
$
so that $\phi(t) \ge 1/2$.

We now introduce the characteristics defined by
$$
X_N(0,z) = 0,
$$
$$
\partial_t X_N(t,z) = - \phi^{-1}(t) H_N(t,X_N(t,z)) -  \alpha z \phi^{-1}(t) \phi'(t) .
$$
Note that on $0$ and $\eta_0$, the characteristics of (\ref{equationHN}) are outgoing. There is no need to prescribe a boundary condition
on $H_N$. Moreover, the characteristics are defined for all $t \le T$ and do not cross since both $H_N$ and $\partial_z H_N$
are bounded for $t \le T_N$ and $0\le z \le \eta_0$ (by some constant which may depend on $N$).
We may therefore make a change of variables and define
$$
K_N(t,z) = H_N(t,X_N(t,z)).
$$
Then, by definition of $K_N$,
$$
\partial_t K_N \le 0 .
$$
In particular, as $K_N$ is positive,
$$
K_N \le \sup_{0 \le z \le \eta_0} H(0,z) = M_0.
$$
Therefore $H_n(t,z)$ is bounded by $M_0$ itself for $0 \le t \le T$ and $z \le \eta_0$.
We can let $N$ pass to the limit. This gives that $H$ exists and is well defined for $z$ and $t$ small enough.
Hence the serie (\ref{Hopf2}) converges provided $e^{\alpha t}$ is small enough.
\end{proof}


\subsection{Application to Euler equations}


We now extend the previous theorem to Euler equations on $\TT^2$. Let $U_s = (U,0)^t$ be a given shear flow. We look for solutions
of Euler equations of the form $U_s + u$, which leads to 
\beq \label{gener1}
\partial_t \omega + {\cal L} \omega = Q(u, \omega) 
\eeq
where 
$$
{\cal L} \omega = U_s \cdot \nabla \omega + u \cdot \nabla \Omega_s , \qquad Q(u, \omega) = - (u \cdot \nabla) \omega .
$$ 
Here and in what follows, velocity $u$ is computed by vorticity $\omega$ through the Biot-Savart law $u = \nabla^\perp \Delta^{-1} \omega$. We assume that the linearized operator ${\cal L}$ has an unstable eigenmode, of the form $\omega_1 e^{\alpha t}$, and
that
$$
(H1) \qquad \qquad Gen \Bigl( (\lambda + {\cal L})^{-1} f \Bigr) \le {C \over \lambda} Gen(f)  \qquad \forall \lambda, \qquad
\Re \lambda >  {3 \over 2} \Re \alpha 
$$
for some universal constant $C$. We start with
$$
\omega_1 = \Re ( \omega_1 e^{\alpha t})
$$
and iteratively construct $\omega_n$ by
\begin{equation}\label{def-Ewnnn}
(\alpha n + {\cal L}) \omega_n = \sum_{1 \le j \le n-1} Q(u_j, \omega_{n-j}) ,
\end{equation}
where $u_n = \nabla^\perp \Delta^{-1}\omega_n$ is the velocity associated to vorticity $\omega_n$.

 \begin{theo}
 Under Assumption (H1), there exists a positive time $T_0$ such that, for $t \le T_0$, the series
 $$
 \sum_{n \ge 0} e^{n \alpha t} \omega_n
 $$
 converges and is a solution of (\ref{gener1}). 
\end{theo}

As a consequence, $U_s$ is nonlinearly unstable in $L^\infty$. The proof of assumption (H1) relies
on a careful of Rayleigh equation, which will be detailed later.

\begin{proof} The proof follows a similar line as done for the toy model equation in the previous section. Indeed, by construction \eqref{def-Ewnnn} and the Assumption (H1), for $n\ge 1$, we 
have
$$
\begin{aligned}
 Gen (\omega_n) 
 &\le \frac{C}{n \alpha}  \sum_{1 \le j \le n-1} Gen (u_j\cdot \nabla \omega_{n-j} )  
 \\
 &\le \frac{C}{n \alpha}  \sum_{1 \le j \le n -1} \Big[ Gen (\omega_j) \partial_{z_1} + Gen (\omega_j) \partial_{z_2} \Big] Gen(\omega_{n-j}) ,
 \end{aligned} $$
upon recalling the elliptic estimates 
$$Gen (u_n) \le C Gen (\omega_n).$$ 
For convenience, set $G^n(z_1,z_2) = Gen (\omega_n)(z_1,z_2)$. The functions $G^n(z_1,z_2)$ are well-defined for sufficiently small $z_1,z_2$, and in addition, 
there exists some universal constant $C_0$ so that 
\beq \label{boundGn}
G^n \le {C_0 \over n} \sum_{1 \le j \le n-1} \Big( G^j \partial_{z_1} G^{n-j} + G^j \partial_{z_2} G^{n-j} \Big) .
\eeq
As in the previous section, for $N\ge 1$, we introduce the partial sum 
$$
G_N(\tau,z_1,z_2)  := \sum_{n=1}^N G^n(z_1,z_2) \tau^{n-1} ,
$$
for $\tau,z_1,z_2\ge 0$. 
Note that $G_N$ is a polynomial in $\tau$, and thus well-defined for all times $\tau\ge 0$. 
We also note that all the coefficients $G^n(z_1,z_2)$ are positive. 
In particular, $G_N(\tau, z_1,z_2)$ is positive, and so are all its time derivatives (when $z_1 > 0$ and $z_2 > 0$). 
Moreover, $G_N(\tau,z_1,z_2)$, and all its derivatives, are increasing in $N$. We also observe that, at $\tau =0$,
$$
G_N(0,z_1,z_2) = G^1(z_1,z_2),
$$
for all $N\ge 1$, and hence, 
$$
G(0,z_1,z_2) = \lim_{N\to \infty} G_N(0,z_1,z_2) = G^1(z_1,z_2).
$$ 
Next, multiplying (\ref{boundGn}) by $\tau^{n-2}$ and summing up the result, we obtain 
the following partial differential inequality
$$
\partial_\tau G_N \le C G_N \partial_{z_1} G_N + C G_N \partial_{z_2} G_N ,
$$
for all $N\ge 1$. That is, the generator function satisfies an Hopf-type equation, or more precisely an Hopf differential inequality. Thus, the theorem follows from the same argument as done in the previous section. 
\end{proof}

\part{Linear analysis}

\chapter{Rayleigh equation}\label{chapter-Rayleigh}

\def\Ray{\mathrm{Ray}}




In this chapter, we study the linearized Euler equations around a fixed, time independent shear flow 
$$
U = \left( \begin{array}{c} 
U(z) \cr
0 \cr 
\end{array} \right).
$$
 The linearization reads
$$
\partial_t v + (U \cdot \nabla v) + (v \cdot \nabla ) U + \nabla p = 0,
$$
$$
\nabla \cdot v = 0 ,
$$
with the boundary condition $v_2 = 0$ at $z = 0$. Taking advantage of the divergence-free condition, we introduce the stream function
$\psi$, in such a way that
$$
v = \nabla^\perp \psi.
$$ 
 Then, the vorticity 
 $$
 \omega = \partial_z v_1 - \partial_x v_2
 $$ 
 satisfies 
$$ 
(\partial_t + U\partial_x) \omega - U'' \partial_x \psi= 0, \qquad \Delta \psi = \omega
$$
with $\psi_{\vert_{z=0}} =0$. 

Our objective is to study the spectrum of the linearized Euler operator and to derive bounds on the corresponding semigroup.
 To proceed, we take the Fourier transform in the tangential variable, with wave number $\alpha$, and search for 
solutions of the form
$$
\psi = e^{i \alpha x} \phi(t,z) , \qquad \omega = e^{i \alpha x} \omega(t,z).
$$
It follows that the vorticity $\omega$ satisfies 
\begin{equation}\label{Euler-vort}
\partial_t \omega = i \alpha ( U '' \phi - U \omega) , \qquad \Delta_\alpha \phi = \omega
\end{equation}
 in which we denote 
$$
\Delta_\alpha = \partial_z^2 - \alpha^2.
$$
The corresponding resolvent equation is 
$$
i \alpha ( U '' \phi - U \omega) = \lambda \omega,
$$
and, setting 
$$
\lambda = -i\alpha c
$$ 
to follow the classical literature, we get 
\begin{equation}\label{Rayleigh-prb}
\left \{ 
\begin{aligned}
(U-c) \Delta_\alpha \phi - U'' \phi &= 0
\\
\phi_{\vert_{z=0}}  = 0, \qquad \lim_{z\to \infty} \phi(z) &=0 .
\end{aligned}
\right.
\end{equation} 
This is the classical Rayleigh problem. 
This chapter is devoted to the study of its spectrum and of its resolvent. 

By symmetry, it is sufficient to study the problem for $\alpha>0$. We often refer to $c$ as 
an eigenvalue of the Rayleigh problem. It is unstable if $\Im c>0$. 

Although some results do not require analyticity, we will consider analytic functions, defined on a "pencil" like domain
$$
\Gamma_{\sigma, r} = \Bigl\{ z \in \cit \quad | \quad | \Im z | \le \min( r \Re z, \sigma) \Bigr\} .
$$
We will assume that $U$ converges exponentially fast to some constant $U_+$ as $z$ goes to $+ \infty$. More precisely for $\eta > 0$
we introduce the space ${\cal A}^\eta$ of smooth functions such that the following norm
$$
\| f \|_\eta = \sup_{z \in \Gamma_{\sigma,r}} | f(z) | e^{\eta z} 
$$
is finite. We will assume that 
\begin{equation}\label{U-Aeta0}
U - U_+ \in {\cal A}^{\eta_0}
\end{equation}
for some positive $\eta_0$.
 
To begin with, we will recall classical criteria on the existence of stable and unstable modes.
Then we will study the resolvent of the Rayleigh equation. 
For the sake of convenience, we introduce the Rayleigh operator, defined by
\begin{equation}\label{def-Ray}
\Ray_\alpha(\phi): = (U-c) \Delta_\alpha \phi - U'' \phi.
\end{equation}


\section{Some inviscid stability criteria}


In the sequel we try to  determine whether there is a pair of solutions 
$(c,\phi)$ to \eqref{Rayleigh-prb} so that the complex number $c$ has a positive imaginary part, 
which corresponds to the existence of a growing solution of the linearized Euler problem.
Let us recall a few well-known instability criteria for the basic velocity profiles $U(z)$. 

\begin{lemma}[Rayleigh's inflexion-point criterium, 1880]  A necessary condition for instability is that the basic
velocity profile must have an inflection point.
\end{lemma}

\begin{proof} Indeed, assume that there is a triple $(c, \alpha,\phi)$ solving \eqref{Rayleigh-prb} with $\Im c>0$.  
 Multiplying by $\bar \phi/(U-c)$  the Rayleigh equation \eqref{Rayleigh-prb} and  integrating by parts
 yield 
 \begin{equation}\label{id-Ray}
 \int_0^\infty (|\partial_z \phi|^2 + \alpha^2 |\phi|^2) \;dz+ \int_0^\infty \frac{U''}{U-c}|\phi|^2 \;dz=0,
 \end{equation}
whose imaginary part reads 
\begin{equation}\label{id-Ray1}
\Im \Bigl( ~c \int_0^\infty \frac{U''}{|U-c|^2}|\phi|^2 \;dz \Bigr) =0.
\end{equation}
Thus, the instability condition $\Im c>0$ must imply that $U''$ changes its sign. 
This gives the Rayleigh criterium. 
Consequently, boundary layer profiles $U(z)$ with no inflection point are spectrally stable for the linearized Euler equations. 
\end{proof}

More precisely,

\begin{lemma}[Fjortoft criterium, 1950] 
A necessary condition for instability is that $U'' (U - U(z_c))<0$ somewhere in the flow, where $z_c$ is a point at which $U''(z_c) =0$.
\end{lemma}

\begin{proof} 
Indeed, assume that there is a triple $(c, \alpha,\phi)$ solving \eqref{Rayleigh-prb} with $\Im c>0$. 
Let $z_c$ be an inflection point and consider the real part of the identity \eqref{id-Ray}: 
$$
\int_0^\infty (|\partial_z \phi|^2 + \alpha^2 |\phi|^2) \;dz+ \int_0^\infty \frac{U'' (U-\Re c)}{|U-c|^2}|\phi|^2 \;dz=0.
$$
Adding to this the identity \eqref{id-Ray1}:
$$ (\Re c - U(z_c))\int_0^\infty \frac{U''}{|U-c|^2}|\phi|^2 \;dz =0,$$
we obtain 
$$ \int_0^\infty \frac{U'' (U-U(z_c))}{|U-c|^2}|\phi|^2 \;dz=-\int_0^\infty (|\partial_z\phi|^2 + \alpha^2 |\phi|^2) \;dz <0,$$
 from which the Fjortoft criterium follows. \end{proof}

There are no general sufficient conditions for instability of the Rayleigh problem. All what we know is that there are at most finitely many unstable eigenvalues and they are confined in a semicircle. Namely, we have the following

\begin{lemma}[Howard's semicircle theorem, 1961] 
Let $M = \sup_{\RR_+} U$ and $m = \inf_{\RR_+} U$. Unstable eigenvalues $c$ must satisfy 
\begin{equation}\label{Howard-semicircle}
 \Big( \Re c - \frac12 (M+m)\Big)^2 + |\Im c|^2 \le \frac14 (M-m)^2.
 \end{equation}
\end{lemma}
\begin{proof} 
Let $(c,\phi)$ be a solution to the Rayleigh problem \eqref{Rayleigh-prb}, with $\Im c>0$. 
Since $U-c$ never vanishes, the adjoint solution $\phi^* = \phi/(U-c)$ is regular and satisfies 
$$ 
\partial_z [(U-c)^2 \partial_z \phi^*] - \alpha^2 (U-c)^2 \phi^* = 0.
$$
Multiplying this equation by $\bar{\phi}^*$ and integrating the result, we obtain 
\begin{equation}\label{Howard-id} \int_0^\infty (U-c)^2 |\nabla_\alpha \phi^*|^2 \; dz = 0.\end{equation}
Taking the imaginary part of \eqref{Howard-id}, together with $\Im c>0$, yields 
$$ 
\int_0^\infty (U-\Re c) |\nabla_\alpha \phi^*|^2 = 0, 
 $$ 
which in particular asserts that $\Re c \in \mathrm{Range}(U)$. In addition, the real part of \eqref{Howard-id} reads
$$
\begin{aligned} 0 &= \int_0^\infty \Big((U-\Re c)^2 - |\Im c|^2\Big) |\nabla_\alpha\phi^*|^2 \; dz
\\
&= \int_0^\infty \Big(U^2 - |\Re c|^2 - |\Im c|^2\Big) |\nabla_\alpha\phi^*|^2 \; dz
\\
&= \int_0^\infty \Big((U-m)(U-M) + (m+M)\Re c- mM -|c|^2\Big) |\nabla_\alpha\phi^*|^2 \; dz.
\end{aligned}$$
Since $(U-m)(U-M) \le 0$, the above proves that $ (m+M)\Re c- mM -|c|^2 \ge 0$, which implies \eqref{Howard-semicircle}. 
\end{proof}

\begin{lemma}\label{lem-spectrumRayleigh} For any $\delta>0$, there are at most finitely many $c$  
so that the Rayleigh problem \eqref{Rayleigh-prb} has a nontrivial solution with $ \Im c \ge \delta$. 
 \end{lemma}

\begin{proof} Let $\Im c > 0$.
Classical arguments on ordinary differential equations indicate
 that the Rayleigh equation has two independent solutions, 
 one decaying and the other growing at an exponential rate $e^{\pm \alpha z}$ as $z$ tends to infinity. 
 Denote the decaying solution to be $\phi(z;\alpha,c)$. 
 A nontrivial solution to the Rayleigh problem exists if and only if the function 
$$
E(\alpha,c): =\phi(0;\alpha,c)
$$ 
vanishes. It follows that $\phi(z;\alpha,c)$ is analytic in the complex half plane $\Im c>0$, 
and hence there are finitely many zeros $c$ lying in  Howard's semicircle, 
with $\Im c\ge \delta>0$, of $E(\alpha,c)$ for each positive $\alpha$. This yields the Lemma. 
\end{proof}


\section{Rayleigh problem: away from critical layers}\label{sec-Rayleigh}


In this section we shall construct an exact inverse for the Rayleigh equation and thus find the complete solution to 
\begin{equation}\label{eq-Raya}
\Ray_\alpha(\phi) = f, \qquad z\ge 0,
\end{equation}
for each fixed $\alpha, c$, under the assumption that
\begin{equation}\label{assump-U} \alpha \ge \theta_0, \qquad d_c: = |\Range(U) - c|\ge \frac{\theta_1}{1+\alpha} 
\end{equation}
for arbitrary fixed positive constants $\theta_0,\theta_1$. We will also assume that $\Im c > 0$.
Recall that the temporal eigenvalue is written as $\lambda = -i\alpha c$. 
Thus, such an assumption asserts that $\lambda$ is away from the continuous spectrum of Euler equations. 
We underline the fact that in this section $c$ is not real, hence $U-c$ does not vanish and remains bounded away from $0$.
As a consequence, we may divide (\ref{eq-Raya}) by $U-c$, which leads to
\beq \label{eq-Raya2}
\Delta_\alpha  \phi - {U'' \over U - c} \phi = {f \over U - c} .
\eeq
Note that if $f$ is holomorphic on some domain $\Gamma_{\sigma,r}$, 
then $\phi$ is also holomorphic in the same domain,
by classical results on complex ordinary differential equations.
Of course, the solution to (\ref{eq-Raya2}) is not unique and defined up to an homogeneous solution.

\subsection{Homogenous solutions}


In this paragraph we detail the construction of two independent solutions $\phi_{\alpha,\pm}$ to the Rayleigh operator. We will prove the following lemma.
\begin{lemma}\label{prop-Rayleigh} 
There exist two independent solutions $\phi_{\alpha,\pm}$ to the homogenous equation 
$\Ray_\alpha(\phi) = 0$ such that 
\beq\label{prop-Rayleigh1}
 \phi_{\alpha,\pm}(z) =e^{\pm \alpha z} \Big[ 1 + \psi_{\alpha,\pm}\Big],
 \eeq
where $\psi_{\alpha,\pm} \in {\cal A}^{\eta}$, for some positive $\eta<\min\{\alpha,\eta_0\}$, with $\eta_0$ as in \eqref{U-Aeta0}, and
$$
\| \psi_{\alpha,\pm } \|_{\eta} \le C.
$$
Here, ${\cal A}^\eta$ denotes the weighted $L^\infty$ function spaces defined in section \ref{defAeta}. 
\end{lemma}

The proof of this Lemma is classical and we will now detail it.

\begin{proof} 
Let us denote 
$$
W = (\phi, \partial_z\phi)^{tr}.
$$ 
The homogenous equation  becomes 
\begin{equation}\label{eqs-W} 
W' = A W, \qquad A = \begin{pmatrix} 0&1 \\ \alpha^2 + \frac{U''}{U-c} &0\end{pmatrix} .
\end{equation}
We denote by $A_+$ the limit of $A(z)$ as $z\to \infty$. Note that
$$
 A_+= \begin{pmatrix} 0&1 \\ \alpha^2 &0\end{pmatrix} .
$$
 Eigenvalues of $A_+$ are $\mu_\pm = \pm \alpha$. Let $V_\pm= (1,\pm \alpha)^{tr}$ be the associated eigenvectors. 
 We now want to construct two independent solutions of \eqref{eqs-W} of the form 
$$ W_\pm = e^{\mu_\pm z} V(z) , \qquad \lim_{z\to \infty}V(z) = V_\pm.$$
Let us first consider the $(-)$ case: $\mu_-=-\alpha$. Plugging the above Ansatz into \eqref{eqs-W}, $V(z)$ solves 
$$ V' = (A_+ - \mu_-) V + (A - A_+) V.$$
Since the eigenvalues of $A_+ - \mu_-$ are $0$ and $2\alpha$, a direct calculation leads to 
$$
\|e^{(A_+ - \mu_-)z} \|\le C_0 \alpha e^{- 2\alpha |z|}, \qquad \forall ~z\le 0.
$$
Here, $\|\cdot \|$ denotes the standard matrix norm. Thanks to the decay assumption on the boundary layer profile $U(z)$, 
there holds 
$$
\|A (z)- A_+\| \le C e^{-\eta_0 |z|}.
$$
The Duhamel formula yields 
$$ 
V(z) = V_- - \int_z^\infty e^{(A_+ - \mu_-) (z-y)}  (A(y) - A_+) V(y) \; dy ,
$$
the integral being taken over the half line $z+ \rit_+$. Denote by $TV(z)$ the right hand side of the above identity. 
Let $M > 0$ and let ${\cal X}^{\eta}$ be the space of functions defined like in \ref{defAeta} by its norm
$$
\| f \|_{\cal X}  = \sup_{z, \Re z > M, | \Im z | \le \min(\sigma \Re z, \sigma r)} |f(z)| e^{\eta \Re z} .
$$
We shall show that $T$ is a contractive map 
from ${\cal X}^{\eta}$ into itself, provided $M$ is  sufficiently large and for some $\eta<\min\{ \alpha,\eta_0\}$.  
Indeed, for $V_1,V_2 \in {\cal X}^{\eta}$, we get 
$$\begin{aligned}
 |TV_1(z) - TV_2(z)| 
 &\le \int_z^\infty \|e^{(A_+ - \mu_-) (z-y)} \| \|(A(y) - A_+)\| \|V_1- V_2\|_{\eta}  e^{-\eta y} \; dy 
 \\
 &\le C_0 \alpha \|V_1- V_2\|_{\eta} \int_z^\infty e^{ - 2\alpha|y-z|}e^{- (\eta+ \eta_0) y}  \; dy 
 \\
 &\le C_0 \alpha e^{- (\eta+ \eta_0) z} \|V_1- V_2\|_{\eta} \int_z^\infty e^{ - 2\alpha|y-z|}  \; dy 
 \\
 &\le C e^{-\eta z}  e^{-\eta_0 M} \|V_1- V_2\|_{\eta}
  \end{aligned}$$
for all $z\ge M$. Taking $M$ large enough so that $C e^{-\eta_0 M} <1/2$, the map $T$ is contractive. 
Hence, there exists a solution $V\in {\cal X}^{\eta}$ such that $V = TV$, or equivalently $e^{-\alpha z}V(z)$ 
solves \eqref{eqs-W} on $[M,\infty)$. By the standard ODE theory, the solution $V$ on $[M,\infty)$ 
can be extended to be a global solution on $[0,\infty)$, since the right-hand side of \eqref{eqs-W} is uniformly bounded.  

Similarly, with the Ansatz solution $W_+ = e^{\mu_+ z} V(z)$, there holds 
$$ V' = (A_+ - \mu_+) V + (A - A_+) V,$$
in which the matrix $A_+ - \mu_+$ has only nonpositive eigenvalues. It follows that 
$\|e^{(A_+ - \mu_+)z} \|\le C_0\alpha e^{-2\alpha|z|}$ for $z\ge 0$. 
The solution is then constructed via the Duhamel formula: 
$$ 
V(z) = V_+ + \int_M^z e^{(A_+ - \mu_+) (z-y)}  (A(y) - A_+) V(y) \; dy .
$$
A similar argument as above shows that such a solution exists and satisfies the claimed bound. 
\end{proof}


\subsection{Exact Rayleigh solver} 


In this section we explicitly construct an inverse $Ray^{-1}$ of the Rayleigh operator. Note that at this point
we do not take care of the boundary condition at $z = 0$.

Let us introduce the (inviscid) Evans function
\begin{equation}\label{def-Evans-inviscid} 
E(\alpha, c): = \alpha^{-1} W[\phi_{\alpha,-},\phi_{\alpha,+}](0) .
\end{equation}
For real arguments $x$ and $z$, a Green kernel $G_{\alpha}(x,z)$ of the $\Ray_\alpha$ operator can be defined by 
$$
G_{\alpha}(x,z) = \frac{1}{ \alpha E(\alpha,c) (U(x)-c) }\left\{ \begin{array}{rrr}  \phi_{\alpha,-}(z) \phi_{\alpha,+}(x), 
\quad \mbox{if}\quad z>x,\\
 \phi_{\alpha,-}(x) \phi_{\alpha,+}(z), \quad \mbox{if}\quad z<x.\end{array}\right.
$$ 
We denote by $G_\alpha^-$ the expression of $G_\alpha$ for $x < z$ and by $G_\alpha^+$ the expression of $G_\alpha$ for
$x > z$.
The exact Rayleigh solver can then be defined by 
\begin{equation}\label{def-RaySa}
\begin{aligned}
Ray^{-1}_\alpha(f) (z)  &: =  \int_0^\infty G_{\alpha}(x,z) f(x) dx.
\end{aligned}
\end{equation}
The following lemma gives bounds on this solution.

\begin{lemma}\label{lem-Raysolver}
For any positive constant $\eta <\alpha$ and for each $f \in {\cal A}^\eta$, there holds on $z \ge 0$:
\begin{equation}\label{est-Ray1} 
|\partial_z^{k} Ray^{-1}_\alpha(f) (z) |\le C \alpha^{k-2} d_c^{-1} E(\alpha,c)^{-1}  \| f\|_{\eta} e^{-\eta z},
\end{equation}
for some positive constant $C$, and $k = 0, 1$, and $d_c$ defined as in \eqref{assump-U}. In addition, we have 
\begin{equation}\label{est-Ray2}
|\Delta_\alpha Ray^{-1}_\alpha(f) (z) | \le  C \Big( 1+ \alpha^{-2} d_c^{-1} E(\alpha,c)^{-1} \Big) d_c^{-1} \| f\|_{\eta} e^{-\eta z} .
\end{equation}
\end{lemma}

\begin{remark}{\em Note that we "gain" bounds on two derivatives, 
which is natural since Rayleigh equation is a second order elliptic equation. 
We also gain $\alpha^{-2}$ for the Rayleigh solution, when $\alpha$ is sufficiently large. 
}
\end{remark}

\begin{proof} By Lemma \ref{prop-Rayleigh} and \eqref{assump-U}, we have
\begin{equation}\label{RGGG}
|\partial_z^k G_{\alpha}(x,z)| \le C \alpha^{k-1} d_c^{-1}E(\alpha,c)^{-1} e^{-\alpha|x-z|},
\end{equation}
for all $x,z\ge 0$, and for $k = 0$ or $k = 1$. This yields 
$$
\begin{aligned}
 \Big| \partial_z^k \int_0^\infty G_{\alpha}(x,z) f(x) dx  \Big| 
& \le C \alpha^{k-1} d_c^{-1} E(\alpha,c)^{-1}  \|f\|_{\eta} \int_0^{\infty} 
e^{-\alpha |x-z|}e^{-\eta x}\; dx
\\& \le C \alpha^{k-2} d_c^{-1} E(\alpha,c)^{-1}  \|f\|_{\eta} e^{-\eta z}.
\end{aligned}$$ 
This proves the first estimate stated in the lemma. 
For the second estimate, using the equation $\Ray_\alpha (\phi) = 0$, we write 
$$
\Delta_\alpha Ray^{-1}_\alpha(f) (z) = \frac{f}{U-c} +  \frac{U''}{U-c}  Ray^{-1}_\alpha(f) (z) .
$$
This proves the second estimate in the lemma.   
\end{proof}

%


\subsection{Behavior for large $\alpha$}


\begin{lemma}\label{lem-Ray-la1} 
For sufficiently large $\alpha$, there are no nontrivial bounded solutions $(\alpha, c, \phi)$ 
to the Rayleigh problem when $\Im c \gg e^{- |\alpha|}$.
\end{lemma}
\begin{proof}
We study the Rayleigh equation when $\alpha \gg 1$. We introduce the change of variables
$$ \tilde z = \alpha z, \qquad \phi(z) = \alpha^{-2}\tilde \phi (\alpha z).$$
Then, $\tilde \phi$ solves the scaled Rayleigh equation
$$ \Ray_{sc}(\tilde \phi): = (U-c) (\partial_{\tilde z}^2 - 1) \tilde \phi - \alpha^{-2}U'' \tilde \phi =0$$
in which $U = U(\alpha^{-1}\tilde z)$ and $U'' = U''(\alpha^{-1}\tilde z)$. 
The inverse of $\Ray_{sc}(\cdot)$ is constructed by induction, starting from the inverse of $(U-c) (\partial_{\tilde z}^2 - 1)$. 
The iteration operator is defined by  
$$ 
T \tilde \phi : = \alpha^{-2} \Big[ (U-c) (\partial_{\tilde z}^2 - 1)\Big]^{-1} \circ U'' \tilde \phi .
$$ 
It suffices to show that $T$ is contractive with respect to the $e^{\tilde z}$ or $e^{-\tilde z}$ weighted sup norm, 
denoted by $\|\cdot \|_{L^\infty_\pm}$, respectively for decaying or growing solutions. 
Indeed, recalling that 
$$
(\partial_{\tilde z}^2 - 1)^{-1} f= e^{-|\tilde z|} \star f,
$$ 
we check that  
$$\begin{aligned}
 |T \tilde \phi  - T\tilde \psi|(z) 
 &\le  \alpha^{-2}  \int_0^\infty  \Big| \frac{U''}{U-c}\Big| e^{-|\tilde x - \tilde z|} e^{\pm|\tilde x|} \| \tilde \phi  - \tilde \psi\|_{L^\infty_\pm}\; d\tilde x
\\
& \le  \alpha^{-2} e^{\pm|\tilde z|}  \| \tilde \phi  - \tilde \psi\|_{L^\infty_\pm}
 \int_0^\infty  \Big| \frac{U''(\alpha^{-1}\tilde x)}{U(\alpha^{-1}\tilde x)-c}\Big| \; d\tilde x
\\
& \le  \alpha^{-1} e^{\pm|\tilde z|}  \| \tilde \phi  - \tilde \psi\|_{L^\infty_\pm} \int_0^\infty  \Big| \frac{U''(x)}{U(x)-c}\Big| \; dx
 \end{aligned}$$ 
in which the triangle inequality was used to deduce $e^{-|\tilde x - \tilde z|} e^{\pm|\tilde x|} \le e^{\pm|\tilde z|} $. 
As for the integral term, when $c$ is away from the range of $U$, it follows that 
$$
 \int_0^\infty  \Big| \frac{U''(x)}{U(x)-c}\Big| \; dx \le C_0 \int_0^\infty  |U''|\; dx \le C.
$$ 
Therefore, if $\alpha$ is small enough, $T$ is contractive.

Note that if $c$ is close to the range of $U$, then 
$$
 \int_0^\infty  \Big| \frac{U''(x)}{U(x)-c}\Big| \; dx \le C + C | \log \Im c | .
 $$
 Therefore, for small $c$,
$$\begin{aligned}
 |T \tilde \phi  - T\tilde \psi|(z) 
& \le C \alpha^{-1} |\log \Im c| e^{\pm|\tilde z|}  \| \tilde \phi  - \tilde \psi\|_{L^\infty_\pm}.
 \end{aligned}$$ 
Hence, $T$ is contractive as long as $\alpha^{-1}\log \Im c$ is sufficiently small. 
Similarly, the same bound holds for derivatives, since the derivative of the Green kernel of $\partial_{\tilde z}^2 - 1$ 
satisfies the same bound. This yields two independent Rayleigh solutions of the form 
\begin{equation}\label{phi0-la} \partial_{\tilde z}^k\tilde \phi^0_{s,\pm} (\tilde z) 
= (-1)^k e^{\pm \tilde z} \Big( 1 + \cO(\alpha^{-1}\log \Im c) \Big), \qquad k=0,1.\end{equation}
In particular, this shows that there are no non-trivial bounded solutions to the Rayleigh problem with the zero boundary condition, 
as long as $\alpha^{-1}\log \Im c\ll1$, since $\tilde \phi_{s,-}^0(0)\not =0$. This yields the lemma. \end{proof}


\section{Rayleigh problem: near critical layers}\label{sec-Rayleigh-cr}


In this part, we construct an exact inverse for the Rayleigh operator $\Ray_\alpha$ 
and thus solve the inhomogenous Rayleigh problem 
\begin{equation}\label{Ray-smalla}
\Ray_\alpha(\phi) = f, \qquad z\ge 0,
\end{equation}
in the presence of critical layers. A critical layer is an area in $z$ where $U - c$ almost vanishes. 
More precisely, it is by definition a small
neighborhood of $z_c$ where $U(z_c) = c$. Note that $z_c$ may be a complex number.
Precisely, we study in this section the case when $$\Im c\ll1.$$
We assume that there exists a complex number $z_c$, with
small imaginary part, such that
$$
U(z_c) = c.
$$
Note that $z_c$ may not be unique. In this case, we have to repeat the following construction near each critical layer.

For $\Im c$ small enough, $z_c \in \Gamma_{\sigma,r}$. The Rayleigh solutions are no longer smooth in $\Gamma_{\sigma,r}$ and 
have a singularity at $z = z_c$.

We shall consider three cases: $\alpha\ll 1$, $\alpha \approx 1$, and $\alpha \gg1$. 
In what follows, we shall first focus on the first case, leaving the last two cases to the last two subsections of this section. 
For the first case, we construct the inverse of the Rayleigh operator 
as a perturbation of the inverse of the Rayleigh operator $\Ray_0$ when $\alpha =0$. This is constructed 
via an explicit Green function. 
We then use this inverse to construct an approximate inverse to the $\Ray_\alpha$ operator 
through the construction  of an approximate Green function. 
Finally, the construction of the exact inverse of $\Ray_\alpha$ follows by  an iterative procedure.


\subsection{Functional spaces}


We shall use the function spaces $X_{\eta,p}$, for $p\ge 0, \eta > 0$, 
to denote the spaces consisting of functions $f = f(z)$ that are $C^p$ smooth and satisfy 
$$
|\partial_z^k f(z)| \le C e^{-\eta z} (1 + |z - z_c |^{-k}) 
$$ 
for all $z\ge 0$ and all $k = 0,\cdots,p$. 
The best constant that satisfies the above inequality defines a norm, denoted by $\|\cdot\|_{\eta,p}$, 
on the Banach space $X_{\eta,p}$. We denote 
$X_\eta, \|\cdot \|_\eta$ in place of $X_{\eta,0}, \|\cdot \|_{\eta,0}$, respectively.

Next,  the function space $(Y_{\eta,p},\|\cdot \|_{Y_{\eta,p}})$ 
consists of all functions $f$ that decay exponentially at the rate $e^{-\eta z}$ and satisfy
$$
|f(z)| \le C, \quad | \partial_z f(z) | \le C (1 + | \log (z - z_c) | ) , \quad  | \partial_z^k f(z) | \le C (1 + | z - z_c |^{1 - k} )
$$
for $z$ near $z_c$ and for $0\le k\le p$.


\subsection{Rayleigh equation: $\alpha = 0$}


As mentioned, we begin with the Rayleigh operator $\Ray_0$ when
$\alpha = 0$ and construct the inverse of $\Ray_0$. More precisely, we will construct the Green function of $\Ray_0$ and solve  
\begin{equation}\label{Ray0} 
\Ray_0 (\phi) = (U-c) \partial_z^2 \phi - U'' \phi = f.
\end{equation}
We first prove the following lemma.

\begin{lemma}\label{lem-defphi012} 
Assume that $\Im c \not =0$. There are two independent solutions $\phi_{1,0}$ and
$\phi_{2,0}$ of $\Ray_0(\phi) =0$ with unit Wronskian determinant 
$$ 
W(\phi_{1,0}, \phi_{2,0}) := \partial_z \phi_{2,0} \phi_{1,0} - \phi_{2,0} \partial_z \phi_{1,0} = 1.
$$
Furthermore, there are holomorphic functions $P_1(z), P_2(z), Q(z)$ 
with $P_1(z_c) = P_2(z_c) = 1$ and $Q(z_c)\not=0$ so that the asymptotic descriptions 
\begin{equation}\label{asy-phi012} 
\phi_{1,0}(z) = (z-z_c) P_1(z) ,\qquad \phi_{2,0}(z) = P_2(z) + Q(z) (z-z_c) \log (z-z_c)
\end{equation}
hold for $z$ near $z_c$, and  \begin{equation}\label{decay-phi012} 
 \phi_{1,0}(z) - V_+ \in {\cal A}^{\eta_0} , \qquad \partial_z \phi_{2,0}(z)  - \frac{1}{V_+}   \in  {\cal A}^{\eta_1}
\end{equation}
for some positive constants $C,\eta_0$, for $V_+ = U_+ - c$, and for any $\eta_1<\eta_0$.
 \end{lemma}
\begin{proof} 
First, we observe that 
$$ 
\phi_{1,0}(z) = U(z)-c
$$
 is an exact solution of $\Ray_0(\phi) =0$, defined and holomorphic on $\Gamma_{\sigma,r}$. 
 In addition, the claimed asymptotic expansion for $\phi_{1,0}$ clearly holds for $z$ near $z_c$ since $U(z_c) = c$. 
 
 We then construct a second particular solution $\phi_{2,0}$, imposing the Wronskian determinant to be one
$$
W[\phi_{1,0},\phi_{2,0}] =  \partial_z \phi_{2,0} \phi_{1,0} - \phi_{2,0} \partial_z \phi_{1,0}  = 1.
$$
From this, the variation-of-constant method $\phi_{2,0} (z)= C(z) \phi_{1,0}(z)$  yields
$$
  \phi_{1,0} C \partial_z  \phi_{1,0} 
+ \ \phi_{1,0}^2 \partial_z C -\partial_z  \phi_{1,0} C  \phi_{1,0}  = 1.
$$
This gives 
$$
\partial_z C(z) = {1 \over \phi^2_{1,0}(z)} = {1 \over (U(z) - c)^2} .
$$ 
Note that this equation is singular at $z = z_c$, with a second order pole.
More precisely,
$$
{1 \over (U(c) - U(z_c))^2}
= {1 \over ( U'(z_c) (z - z_c) + U''(z_c) (z - z_c)^2 / 2 + ...)^2}
$$
$$
= {1 \over U'(z_c)^2 (z-z_c)^2} 
- {U''(z_c) \over U'(z_c)^3} {1 \over z - z_c} + holomorphic .
$$
Hence,
\begin{equation} \label{defiphi2bis}
\phi_{2,0} = - {U(z) - c \over U'(z_c)^2 (z - z_c)}
- {U''(z_c) \over U'(z_c)^3} (U(z) -  c)\log (z - z_c)   + holomorphic .
\end{equation}
The Lemma then follows from the explicit expression (\ref{defiphi2bis}) of $\phi_{2,0}$.
\end{proof}

To go back to $\Gamma_{\sigma,r}$ we have to make
a choice of the logarithm. We choose to define it on $\mathbb{C} - \{ z_c + \mathbb{R}^-  \}$.
Then $\phi_{2,0}$ is holomorphic on 
$$
\tilde \Gamma = \Gamma_{\sigma,r} \cap \Bigl(  \mathbb{C} - \{ z_c + \mathbb{R}^-  \} \Bigr) .
$$
Note that because of the singularity at $z_c$, $\phi_{2,0}$ can not be made holomorphic on 
$\Gamma_{\sigma,r}$.

In particular if $\Im z_c = 0$, $\phi_{2,0}$ is holomorphic in $z$ excepted on the half line $z_c + \mathbb{R}^- $.
For $z \in \mathbb{R}$, $\phi_{2,0}$ is holomorphic as a function of $c$ excepted if $z - z_c$ is real and negative,
namely excepted if $z < z_c$. 
For a fixed $z$, $\phi_{2,0}$ is an holomorphic function of $c$ provided $z_c$ does not cross
$\mathbb{R}^+$, and provided $z - z_c$ does not cross $\mathbb{R}^-$.

Let $\phi_{1,0},\phi_{2,0}$ be constructed as in Lemma \ref{lem-defphi012}. Then 
for real arguments $x$ and $z$, the Green function $G_{R,0}(x,z)$ of the $\Ray_0$ operator can be explicitly defined by 
$$
G_{R,0}(x,z) = \left\{ \begin{array}{rrr} (U(x)-c)^{-1} \phi_{1,0}(z) \phi_{2,0}(x), 
\quad \mbox{if}\quad z>x,\\
(U(x)-c)^{-1} \phi_{1,0}(x) \phi_{2,0}(z), \quad \mbox{if}\quad z<x.\end{array}\right.
$$ 
We will denote by $G_{R,0}^\pm$ the expressions of $G_{R,0}$ for $x > z$ and $x < z$. 
Here we note that $c$ is complex with $\Im c \not=0$ and so the Green function $G_{R,0}(x,z)$ is a well-defined function in $(x,z)$,
 continuous across $x=z$, and its first derivative has a jump across $x=z$. Let us now introduce the inverse of $\Ray_0$ as 
\begin{equation}\label{def-RayS0}
\begin{aligned}
RaySolver_0(f) (z)  &: =  \int_0^{+\infty} G_{R,0}(x,z) f(x) dx.
\end{aligned}
\end{equation}
As $G_{R,0}^\pm$ are holomorphic,
$$
RaySolver_0(f) (z)   =  \int_{\gamma^-} G_{R,0}^-(x,z) f(x) dx +  \int_{\gamma^+} G_{R,0}^+(x,z) f(x) dx
$$
where $\gamma^-$ is a smooth path between $0$ and $z$ and $\gamma_+$ is a smooth path between $z$ and $+ \infty$.
Using such an expression, $RaySolver_0(f)$ can be extended to $\tilde \Gamma$ is $f$ is holomorphic on $\Gamma_{\sigma,r}$
or holomorphic on $\tilde \Gamma$.

The following Lemma asserts that the operator $RaySolver_0(\cdot)$ is in fact well-defined from $X_0^\eta$ to $Y^0_2$, 
which in particular shows that $RaySolver_0(\cdot)$ gains two derivatives, but losses the fast decay at infinity.  

More precisely, if $f$ is bounded near $z_c$,  then, assuming that $\phi$ is bounded, we obtain that
$\partial_z^2 \phi$ behaves like $(z - z_c)^{-1}$ near $z_c$. Hence $\partial_z \phi$
behaves like $\log ( z -z_c)$ near $z_c$, and $\phi$ is bounded. Hence if $f \in X_0^\eta$ then 
we expect $u$ to be in $Y_2^0$. Note that if $f$ decays exponentially, then $\partial_z^2 u$ also decays exponentially,
but this does not imply that $u$ goes to $0$ exponentially. In fact it may converge exponentially fast to some
non vanishing constant. We therefore loose the exponential decay.

%


%

\begin{lemma}\label{lem-RayS0} 
Assume that $\Im c \not =0$. For any $f\in {X_0^\eta}$,  the function $RaySolver_0(f)$ is a solution to the Rayleigh problem \eqref{Ray0}, defined on $\tilde \Gamma$. In addition, $RaySolver_0(f) \in Y^0_2$, and there holds  
$$
\| RaySolver_0(f)\|_{Y^0_2} \le C (1+|\log \Im c|) \|f\|_{{X_0^\eta}},
$$ 
for some universal constant $C$. 
\end{lemma}
\begin{proof} As long as it is well-defined, the function $RaySolver_0(f)(z)$ solves the equation \eqref{Ray0} at once by a direct calculation, upon noting that 
$$ \Ray_0 (G_{R,0}(x,z) ) = \delta_x(z),$$
for each fixed $x$. 

Next, by scaling, we assume that $ \| f\|_{X_0^\eta} = 1$. By Lemma \ref{lem-defphi012}, it is clear that $\phi_{1,0}(z)$ and $\phi_{2,0}(z)/(1+z)$
  are uniformly bounded. Thus, by direct computations, we have 
\begin{equation}\label{est-Gr0}|G_{R,0}(x,z)| \le  C \max\{ (1+x), |x-z_c|^{-1} \}.\end{equation}
That is, $G_{R,0}(x,z)$ grows linearly in $x$ for large $x$ and has a singularity of order $|x-z_c|^{-1}$ 
when $x$ is near $z_c$, for arbitrary $z \ge 0$.  Since $|f(z)|\le e^{-\eta z}$, the integral \eqref{def-RayS0} is well-defined and satisfies 
$$
|RaySolver_0(f) (z)| \le  C \int_0^\infty e^{-\eta x} \max\{ (1+x), |x-z_c|^{-1} \}  \; dx
$$
$$
\le C (1+|\log \Im c|),
$$
in which we used the fact that $\Im z_c \approx \Im c$. 
  
Finally,  for derivatives, we need to check the order of singularities for $z$ near $z_c$. 
We note that $|\partial_z \phi_{2,0}| \le C (1+|\log(z-z_c)|)$, and hence 
  $$|\partial_zG_{R,0}(x,z)| \le  C \max\{ (1+x), |x-z_c|^{-1} \} (1+|\log(z-z_c)|).$$
Thus, $\partial_z RaySolver_0(f)(z)$ behaves as $1+|\log(z-z_c)|$ near the critical layer. In addition, from the $\Ray_0$ equation, 
we have 
\begin{equation}\label{identity-R0f} 
\partial_z^2 (RaySolver_0(f)) = \frac{U''}{U-c} RaySolver_0(f) + \frac{f}{U-c}.
\end{equation}
This proves that $RaySolver_0(f) \in Y_2^0$ by definition of the function space $Y_2^0$. \end{proof}

\begin{lemma}\label{lem-derRayS0}Assume that $\Im c \not =0$. Let $p$ be in $\{0,1,2\}$. For any $f \in X_p^\eta$, we have 
$$
\begin{aligned}
 \| RaySolver_0(f) \|_{Y_{p+2}^0} \le C \|f\|_{X_p^\eta}(1+|\log (\Im c)| )
 \end{aligned}$$ 
\end{lemma}
\begin{proof} This is Lemma \ref{lem-RayS0} when $p=0$. When $p=1$ or $2$, the lemma follows directly from the identity \eqref{identity-R0f}.
\end{proof}


\subsection{Approximate Green function for Rayleigh: $\alpha \ll1$}


Let $\phi_{1,0}$ and $\phi_{2,0}$ be the two solutions of $\Ray_0(\phi) = 0$ 
that are constructed above in Lemma \ref{lem-defphi012}. We note that  solutions of $\Ray_0(\phi) = f$ tend to a constant
value as $z \to + \infty$ since $\phi_{1,0} \to U_+-c$. 
We now construct  solutions to the Rayleigh equation \eqref{Ray-smalla} with $\alpha \not=0$.
By looking at the spatially asymptotic limit of the Rayleigh equation, 
we observe that there are two normal mode solutions of \eqref{Ray-smalla}
 whose behaviors are as $e^{\pm\alpha z}$ at infinity. 
 In order to study the mode which behaves like $e^{-\alpha z}$ we introduce
\begin{equation}\label{def-phia12}
\phi_{1,\alpha } = \phi_{1,0} e^{-\alpha z} ,\qquad \phi_{2,\alpha} = \phi_{2,0} e^{-\alpha z}.
\end{equation}
A direct calculation shows that the Wronskian determinant 
$$
W[\phi_{1,\alpha},\phi_{2,\alpha}] =  \partial_z \phi_{2,\alpha} \phi_{1,\alpha} - \phi_{2,\alpha} \partial_z \phi_{1,\alpha}  = e^{-2\alpha z}
$$ is  non zero. In addition, we can check that 
\begin{equation}\label{Ray-phia12}
\Ray_\alpha(\phi_{j,\alpha}) = - 2 \alpha (U-c) \partial_z \phi_{j,0} e^{-\alpha z} 
\end{equation}
We are then led to introduce an approximate Green function $G_{R,\alpha}(x,z)$ defined by 
$$
G_{R,\alpha}(x,z) = \left\{ \begin{array}{rrr} (U(x)-c)^{-1} e^{-\alpha (z-x)}  \phi_{1,0}(z) \phi_{2,0}(x), \quad \mbox{if}\quad z>x\\
(U(x)-c)^{-1} e^{-\alpha (z-x)}  \phi_{1,0}(x) \phi_{2,0}(z), \quad \mbox{if}\quad z< x.\end{array}\right.
$$
Again, like $G_{R,0}(x,z)$, the Green function $G_{R,\alpha}(x,z)$ is ``singular'' near $z = z_c$ with two sources 
of singularities: one arising from $1/ (U(x) - c)$ for $x$ near $ z_c$ and the other coming from the $(z - z_c) \log (z- z_c)$, singularity
of $\phi_{2,0}(z)$. By a view of \eqref{Ray-phia12}, it is clear that 
\begin{equation}\label{id-Gxz}
\Ray_\alpha (G_{R,\alpha}(x,z)) = \delta_{x} -2\alpha (U- c) E_{R,\alpha}(x,z),
\end{equation}
for each fixed $x$. Here the error term $E_{R,\alpha}(x,z)$ is defined by  
$$
E_{R,\alpha}(x,z) = 
 \left\{ \begin{array}{rrr}
 (U(x)-c)^{-1} e^{-\alpha (z-x)} \partial_z \phi_{1,0}(z) \phi_{2,0}(x), \quad \mbox{if}\quad z>x\\
 (U(x)-c)^{-1}e^{-\alpha (z-x)}\  \phi_{1,0}(x) \partial_z \phi_{2,0}(z), \quad \mbox{if}\quad z< x.\end{array}\right.   .
$$
We then introduce an approximate inverse of the operator $\Ray_\alpha$, defined by
\begin{equation}\label{def-RaySa}
RaySolver_\alpha(f)(z) 
:= \int_0^{+\infty} G_{R,\alpha}(x,z) f(x) dx
\end{equation}
and the error remainder 
\begin{equation}\label{def-ErrR}
Err_{R,\alpha}(f)(z) := 2\alpha (U(z) - c) \int_0^{+\infty} E_{R,\alpha}(x,z) f(x) dx
\end{equation}

%

\begin{lemma}\label{lem-RaySa} Assume that $\Im c \not =0$, and let $p$ be $0,1,$ or $2$. For any $f\in {X_p^\eta}$,  with $\alpha<\eta$, the function $RaySolver_\alpha(f)$ is well-defined in $Y^\alpha_{p+2}$, satisfying 
$$ \Ray_\alpha(RaySolver_\alpha(f)) = f + Err_{R,\alpha}(f).$$
Furthermore, there hold  
\begin{equation}\label{est-RaySa}
\| RaySolver_\alpha(f)\|_{Y^\alpha_{p+2}} \le C (1+|\log \Im c|) \|f\|_{{X_p^\eta}},
\end{equation}
and 
\begin{equation}\label{est-ErrRa} 
\|Err_{R,\alpha}(f)\|_{Y_{p}^\eta} \le C\alpha  (1+|\log (\Im c)|)  \|f\|_{X_p^\eta} ,
\end{equation}
for some universal constant $C$. 
\end{lemma}
\begin{proof}  The proof follows similarly to that of Lemmas \ref{lem-RayS0} and \ref{lem-derRayS0}. In fact, the proof of the right order of singularities near the critical layer follows identically from that of Lemmas \ref{lem-RayS0} and \ref{lem-derRayS0}.  

Let us check the right behavior at infinity. Consider the case $p=0$ and assume $\|f \|_{X_0^\eta} =1$. Similarly to the estimate \eqref{est-Gr0}, Lemma \ref{lem-defphi012} and the definition of $G_{R,\alpha}$ yield
  $$|G_{R,\alpha}(x,z)| \le  C e^{-\alpha|x-z|} \max\{ (1+x), |x-z_c|^{-1} \}.$$
Hence, by definition, 
$$ |RaySolver_\alpha (f)(z) |\le C e^{-\alpha z} \int_0^\infty e^{\alpha x} e^{-\eta x}\max\{ (1+x), |x-z_c|^{-1} \}\; dx $$ 
which is clearly bounded by $C(1+|\log \Im c|) e^{-\alpha z}$.  This proves the right exponential decay of $RaySolver_\alpha (f)(z)$ at infinity, for all $f \in X_0^\eta$.

Next, by definition, we have 
$$\begin{aligned}
Err_{R,\alpha}(f)(z) &= -2\alpha (U(z) - c)  \partial_z \phi_{2,0}(z)   \int_z^\infty  e^{-\alpha (z-x)} \phi_{1,0}(x)\frac{f(x)}{U(x)-c}\; dx  
\\ & \quad - 2\alpha (U(z) - c)  \partial_z \phi_{1,0}(z) \int_0^ze^{-\alpha (z-x)} \phi_{2,0}(x) {f(x) \over U(x) - c} dx .
\end{aligned}$$
Since $f(z), \partial_z \phi_{1,0}(z)$ decay exponentially at infinity, the exponential decay of $Err_{R,\alpha}(f)(z)$ follows directly from the above integral representation. It remains to check the order of singularity near the critical layer. Clearly, for bounded $z$, we have 
$$ |E_{R,\alpha}(x,z) | \le C (1+ |\log (z-z_c)| ) e^{\alpha x} \max \{ 1, |x-z|^{-1}\} . $$
The lemma then follows at once, using the extra factor of $U-c$ in the front of the integral \eqref{def-ErrR} to bound the $\log (z-z_c)$ factor. The estimates for derivatives follow similarly. \end{proof}

%
%
%

\subsection{The exact solver for Rayleigh: $\alpha \ll1$}


We are ready to construct the exact solver for the Rayleigh problem. Precisely, we obtain the following.

\begin{proposition}\label{prop-exactRayS}
 Let $p$ be in $\{0,1,2\}$ and $\eta>0$. Assume that $\Im c \not =0$ and $\alpha |\log \Im c|$ is sufficiently small. Then, there exists an operator $RaySolver_{\alpha,\infty} (\cdot) $ from $X_p^\eta$ to $Y^\alpha_{p+2}$ (defined by \eqref{def-exactRayS}) so that 
\begin{equation}\label{eqs-RaySolver}
\begin{aligned}
 \Ray_\alpha (RaySolver_{\alpha,\infty} (f)) &= f.
\end{aligned} \end{equation} 
In addition, there holds 
$$\| RaySolver_{\alpha,\infty}(f)\|_{Y^\alpha_{p+2}} \le C \|f\|_{X_p^\eta}(1+|\log (\Im c)|) ,$$
for all $f \in X_p^\eta$.
\end{proposition}

\begin{proof} The proof follows by iteration. Let us denote 
$$ S_0(z) : = RaySolver_\alpha(f)(z),\qquad E_0(z): = Err_{R,\alpha}(f)(z).$$
It then follows that $\Ray_\alpha (S_0) (z) = f(z) + E_0(z)$. Inductively, we define
$$ S_n(z): = - RaySolver_\alpha(E_{n-1})(z), \qquad E_n(z): = - Err_{R,\alpha}(E_{n-1})(z) ,$$
for $n\ge 1$. It is then clear that for all $n\ge 1$, 
\begin{equation}\label{eqs-appSn} \Ray_\alpha  \Big( \sum_{k=0}^n S_k(z)\Big) = f(z) + E_n(z).\end{equation}
This leads us to introduce the exact solver for Rayleigh defined by 
\begin{equation}\label{def-exactRayS} RaySolver_{\alpha,\infty} (f) := RaySolver_\alpha(f)(z) - \sum_{n\ge 0} (-1)^n RaySolver_{\alpha} (E_n)(z).\end{equation}
By a view of \eqref{est-ErrRa}, we have 
$$\| E_n \|_\eta = \| (Err_{R,\alpha})^n (f)\|_\eta \le C^n\alpha^n  (1+|\log (\Im c)|)^n  \|f\|_\eta,$$
which implies that $E_n\to 0$ in $X_\eta$ as $n \to \infty$ as long as $\alpha \log \Im c$ is sufficiently small. In addition, by a view of  \eqref{est-RaySa},
$$\|RaySolver_{\alpha} (E_n)\|_{Y^\alpha_2} \le C C^n\alpha^n  (1+|\log (\Im c)|)^n  \|f\|_\eta .$$
This shows that the series
$$\sum_{n\ge 0} (-1)^n RaySolver_{\alpha} (E_n)(z)$$
converges in $Y_2^\alpha$, assuming that $\alpha \log \Im c$ is small.

Next, by taking the limit of $n\to \infty$ in \eqref{eqs-appSn}, the equation \eqref{eqs-RaySolver} holds by definition at least in the distributional sense. The estimates when $z$ is near $z_c$ follow directly from the similar estimates on $RaySolver_\alpha(\cdot)$; see Lemma \ref{lem-RaySa}. The proof of Proposition \ref{prop-exactRayS} is thus complete. 
\end{proof}

\subsection{Exact Rayleigh solutions: $\alpha  \ll1$}\label{sec-exactRayleigh}
We shall construct two independent exact Rayleigh solutions by iteration, starting from the approximate Rayleigh solutions $\phi_{j,\alpha}$ defined as in \eqref{def-phia12}. 

\begin{lemma}\label{lem-exactphija} For $\alpha$ small enough so that $\alpha |\log \Im c| \ll 1$, 
there exist two independent functions $\phi_{Ray,\pm} \in e^{\pm \alpha z}L^\infty$ such that
$$
\Ray_\alpha ( \phi_{Ray,\pm} ) = 0, \qquad W[\phi_{Ray,+},\phi_{Ray,-}](z) = \alpha.
$$
Furthermore, we have the following expansions in $L^\infty$: 
$$
\begin{aligned}
\phi_{Ray,-} (z)&=  e^{-\alpha z} \Big (U-c + O(\alpha )\Big).
\\
\phi_{Ray,+} (z)&=  e^{\alpha z} \mathcal{O}(1),
\end{aligned}$$
as $z\to \infty$. At $z = 0$, there hold
$$ 
\begin{aligned}
\phi_{Ray,-}(0) &= U_0 - c + \alpha (U_+-U_0) ^2  \phi_{2,0}(0) + \mathcal{O}(\alpha(\alpha + |z_c|))
\\
\phi_{Ray,+}(0) &= \alpha  \phi_{2,0}(0) +  \mathcal{O}(\alpha^2)
\end{aligned}$$
with 
$$
\phi_{2,0}(0) =  {1 \over {U'_c}}  -  {2U''_c\over {U'_c}^2 } z_c \log z_c + \mathcal{O}(z_c).
$$ 
 \end{lemma}
\begin{proof} 
Let us start with the decaying solution $\phi_{Ray,-}$, which is now constructed by induction. Let us introduce 
$$ \psi_{0} =  e^{-\alpha z} (U-c), \qquad e_{0} = - 2\alpha (U-c) U' e^{-\alpha z},$$
and inductively for $k \ge 1$, 
$$ \psi_{k} = - RaySolver_\alpha (e_{k-1}), \qquad e_{k} = - Err_{R,\alpha} (e_{k-1}).$$
We also introduce 
$$ \phi_{N} = \sum_{k=0}^N \psi_{k} .$$
By definition, it follows that 
$$ \Ray_\alpha (\phi_{N}) = e_{N}, \qquad \forall ~N\ge 1.$$
%
We observe that $\|e_{0}\|_{\eta+\alpha} \le C \alpha$ and $\|\psi_{0}\|_\alpha \le C$. Inductively for $k\ge 1$, by the estimate \eqref{est-ErrRa}, we have 
$$\| e_{k}\|_{\eta + \alpha}  \le C\alpha  (1+|\log (\Im c)|)  \|e_{k-1}\|_{\eta+\alpha} \le C \alpha(C\alpha  (1+|\log (\Im c)|))^{k-1} ,
$$
and by Lemma \ref{lem-RaySa}, 
$$\| \psi_{k} \|_\alpha \le C(1+|\log (\Im c)|)  \| e_{k-1}\|_{\eta + \alpha} \le (C\alpha  (1+|\log (\Im c)|))^{k} .$$
Thus, for sufficiently small $\alpha$, the series $\phi_{N}$ converges in $X_\alpha$ and the error term $e_{N}\to 0$ in $X_{\eta+\alpha}$. This proves the existence of the exact decaying Rayleigh solution $\phi_{Ray,-}$ in $X_\alpha$, or in $e^{-\alpha z}L^\infty$.

As for the growing solution, we simply define 
$$ \phi_{Ray,+}  =\alpha \phi_{Ray,-}(z) \int_{1/2} ^ z \frac{1}{\phi^2_{Ray,-} (y) }\; dy.$$
By definition, $\phi_{Ray,+} $ solves the Rayleigh equation identically. Next, since $\phi_{Ray,-}(z)$ tends to $e^{-\alpha z} (U_+ - c + \mathcal{O}(\alpha))$, $\phi_{Ray,+} $ is of order $e^{\alpha z}$ as $z \to \infty$.

Finally, at $z=0$, we have 
$$
\begin{aligned}
\psi_1(0) &= - RaySolver_\alpha(e_0) (0) = - \phi_{2,\alpha}(0) \int_0^{+\infty} e^{2\alpha x}\phi_{1,\alpha}(x) {e_0(x) \over U(x) - c} dx 
\\
&=  2 \alpha \phi_{2,0}(0) \int_0^{+\infty}  U' (U-c)dz = \alpha (U_+-U_0) (U_+ + U_0 - 2c) \phi_{2,0}(0)
\\
&  = \alpha (U_+-U_0)^2  (U_+ + U_0 - 2c) \phi_{2,0}(0) +  2\alpha (U_+-U_0) (U_0 - c) \phi_{2,0}(0).
\end{aligned}
$$
From the definition, we have $\phi_{Ray,-}(0) = U_0 -c + \psi_1(0) + \mathcal{O}(\alpha^2)$. 
This proves the lemma, upon using that $U_0 - c = \mathcal{O}(z_c)$. 
\end{proof}

\chapter{A survey of Orr Sommerfeld equation}\label{chapter-OS-intro}

In this chapter, we study the Orr-Sommerfeld equation

%
\begin{equation}\label{OS}
\left \{ 
\begin{aligned}
(U-c) \Delta_\alpha \phi - U'' \phi &= \epsilon \Delta_\alpha^2 \phi , \qquad \epsilon = \frac{1}{i\alpha R} 
\\
\phi_{\vert_{z=0}} = \phi'_{\vert_{z=0}} &= 0, \qquad \lim_{z\to \infty} \phi(z) =0
\end{aligned}
\right.
\end{equation} with $\Delta_\alpha = \partial_z^2 - \alpha^2$. 
%
%
%
 When $\epsilon \gg 1$ (that is, in the low Reynolds number regime), a simple energy estimate shows that 
 there is no such an unstable solution: 
 $U(z)$ is spectrally stable. All the flows at a sufficiently low Reynolds number are stable. 
 
  We are interested in the singular perturbation  $\epsilon \to 0$ (or equivalently, large Reynolds number limit 
  $R = 1 /  \nu \to \infty$).  Note that as $\epsilon \to 0$, Orr Sommerfeld equations degenerate
  into Rayleigh equations, up to the boundary conditions.
  
 Two cases appear, depending on whether  $U$ is  stable or unstable to the Rayleigh problem. We will
 discuss these two cases, without any proof, giving a flavor of the next chapters.


\section{Unstable case: formal analysis}


If the profile is unstable for the Rayleigh equation, then there exist a wave number $\alpha_\infty$ and 
an eigenvalue $c_\infty$ with
$\I c_\infty > 0$, with a corresponding eigenfunction $\phi _\infty$, solution to the Rayleigh problem. 
This unstable mode for Rayleigh equation will be the starting point of
 a perturbative analysis to construct an eigenmode $\phi _R$ of
the Orr-Sommerfeld equation with an eigenvalue $\I c_R >0$ for any large enough $R$.
In this paragraph we briefly present on overview of this construction, without going into the
details, and without any justification.

The main point is that $\partial_z \phi _R$ vanishes on the boundary whereas $\partial_z\phi _\infty$
does not necessarily vanish. We therefore need to add a boundary layer to
correct $\phi _\infty$. This boundary layer comes from the balance between the terms
$\partial_z^4 \phi  / \alpha R$ (diffusion) and $(U-c) \partial_z^2 \phi $ (convection) 
in the equation (\ref{OS}), and is therefore of size
$$
\sqrt{ U_0 -c_\infty \over \alpha_\infty R} = \cO( R^{-1/2}) = \cO(\nu^{1/2}).
$$
The existence and study of the viscous sublayer is a classical issue in fluid mechanics.

When $\phi _R$ is constructed and corrected by this sublayer, it satisfies the Orr-Sommerfeld boundary conditions 
exactly and solves the Orr-Sommerfeld equation
up to an error with size of order $\cO(R^{-1})$. Perturbative arguments allow to construct an unstable
mode with 
\beq \label{perturb1}
c_R = c_\infty + \cO(R^{-1}) .
\eeq
Next, starting from $\phi _R$, we can  construct an unstable mode for the linearized Navier
Stokes equations, and further obtain instability results in strong Sobolev norms for the nonlinear Navier
Stokes equations. In particular, in $L^\infty$ norm, 
the nonlinear unstable solution can reach to order $\nu^{1/2}$ in its amplitude. 
This has been carried out in detail by E. Grenier in \cite{Grenier00CPAM}.

The appearance of the blayer of size $\nu^{1/2}$ prevents the previous analysis (\cite{Grenier00CPAM})
 to conclude the instability of order one in $L^\infty$ norm and limits this approach to $O(\nu^{1/2})$. 
 Indeed, the energy of a perturbation of the unstable solution is governed by 
$$
J = \exp \int_0^t \| \nabla u_R \|_{L^\infty} 
$$
in which $u_R = \nabla^\perp \phi_R$. Due to the presence of the $\nu^{1/2}$ layer,
$$
\|\nabla u_R \|_{L^\infty} \sim 1 + \nu^{-1/2} \nu^N \exp( \nu^{-1/2} t) 
$$
for arbitrary size of order $\nu^N$ of initial perturbations.  In particular $J$ is unbounded before the size of the perturbation
reaches a size of $O(1)$. Energy methods fail to prove the $L^\infty$ instability up to $O(1)$.


\section{Stable case: formal analysis}


Some profiles are stable for the Rayleigh equation; in particular, shear profiles
without inflection points as a direction consequence of the Rayleigh's inflexion-point criterium
(see Chapter \ref{chapter-Rayleigh}). For stable profiles, all the spectrum of the Rayleigh equation
is imbedded on the axis: $\I c = 0$ (or precisely, on the range of $U$).
At a first glance, we may believe that (\ref{perturb1}) still holds true, which would mean
that any eigenvalue of the Orr-Sommerfeld problem would have an imaginary part $\I c_R$ of
order $\cO(R^{-1}) = \cO(\nu)$. However, this is not the case,and in fact, in the stable case, 
$\I c_R$ appears to be much larger than that of \eqref{perturb1}.

\medskip

 Let us detail  now this point. Orr Sommerfeld is a fourth order ordinary differential equation. It therefore
 admits for independent solutions, two of which going to $0$ at infinity, and two of which going to infinity
 in module as $z$ goes to $+ \infty$. Looking at the equation at infinity, we see that it admits two
 solutions which a "fast" behavior, and two solutions with a "slow behavior".
 
 Let $\phi_{s,\pm}$ be the "slow behavior" solutions, which go to $0$ ($\phi_{s,-}$) or infinity ($\phi_{s,+}$)
 at infinity. Let $\phi_{f,\pm}$ be the corresponding "fast" solutions.
 
 An eigenmode of Orr Sommerfeld is a combination of these four solutions which goes to $0$ at infinity and
 which satisfies the right boundary conditions at $z = 0$. It is therefore a combination of the decaying modes
 $\phi_{s,-}$ and $\phi_{s,f}$ of the form
 $$
 \phi = A \phi_{s,-} + B \phi_{f,-}.
 $$
 Note that both $\phi(0)$ and $\partial_z \phi(0)$ must vanish, hence the determinant
 $$
 \left| \begin{array}{cc} 
 \phi_{s,-}(0)  & \phi_{f,-}(0) \cr
 \partial_z \phi_{s,-}(0) & \partial_z \phi_{f,-}(0) \cr
 \end{array} \right|
 $$
 must vanish.
 
 It turns out that $\phi_{s,-}$ is close to the Rayleigh solution $\phi_{\alpha,-}$. The fast solution $\phi_{f,-}$
 is located in a boundary layer, where the highest order terms balance. 

Near the boundary,  $\partial_z^4\phi  / i \alpha R$ must balance with $U'(z_c) (z - z_c) \partial_z^2 \phi $.
This leads  to the introduction of another boundary layer of size $(\alpha R)^{1/3}$, near $z_c$,
 satisfying the equation
\beq \label{Airy}
\partial_Y^2 \Phi_R = Y \Phi_R, \qquad \Phi_R := \partial_z^2 \phi _R
\eeq
where
$$
Y := {z - z_c \over (\alpha R U'(z_c))^{1/3}}. 
$$
This layer is called {\it critical layer}.
Note that (\ref{Airy}) is simply the classical Airy equation.Therefore $\phi_{f,-}$ may be expressed
in terms of the Airy function. As a consequence, $(\alpha R)^{1/3}$ is an important parameter,
and similarly to the unstable case, we could prove
\beq \label{exp2}
c_R = c_\infty + \cO( (\alpha R)^{-1/3}) + \cO(R^{-1}) .
\eeq
Hence, the situation is very delicate. It has been intensively studied in the period
$1940 - 1960$ by many physicists, including Heisenberg, C.C. Lin, Tollmien, Schlichting, among others.
Their main objective was to compute the critical Reynolds number of shear layer flows,
namely the Reynolds number $R_c$ such that for $R > R_c$ there exists an unstable
growing mode for the Orr-Sommerfeld equation. Their analysis requires a careful study
of the critical layer.

\medskip

From their analysis, it turns out that there exists some $R_c$ (depening on the profile)
such that for $R > R_c$ there are solutions $\alpha(R)$, $c(R)$ and $\phi _R$ to the Orr-Sommerfeld equations with
$\I c(R) > 0$. Their formal analysis has been compared with modern numerical experiments
and also with experiments, with very good agreement. Note that physicists
are interested in the computation of the critical Reynolds number, since
any shear flow is unstable if the Reynolds number is larger than this critical
Reynolds number. 
{\em In this program, we are interested in the high Reynolds limit,
which is a different question.}
This limit is not a physical one, since any flow has a finite Reynolds
number, and not in any physical case can we let the Reynolds go to very very
high values. Physical Reynolds numbers may be large (of several millions
or billions), much larger than the critical Reynolds number, but despite their
large values, they are too small to enter the mathematical limit $R \to + \infty$ that
we are considering. Fluids would enter the mathematical asymptotic regime
if $R^{1/7}$ or $R^{1/11}$ (see below) are large numbers, which 
leads Reynolds numbers to be of order of billions of billions, much larger than
any physical Reynolds number! 

\medskip

It is thus important to keep in mind that the mathematical limit is not physically pertinent.
Physically, the most important phenomena are: the existence of a critical Reynolds
number (above which the shear flow is unstable), the transition from laminar
to turbulent boundary layers, the separation of the boundary layer from the boundary.
All of these occur near the critical Reynolds number, which is large, but not in the
asymptotic regime which we will now consider.

\medskip

The problem is now to study rigorously the asymptotic behavior of $\alpha$ and $c$ as
$R \to \infty$.
Let us present now some classical physical results. These results can be found, for example, in the book of Drazin and Reid \cite{Reid} or of H. Schlichting \cite{Sch}.

\medskip

For $R$ large enough there exists an interval $[\alpha_1(R),\alpha_2(R)]$
such that for every $\alpha$ in this interval there exists an unstable mode
with $\I c(R) > 0$. These values of $\alpha_1(R)$ and $\alpha_2(R)$ are so called lower and upper marginal stability curves. 
The asymptotic behavior of $\alpha_1$ and $\alpha_2$ depends
on the shear profile.

\medskip

\begin{itemize}

\item For plane Poiseuille flow: $U(z) = 1- z^2 $ on the bounded channel $z\in [-1,1]$, we have
$$
\alpha_1(R) \sim C_1 R^{-1 / 7}, \qquad
\alpha_2(R) \sim C_2 R^{-1/11} .
$$
\item For boundary layer profiles:
$$
\alpha_1(R) \sim C_1 R^{-1 / 4}, \qquad
\alpha_2(R) \sim C_2 R^{-1/6} .
$$
\item For the Blasius (a particular boundary layer) profile:
$$
\alpha_1(R) \sim C_1 R^{-1 / 4}, \qquad
\alpha_2(R) \sim C_2 R^{-1/10} .
$$

\end{itemize}

\begin{figure}[t]
\centering
\includegraphics[scale=.4]{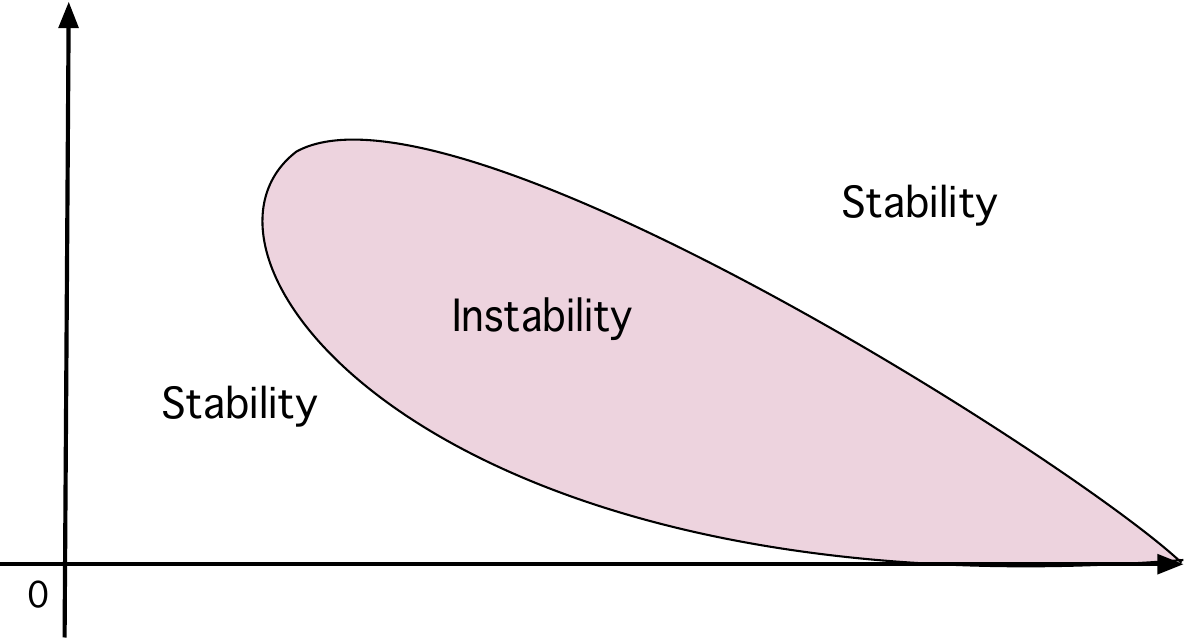}
\put(-20,1){$R^{1/5}$}
\put(-210,117){$\alpha^2$}
\put(-205,30){$\alpha_1\approx R^{-1/4}$}
\put(-70,70){$\alpha_2\approx R^{-1/6}$}
\caption{\em Illustrated are the marginal stability curves.}
\label{fig-shear}
\end{figure}
More precisely, in the $\alpha, R$ plane, the area where unstable modes exist is shown on Figure \ref{fig-shear}. 
For small $R$, all the $\alpha$ are stable. Above some critical Reynolds number, there 
is a range $[\alpha_1(R),\alpha_2(R)]$ where instabilities occur. 
Associated with the range of $\alpha(R)$, we have to determine the behavior of the eigenvalues
 $c(R,\alpha)$, or more precisely the imaginary part of $c(R,\alpha)$. 
 The complete mathematical justification of the construction of unstable modes will be detailled later.
  Here we just want to present a quick and as simple as possible  construction of the unstable Orr 
 Sommerfeld modes. We will skip all difficulties and only focus on the backbone of the instability.

 
 \subsection{The lower branch $\beta = 1/4$}
 
 
 We recall that there is no mathematically rigorous arguments of this section. We just show the main
 ingredients of the instability, keeping under silence any other term. We assume that our profile
 $U(z)$ is stable for Rayleigh equation. We focus on the lower
 marginal stability curve. In this case, we have 
 $$
 \alpha \sim A R^{-1/4} 
 $$
 and so the size of critical layers is 
 $$
 \delta \sim (\alpha R)^{-1/3} \sim A^{-1/3} R^{-1/4} .
 $$
 Let us assume that $R$ is very large, and so $\alpha$ very small. For small $\alpha$, the Rayleigh
 equation is very close to
 \beq \label{RayT}
  (U - c) \partial_z^2 \phi = U'' \phi
  \eeq
  which has an obvious solution
  $$
  \phi_1 = U - c.
  $$
There exists another independent particular solution to (\ref{RayT}),  but it turns out that this second
 solution grows linearly as $z$ increases, and may therefore be discarded.  Note that
 $\phi_1$ is a smooth function and an approximate solution of the Orr Sommerfeld problem.
 
 We next focus on Airy equation (\ref{Airy}). It has two particular fast decaying and growing solutions $\Phi_R = Ai (Y)$ and $Bi(Y)$, respectively.
 Only $Ai(Y)$ goes to $0$ as $Y$ goes to infinity, hence $Bi(Y)$ may be discarded.
 Let us denote by $Ai(1,Y)$ the primitive of $Ai(Y)$ and $Ai(2,Y)$ the primitive of $Ai(1,Y)$.
 Then
 $$
 \phi_3 = Ai(2,Y) = Ai(2, \delta^{-1} (z - z_c) ) 
 $$
 is a particular solution of the Airy equations $\epsilon \partial_z^4 \phi - U'(z_c)(z-z_c)\partial_z^2 \phi =0$, and hence is an approximate solution of the Orr Sommerfeld equation, recalling the critical layer thickness $\delta = \epsilon^{1/3}/(U'(z_c))^{1/3}$.
 Now we look for an eigenmode of the Orr Sommerfeld problem which is a combination of $\phi_1$ and
 $\phi_3$ of the form:
 $$
 \phi = A \phi_1 + B \phi_3 .
 $$
  It has the good, decaying behavior as $z$ goes to infinity. It remains to know whether we can find $A$ and $B$
  such that the boundary conditions hold: $\phi(0) = \partial_z \phi(0) = 0$. This happens if the dispersion relation
  $$
  \phi_1(0) \partial_z \phi_3(- \delta^{-1} z_c) = \partial_z \phi_1(0) \phi_3(- \delta^{-1} z_c)
  $$
holds,  or equivalently 
  \beq \label{disper}
  {\phi_1(0) \over \partial_z \phi_1(0) } 
  =
  \delta {Ai(2,- \delta^{-1} z_c) \over Ai(1,- \delta^{-1} z_c)} .
  \eeq
  The left hand side of (\ref{disper}) is 
  $$
    {\phi_1(0) \over \partial_z \phi_1(0) }  = {U_0 - c \over U_0'}
    $$
   and its imaginary part is simply $- \Im c / U_0'$, with $U_0' >0$. Here, $U_0 = U(0)$ and $U'_0 = U'(0)$.
   Now the ratio $Ai(2,Y) / Ai(1,Y)$ is precisely the classical Tietjens function 
$$
T(Y)= {Ai(2,Y) \over Ai(1,Y) }.
$$
The main point is that $\Im T(Y)$ changes sign as $Y$ goes from zero to infinity. It is positive for small $Y$ and
negative for large $Y$. As a consequence $\Im c$ changes sign as $z_c / \delta$ increases.
{\em This change of sign leads to the existence of unstable modes.}

It remains to link the size of $z_c / \delta$ with $R$ and to prove that $z_c / \delta$ goes to infinity as
$R$ increases. For this we first have to refine $\phi_1$. Namely $\phi_1$ does not go to $0$ as
$z$ goes to infinity. For small $\alpha $, we may construct a solution $\phi_{1,\alpha}$ of the Rayleigh
equation which is a perturbation of $\phi_1$ and decreases like $\exp(-\alpha z)$.
A classical perturbative analysis leads to
\begin{equation}\label{OSRay-expansion}
\phi_{1,\alpha}(0) = U - c + \alpha { (U(\infty) - U_0)^2 \over U_0'} + \cO(\alpha^2\log \alpha).
\end{equation}
Moreover as $Y$ goes to infinity,
$$
T(Y) \sim C Y^{-1/2} .
$$
Hence the dispersion relation now takes the form
\beq \label{disper2}
{U_0 - c \over U_0'} + \alpha { (U(\infty) - U_0)^2 \over {U_0'}^2} + \cO(\alpha^2\log\alpha)
\quad \sim\quad \delta (1 + | z_c  / \delta | )^{-1/2} .
\eeq
Assuming that $\alpha$ is much larger than the right hand side and that the right hand side is not disturbed by the 
$\cO(\alpha^2\log\alpha)$ term, which is the case if $A$ is
large enough but remains finite, this gives that
$| U_0 - c |$  is of order $\alpha$, and hence $z_c$, defined by $U(z_c) = c$ is also of order $\alpha$. Hence
$$
z_c / \delta \sim A^{4/3} .
$$
Therefore, provided as $A$ increases, $\Im c$ changes from negative to positive values: there exists 
an threshold $A_{1c}$ such that $\Im c > 0$, whenever $A$ is finite and $A > A_{1c}$.
This ends our overview of the lower marginal curve.

 
 \subsection{The upper branch $\beta = 1/6$}
 

 The instability remains as long as $|z_c/\delta|$ is sufficiently large and the $\cO(\alpha^2\log \alpha)$ term appearing on the left hand side of the dispersion relation \eqref{disper2} remains neglected. We note that in all cases, $z_c \approx \alpha$. The computation of the upper stability branch thus follows from the approximation that 
 $$ \delta (1+ |\alpha/\delta|)^{-1/2} \approx \alpha^2 $$
 which yields $\delta \approx \alpha^{5/3}$ or equivalently, $\alpha \approx R^{-1/6}$. 
 
%
%

To show that $\alpha \approx R^{-1/6}$ is indeed the upper branch of marginal stability is much more delicate. Roughly speaking, beyond this range of $\alpha$, the effect of the critical layers is neglected and the slow dynamics of the Rayleigh solutions becomes dominant. One expects to recover the stability of Orr-Sommerfeld equations from that of the Rayleigh problem. 


\subsection{Blasius boundary layer}

In the case of the classical Blasius boundary layer, we have additional information: $U''(0) = U'''(0) = 0$. Hence, $U''(z_c) = \cO(z_c^2)$, and so by a view of \eqref{OSRay-expansion}, the expansion for the Rayleigh solution now becomes
$$
\phi_{1,\alpha}(0) = U - c + \alpha { (U(\infty) - U_0)^2 \over U_0'} + \cO(\alpha^4\log \alpha).
$$
That is, the $\cO(\alpha^2 \log \alpha)$ term in the dispersion relation \eqref{disper2} is in fact much smaller, of order $\cO(\alpha^4 \log \alpha)$, which does not disturb the right hand side of \eqref{disper2}, as long as 
 $$ \delta (1+ |\alpha/\delta|)^{-1/2} \approx \alpha^4. $$
This yields the larger upper stability branch $\alpha \sim R^{-1/10}$ in the case of the Blasius boundary layer.


\chapter{Orr Sommerfeld equations for unstable profiles}\label{chapter-OS-unstable}

 \def\Airy{\mathrm{Airy}}
 \def\OS{\mathrm{OS}}
 \def\Ray{\mathrm{Ray}}
 \def\Iter{\mathrm{Iter}}


\section{Introduction and main result}


In this chapter, we shall derive pointwise bounds on the Green function of the Orr-Sommerfeld problem introduced
 in the previous chapter. For convenience, we introduce
\begin{equation}\label{def-OS}
\OS_{\alpha,\lambda}(\phi): = (\lambda + i\alpha U) \Delta_\alpha \phi - i\alpha U'' \phi  -  \nu \Delta_\alpha^2 \phi 
\end{equation}
where 
$$
 \Delta_\alpha: = \partial_z^2 - \alpha^2 .
 $$
For each fixed $\alpha \in \NN^*$ and $\lambda\in \CC$, we denote by $G_{\alpha,\lambda}(x,z)$ the corresponding Green kernel 
of the Orr-Sommerfeld problem. 
By definition, for each $x\in \RR_+$, $G_{\alpha,\lambda}(x,z)$ solves 
$$ 
\OS_{\alpha,\lambda}(G_{\alpha,\lambda} (x,\cdot)) = \delta_x (\cdot)
$$
on $z\ge 0$, together with the boundary conditions:
\begin{equation}\label{G-noslip}
G_{\alpha,\lambda}(x,0) = \partial_z G_{\alpha,\lambda} (x,0) =0, \qquad \lim_{z\to \infty} G_{\alpha,\lambda}(x,z) =0.
\end{equation}
To construct the Green function, let us first note that as $z \to +\infty$ the homogenous Orr-Sommerfeld equation "converges" 
to  the following constant-coefficient equation
\begin{equation}\label{OS-plus} 
\OS_+(\phi) = (\lambda + i \alpha U_+ ) \Delta_\alpha \phi  -  \nu \Delta_\alpha^2 \phi =0,
\end{equation}
where $U_+ = \lim_{z\to \infty}U(z)$. This constant-coefficient equation has four independent solutions $e^{\pm\mu_s z}$ and $e^{\pm \mu_f^+ z}$, with 
\begin{equation}\label{def-sfrate}
\mu_s = | \alpha |, \quad
 \mu_f (z)=  \nu^{-1/2} \sqrt{ \lambda + \nu \alpha^2+i \alpha U(z) },\quad
  \mu_f^+  = \lim_{z\to \infty}\mu_f(z),
 \end{equation}
in which we take the positive real part of the square root. 

As will be proved later, there exist four solutions to the homogenous Orr-Sommerfeld equation $\OS_{\alpha,\lambda} (\phi) =0$ which have either
a ``slow behavior'' $e^{\pm\mu_s z}$ 
or a ``fast behavior'' $e^{\pm \mu_f^+ z}$ as $z \to + \infty$. The two slow modes appear to be perturbations of solutions
of the Rayleigh equation 
$$
Ray_{\alpha,\lambda}(\phi) = (\lambda + i \alpha U)\Delta_\alpha\phi - i\alpha U''\phi =0,
$$
whereas the two fast modes are linked to the Airy type equation
$$(\lambda + i\alpha U- \nu \Delta_\alpha
)\Delta_\alpha \phi =0,$$
or recalling $\mu_f$ introduced in \eqref{def-sfrate}, 
\beq \label{Air}
\nu (\partial_z^2 - \mu_f^2)\Delta_\alpha \phi =0.
\eeq
Let us first consider  the Rayleigh equation $Ray_{\alpha,\lambda}(\phi) = 0$. 
As $z$ goes to $+\infty$, this equation "converges"
to $\Delta_\alpha \phi = 0$, hence $Ray_{\alpha,\lambda}(\phi) = 0$ has two solutions $\phi_{\alpha,\pm}$, with respective
behaviors $e^{\pm |\alpha| z}$ at infinity. We define the Evans function $E(\alpha,\lambda)$ by
\begin{equation}\label{def-EvansE}
E(\alpha,\lambda) = \phi_{\alpha,-}(0) .
\end{equation}
Note that the Rayleigh equation degenerates at points where $\lambda + i\alpha U (z)$ vanishes. 
In this paper, we restrict ourselves to the case when $\lambda$ is away from the range of $-i\alpha U$. 
Precisely, throughout the paper, letting $\epsilon_0$ be an arbitrarily small, but fixed, positive constant, 
we shall consider the range of $(\alpha,\lambda)$ in $\RR\setminus\{0\}\times \CC$ so that 
\begin{equation}\label{range-c} 
d(\alpha,\lambda) =\inf_{z\in \RR_+}| \lambda + i \alpha U(z)| \ge \epsilon_0.
\end{equation}
Note that $d(\alpha,\lambda) = \Re \lambda$ if $\Im \lambda \in - \alpha \mathrm{Range}(U)$. In any case,
\beq \label{range-c1}
d(\alpha,\lambda) \ge | \Re \lambda | .
\eeq
We are mainly interested in getting bounds on the Green function when $\lambda$ has a small positive real part. 
In this case, the condition \eqref{range-c} implies
\begin{equation}\label{comp-musf} \Re \mu_f(z) = \nu^{-1/2} \Re \sqrt{\lambda + \nu \alpha^2 + i\alpha U(z)} \ge \nu^{-1/2} \sqrt{\epsilon_0/2} \gg \mu_s\end{equation}
for sufficiently small $\nu$ and for $| \alpha | \le \nu^{-\zeta}$
for some $\zeta < 1/2$. We may also use these Green function bounds in order to obtain bounds on the solutions of linearized
Navier Stokes equations, through contour integrations. 


Our main result in this chapter is the following. 

\begin{theorem}\label{theo-GreenOS-unstable}
For each $\alpha,\lambda$, let by $G_{\alpha,\lambda}(x,z)$ be
the Green kernel of the Orr-Sommerfeld equation, with source term in $x$, and let 
\begin{equation}\label{def-mMf}
\mu_s = | \alpha| , \qquad   \mu_f(z) =  \nu^{-1/2} \sqrt{\lambda + \nu \alpha^2 + i \alpha U(z)}, 
\end{equation}
where we take the square root with positive real part. Let $0 < \theta_0 < 1$ and $\zeta < 1/2$. Let $\sigma_0 > 0$ be arbitrarily small.
Then, there exists $C_0 > 0$ so that 
\begin{equation}\label{est-GrOS}
 |G_{\alpha,\lambda}(x,z)| 
  \le \frac{C_0}{  \mu_s  d(\alpha,\lambda) }e^{-\theta_0 \mu_s |x-z|} 
  +    \frac{C_0}{ |\mu_f(x) |  d(\alpha,\lambda)}  e^{- \theta_0 | \int_x^z \Re \mu_f \; dy|} 
\eeq
uniformly for all $x,z\ge 0$ and $0 < \nu \le 1$, and uniformly in $(\alpha,\lambda)\in \RR\setminus\{0\}\times \CC$ so that $| \alpha | \le\nu^{-\zeta}$, \eqref{range-c} holds, and
$$
| E(\alpha,\lambda)| > \sigma_0 .
$$
In particular, for $\Re \lambda > 0$, we have
\begin{equation}\label{est-GrOSsimply}
 |G_{\alpha,\lambda}(x,z)| 
  \le \frac{C_0}{  \mu_s  | \Re \lambda |}e^{-\theta_0 \mu_s |x-z|} 
  +    \frac{C_0}{ |\mu_f(x) |  | \Re \lambda |}  e^{- \theta_0 | \int_x^z \Re \mu_f \; dy|} .
\eeq
In addition, there hold the following derivative bounds
\begin{equation}\label{est-GrOS-d}
 | \partial_x^k \partial_z^\ell G_{\alpha,\lambda}(x,z)|  
 \le \frac{C_0\mu_s^{k+\ell}}{   \mu_s  d(\alpha,\lambda)}  e^{-\theta_0 \mu_s |x-z|} 
 +    \frac{C_0 | \mu_f(z) |^{k+\ell}}{| \mu_f(x)| d(\alpha,\lambda) }  e^{- \theta_0 | \int_x^z \Re \mu_f \; dy|}  
\eeq
for all $x,z\ge 0$ and $k, \ell \ge 0$, in which $M_f = \sup_z \Re  \mu_f(z) $. 
Moreover, 
\begin{equation}\label{est-GrOS-delta}
 | \Delta_\alpha G_{\alpha,\lambda}(x,z)|  
 \le \frac{C_0}{   d(\alpha,\lambda)^2 } e^{-\theta_0 \mu_s |x-z|} 
 +    \frac{C_0 | \mu_f(z)|^2}{| \mu_f(x)| d(\alpha,\lambda) }   e^{- \theta_0 | \int_x^z \Re \mu_f \; dy|}  
\eeq
where we "gain" a factor $\mu_s$ in the first term on the right hand side.
\end{theorem}

We believe that the $\theta_0$ factor is purely technical, and that this Theorem holds true for $\theta_0 = 1$. 
In addition, we note that on the slow modes we have mainly to invert $\partial^2 - \alpha^2$ which
leads to a gain of $\mu_s  = | \alpha |$, while the inversion of the fast modes leads to a gain of  $\mu_f$.
Moreover, $\Delta_\alpha G_{\alpha,\lambda}$ enjoys better bounds since $\Delta_\alpha e^{\pm |\alpha| z}= 0$ 
and $\Delta_\alpha \phi_{s} \approx i\alpha \phi_{s} / (\lambda + i\alpha U)$ for slow modes $\phi_s$, 
which gains a prefactor $\alpha/d(\alpha,\lambda)$ in the claimed estimate. 
We also note that in view of \eqref{comp-musf}, the fast modes are exponentially localized in the sense 
\begin{equation}\label{fast-local}
e^{- \theta_0 | \int_x^z \Re \mu_f \; dy|}  \le e^{-\theta_1 |x-z|/\sqrt \nu} .
\eeq
To prove this Theorem we first construct approximate solutions to the Orr Sommerfeld equation, and then construct
an approximate Green function. An iteration argument yields the exact Green function together with the stated bounds. 
Our construction of the Green function for the Orr-Sommerfeld problem  was inspired by the pointwise Green function 
approach introduced by Zumbrun-Howard \cite{ZumbrunHoward} and Zumbrun \cite{Z1,Z2}. 

We are also interested in the construction of a pseudo inverse of Orr Sommerfeld operator near a simple eigenvalue, 
a construction which is detailed in Section \ref{sectionpseudo}.


\section{Approximate solutions of Orr-Sommerfeld}


In this section, we  construct  four independent approximate solutions to the Orr Sommerfeld equations $\OS_{\alpha,\lambda}(\phi) =0$, 
two with a "fast" behavior and two with a "slow" one.

The fast modes are constructed using geometrical optics methods, namely following the BKW method.

For the slow modes we will distinguish between three regimes:

\begin{itemize}

\item bounded $| \alpha |$. In this case the  slow modes are perturbations of the eigenmodes of Rayleigh equations.

\item $ 1 \ll | \alpha | \le \nu^{-1/4}$ (or any small negative power of $\nu$). 
We use the fact that Rayleigh equation is a perturbation of $\Delta_\alpha$.
The slow modes are perturbations of the eigenmodes of $\partial^2_z - \alpha^2$, namely
$e^{\pm |\alpha| z}$.

\item $ \nu^{-1/4} \le | \alpha | \le \nu^{-\zeta}$, for $\zeta <1/2$.  
In this case $e^{\pm |\alpha| z}$ is a sufficient approximation.

\end{itemize}
Solutions will be constructed in the function spaces $L^\infty_\eta$, for $\eta>0$, 
that consist of smooth functions $f$ so that the norm 
$$
\|f\|_\eta: = \sup_{z\ge 0} e^{\eta |z|} |f(z)| 
$$ 
is finite.


\subsection{Fast modes}\label{sec-fast}


The two "fast" solutions come from the Airy equation \eqref{Air}. 
This equation degenerates when $\lambda + \alpha^2 \nu + i\alpha U$ gets small.
Points $z_c$ such that 
$$
\alpha U(z_c) = - \Im \lambda
$$ 
are called "critical layers". The behavior of Airy equation changes
as we approach these points, and in this paper we only study this equation away from these critical layers.
Let us quantify this notion. The Airy's equation has a typical length scale
$$
\delta(z)  = \frac{1}{\mu_f(z)}= \sqrt{\nu \over \lambda + \nu \alpha^2 + i \alpha U(z)} .
$$
The BKW method fails when $\delta(z)$ 
varies within a length $\delta(z)$, namely when $\delta'(z) \delta(z) \sim \delta(z)$.
We therefore restrict our constructions to the case when $\delta' \ll 1$, namely
\beq \label{singu}
\delta'(z) = {-i\sqrt \nu \alpha U'(z) \over 2 (\lambda + \nu \alpha^2 + i \alpha U(z))^{3/2} } \ll 1.
\eeq
 or equivalently 
to $\alpha \ll \nu^{-1/2}$.

If $\alpha \sim \nu^{-1/2}$  then the nature of the construction changes (see \eqref{singu} for more details).
In this paper we restrict to the case $| \alpha | \ll \nu^{-1/2}$ or more precisely on $| \alpha | \le \nu^{-\zeta}$
for some $\zeta < 1/2$. We leave the case $\alpha \sim \nu^{-1/2}$ open, since it will not be necessary
in the construction of linear and nonlinear instabilities.

\begin{prop}\label{prop-fastOrrapp} 
Let $N > 0$ be arbitrarily large. Then for sufficiently small $\nu$ and for $|\alpha|\le \nu^{-\zeta}$ with $\zeta <1/2$, 
there exist two approximate  solutions $\phi_{f,\pm}^{app}(z)$ which solve Orr-Sommerfeld equations up to a small error term
$$
\OS_{\alpha,\lambda}(\phi_{f,\pm}^{app}) = O(\nu^N | \phi_{f,\pm}^{app}|),
$$
with $\phi_{f,\pm}^{app}(0) = 1$ and
\begin{equation}\label{fast-mode1} 
\phi_{f,\pm}^{app} (z) =  e^{\pm \int_0^z \mu_{f}(y) \; dy  } \Big( 1 + \phi_{\pm}(z)\Big) ,
\end{equation}
where $\phi_{\pm}$ and their derivatives are uniformly bounded in $\alpha$, $\nu$ and $z$, 
and converge exponentially fast to $0$ at $z=+ \infty$.
\end{prop}

\begin{proof}
Following a semi classical approach, we look for $\phi_{f,\pm}^{app}$ under the form
$$
\phi_{f,\pm}^{app} = \exp \Bigl({\theta_{\pm}^{app} \over \sqrt{\nu}} \Bigr) .
$$
Let $\theta = \theta_{\pm}^{app}$ to simplify the notations. We compute 
$$
\partial_z^2 \phi_{f,\pm}^{app} = \Bigl( {\theta'^2 \over \nu} + {\theta'' \over \sqrt{\nu}} \Bigr) \phi_{f,\pm}^{app} 
$$
and
$$
\nu \partial_z^4 \phi_{f,\pm}^{app} 
= \Bigl( {\theta'^4 \over \nu} + 6 {\theta'^2 \theta'' \over \sqrt\nu} 
+ 4 \theta' \theta''' + 3 \theta''^2 + \sqrt\nu \theta'''' \Bigr) \phi_{f,\pm}^{app} .
$$
We now expand $\theta$ in powers of $\sqrt\nu$; namely, 
$$
\theta = \sum_{i = 0}^N \theta_j \nu^{j/2},
$$
where the functions $\theta_j$ will themselves depend on $\alpha$ and $\lambda$. Putting the Ansatz into the Orr-Sommerfeld equations, at leading order, we obtain
$$
(\lambda + i\alpha U ) (\theta_0'^2 - \nu \alpha^2) - \Bigl( \theta_0'^4 - 2  \nu \alpha^2 \theta_0'^2 + \nu^2 \alpha^4  \Bigr) = 0.
$$
Factorizing by $\theta_0'^2 - \nu \alpha^2$ we get
$$
\theta_0'^2 =  \lambda + \nu \alpha^2 + i \alpha U = \nu \mu_f^2(z),
$$
which gives
$$
\theta_0' = \pm \sqrt{\nu}  \mu_f(z).
$$
Note that $\theta_0'$ converges exponentially fast to $\pm \sqrt{\nu} \mu_f^+$ and $\theta_0''$ converges exponentially fast
to $0$. To obtain $\theta_1$ we equate the powers in $\sqrt{\nu}^{-1}$ and get
$$
- 4  \theta_0'^3 \theta_1' + 4  \nu \alpha^2 \theta_0' \theta_1' + 2 ( \lambda + i\alpha U ) \theta_0' \theta_1' = S,
$$ 
where the source term $S = 6 \theta_0'^2 \theta_0''$ only depends on $\theta_0'$ and its derivatives.
This leads to
$$
\theta_1' =   {S  \over (- 4 \theta_0'^2  +4  \nu \alpha^2  + 2 (\lambda + i\alpha U ) ) \theta_0'}
= -{S  \over 2 (\lambda + i\alpha U ) \theta_0'} .
$$
As $\theta_0'$ is bounded away from $0$, $\theta_1'$ is correctly defined. Moreover, 
$\theta_1$ converges exponentially at infinity, as well as all its derivatives, and as 
$\theta_0'' = O(\alpha)$, $\theta_1 = O(\alpha)$. This leads to
\beq \label{bornealpha}
\theta^{app}_\pm = \theta_0 + O(\alpha \nu^{-1/2}).
\eeq
We then obtain equations and similar estimates on the remaining $\theta_j$ by equaling successive powers of $\nu$.
 The Proposition follows.
\end{proof}


\subsection{Slow modes}


\begin{prop} \label{slow1}
There exist
two solutions $\phi_{s,\pm}^{app}$ which approximately solve the Orr Sommerfeld equations: precisely, for any $N$, 
$$
| \OS_{\alpha,\lambda}(\phi_{s,\pm}^{app}) | \le C_N \nu^N e^{\pm | \alpha | z - \eta z} 
$$
and behave like $e^{\pm | \alpha | z}$ as $z$ goes to $+ \infty$: for any $n$, 
$$
| \partial_z^n \phi_{s,\pm}^{app} (z) | \le C_n e^{\pm | \alpha | z} . 
$$
\end{prop}
For the proof of Proposition \ref{slow1}, we shall distinguish three cases: bounded $\alpha$, moderate $\alpha$, 
and large $\alpha$, that will be detailed in the next
sections. We restrict ourselves to $\alpha > 0$, the opposite case being similar.


\subsubsection{Approximate slow modes for bounded $\alpha$ and $\lambda$}


As $z$ goes to $+ \infty$, the Rayleigh equation "converges" to $\Delta_\alpha \phi$. 
Therefore the Rayleigh equation admits two particular equations, called $\phi_{\alpha,\pm}$ which behave like
$e^{\pm | \alpha | z}$ as $z \to + \infty$. Moreover $|\partial^n_z \phi_{\alpha,\pm}(z) | \le C_n e^{\pm | \alpha | z}$ for every
positive $n$. Note that
$$
\OS_{\alpha,\lambda}(\phi_{\alpha,\pm}) = - \nu \Delta_\alpha^2 \phi_{\alpha,\pm} .
$$
Using the Rayleigh equation, we compute 
$$
\Delta_\alpha \phi_{\alpha,\pm} = {i\alpha U'' \phi_{\alpha,\pm} \over \lambda + i\alpha U},
$$
which gives
$$
\OS_{\alpha,\lambda}(\phi_{\alpha,\pm}) =  - \nu \Delta_\alpha \Bigl( {i\alpha U'' \phi_{\alpha,\pm} \over \lambda + i\alpha U} \Bigr) 
$$
$$
 = - \nu  \Bigl( {i\alpha U''  \over \lambda + i\alpha U} \Bigr)^2 \phi_{\alpha,\pm}
 - 2  \nu  \partial_z \phi_{\alpha,\pm} \partial_z \Bigl( {i\alpha U'' \over \lambda + i\alpha U} \Bigr) 
- \nu \phi_{\alpha,\pm} \partial_z^2  \Bigl( {i\alpha U'' \over \lambda + i\alpha U} \Bigr) .
$$
Note that $\lambda + i\alpha U$ is bounded away from $0$, therefore 
$$
| \OS_{\alpha,\lambda}(\phi_{\alpha,\pm}) | \le C \nu e^{\pm | \alpha | z - \eta z} ,
$$
and similarly for all its derivatives.

We now look for approximate solutions of Orr Sommerfeld solutions $\phi_{s,\pm}^{app}$ of the form
$$
\phi_{s,\pm}^{app} = \sum_{j=0}^N \phi_{\alpha,\pm}^j
$$
for arbitrarily large $N$, starting with $\phi_{\alpha,\pm}^0 = \phi_{\alpha,\pm}$.
We have
$$
Ray_\alpha(\phi_{\alpha,\pm}^{j+1}) = - \OS_{\alpha,\lambda}(\phi_{\alpha,\pm}^j) .
$$
Note that 
\begin{equation}
\OS_{\alpha,\lambda}(\phi_{s,\pm}^{app}) = - \nu \Delta_\alpha^2 \phi_{\alpha,\pm}^N .
\end{equation}
We will focus on the construction of $\phi_{s,-}^{app}$, the construction of $\phi_{s,+}^{app}$ being similar.
To end the proof of Proposition \ref{slow1} we need to bound the various $\phi_{\alpha,-}^i$,
which is done through the iterative use of the following Proposition.

\begin{prop} 
There exist constants $C_n$ such that the following assertion is true.
For any $\beta>0$ and any smooth function $\psi$ there exists a smooth solution $\phi$ of $Ray_\alpha(\phi) = \psi$
such that 
$$
\sup_{k \le n} \| \partial_z^k \phi \|_{\alpha} 
+ \sup_{k \le n} \| \partial_z^n \Delta_\alpha \phi \|_{\alpha + \beta}
\le {C_n \over E(\alpha,\lambda)} \sup_{k \le n} \| \partial_z^k \psi \|_{\alpha+\beta}
$$
where $\|\phi\|_\eta = \sup_{z\ge 0} e^{\eta |z|}|\phi(z)|$. \end{prop}

\begin{proof}
We first construct the Green function of the Rayleigh operator. Let
$$
\widetilde \phi_{\alpha,+}(z) = \phi_{\alpha,-}(0) \phi_{\alpha,+}(z)
- \phi_{\alpha,+}(0) \phi_{\alpha,-}(z)
$$
Then $\widetilde \phi_{\alpha,+}(0) = 0$ and the Wronskian of $\widetilde \phi_{\alpha,+}$ and 
$\phi_{\alpha,-}$ equals
$$
W(\widetilde \phi_{\alpha,+}, \phi_{\alpha,-}) = \phi_{\alpha,-}(0) W(\phi_{\alpha,+},\phi_{\alpha,-})
= 2 \alpha \phi_{\alpha,-}(0)
$$
evaluating this latest Wronskian at infinity. The Green function of the Rayleigh operator is therefore
$$
G(x,z) = {1 \over 2 \alpha \phi_{\alpha,-}(0)} \phi_{\alpha,-}(x) \widetilde \phi_{\alpha,+}(z) \qquad
\hbox{if} \qquad z < x
$$
$$
G(x,z) = {1 \over 2 \alpha \phi_{\alpha,-}(0)} \widetilde \phi_{\alpha,+}(x) \phi_{\alpha,-}(z) \qquad
\hbox{if} \qquad z  > x.
$$
We then have
$$
\phi(z) = \int_0^{+ \infty} G(x,z) \psi(x) dx .
$$
Using the asymptotic behavior of $\phi_{\alpha,\pm}$ we get the claimed bounds on $\| \partial_z^n \phi \|_{\alpha}$
with $n = 0$ and $n = 1$ by a direct computation. Higher derivatives are obtained by differentiating
$$
\partial_y^2 \phi = \alpha^2 \phi + {i\alpha U'' \over \lambda + i\alpha U} \phi + \psi,
$$
keeping in mind that $\alpha$ is bounded and $\lambda$ is away from the range of $-i\alpha U$. 
Next, we write 
$$
\Delta_\alpha \phi = {i\alpha U'' \over \lambda + i\alpha U} \phi + \psi
$$
which gives the desired bounds on $\Delta_\alpha \phi$.
\end{proof}


\subsubsection{Approximate slow modes for $1 \ll |\alpha| \le \nu^{-1/4}$ or large $\lambda / \alpha$}


For large $\alpha$, or for large $\lambda / \alpha$,
 the Rayleigh operator is a small perturbation of $\partial_z^2 - \alpha^2$
 and we can construct approximate eigenmodes $\phi_{s,\pm}^{app}$ using a perturbative construction.
Namely, the Rayleigh equation may be rewritten as 
$$ 
\Delta_\alpha \phi = { i\alpha U'' \phi \over  \lambda + i \alpha U} .
$$
Note that $\alpha^{-1} e^{-\alpha | x - z|}$ is a Green function for $\Delta_\alpha$.
We therefore define the following operator $\mathcal{T}$ by
$$ 
\mathcal{T}[\phi](z) := \int_0^\infty \alpha^{-1} e^{-\alpha |x-z|}  {i \alpha U'' \phi(x) \over \lambda + i \alpha U } \; dx .
$$
We shall prove that for sufficiently large $\alpha$, the map $\mathcal{T}$ is well-defined and contractive 
from $L^\infty_{\alpha +\eta}$ to itself. 
Indeed, for $\phi \in L^\infty_{\alpha + \eta}$, as $\lambda + i\alpha U$ is bounded away from $0$, we have 
$$
 | \mathcal{T}[\phi](z)| 
 \le C_0 \int_0^\infty  e^{-\alpha |x-z|} e^{-\eta x - \alpha x}   \| \phi\|_{\alpha + \eta} \; dx 
\le C_0 \alpha^{-1} \| \phi\|_{\alpha + \eta}  e^{-\eta z - \alpha z}. 
$$
This proves that $\mathcal{T}[\phi] \in L^\infty_{\alpha + \eta}$. 
If $\alpha$ is large enough then $\mathcal{T}$ is a contraction in this space. On the other hand, if $\lambda / \alpha$ is large enough we rewrite
$$
 {i\alpha U'' \phi(x) \over  \lambda + i\alpha U} =  { U'' \phi(x) \over  U- i \alpha^{-1} \lambda }
 $$
 which is bounded by $C / (\alpha^{-1} \lambda)$. Hence $\mathcal{T}$ is a contraction if $\lambda / \alpha$ is large enough.
 
We now construct two independent solutions of the Rayleigh equation, which behaves like
$e^{\pm \alpha z}$ for large $z$. Let us detail the "-" case. We look for $\phi_{s,-}$ under the form
$$
\phi_{s,-} = \sum_{n \ge 0} \phi_{-}^n
$$
with $\phi_-^0 = e^{- \alpha z}$ and 
$\phi_{-}^{n+1} = \mathcal{T}[\phi_{-}^n]$. As $\mathcal{T}$ is contractive, the previous sum converges in $L^{\infty}_{\alpha + \eta}$.
Note that in particular
$$
\phi_{\alpha,-} = e^{- \alpha z}  ( 1 + O(\alpha^{-1})_{L^{\infty}_{\alpha+\eta}}),
$$ 
and similarly for its derivatives. The construction of $\phi_{\alpha,+}$ is similar.

The construction of approximate solutions of Orr Sommerfeld is similar to that of the previous section. We start with
$\phi_{s,-}$ and note that
$$
 \nu \| \Delta_\alpha^2 \phi_{s,-} \|_{\alpha + \eta} \le C \nu | \alpha |^2 \lesssim \nu^{1/2}.
$$
We then introduce $\phi_{s,-}^1$, defined by
$$
Ray_\alpha(\phi_{s,-}^1) = - \nu \Delta_\alpha^2 \phi_{s,-} ,
$$
which can be bounded using the $\mathcal{T}$ operator.  To  end the proof of Proposition \ref{slow1},
we iterate the construction as in the previous section.


\subsubsection{Approximate slow modes for $\nu^{-1/4} \le |\alpha|  \ll \nu^{-1/2}$}


We look for eigenmodes of the form
$$
\phi_{s,\pm}^{app} = \exp ( \alpha \theta_{\pm}^{app} )
$$
where $\theta_\pm^{app}$ may be expanded in powers of $\alpha^{-1}$.
As in Section \ref{sec-fast}, we get
$$
- \nu \alpha^4 \theta_0'^4 +2  \nu \alpha^4 \theta_0'^2 - \nu \alpha^4  
+ (\lambda + i\alpha U) (\alpha^2 \theta_0'^2 - \alpha^2) = 0,
$$
This time we choose $\theta_0 = \pm 1$ and iterate as in Section \ref{sec-fast} to prove Proposition \ref{slow1}.
Note again that the leading order of $\Delta_\alpha \phi_{s,\pm}^{app}$ vanishes.


\section{Approximate Green function}\label{sec-Greenapp}


We now construct an approximate Green function $H^{app}$ using the approximate solutions
$\phi_{s,\pm}^{app}$ and $\phi_{f,\pm}^{app}$.
We will decompose this Green function into two components
$$
H^{app} = G^{app} + \hat G^{app}
$$
where $G^{app}$ does not take into account the boundary conditions and focus on the discontinuity at $y = x$,
and where $\hat G^{app}$ restores the proper boundary conditions.

Hence, first forgetting the boundary condition, we look for $G^{app}(x,y)$ of the form
\begin{equation}\label{def-GappX}
\begin{aligned}
G^{app}(x,y) &= a_+(x)  {\phi_{s,+}^{app}(y) \over c_2}
+ {b_+(x) }  {\phi_{f,+}^{app}(y) \over \phi_{f,+}^{app}(x)} 
\quad \hbox{for} \quad y < x,
\\
G^{app}(x,y) &= a_-(x)  {\phi_{s,-}^{app}(y) \over c_1} 
+ {b_-(x)} {\phi_{f,-}^{app}(y) \over \phi_{f,-}^{app}(x)} 
\quad \hbox{for} \quad y  > x,
\end{aligned}\end{equation}
where the normalization constants $c_1$ and $c_2$ will be fixed later.
Let 
\begin{equation}\label{def-vvvapp}
v(x) = (- a_-(x), a_+(x), - b_-(x), b_+(x) ) .
\end{equation}
By definition, $G^{app}$, $\partial_y G^{app}$, $\sqrt{\nu} \partial_y^2 G^{app}$ are continuous at $x = y$
and $\nu \partial_y^3 G^{app}$ has a jump at $x = y$, of magnitude $1$. 
Let
\beq \label{matriceM}
M = \left( \begin{array}{cccc} 
\phi_{s,-} / c_1& \phi_{s,+} / c_2 & \phi_{f,-}  & \phi_{f,+}  \cr
\partial_y \phi_{s,-} / c_1 \mu_f & \partial_y\phi_{s,+} / c_2 \mu_f
& \partial_y\phi_{f,-} /  \mu_f & \partial_y\phi_{f,+} / \mu_f  \cr
\partial_y^2 \phi_{s,-} / c_1 \mu_f^2 &\partial_y^2 \phi_{s,+} / c_2 \mu_f^2 
& \partial_y^2 \phi_{f,-} /  \mu_f^2  & \partial_y^2 \phi_{f,+} /  \mu_f^2 \cr
 \partial_y^3 \phi_{s,-} / c_1 \mu_f^3 &  \partial_y^3 \phi_{s,+} / c_2  \mu_f^3
& \partial_y^3 \phi_{f,-}  /  \mu_f^3 &  \partial_y^3 \phi_{f,+} /  \mu_f^3 \cr 
\end{array} \right) ,
\eeq
where the functions $\phi_{s,\pm} = \phi_{s,\pm}^{app}$ and $\phi_{f,\pm} = \phi_{f,\pm}^{app}$ and their derivatives are evaluated at $y=x$.
Then 
\beq \label{Mv}
M v = (0,0,0,1/ \nu \mu_f^3) .
\eeq
In the following sections we will bound the solution $v$ of (\ref{Mv}).
Let us define the four two by two matrices $A$, $B$, $C$ and $D$  by
$$
M = \left( \begin{array}{cc} 
A & B \cr
C & D \cr \end{array} \right) .
$$
Note that, using (\ref{bornealpha}),
$$
D = \left( \begin{array}{cc}
1  & 1 \cr
- 1 & 1 \cr 
\end{array} \right) + O(\alpha \mu_f^{-1}).
$$
 Hence the matrix $D$ is bounded and invertible, upon recalling that $\alpha \ll \mu_f$ in the range of $\alpha$ that we consider
 (see \eqref{comp-musf}). Moreover its inverse is bounded and equals
$$
D^{-1} = {1 \over 2}  \left( \begin{array}{cc}
1  & - 1 \cr 
1 & 1 \cr 
\end{array} \right) + O(\alpha \mu_f^{-1}).
$$
We shall consider two cases: bounded $\alpha$ and unbounded $\alpha$. 


\subsection{Case 1: 
bounded $\alpha$}\label{sec-bda}


We take $c_1 = c_2 = 1$. Note that $A = A_1 A_2$
where
$$
A_1 = \left( \begin{array}{cc}  
1 & 0 \cr
0 & \mu_f^{-1} \cr
\end{array} \right),
\quad 
A_2 = \left( \begin{array}{cc}  
\phi_{s,-} & \phi_{s,+} \cr
\partial_y \phi_{s,-} & \partial_y \phi_{s,+} \cr
\end{array} \right) .
$$
The determinant $E^{app}(\alpha,\lambda)$ of $A_2$ is a perturbation of the Evans function $E(\alpha,\lambda)$
in the sense
$$
E^{app}(\alpha,\lambda) = E(\alpha, \lambda) + O(\nu^\sigma),
$$
for some positive $\sigma$.
Hence if $E(\alpha, \lambda) \ne 0$, then $A_2$ and $A$ 
are invertible provided $\nu$ is small enough, and $A_2^{-1}$ is bounded. 
Moreover the matrix $M$ has an approximate inverse
$$
\widetilde M = \left( \begin{array}{cc}
A^{-1} & - A^{-1} B D^{-1}  \cr
0 & D^{-1} \cr 
\end{array} \right) 
$$
in the sense that $M  \widetilde M = Id + N$ where
$$
N =  \left( \begin{array}{cc}
0 & 0 \cr 
 C A^{-1} &  - C A^{-1} B D^{-1} \cr 
\end{array} \right) .
$$
Note that $C$ is of order $O(\mu_f^{-2})$ since $\alpha$ is bounded, that $B$ is bounded 
and that $A^{-1} = A_2^{-1} A_1^{-1}$ is of order $O(\mu_f)$.
Hence we have  $N = O(\mu_f^{-1})$. Therefore $(Id + N)^{-1}$
is well defined and uniformly bounded for $\nu$ small enough provided $E(\alpha,\lambda) \ne 0$.
As a consequence, 
$$
M^{-1} = \widetilde M (Id + N)^{-1} = \widetilde M \sum_n N^n.
$$
Note that the two first lines of $N^n$ vanish. Therefore
$$
(Id + N)^{-1} (0,0,0, 1 / \nu \mu_f^3) = \Bigl( 0, 0, O(1 / \nu \mu_f^4), 1 / \nu \mu_f^3  + O(1 / \nu \mu_f^4) \Bigr) .
$$
As $D^{-1}$ is bounded and $A^{-1} B D^{-1}$ is of order $O(\mu_f)$, we obtain that 
$a_\pm$ and $b_\pm$ are respectively of order $O(1 / \nu \mu_f^2)$ and $O(1/ \nu \mu_f^3)$.
Note that $\alpha$ is bounded in this case, which give the desired bounds since
$$
\nu \mu_f^2 = \lambda + \nu \alpha^2 + i  \alpha U 
$$
hence
$$
| \nu \mu_f^2 | \ge d(\alpha,\lambda),
$$
which ends this first case.


\subsection{Case $2$: 
large $\alpha$}


We take $c_1 = \phi_{s,+}^{app}(x)$ and $c_2 = \phi_{s,-}^{app}(x)$.
In this case $A$ is of the form
$$
A = \left( \begin{array}{cc}
1 & 1 \cr
- \alpha \mu_f^{-1} & \alpha \mu_f^{-1} \cr \end{array} \right) (1 + o(1)).
$$
Its inverse $A^{-1}$ equals
$$
A^{-1} = {1 \over 2} \left( \begin{array}{cc}
1 & - \alpha^{-1} \mu_f \cr
1 &  \alpha^{-1} \mu_f \cr \end{array} \right) (1 + o(1)).
$$
Note that $D^{-1}$ and $B$ are bounded and $A^{-1}$ is order $O(\mu_f / \alpha)$. 
As $C$ is of order $O(\alpha^2 / \mu_f^{2})$, $N$ (defined in the previous section) is of order $O(\alpha / \mu_f)$.
Hence, as $| \alpha |\ll \mu_f$ in view of \eqref{comp-musf}, we have 
$$
(Id + N)^{-1} = \sum_n (-1)^n N^n.
$$
This leads to
\beq \label{IdN1}
(Id + N)^{-1}(0,0,0,1/\nu \mu_f^4) =  \Bigl( 0,0,O(\alpha / \nu \mu_f^4), O(1 / \nu \mu_f^3) \Bigr) .
\eeq
It remains to evaluate the image of this vector by $\widetilde M$. As $D^{-1}$ is bounded, we obtain that
$b_{\pm}$ are of order $O(1 / \nu \mu_f^3) = O(1 / \mu_f d(\alpha,\lambda))$.

Moreover, we compute 
$$
D^{-1}(0, O(1 / \nu \mu_f^3)) = \Bigl[ (-1,1) + O(\alpha \mu_f^{-1}) \Bigr]O(1 / \nu \mu_f^3).
$$
As
$$
B = \left( \begin{array}{cc}
1 & 1 \cr
-1 & 1 \cr
\end{array} \right) (1 + O(\mu_f^{-1})),
$$
we obtain
$$
B D^{-1}(0, O(1 / \nu \mu_f^3)) =  \Bigl[ (0,1) + O(\alpha \mu_f^{-1}) \Bigr]O(1 / \nu \mu_f^3).
$$
As a consequence, we obtain 
$$
A^{-1} B D^{-1} (0, O(1 / \nu \mu_f^3)) = O(1 / \alpha \nu \mu_f^2).
$$
It remains to bound the images of the $O(\alpha / \nu \mu_f^4)$ term in the equation (\ref{IdN1}). We have 
$$
D^{-1}(O(\alpha / \nu \mu_f^4),0) = \Bigl[ (1,1) + O(\alpha \mu_f^{-1}) \Bigr]O(\alpha / \nu \mu_f^4).
$$
Hence
$$
B D^{-1} (O(\alpha / \nu \mu_f^4),0)  =  \Bigl[ (1,0) + O(\alpha \mu_f^{-1}) \Bigr]O(\alpha / \nu \mu_f^4)
$$
and $A^{-1} B D^{-1} (O(\alpha / \nu \mu_f^4),0) = O(\alpha / \nu \mu_f^4)$. Using again $\alpha \ll \mu_f$, we obtain
that $a_{\pm}$ are of order $O(1 / \nu \mu_f^2 \alpha) = O(1 / \alpha d(\alpha,\lambda))$.


\subsection{Boundary condition}


We now add to $G^{app}$ another approximate Green function $\hat G^{app}$ to handle the boundary conditions.
We look for $\hat G^{app}$ under the form
$$
\hat G^{app}(y) = d_s {\phi_{s,-}(y) \over d_1} + d_f {\phi_{f,-}(y)  \over \phi_{f,-}(0)}
$$
where the normalization constant $d_1$ will be fixed later, 
and look for $d_s$ and $d_f$ such that
\beq \label{Greenb1}
G^{app}(x,0) + \hat G^{app}(0) = 
\partial_y G^{app}(x,0) + \partial_y \hat G^{app}(0) =  0.
\eeq
Let
$$
\hat M = \left( \begin{array}{cc} \phi_{s,-} / d_1 & \phi_{f,-} / \phi_{f,-}(0) \cr
\partial_y \phi_{s,-} / d_1 & \partial_y \phi_{f,-} / \phi_{f,-}(0) \cr \end{array} \right) ,
$$
the functions being evaluated at $y = 0$. Then (\ref{Greenb1}) can be rewritten as
$$
\hat M d = - (G^{app}(x,0), \partial_y G^{app}(x,0)) 
$$
where $d = (d_s,d_f)$. Note that
$$
(G^{app}(x,0), \partial_y G^{app}(x,0))  = Q (a_+,b_+)
$$
where
$$
Q = \left( \begin{array}{cc}
 \phi_{s,+}(0) /c_2  & 1 \cr
 \partial_y \phi_{s,+}(0) / c_2 & \partial_y \phi_{f,+}(0) / \phi_{f,+}(0) \cr
\end{array} \right) .
$$
By construction
\beq \label{defid}
d = -  \hat M^{-1}  Q (a_+,b_+) .
\eeq
Let us first consider bounded $\alpha$. We take $d_1 = 1$. This leads to
$$
\hat M = \left( \begin{array}{cc} \phi_{s,-}(0) & 1 \cr
\partial_y \phi_{s,-}(0) & - \mu_f + O(1) \cr \end{array} \right) .
$$
Note that  $\hat M = M_1 M_2$ with
$$
M_1 = \left( \begin{array}{cc} 1 & 0 \cr 0 &  \mu_f \cr \end{array} \right), \qquad
M_2 = \left( \begin{array}{cc} \phi_{s,-}(0) & 1 \cr
\partial_y \phi_{s,-}(0) / \mu_f & -1 + O( 1 / \mu_f) \end{array} \right) .
$$
The determinant of $M_2$ equals to $-E(\alpha,\lambda) = -\phi_{s,-}(0)$, up to a small term of order $\mu_f^{-1}\sim \sqrt \nu$, recalling that $\alpha$ is bounded.
Hence $M_2$ is invertible, and $M_2^{-1}$ is bounded if $E(\alpha,\lambda) \ne 0$, provided $ \nu$ is small enough.
Then
$$
\hat M^{-1} Q = M_2^{-1} M_1^{-1} Q.
$$
Note that $(a_+,b_+) = (O(1/\nu \mu_f^2), O(1 / \nu \mu_f^3))$. Hence $Q (a_+,b_+) = O(1 / \nu \mu_f^2)$.
Therefore $M_1^{-1} Q (a_+,b_+) = (O(1/ \nu \mu_f^2), O(1/ \nu \mu_f^3))$. Hence, as the second term of the first
column of $M_2$ is of order $O(1 / \mu_f)$ we get, as desired, that 
\beq \label{boundofd}
d = (O(1 / \alpha \nu \mu_f^2), O(1 / \nu \mu_f^3)),
\eeq
keeping in mind that $\alpha$ is bounded.

For large $\alpha$ we choose $d_1 = \phi_{s,-}(0)$. Then
$$
Q = \left(\begin{array}{cc} 1 & 1 \cr
\alpha + O(1) & \mu_f + O(1) \cr
\end{array} \right),
$$
$$
\hat M = \left( \begin{array}{cc} 1 & 1  \cr  - \alpha + O(1)  & - \mu_f + O(1) \cr \end{array} \right)
$$
and
$$
\hat M^{-1} = {1 \over \mu_f - \alpha + O(1)} \left( \begin{array}{cc} \mu_f  + O(1) & 1 \cr
- \alpha + O(1) & -1 \cr \end{array} \right) .
$$
In this case 
$$
(a_+,b_+) = (O(1 / \alpha  \nu \mu_f^2),O(1/ \nu \mu_f^3)).
$$
A direct computation of  $\hat M^{-1} Q (a_+,b_+)$ again gives (\ref{boundofd}).
Combining all the previous estimates ends the proof.


\section{Exact Green function}\label{sec-exactGreen}


Let 
$$
H^{app} = G^{app} + \hat G^{app}
$$ 
be the complete approximate Green function. By construction, $H^{app}$ satisfies the zero boundary conditions \eqref{G-noslip}. 
We now construct the exact Green function $G(x,z)$ as an infinite sum
\beq \label{defG}
G(x,z) = \sum_{n \ge 0} G_n(x,z),
\eeq
where $G_0 = H^{app}$, 
$$
G_1 = - H^{app} \star (\OS_{\alpha,\lambda}(H^{app}) - \delta_{y=x}),
$$ 
and $G_n$ is defined by iteration through
$$
G_{n+1} = - H^{app} \star \OS_{\alpha,\lambda}(G_n).
$$
Hence, it suffices to prove that the series \eqref{defG} converges in a suitable function space, which follows immediately from the following lemma. The stated bounds for $G(x,z)$ in Theorem \ref{theo-GreenOS-unstable} then follow from those on $H^{app}(x,z)$. 

\begin{lem}
For each $x$, assume that
$$
| f^x(y) | \le e^{- \alpha'  |x-y | } 
$$
for some $\alpha'$ such that $\alpha' < | \alpha |$ and $\alpha' < \Re \mu_f$.
Then
$$
| \OS_{\alpha,\lambda}(G^{app} \star f^x) (y)| \le C \nu^{N-2} e^{- \alpha' |x-y| } .
$$
\end{lem}
\begin{proof}
Note that
$$
\OS_{\alpha,\lambda}(G^{app} \star f^x)(y)  = \int \OS_{\alpha,\lambda}(G^{app})(z,y) f^x(z) dz .
$$
However we recall that $\phi_{s,\pm}^{app}$ satisfy
$$
| \OS_{\alpha,\lambda}(\phi_{s,\pm}^{app})  | \le C \nu^N e^{\pm |\alpha|  z },
$$
$$
| \OS_{\alpha,\lambda}(\phi_{f,\pm}^{app}) |\le C \nu^N |\phi_{f,\pm}^{app} | .
$$
Using the bounds on the coefficients on $G^{app}(z,y)$, this leads to
$$
| \OS_{\alpha,\lambda}(G^{app}(z,y)) | \le C \nu^{N-2}  e^{- \alpha | y - z | }.
$$
The Lemma follows by convolution.
\end{proof}


\section{Construction of a pseudo inverse \label{sectionpseudo}}


We now focus on the case when $\lambda$ is close to a simple eigenvalue $\lambda_0$.

\begin{theo} \label{theopseudo}
Let $\alpha$ be fixed.
Let $\lambda_0$ be a simple eigenvalue of $Orr_{\alpha,\lambda}$ with corresponding eigenmode
$\phi_{\alpha,\lambda_0}$. Then there exists a bounded family of linear forms $l^\nu$
and a family of pseudoinverse operators $Orr^{-1}_{\alpha,\lambda}$ such that for any stream function $\phi$,
$$
Orr_{\alpha,\lambda}\Bigl(Orr^{-1}_{\alpha,\lambda}(\phi) \Bigr)  = \phi - l^\nu(\phi) \phi_{\alpha,\lambda_0} 
$$
for $\lambda$ near $\lambda_0$. Moreover, the pseudoinverse $Orr_{\alpha,\lambda}^{-1}$ may be defined through a Green function $\widetilde G_{\alpha,\lambda}(x,z)$ which satisfies the same bounds in (\ref{est-GrOS}).
\end{theo}


\subsection{Principle of the construction}


Let us sketch the principle of the proof on a simplified case. Let $A_0$ be a $N \times N$ matrice of rank $N-1$
(which is a toy model for the Rayleigh operator when $\lambda$ is a simple eigenvalue), and let $A(\eps)$
be a bounded family of $N \times N$ matrices (toy model for Orr Sommerfeld equation). We want to construct an inverse for 
$$
A^\eps = A_0 + \eps A(\eps).
$$
Let us first invert $A_0$. Let $v$ be a unit vertor, orthogonal to the image of $A_0$. Let $P$ be the orthogonal projector
on the image of $A$, namely
$$
P v =  f - (f. v) v.
$$
Let $B$ be a pseudo inverse of $A_0$, namely a matrix such that, on the image of $A_0$, $A_0 B =Id$. 
Then  $u = B P f$ solves 
$$
A_0 u = f - (f.v) v .
$$
We now fulfill a similar construction for $A^\eps$ for small $\eps$. 
Let  $u_0 = B  P f$. Then 
$$
A^\eps u_0 = f - (f.v) v  + \eps A(\eps) u_0.
$$
We know define $u_1 = -  B P A(\eps) u_0$. Then $u_0 + u_1$ solves
$$
A^\eps (u_0 + \eps u_1) = f - (f.u_0) v + \eps (A(\eps) u_0. v) v - \eps^2 A(\eps) B P A(\eps) u_0 
$$
and the construction follows by iteration.


\subsection{Rayleigh equation}


In this section we fix $\alpha$ and investigate the Rayleigh operator $Ray_{\alpha,\lambda}$ 
when $\lambda$ is near a simple eigenvalue $\lambda_0$ of $Ray_{\alpha,\lambda}$. 
We will also assume that $Ker(Ray_{\alpha,\lambda_0}^2) = \cit \phi_{\alpha,\lambda_0,\pm}$.
At $\lambda = \lambda_0$, $\phi_{\alpha,\lambda_0,\pm}$ are colinear (that is, the Jacobian of $\phi_{\alpha,\lambda_0,\pm}$ vanishes). 
Up to a renormalisation we may assume that
$\phi_{\alpha,\lambda_0,+} = \phi_{\alpha,\lambda_0,-}$.  
For $\lambda \ne \lambda_0$ the solution of $Ray_{\alpha,\lambda}(\phi) = \psi$ is explicitely given by
\begin{equation}\label{def-psRphi}
\phi(z) = \phi_{\alpha,\lambda,+}(z) \int_z^{+ \infty} {\phi_{\alpha,\lambda,-}(x) \over Jac(x)} \psi(x) dx
+ \phi_{\alpha,\lambda,-}(z) \int_0^z {\phi_{\alpha,\lambda,+}(x) \over Jac(x)} \psi(x) dx
\end{equation}
where
$$
Jac(x) := \phi_{\alpha,\lambda,-}(x) \partial_x \phi_{\alpha,\lambda,+}(x) 
- \phi_{\alpha,\lambda,+}(x) \partial_x \phi_{\alpha,\lambda,-}(x) 
$$
is the Jacobian of $\phi_{\alpha,\lambda,\pm}$. 
Note that, as $\lambda_0$ is a simple eigenvalue, $Jac(\lambda_0) = 0$ and that for $\lambda$ near $\lambda_0$,
$$
Jac(\lambda) = (\lambda - \lambda_0) \widetilde  Jac(\lambda)
$$
where $\widetilde  Jac(\lambda)$ is a smooth function with $\widetilde  Jac(\lambda_0)\not =0$ since $\lambda_0$
is a simple eigenvalue. Let us also define
$$
\widetilde  \phi_{\alpha,\lambda,\pm} = {\phi_{\alpha,\lambda,\pm} - \phi_{\alpha,\lambda_0,\pm} \over \lambda - \lambda_0} .
$$
Then it follows from \eqref{def-psRphi} that 
$$
\phi(z) = {\phi_{\alpha,\lambda_0,+}(z) \over \lambda - \lambda_0} 
\int_0^{+ \infty} {\phi_{\alpha,\lambda_0,+}(x) \over \widetilde  Jac(x)} \psi(x) dx + \widetilde  \phi(z) 
$$
where 
\begin{equation}\label{def-psphi}
\widetilde  \phi(z) = \int_0^{+ \infty} \widetilde  G(x,z) \psi(x) dx 
\end{equation}
with $$
\widetilde  G(x,z) = {\widetilde  \phi_{\alpha,\lambda,+}(z) \phi_{\alpha,\lambda,-}(x) + \phi_{\alpha,\lambda,+} (z) \widetilde  \phi_{\alpha,\lambda,-}(x) 
+ (\lambda - \lambda_0) \widetilde  \phi_{\alpha,\lambda,+}(z) \widetilde  \phi_{\alpha,\lambda,-}(x) \over \widetilde  Jac(x)}
$$
if $x > z$, and a similar expression if $x < z$.
This computation may be rewritten as follows.
 Let $l$ be the linear form defined by
$$
l(\psi) =  \int_0^{+ \infty} {\phi_{\alpha,\lambda_0,+}(x) \over \widetilde  Jac(x)} \psi(x) dx.
$$
Then, for any $\psi$,  if $l(\psi) = 0$ then $\widetilde  \phi$ solves
$Ray_{\alpha,\lambda}(\widetilde  \phi) = \psi$. 
In particular, as the image of the Rayleigh operator $Im(Ray_{\alpha,\lambda_0})$ is of codimension $1$, 
$Ker(l) = Im(Ray_{\alpha,\lambda_0})$. 
Note that, as $\lambda_0$ is a simple eigenvalue,
$\phi_{\alpha,\lambda_0,+}$ is not in $Im(Ray_{\alpha,\lambda_0})$. Therefore, $l(\phi_{\alpha,\lambda_0,+} ) \ne 0$.
As a consequence
$$
\tilde \psi = \psi - {l(\psi) \over l(\phi_{\alpha,\lambda_0,+} )}  \phi_{\alpha,\lambda_0,+} \in Im(Ray_{\alpha,\lambda})
$$ 
since the image by $l$ of this function vanishes.
We then have
\beq \label{decomp1}
Ray_{\alpha,\lambda}(\widetilde  \phi) = \psi -  {l(\psi) \over l(\phi_{\alpha,\lambda_0,+} )}   \phi_{\alpha,\lambda_0,+},
\eeq
where
$$
\tilde \phi(z) = \int_0^{+ \infty} \tilde G(x,z) \tilde \psi(x) dx .
$$
That is, $\widetilde \phi$ defines the pseudoinverse $Ray_{\alpha,\lambda}^{-1}$ of $Ray_{\alpha,\lambda}$ for $\lambda$ near $\lambda_0$. We shall now fulfill a similar analysis for the $Orr_{\alpha,\lambda}$ operator.


\subsection{Orr Sommerfeld equation}


Let us now prove Theorem \ref{theopseudo}. 
We follow the analysis in the previous section to construct the Green function $\widetilde  G_{\alpha,\lambda}(x,z)$ 
for the pseudoinverse of $Orr_{\alpha,\lambda}$. Let $\lambda_0^{app}$ be a simple eigenvalue of the approximate
Evans function $E^{app}$. $Ray_{\alpha,\lambda}$ operator. To simplify the notation we drop the "app" and set
$\lambda_0 = \lambda_0^{app}$.
At $\lambda = \lambda_0$, the matrix $M$, defined by (\ref{matriceM}), is singular since its first two columns are colinear.
Up to the multiplication by a constant of $\phi_{s,-}$, we may assume that $\phi_{s,\pm}$ coincide at $\lambda = \lambda_0$.
To desingularize it we introduce
$$
\Lambda = \left( \begin{array}{cccc} 
(\lambda - \lambda_0)^{-1} & 1 & 0 & 0 \cr
- (\lambda - \lambda_0)^{-1} & 1 & 0 & 0 \cr
0 & 0 & 1 & 0 \cr
0 & 0 & 0 & 1 \cr
\end{array} \right) .
$$
Then, recalling \eqref{def-vvvapp} and defining $\widetilde  M = M \Lambda$, with the notation of (\ref{Mv}), we have
\begin{equation}\label{def-psv}
v = \Lambda \widetilde  M^{-1} (0,0,0,1/ \nu \mu_f^3) ,
\end{equation}
when $\lambda \ne \lambda_0$.
The arguments applied to the matrix $M$ in Section $3$ may now be applied to $\widetilde  M$ since the corresponding matrix
$$
\widetilde  A_2 = A_2 \Lambda
= \left( \begin{array}{cc} (\phi_{s,-} - \phi_{s,+}) / (\lambda - \lambda_0) & \phi_{s,+} \cr
(\partial_y \phi_{s,-} - \partial_y \phi_{s,+}) / (\lambda - \lambda_0) & \partial_y \phi_{s,+} \cr
\end{array} \right)
$$  
is non singular near $\lambda = \lambda_0$, keeping in mind that $\lambda_0$ is a simple eigenvalue. 

Let $l_4 = (l_{4,1},...,l_{4,4})$ 
be the fourth line of the inverse of $\widetilde  M$. 
It  follows from \eqref{def-psv} that 
$$
v = \Lambda l_4(x) / \nu \mu_f^3.
$$
The singular part $v^s$ of $v$, namely the terms involving $(\lambda - \lambda_0)^{-1}$, is 
$$
v^s ={1 \over \nu \mu_f^3} l_{4,1}(x) (1 , -1, 0, 0).
$$ 
Let us now compute $l_{4,1}(x)$. We have to evaluate $\Lambda^{-1} A_2^{-1} A_1^{-1} B D^{-1}(0,1)$ (see Section \ref{sec-bda}).
But, up to higher order terms, $A_1^{-1} B D^{-1} \sim (0,\mu_f)$.  Note that
$$
A_2^{-1} = {1 \over E^{app}(\alpha,\lambda)} \left( \begin{array}{cc}
\partial_y \phi_{s,+} & - \phi_{s,+} \cr
- \partial_y \phi_{s,-} & \phi_{s,-} \cr \end{array} \right).
$$
Hence when $\lambda$ is close to $\lambda_0$, 
$$
A_2^{-1} A_1^{-1} B D^{-1}(0,1) \sim {\mu_f \over E^{app}(\alpha,\lambda)} \phi_{s,+} (-1,1)
$$ 
namely like $C (-\mu_f,  \mu_f) \phi_{s,+} / (\lambda - \lambda_0)$.
At leading order, the computation is exactly the same as in the previous section.
Let 
$$
L(\psi) = -  \int_0^{+ \infty} l_{4,1}(x) \psi(x) dx .
$$
Then, at leading order, $L = l$. Moreover, the regular part $v^r$ of $v = v^r + v^s$ is
$$
v^r =  {1 \over \nu \mu_f^3}  \Bigl( l_{4,2},l_{4,2},l_{4,3},l_{4,4} \Bigr) .
$$
We now define $\widetilde  G^{app}(x,z)$ to be the approximate Green kernel that corresponds to the regular part $v^r$, 
recalling the Green function construction in \eqref{def-GappX}-\eqref{def-vvvapp}.
Setting 
$$
\widetilde  \psi =  \psi - {L(\psi) \over L(\phi_{\alpha,\lambda_0,+})}  \phi_{\alpha,\lambda_0,+}  ,
$$
we have $L(\widetilde  \psi) = 0$ and so 
$$
Orr_{\alpha,\lambda}(\widetilde  G^{app} \star \psi) = \widetilde  \psi .
$$
The exact Green function $\widetilde G_{\alpha,\lambda}(x,z)$ then follows by iteration
as in Section \ref{sec-exactGreen}.

%
%
%
%
%
%
%
%
%

\chapter{Orr Sommerfeld solutions for stable profiles}\label{chapter-OS-stable}

In this chapter, we turn to the most delicate case, namely to the study of classical Orr-Sommerfeld solutions 
near critical layers. We recall 
\beq \label{OS-long}
\OS_{\alpha,c}(\phi_\alpha) := 
- \eps \Delta_\alpha^2 \phi_\alpha + (U - c) \Delta_\alpha \phi_\alpha - U'' \phi_\alpha = 0
\eeq
where
$$
\eps = {\nu \over i \alpha},
$$
together with the boundary conditions
\beq \label{OS2}
{\phi_\alpha}_{\vert_{z=0}} = \partial_z{ \phi_\alpha}_{\vert_{z=0}} = 0, \qquad \lim_{z\to \infty}\phi_\alpha(z) =0.
\eeq
The aim of this chapter is to give bounds on the Green function of the Orr Sommerfeld equation 
when $\alpha$ is of order  $\nu^{1/4}$ and $c$ is of the same order, which corresponds to
an instability area. 
As it turns out, this restricted study appears to be sufficient to construct linear and nonlinear instabilities 
for the full nonlinear Navier Stokes equations.  

To construct the Green function we first construct 
two approximate solutions $\phi_{s,\pm}$ with a "slow behavior", and two
approximate solutions $\phi_{f,\pm}$ with a "fast behavior" (the "-" solutions going to $0$ as $z$ goes to $+ \infty$). The slow approximate solutions will be solutions of the Rayleigh equation
\beq \label{Rayleigh}
(U - c) \Delta_\alpha \phi - U'' \phi = 0
\eeq
with boundary condition $\phi(0) = 0$. They will be constructed by perturbation of the case
$\alpha = 0$ where the Rayleigh equation degenerates in
\beq \label{limitRay}
Ray_0(\phi) = (U - c) \partial_z^2 \phi - U'' \phi .
\eeq
The main observation is that $\phi_{1,0} = U - c$ is a particular solution of (\ref{limitRay}). 
Let $\phi_{2,0}$ be the other solution of this equation such that the Wronskian $W[\phi_{1,0},\phi_{2,0}]$ equals $1$. 
We will construct approximate solutions starting from $\phi_{1,0}$ and prove that
\beq \label{defiEvans1}
\phi_{s,-}(0) = U(0) - c + \alpha(U_+ - U(0))^2 \phi_{2,0}(0) + O(\alpha^2), 
\eeq
\beq \label{defiEvans2}
\partial_z \phi_{s,-}(0) = U'(0) + O(\alpha).
\eeq
The "fast approximate solutions" will emerge in the balance between $- \eps \Delta_\alpha^2 \phi$
and $(U - c) \Delta_\alpha \phi$. Keeping in mind that $\alpha$ is small, they will be constructed starting
from solutions of the simplified equation
$$
- \eps \partial_z^4 \phi + (U - c) \partial_z^2 \phi = 0.
$$
As $c$ is small, and as $U'(0) \ne 0$, there exists a unique $z_c \in \cit$ near $0$ such that
\beq \label{critic1}
U(z_c) = c .
\eeq
Such a $z_c$ is called a "critical layer" in the physics literature.
It turns out that all the instability is driven by what happens near this critical layer.
Equation (\ref{critic1}) is a perturbation of the Airy equation
\beq \label{Airy}
- \eps \partial_z^2 \psi + U'(0) (z - z_c) \psi = 0
\eeq
posed on $\psi = \partial_z^2 \phi$. The fast approximate solutions are thus constructed as perturbations
of second primitives of classical Airy functions.

The  main result is as follows. 

\begin{theorem}\label{theo-GreenOS-stable} 
Let $\sigma_0$ be arbitrarily small.
Let $\alpha = O(\nu^{1/4})$, and let $c = O(\nu^{1/4})$, with $| \Im c | \ge \sigma_0 \nu^{1/4}$, 
such that
\beq \label{disp}
{| W[ \phi_{s,-}, \phi_{f,-} ]| \over | \phi_{f,-}(0) | } \ge \sigma_0 \nu^{1/4}.
\eeq
Let $G_{\alpha,c}(x,z)$ be  the Green function of the Orr-Sommerfeld problem. 
Then, there exists a smooth function $P(x)$ and there are universal positive constants $\theta_0, C_0$ so that 
\begin{equation}\label{est-GrOS-stable}
\begin{aligned}
 |G_{\alpha,c}(x,z) - \nu^{-1/4} P(x) \phi_{s,-}(z) |  &\le 
  C_0 \Big( e^{-\theta_0  |x-z|} + 
 {1 \over | \mu_f(x) |} e^{ - | \int_x^z \Re \mu_f(y) dy | } \Big)
\end{aligned}
\end{equation}
uniformly for all $x,z\ge 0$. 
Similarly, 
\begin{equation}\label{est-GrOS-stable2}
\begin{aligned}
 | \partial_z G_{\alpha,c}(x,z)  - \nu^{-1/4} P(x) \partial_z \phi_{s,-}(z) |  &\le 
  C_0 \Big( e^{-\theta_0  |x-z|} + 
 {| \mu_f(z) | \over | \mu_f(x) |} e^{ - | \int_x^z \Re \mu_f(y) dy | } \Big)
\end{aligned}
\end{equation}
where $$
\mu_f(z) = \sqrt{ { U(z) - c \over \eps}},
$$
taking the square root with a positive real part.
 \end{theorem}

Let us comment (\ref{disp}). We have
$$
W[ \phi_{s,-}^{app}, \phi_{f,-}^{app} ]
= \gamma \psi_{s,0}^{app}(0) Ti(- \gamma z_c)  \phi_{f,-}^{app}(0)   - \partial_z \phi_{s,-}^{app}(0)  \phi_{f,-}^{app}(0) 
$$
$$
= - \Bigl( \gamma c Ti(- \gamma z_c) + U'(0) \Bigr) Ai(2,-\gamma z_c) + O(\nu^{1/4})
$$
Note that both terms under the brackets are of order $O(1)$, since $\gamma c$ is of order $O(1)$.
The Wronskian vanishes if there exists a linear combination of
$\phi_{s,-}^{app}$ and $\phi_{f,-}^{app}$ which satisfies the boundary conditions, namely if there exists an approximate
eigenmode of $\OS_{\alpha,c}$ (recalling that $\phi_{s,-}^{app}$ and $\phi_{f,-}^{app}$ are only approximate solutions of $\OS_{\alpha,c}$).
We have to remain away from such approximate modes, since nearby there exists true eigenmodes where $\OS_{\alpha,c}$
is no longer invertible. Note that $\sigma_1$ may be taken arbitrarily small.

Note that in this Theorem we are at a distance $O(\nu^{1/4})$ from a simple eigenmode $\psi_0$. It is therefore expected that
$\OS_{\alpha,c}$ is of order $O(\nu^{-1/4})$ and that
\beq \label{inverseOrr}
\OS^{-1}_{\alpha,c}(\psi) = \nu^{-1/4} \Bigl( \int_0^{ +\infty} P(z) \psi(z) dz \Bigr) \psi_0 + O(1) .
\eeq
As $\psi_0 = \phi_{s,-} + O(\nu^{1/4})$, $G_{\alpha,c}$ is only bounded
by $O(\nu^{-1/4})$, and its main component is $ \nu^{-1/4} P \phi_{s,-}$.


\section{The Airy operator \label{sectionAiry}}


In this section, we construct two approximate solutions of Orr Sommerfeld equation, called
$\phi_{f,\pm} = \phi_{f,\pm}^{app}$, with fast increasing or decreasing behaviors. For these approximate solutions, it turns out
that the zeroth order term $U'' \phi_{f,\pm}$ may be neglected. 
Moreover, as $\alpha$ is small, $\alpha^2$ terms may also be neglected.
This simplifies the Orr Sommerfeld operator in the so called modified Airy operator defined by
\begin{equation}\label{Airy-de}
\Airy = \cA \partial_z^2 ,
\end{equation}
where
\begin{equation}\label{def-cA}
\mathcal{A}: = - \eps \partial_z^2 + (U-c) .
\end{equation}
Note that
\beq \label{Airy-de-2}
\OS_{\alpha,c} = \Airy + \hbox{OrrAiry} 
\eeq
where
$$
\hbox{OrrAiry} = 2 \eps \alpha^2 \partial_z^2 - \eps \alpha^4  - \alpha^2 (U-c) - U'' .
$$
Note also that $U-c$ behaves like $U'(z_c) ( z - z_c)$ for $z$ near $z_c$, hence $\cA$ is very similar to the classical Airy
operator $\partial_z^2 - z$ when $z$ is close to $z_c$. The main difficulty lies in the fact that the "phase" $U(z) - c$ almost
vanishes when $z$ is close to $\Re z_c$, hence we have to distinguish between two cases: $z \le \sigma_1$ and
$z \ge \sigma_1$ for some small $\sigma_1$. 
The first case is handled through a Langer transformation, which reduces (\ref{Airy-de}) to the classical
Airy equation. The second case may be treated using a classical WKB expansion.

We will prove the following proposition. 

\begin{proposition} \label{prop-maphif}
There exist two smooth functions $\phi^{app}_{\pm}(z)$ such that
\beq \label{fastapp1}
|{\cal A} \partial_z^2 \phi^{app}_{\pm} | \le C \nu^N | \phi^{app}_{\pm} |,
\eeq
\beq  \label{fastapp1b}
| \OS_{\alpha,c}(\phi_{\pm}^{app}) | \le C | \phi^{app}_{\pm} |,
\eeq
for arbitrarily large $N$.
Moreover for $z \gg \nu^{1/4}$ and for $k =1,2,3$, as $\nu \to 0$,
\beq \label{fastapp1c}
{\partial_z^k \phi^{app}_-(z) \over \phi^{app}_-}(z) \sim (-1)^k \mu_f^k(z),
\eeq
and similarly for $\phi^{app}_+$ without the $(-1)^k$ factor.
For $k = 1, 2, 3$ and  any $x_1 < x_2$,
\beq \label{fastapp2}
\Bigl| {\phi^{app}_{+}(x_2) \over \phi^{app}_{+}(x_1)} \Bigr| \le C \exp \Bigl(
\int_{x_1}^{x_2} \Re \mu_f(y) dy \Bigr) 
\eeq
and similarly for $\phi^{app}_-$.
\end{proposition}

To prove this proposition we construct $\psi^{app}_\pm = \partial_z^2 \phi^{app}_\pm$ for 
$z < z_c$ in Section \ref{Airypoints} using the Langer's transformation
introduced in (\ref{sec-Langer}) and for $z > z_c$ in Section \ref{sec-WKB} using the classical WKB method. We then match these
two constructions in Section \ref{sec-match}, integrate them twice in Section \ref{AA}
and detail the Green function of Airy operator in Section \ref{GreenAiry}.


\subsection{A primer on Langer's transformation}\label{sec-Langer}


The first step is to construct approximate solutions to ${\cal A} \psi = 0$, 
starting from solutions of the genuine Airy equation
$\eps \psi'' = y \psi$, thanks to the so called Langer's transformation that we will now detail.
Let $B(x)$ and $C(x)$ be two smooth functions. 
In $1931$,  Langer introduced the following method to build approximate solutions to the varying coefficient
Airy type equation
\beq \label{nonlinC}
- \eps \phi'' + C(x) \phi = 0
\eeq
starting from solutions to the similar Airy type equation
\beq \label{nonlinB}
- \eps \psi'' + B(x) \psi = 0.
\eeq
We assume that  both $B$ and $C$ vanish at some point $x_0$, and that their derivatives at $x_0$ does not vanish.
Let $\psi$ be any solution to (\ref{nonlinB}). Let $f$ and $g$ be two
smooth functions, to be chosen later. Then 
$$
\phi(x) = f(x) \psi(g(x))
$$ 
satisfies
$$
- \eps \phi'' + C(x) \phi = - \eps f'' \psi -  2 \eps f' \psi' g' - B(g(x))  (g')^2 f \psi - \eps f \psi' g'' +  C(x) f\psi.
$$
Note that $f$ may be seen as a modulation of amplitude and $g$ as a change of phase.
If we choose $g$ such that 
\beq \label{eqg}
B(g(x)) (g')^2 = C(x) 
\eeq
and $f$ such that
\beq \label{eqf}
2 f' g' + f g'' = 0,
\eeq
we have
$$
- \eps \phi'' + C(x) \phi = - \eps f'' \psi.
$$
Hence $\phi$ may be considered as an approximate solution to $- \eps \phi'' + C(x) \phi = 0$.

Note that (\ref{eqf}) may be solved, yielding
\beq \label{eqf1}
f(x) = {1 \over \sqrt{g'(x)}} .
\eeq
Now let $B_1$ be the primitive of $\sqrt{B}$ which vanishes at $x_0$
and let $C_1$ be the primitive of $\sqrt{C}$ which vanishes at $x_0$.
Then (\ref{eqg}) may be rewritten as
\beq \label{eqg2}
B_1(g(x))  = C_1(x).
\eeq
Note that both $B_1$ and $C_1$ behave like $C_0(x - x_0)^{3/2}$ near $x_0$. Hence (\ref{eqg2}) may be solved for
$x$ near $x_0$. This defines a smooth function $g$ which satisfies $g(x_0) = x_0$. Moreover if $B'(x_0) = C'(x_0)$
then $g'(x_0) = 1$.



\subsection{Airy critical points \label{Airypoints}}


In this section we  use Langer's transformation to construct approximate solutions to ${\cal A} \psi = 0$ starting
from solutions of the genuine Airy equation.

Let $c$ be of order $\nu^{1/4}$. Then there exists an unique $z_c \in \cit$ near $0$ such that $U(z_c) = c$. 
Note that $z_c$ is also of order
$\nu^{1/4}$ since $U'(0) \ne 0$. Expanding $U$ near $z_c$ at first order we get the approximate equation
\beq \label{order2l}
- \eps \partial_z^2 \psi + U'(z_c) (z - z_c) \psi = 0 
\eeq
which is the classical Airy equation. Let us assume that $\Re U'(z_c) > 0$, the opposite case being similar.
A first solution is given by  
$$
A(z) = Ai ( \gamma (z - z_c) )
$$
where $Ai$ is the classical Airy function, solution of $Ai'' = x Ai$, and where
$\eps \gamma^3 =  U'(z_c)$, namely
$$
\gamma = \Bigl(  {i \alpha U'(z_c) \over \nu} \Bigr)^{1/3} .
$$
Note that since $\alpha$ is of order $\nu^{1/4}$, $\gamma$ is of  order $\nu^{-1/4}$ and that 
$$
\arg(\gamma) =  + \pi / 6 + O(\nu^{-1/4}).
$$
Moreover, as $x$ goes to $\pm \infty$, with argument $i \pi / 6$,
$$
Ai(x) \sim {1 \over 2 \sqrt{\pi}}  { e^{- 2 x^{3/2}  / 3} \over x^{1/4}} .
$$
In particular, $Ai'(x) / Ai(x) \sim - x^{1/2}$ for large $x$. Hence, as $\gamma(z - z_c)$ goes to infinity, $A(z)$ goes to $0$ and 
\beq \label{asymplog}
{A'(z) \over A(z)} \sim - \gamma^{3/2} (z - z_c)^{1/2} = 
-  \Bigl(  {i \alpha U'(z_c) \over \nu} \Bigr)^{1/2} (z - z_c)^{1/2} \sim - \sqrt{B(z)},
\eeq
with 
$$
B(z) = \eps^{-1} U'(z_c) (z - z_c).
$$ 
More precisely, we get 
$$
{A'(z) \over A(z)}  = - \sqrt{B(z)} (1 + O(\nu^{1/4}))
$$
for $z \gg \nu^{1/4}$ (on which $\gamma (z-z_c) \gg1$). 

An independent solution is given by $Ci (\gamma (z - z_c))$ where
$$
Ci = - i \pi (Ai + i Bi),
$$
with $Bi(\cdot)$ being the other classical Airy function. In this case $| Ci (\gamma (z - z_c)) |$ goes to $+ \infty$ as $z - z_c$ goes to $+ \infty$,
with a plus instead of the minus in the corresponding formula (\ref{asymplog}).

We now use Langer's transformation introduced in the previous section. As $U(z)$ and $U'(z_c) (z - z_c)$ vanish at the same point
with the same derivative at that point, we use Langer's transformation with
$$
C(z) = \eps^{-1} (U(z) - c)
$$
and
$$
B(z) = \eps^{-1} U'(z_c) (z - z_c).
$$
Then, $g(z)$ is locally well defined, for $0 \le z \le \sigma_1$ for some positive $\sigma_1$. 
Moreover $g( z_c) = z_c$ and $g'( z_c) = 1$. Now
$$
\tilde Ai (z) = {1 \over \sqrt{g'(z)}} Ai \Bigl(\gamma g(z) \Bigr) 
$$
and
$$
\tilde Ci(z) = {1 \over \sqrt{g'(z)}} Ci \Bigl(\gamma g(z) \Bigr) 
$$
are two approximate solutions of $\cA \phi = 0$ in the sense that
$$
\cA \tilde Ai = - \eps f'' Ai(\gamma g(z)) , \qquad \cA \tilde Ci = - \eps f'' Ci(\gamma g(z)).
$$
Note that the error term is of order $\eps \sim \nu^{3/4}$. 
Note also that at first order, for $z$ of order $\nu^{1/4}$,
$\tilde Ai(z)$ equals $Ai(\gamma (z - z_c))$ since $g'(z_c) = 1$. 

Moreover, for $z \gg \nu^{1/4}$, using (\ref{eqg}), we get
\beq \label{polarloin}
{\partial_z \tilde Ai(z) \over \tilde Ai(z)} \sim  g'(z) {A'( g(z)) \over A( g(z))} \sim - g'(z) \sqrt{B(g(z))} 
\sim  - \sqrt{C(z)} \sim - \mu_f(z),
\eeq
and more precisely
$$
{\partial_z \tilde Ai(z) \over \tilde Ai(z)} \sim - \mu_f(z) (1 + O(\nu^{1/4})) .
$$
Similarly for $z \gg \nu^{1/4}$, we get
\beq \label{polarloin2}
{\partial_z^k \tilde Ai(z) \over \tilde Ai(z)} \sim (-1)^k \mu_f^k(z) .
\eeq


\subsection{Away form the critical layer \label{sec-WKB}}


If $z - z_c$ is small then $g$ is well defined, precisely on $[0,\sigma_1]$ for some small $\sigma_1$ as in the previous section. However, if $z > \sigma_1$, then Langer's transformation is no longer useful, and we may directly use a WKB expansion.
We look for solutions $\psi$ of the form
$$
\psi(z) = e^{ \theta(z) / \eps^{1/2}} 
$$
to the equation $\cA\psi = \eps \partial_z^2 \psi - (U - c) \psi = 0$.
Note that
$$
\eps \partial_z^2 \psi= \Bigl( \theta'^2 + \eps^{1/2} \theta'' \Bigr) \psi.
$$
Hence we look for $\theta$ such that
\beq \label{theta2}
\theta'^2 + \eps^{1/2} \theta'' = (U - c) .
\eeq
As we are only interested in approximate solutions, we solve (\ref{theta2}) in an approximate way, and look
for $\theta$ of the form
$$
\theta = \sum_{i = 0}^M \eps^{i/2} \theta_i 
$$
for some arbitrarily large $M$. The $\theta_i$ may be constructed by iteration, starting from
$$
\theta_0' = \pm \sqrt{U (z) - c} .
$$
If we keep the positive real part to the square root, the $-$ choice leads to a solution going to $0$ at $+ \infty$
and the $+$ choice to a solution going to $+ \infty$ at $+ \infty$.
This construction gives a solution $\psi^{app}_{f,\pm}$ such that
$$
| \cA\psi^{app}_{f,\pm} | \le \nu^N | \psi^{app}_{f,\pm} | ,
$$
where $N$ can be chosen arbitrarily large provided $M$ is sufficiently large.
Note that
\beq \label{pola2}
\partial_z \psi^{app}_{f,\pm}(\sigma_1) = \pm  \mu_f(\sigma_1) (1 + O(\nu^{1/4}) )
\psi^{app}_{f,\pm}(\sigma_1) .
\eeq
More generally,
\beq \label{pola2b}
\partial_z^k \psi^{app}_{f,-}(z) = (-1)^k \mu_f^k(z) \psi^{app}_{f,\pm}(z) (1 + O(\nu^{1/4}))
\eeq
for any $z \ge \sigma_1$ and any $k$, and similarly for $\psi^{app}_{f,+}$.


\subsection{Matching at $z = z_c$ \label{sec-match}}


It remains to match at $z = z_c$ the solutions constructed with the WKB method for $z \ge \sigma_1$ 
with the solutions construct thanks to Langer's transformation for $z \le \sigma_1$. 
We look for constants $a$ and $b$ such that
$$
a {\tilde Ai(z) \over \tilde Ai(\sigma_1)} + b {\tilde Ci(z) \over \tilde Ci(\sigma_1)}
$$
and $\psi^{app}_{f,-} /  \psi^{app}_{f,-}(\sigma_1)$ and their first derivatives match at $z = \sigma_1$, which leads to
$$
a + b  = 1
$$
$$
a {\partial_z \tilde Ai(\sigma_1) \over \tilde Ai(\sigma_1)} + b {\partial_z \tilde Ci (\sigma_1)\over \tilde Ci(\sigma_1)}
= {\partial_z \psi^{app}_{f,-} (\sigma_1) \over  \psi^{app}_{f,-}(\sigma_1)}
$$
We now use (\ref{asymplog}) and (\ref{pola2}) to get $a \sim 1$ and $b = O(\mu_f(\sigma_1)^{-1})$.
We then multiply $a$ and $b$ by $\psi^{app}_{f,-}(\sigma_1)$ to get an extension of $\psi^{app}_{f,-}$
from $z > \sigma_1$ to the whole line.
The construction is similar to extend $\psi^{app}_{f,+}$.


\subsection{From ${\cal A}$ to Airy \label{AA}}


We have now constructed global approximate solutions, that we again call $\psi_{f,\pm}^{app}$.
It remains to solve
\beq \label{doubleprimitive}
\partial_z^2   \phi_{f,\pm}^{app}(z) = \psi_{f,\pm}^{app}(z).
\eeq
Let us focus on the $-$ case, the other being similar.
For $z \ge \sigma_1$, we look for solutions $\phi_{f,\pm}^{app}$ of the form
$$
\phi_{f,\pm}^{app} = h(x)\psi_{f,\pm}^{app} =  h(x) e^{\theta(x) / \eps^{1/2}}
$$
which leads to
$$
h'' + 2 h' \theta'(x) \eps^{-1/2} + h \theta''(x) \eps^{-1/2} + h \theta'^2(x) \eps^{-1} = 1.
$$ 
Hence $h$ may be expanded as a series in $\eps^{1/2}$; namely, 
$$ h(x) = \sum_{i=0}^M \epsilon^{i/2} h_i(x) $$
for some arbitrarily large $M$. The leading term $h_0(x)$ is defined by $$
h_0(x) = {\eps \over \theta'^2(x) },
$$
while the other terms are computed similarly. We may thus write a complete WKB expansion for $\phi_{f,\pm}^{app}$.
In particular
$$
{\phi_{f,-}^{app}(y) \over \phi_{f,-}^{app}(x) } \le e^{- \int_x^y \Re \mu_f(z) dz} 
$$
provided $y > x \ge \sigma_1$. 

For $z < \sigma_1$, we integrate once (\ref{doubleprimitive}) which gives
$$
\partial_z\phi_{f,-}^{app}(z) = \partial_z\phi_{f,-}^{app}(\sigma_1) - \int_z^{\sigma_1} \psi_{f,-}^{app}(t) dt .
$$ 
Now $\psi_{f,-}^{app}$ is a combination of $\tilde Ai$ and $\tilde Ci$ for $z < \sigma_1$.
Let us focus on the $\tilde Ai$ term. We have to study
$$
\int_z^{\sigma_1} \tilde Ai(t) dt =  \int_z^{\sigma_1}
{1 \over \sqrt{g'(t)}} Ai(\gamma g(t)) dt.
$$
Let $s = \gamma g(t)$. Then $ds = \gamma g'(t) dt$, hence
$$
 \int_z^{\sigma_1}
{1 \over \sqrt{g'(t)}} Ai(\gamma g(t)) dt = \gamma^{-1}  \int_{\gamma g(z)}^{\gamma g(\sigma_1)} {1 \over g'(t)^{3/2}} Ai(s) ds.
$$
As $\gamma$ is large, the integral term is equivalent to
$$
{\gamma^{-1} \over g'(z)^{3/2}}   \int_{\gamma g(z)}^{\gamma g(\sigma_1)} Ai(s) ds
\sim {\gamma^{-1}\over g'(z)^{3/2}} \Bigl[ Ai(1,\gamma g(\sigma_1)) - Ai(1,\gamma g(z)) \Bigr]
$$
where we introduced the primitive $Ai(1,x)$ of $Ai$. This leads to
\beq \label{dzphi10}
\partial_z \phi_{f,-}^{app}(z) \sim {\gamma^{-1}\over g'(z)^{3/2}}  Ai(1,\gamma g(z)) .
\eeq
We integrate one again $\partial_z \phi_{f,-}^{app}$ and introduce $Ai(2,x)$, the second primitive of
$Ai$ and obtain
\beq \label{phi10}
\phi_{f,-}^{app}(z) \sim {\gamma^{-2}\over g'(z)^{5/2}}  Ai(2,\gamma g(z)) .
\eeq
The study of $\phi_{f,+}$ is similar.
As the asymptotic expansion of $Ai(x)$ is known, we can compute the asymptotic expansions of $Ai(1,x)$ and
$Ai(2,x)$.


\subsection{End of proof of Proposition \ref{prop-maphif}}


We now multiply $\phi_{f,\pm}^{app}$ by $\gamma^2$, such that, after this normalization, we have
(\ref{valuephi}) and (\ref{valuephid}).
In particular
\beq \label{normalfast}
\phi_{f,\pm}^{app}(0) = O(1) .
\eeq
Using (\ref{asymplog}) and (\ref{pola2b}) we get that
$$
{\partial_z \phi_{f,+}^{app}(z) \over \phi_{f,+}^{app}(z)} = \mu_f(z) (1 + O(\nu^{1/4}))
$$
as soon as $z \gg \nu^{1/4}$. As $\mu_f(z)$ is of order $O(\nu^{-1/4})$ for $z$ of order $\nu^{1/4}$, 
we obtain for any $0 \le z \le z'$,
\beq \label{boundsphif}
\Bigl|{\phi_{f,+}^{app} (z') \over \phi_{f,+}^{app}(z)} \Bigr| \le C \exp \Bigl| \int_z^{z'} \Re \mu_f(s) ds \Bigr| 
\eeq
for some constant $C$, and similarly for $\phi_{f,-}$, which gives (\ref{fastapp2}).

Moreover (\ref{fastapp1}) and (\ref{fastapp1c}) have already been proven.
As $\partial_z \phi_{f,+}^{app}(z)$ is bounded by $C \nu^{-1/4} \phi_{f,+}^{app}(z)$, (\ref{fastapp1}) combined with
(\ref{Airy-de-2}) gives (\ref{fastapp1b}), which ends the proof of Proposition \ref{prop-maphif}.


\subsection{Green function for Airy \label{GreenAiry}}


We will now construct an approximate Green function for the $\Airy$ operator.
We first construct an approximate Green function for ${\cal A}$. Let 
$$
G^{Ai}(x,y) = {1 \over \eps W^{Ai}(x)}  \left\{ \begin{aligned} 
{\psi^{app}_+(y) \over \psi^{app}_+(x)} \quad \hbox{if} \quad y < x,
\\
 {\psi^{app}_-(y) \over \psi^{app}_-(x)} \quad \hbox{if} \quad y > x,
 \end{aligned}\right.
$$
where $W^{Ai}$ is the Wronskian of $\psi^{app}_\pm(x)$. Note that this Wronskian is independent of $x$ and of order
$$
W^{Ai}(x) \sim \gamma = O(\nu^{-1/4}).
$$ 
In particular, we have
$$
G^{Ai}(x,y) = O(\nu^{-1/2})  \exp \Bigl( - C  \Bigl| \int_x^y \Re \mu_f(z) dz \Bigr| \Bigr),
$$
therefore $G^{Ai}$ is rapidly decreasing in $y$ on both sides of $x$, within scales of order $\nu^{1/4}$.
By construction, 
$$
{\cal A} G^{Ai}(x,y) = \delta_x + O(\nu^{3/4}) G^{Ai}(x,y) .
$$
We then integrate twice $G^{Ai}$ in $y$ to get an approximate Green function for the $\Airy$ operator.
More precisely, let
$$
G^{Ai,1}(x,y) = \int_{y}^{+\infty} G^{Ai}(x,z) dz
$$
and similarly for $G^{Airy} = G^{Ai,2}$, the primitive of $G^{Ai,1}$, so that $\pa_y^2 G^{Ai,2}(x,y) = G^{Ai}(x,y)$. We have
$$
G^{Ai,1}(x,y) = O(\nu^{-1/4})  \exp \Bigl( - C  \Bigl| \int_x^y \Re \mu_f(z) dz \Bigr| \Bigr) + O(\nu^{-1/4}) 1_{y < x}
$$
and similarly for $G^{Ai,2}$
 $$
G^{Ai,2}(x,y) = O(1)  \exp \Bigl( -  C \Bigl| \int_x^y \Re \mu_f(z) dz \Bigr| \Bigr) + O(\nu^{-1/4}) 1_{y < x}  x.
$$
Note that, taking into account the fast decay of $G^{Ai}$ near $x$,
\beq \label{GreenAiry2}
\begin{aligned}
\Airy (G^{Ai,2}) &= \delta_x + O(\nu^{3/4}) G^{Ai}(x,y) 
\\&= \delta_x + O(\nu^{1/4} ) \exp \Bigl( -  C \Bigl| \int_x^y \Re \mu_f(z) dz \Bigr| \Bigr) 
\\&= \delta_x + O(\nu^{1/4}).
\end{aligned}\eeq
We define the $AirySolve$ operator by
\beq \label{AirySolve}
AirySolve(f) (y) = \int_0^{+ \infty} G^{Ai,2}(x,y) f(x) dx
\eeq
and the associated error term
\beq \label{ErrorAiry}
\begin{aligned}
ErrorAiry(f) (y) 
& = \int_0^{+ \infty}  O(\nu^{3/4}) G^{Ai}(x,y)  f(x) dx
\end{aligned}
\eeq
the $\Airy$ operator acting on the $y$ variable.
These operators will be used in Section \ref{sec35}.



\section{Rayleigh solutions near critical layers}\label{sec-Rayleigh}


In this section, we construct two approximate solutions $\phi_{s,\pm}^{app}$ to the Orr Sommerfeld equation,
whose modules  respectively go to $+ \infty$ and $0$ as $z \to + \infty$. More precisely, we prove the following Lemma

\begin{lemma}\label{lem-exactphija} For $\nu$ small enough
there exist two independent functions $\phi_{s,\pm}^{app}$ such that
$$
W[\phi_{s,+}^{app},\phi_{s,-}^{app}](z) = 1 +  o(1),
$$
$$
\OS_{\alpha,c}(\phi_{s,\pm}^{app}) = O(\nu^{1/2}).
$$
Furthermore, we have the following expansions in $L^\infty$
$$
\begin{aligned}
\phi_{s,-}^{app} (z)&=  e^{-\alpha z} \Big (U-c + O(\nu^{1/4} )\Big).
\\
\phi_{s,+}^{app} (z)&=  \alpha^{-1} e^{\alpha z} O(1),
\end{aligned}$$
as $z\to \infty$. At $z = 0$, there hold
$$ 
\begin{aligned}
\phi_{s,-}^{app}(0) &=  - c + \alpha {U_+^2 \over U'(0)}   + O(\nu^{1/2})
\\
\phi_{s,+}^{app}(0) &=   - {1 \over U'(0)} +  O(\nu^{1/2}).
\end{aligned}$$
with 
 \end{lemma}
The construction of approximate solutions for Orr Sommerfeld equation starts with the construction of approximate solutions
for the Rayleigh operator.

For small $\alpha$, the construction of solutions to the Rayleigh equation is
a perturbation of the construction for $\alpha = 0$, which is explicit. We will now
detail the construction of an inverse of $Ray_0$ and then of an approximate inverse of $Ray_\alpha$ for small $\alpha$


\subsection{Function spaces}\label{sec-space}


In the next sections we will denote
$$
X^\eta = L_\eta^{\infty} = \Bigl\{ f \quad | \quad \sup_{z \ge 0} | f(z) | e^{\eta z} < + \infty \Bigr\} .
$$
The highest derivative of the Rayleigh equation vanishes at $z=z_c$, since $U(z_c) = c$. 
To handle functions which have large derivatives when $z$ is close to $\Re z_c$, we introduce the space
$Y^\eta$ defined as follows. Note that in our analysis, $z_c$ is never real, so $z - z_c$ never vanishes. 
We are close to a singularity but never reach it.

We say that a function $f$ lies in $Y^\eta$ if for any $z \ge 1$,
$$
|f(z)| + |\partial_z f(z) | + | \partial_z^2 f (z) |  \le C e^{-\eta z}
$$
and if for $z \le 1$,
$$
| \partial_z f(z) | \le C (1 + | \log (z - z_c) |  ) , 
$$
and
$$| \partial_z^2 f(z) | \le C (1 + | z - z_c |^{-1} ).
$$
The best constant $C$ in the previous bounds defines the norm $\| f \|_{Y^{\eta}}$.


\subsection{Rayleigh equation when $\alpha = 0$}


In this section, we  study the Rayleigh operator $\Ray_0$. More precisely, we solve  
\begin{equation}\label{Ray0} 
\Ray_0 (\phi) = (U-c) \partial_z^2 \phi - U'' \phi = f.
\end{equation}
The main observation is that
$$
\Ray_0(U - c) = 0.
$$
Therefore 
$$
\phi_{1,0} = U-c
$$ 
is a first explicit solution. The second one is obtained through the Wronskian equation
$$
W[\phi_{1,0},\phi_{2,0}] = 1 .
$$ 
This leads to  the following Lemma whose proof is given in \cite[Lemma 3.2]{GGN3}
\begin{lemma}[\cite{GGN2,GGN3}] \label{lem-defphi012} 
Assume that $\Im c \not =0$. There exist two independent solutions $\phi_{1,0} = U - c$ and $\phi_{2,0}$ of $\Ray_0(\phi) =0$ 
with unit Wronskian determinant 
$$ 
W(\phi_{1,0}, \phi_{2,0}) := \partial_z \phi_{2,0} \phi_{1,0} - \phi_{2,0} \partial_z \phi_{1,0} = 1.
$$
Furthermore, there exist smooth functions $P(z)$ and $Q(z)$  
with $P(z_c) \ne 0$ and $Q(z_c)\not=0$, so that, near $z = z_c$,
\begin{equation}\label{asy-phi012} 
\phi_{2,0}(z) = P(z) + Q(z) (z-z_c) \log (z-z_c) .
\end{equation}
Moreover
$$
\phi_{2,0}(0) = - {1 \over U'(0)} 
$$
and
\begin{equation}\label{decay-phi012} 
 \partial_z \phi_{2,0}(z)  + \frac{1}{U_+}   \in  Y^{\eta_1}
\end{equation}
for some $\eta_1 > 0$.
 \end{lemma}
Let $\phi_{1,0},\phi_{2,0}$ be constructed as in Lemma \ref{lem-defphi012}. 
Then the Green function $G_{R,0}(x,z)$ 
of the $\Ray_0$ operator can be explicitly defined by 
$$
G_{R,0}(x,z) = \left\{ \begin{array}{rrr} - (U(x)-c)^{-1} \phi_{1,0}(z) \phi_{2,0}(x), 
\quad \mbox{if}\quad z>x,\\
- (U(x)-c)^{-1} \phi_{1,0}(x) \phi_{2,0}(z), \quad \mbox{if}\quad z<x.\end{array}\right.
$$ 
The inverse of $\Ray_0$ is explicitly given by
\begin{equation}\label{def-RayS0}
\begin{aligned}
RaySolver_0(f) (z)  &: =  \int_0^{+\infty} G_{R,0}(x,z) f(x) dx.
\end{aligned}
\end{equation}
Note that the Green kernel $G_{R,0}$ is singular at $z_c$.
The following lemma asserts that the operator $RaySolver_0(\cdot)$ 
is in fact well-defined from $X^{\eta}$ to $Y^{0}$, 
which in particular shows that $RaySolver_0(\cdot)$ gains two derivatives, but losses the fast decay at infinity.  
It transforms a bounded function into a function which behaves like $(z-z_c) \log(z-z_c)$ near $z_c$.

\begin{lemma}\label{lem-RayS0} 
Assume that $\Im c \not =0$. For any $f\in {X^{\eta}}$,   $RaySolver_0(f)$ 
is a solution to the Rayleigh problem \eqref{Ray0}. In addition, $RaySolver_0(f) \in Y^{0}$, and there holds  
$$
\| RaySolver_0(f)\|_{Y^{0}} \le C (1+|\log \Im c|) \|f\|_{{X^{\eta}}},
$$ 
for some constant $C$. 
\end{lemma}
\begin{proof} 
Using \eqref{decay-phi012}, it is clear that $\phi_{1,0}(z)$ and $\phi_{2,0}(z)/(1+z)$
  are uniformly bounded. Thus, considering the cases $x < 1$ and $x > 1$, we obtain 
\begin{equation}\label{est-Gr0}
|G_{R,0}(x,z)| \le  C \max\{ (1+x), |x-z_c|^{-1} \}.
\end{equation}
That is, $G_{R,0}(x,z)$ grows linearly in $x$ for large $x$ and has a singularity of order $|x-z_c|^{-1}$ when $x$ is near $z_c$.
As $|f(z)|\le e^{-\eta z} \| f\|_{X^{\eta}}$, the integral \eqref{def-RayS0} is well-defined and we have 
$$
|RaySolver_0(f) (z)| \le  C\| f\|_{X^{\eta}} \int_0^\infty e^{-\eta x} \max\{ (1+x), |x-z_c|^{-1} \}  \; dx
$$
$$
\le C (1+|\log \Im c|) \| f\|_{X^{\eta}},
$$
in which we used the fact that $\Im z_c \approx \Im c$. 
  
To bound the derivatives, we need to check the order of the singularity for $z$ near $z_c$. 
We note that 
$$
|\partial_z \phi_{2,0}| \le C (1+|\log(z-z_c)|),
$$ 
and hence 
$$
|\partial_zG_{R,0}(x,z)| \le  C \max\{ (1+x), |x-z_c|^{-1} \} (1+|\log(z-z_c)|).
$$
Thus, $\partial_z RaySolver_0(f)(z)$ behaves as $1+|\log(z-z_c)|$ near the critical layer. 
In addition, from the $\Ray_0$ equation, we have 
\begin{equation}\label{identity-R0f} 
\partial_z^2 (RaySolver_0(f)) = \frac{U''}{U-c} RaySolver_0(f) + \frac{f}{U-c}.
\end{equation}
This proves that $RaySolver_0(f) \in Y^{0}$ and gives the desired bound. 
\end{proof}


\subsection{Approximate Green function when $\alpha \ll1$}


Let $\phi_{1,0}$ and $\phi_{2,0}$ be the two solutions of $\Ray_0(\phi) = 0$ that are constructed above, in Lemma \ref{lem-defphi012}. 
We now construct an approximate Green function to the Rayleigh equation for $\alpha > 0$.
To proceed, let us introduce
\begin{equation}\label{def-phia12}
\phi_{1,\alpha } = \phi_{1,0} e^{-\alpha z} ,\qquad \phi_{2,\alpha} = \phi_{2,0} e^{-\alpha z}.
\end{equation}
A direct computation shows that their Wronskian determinant equals
$$
W[\phi_{1,\alpha},\phi_{2,\alpha}] =  \partial_z \phi_{2,\alpha} \phi_{1,\alpha} - \phi_{2,\alpha} \partial_z \phi_{1,\alpha}  = e^{-2\alpha z}.
$$ 
Note that the Wronskian vanishes at infinity since both functions have the same behavior at infinity. 
In addition, 
\begin{equation}\label{Ray-phia12}
\Ray_\alpha(\phi_{j,\alpha}) = - 2 \alpha (U-c) \partial_z \phi_{j,0} e^{-\alpha z} 
\end{equation}
We are then led to introduce an approximate Green function $G_{R,\alpha}(x,z)$, defined by 
$$
G_{R,\alpha}(x,z) = \left\{ \begin{array}{rrr} (U(x)-c)^{-1} e^{-\alpha (z-x)}  \phi_{1,0}(z) \phi_{2,0}(x), \quad \mbox{if}\quad z>x\\
(U(x)-c)^{-1} e^{-\alpha (z-x)}  \phi_{1,0}(x) \phi_{2,0}(z), \quad \mbox{if}\quad z< x.\end{array}\right.
$$
Again, like $G_{R,0}(x,z)$, the Green function $G_{R,\alpha}(x,z)$ is ``singular'' near $z_c$.
By a view of \eqref{Ray-phia12}, 
\begin{equation}\label{id-Gxz}
\Ray_\alpha (G_{R,\alpha}(x,z)) = \delta_{x}  + E_{R,\alpha}(x,z),
\end{equation}
for each fixed $x$, where the error kernel $E_{R,\alpha}(x,z)$ is defined by  
$$
E_{R,\alpha}(x,z) = 
 \left\{ \begin{array}{rrr}
- 2 \alpha (U(z) - c) (U(x)-c)^{-1} e^{-\alpha (z-x)} \partial_z \phi_{1,0}(z) \phi_{2,0}(x), \quad \mbox{if}\quad z>x\\
 - 2 \alpha (U(z) - c) (U(x)-c)^{-1}e^{-\alpha (z-x)}\  \phi_{1,0}(x) \partial_z \phi_{2,0}(z), \quad \mbox{if}\quad z< x.\end{array}\right.
$$
We then introduce an approximate inverse of the operator $\Ray_\alpha$ defined by
\begin{equation}\label{def-RaySa}
RaySolver_\alpha(f)(z) 
:= \int_0^{+\infty} G_{R,\alpha}(x,z) f(x) dx
\end{equation}
and the related error operator 
\begin{equation}\label{def-ErrR}
Err_{R,\alpha}(f)(z) := 2\alpha (U(z) - c) \int_0^{+\infty} E_{R,\alpha}(x,z) f(x) dx
\end{equation}

\begin{lemma}\label{lem-RaySa} 
Assume that $\Im c > 0$. For any $f\in {X^{\eta}}$,  with $\alpha<\eta$, 
the function $RaySolver_\alpha(f)$ is well-defined in $Y^{\alpha}$, and satisfies 
$$ 
\Ray_\alpha(RaySolver_\alpha(f)) = f + Err_{R,\alpha}(f).
$$
Furthermore, there hold  
\begin{equation}\label{est-RaySa}
\| RaySolver_\alpha(f)\|_{Y^{\alpha}} \le C (1+|\log \Im c|) \|f\|_{{X^{\eta}}},
\end{equation}
and 
\begin{equation}\label{est-ErrRa} 
\|Err_{R,\alpha}(f)\|_{Y^{\eta}} \le C | \alpha |   (1+|\log (\Im c)|)  \|f\|_{X^{\eta}} ,
\end{equation} 
for some universal constant $C$. 
\end{lemma}
\begin{proof}  
The proof follows that of Lemma \ref{lem-RayS0}.
 Indeed, since 
 $$
 G_{R,\alpha}(x,z)  = e^{-\alpha (z-x)} G_{R,0}(x,z),
 $$ 
 the behavior near the critical layer $z=z_c$ is the same for these two Green functions, 
 and hence the proof of \eqref{est-RaySa} and \eqref{est-ErrRa} near the critical layer identically follows  
 from that of Lemma \ref{lem-RayS0}.  

Let us check  the behavior at infinity. 
Consider the case $p=0$ and assume $\|f \|_{X^{\eta}} =1$. 
Using \eqref{est-Gr0}, we get
$$
|G_{R,\alpha}(x,z)| \le  C e^{-\alpha (z - x)} \max\{ (1+x), |x-z_c|^{-1} \}.
$$
Hence, by definition, 
$$
 |RaySolver_\alpha (f)(z) |\le C e^{-\alpha z} \int_0^\infty e^{\alpha x} e^{-\eta x}\max\{ (1+x), |x-z_c|^{-1} \}\; dx 
 $$ 
which is  bounded by $C(1+|\log \Im c|) e^{-\alpha z}$, upon recalling that $\alpha<\eta$.  
This proves the right exponential decay of $RaySolver_\alpha (f)(z)$ at infinity, for all $f \in X^{\eta}$. 

The estimates on $Err_{R,\alpha}$ are the same, once we notice that 
$(U(z) - c) \partial_z \phi_{2,0}$ has the same bound as that for $\phi_{2,0}$, and similarly for $\phi_{1,0}$. 
\end{proof}

\begin{remark}\label{rem-Ray} For $f(z) = (U-c) g(z)$ with $g\in {X^{\eta}}$, the same proof as done for Lemma  \ref{lem-RaySa} yields 
\begin{equation}\label{rem-RaySa}
\begin{aligned}
\| RaySolver_\alpha(f)\|_{Y^{\alpha}} &\le C \|g\|_{{X^{\eta}}},
\\
\|Err_{R,\alpha}(f)\|_{Y^{\eta}} &\le C | \alpha |   \|g\|_{X^{\eta}} 
\end{aligned}\end{equation} 
which are slightly better estimates as compared to \eqref{est-RaySa} and \eqref{est-ErrRa}. 
\end{remark}

\subsection{Construction of $\phi^{app}_{s,-}$ }\label{sec-exactRayleigh}


Let us start with the decaying solution $\phi_{s,-}$. We note that 
$$ 
\psi_{0} =  e^{-\alpha z} (U-c)
$$
is only a $O(\alpha)$ smooth approximate solution to Rayleigh equation since
$$
e_{0} =  Ray_\alpha (\psi_0) = - 2\alpha (U-c) U' e^{-\alpha z}.
$$
Similarly, a direct computation shows that 
$$
\OS_{\alpha,c}(\psi_0) = O(\alpha) = O(\nu^{1/4}).
$$ 
This is not sufficient for our purposes, and we have to go to the next
order. We therefore introduce
$$
\psi_1 = - RaySolver_\alpha(e_0).
$$
Note that $\psi_1$ is of order $O(\alpha)$ in $Y^\eta$, and behaves like $\alpha (z - z_c) \log (z - z_c)$ near $z_c$.
It particular $\psi_1$ is not a smooth function near $z_c$. Its fourth order derivative behaves like
$\alpha / (z - z_c)^3$ in the critical layer. We have 
$$
\OS_{\alpha,c}(\psi_1) = \eps (\partial_z^2 - \alpha^2)^2 \psi_1 + Ray_\alpha(\psi_1) .
$$
hence 
\beq \label{errr1}
\OS_{\alpha,c}(\psi_0 + \psi_1) = \eps (\partial_z^2 - \alpha^2)^2 \psi_1 + Err_{R,\alpha}(e_0) .
\eeq
Note that
\beq \label{Raypsi1}
 Err_{R,\alpha}(e_0) =  O ( \alpha^2 | \log(\alpha) |)_{Y^\eta} .
\eeq
Moreover, using Rayleigh equation,
$$
(\partial_z^2 - \alpha^2) \psi_1 = {Ray_\alpha(\psi_1) - U'' \psi_1 \over U -c },
$$
hence
\beq \label{errr}
E := \eps (\partial_z^2 - \alpha^2)^2 \psi_1 
= \eps (\partial_z^2 - \alpha^2)   \Bigl\{ {Ray(\psi_1) - U'' \psi_1 \over U -c } \Bigr\} .
\eeq
In view of Remark \ref{rem-Ray}, $Ray_\alpha(\psi_1)$ and $U'' \psi_1$ are of order $O(\alpha)$ in $X^\eta$. We thus have
$$
\eps \alpha^2 \Bigl| {Ray(\psi_1) - U'' \psi_1 \over U -c } \Bigr|
\le C {\eps \alpha^2 \over | z - z_c | } 
\le C {\eps \alpha^2 \over | \Im c | } \le C \eps \alpha = O(\nu)_{X^\eta}.
$$
Next we expand $\partial_z^2$ in (\ref{errr}) which gives three terms. The first one is
$$
\eps {\partial_z^2 Ray(\psi_1) - \partial_z^2(U'' \psi_1) \over U - c} .
$$
As $Ray_\alpha(\psi_1)$ and $\psi_1$ are of order $O(\alpha)$ in $Y^\eta$, this quantity is bounded by
\beq \label{bound1}
C \eps \Bigl( 1 + {\alpha | \log \Im c | \over |z - z_c|}+ {\alpha  \over |z - z_c|^2} \Bigr) \le C {\eps \alpha  \over | \Im c|^2} = O(\alpha^2 ). 
\eeq
The third term in the expansion of (\ref{errr}) is
$$
\eps \Bigl[  Ray(\psi_1) - U'' \psi_1 \Bigr] (z - z_c)^{-3}
$$
which is bounded by $O(\alpha)$. 
Thus, we can write the error term as 
$$ E = E_1 + E_2, \qquad E_1 = O(\alpha^2), \qquad E_2\le C \eps \alpha | z - z_c |^{-3} .
$$
This error term $E_2$ is therefore too large for our purposes. However, it is located near $z = z_c$, namely in the critical
layer.  We therefore correct $\psi_0 + \psi_1$ by $\psi_2$ by approximately inverting the Airy operator in this layer. 
More precisely, let
$$
\psi_2 = - AirySolve(E_2) ,
$$
which will create an error term
$$
\begin{aligned}
E_3 &=  \OS_{\alpha,c} (\psi_2) + E_2 
\\& = \Airy (\psi_2) + \hbox{OrrAiry} (\psi_2) + E_2
\\&=  \hbox{OrrAiry} (\psi_2) + ErrorAiry(E_2).
\end{aligned}$$
Let us now bound $\psi_2$. Using \eqref{AirySolve}, we have
$$
| \psi_2 (y) |\le C \eps \alpha \int_0^{+ \infty} | x - z_c |^{-3}
 \Bigl(e^{ - | \int_x^y \Re \mu_f(z)dz |}   + O(\nu^{-1/4}) 1_{y < x} x \Bigr) dx .
$$
Writing $1_{y < x} x = 1_{y < x} (x - z_c) + 1_{y < x} z_c$, we thus have 
$$
\begin{aligned}
| \psi_2 (y) | &\le C \eps \alpha \int_0^{+ \infty} \Bigl ( | x - z_c |^{-3} + \nu^{-1/4} |x-z_c|^{-2}\Bigr )
 dx 
 \\
 &\le C \eps \alpha \Bigl( |\Im c|^{-2}+ \nu^{-1/4} |\Im c|^{-1}\Bigr )  = O( \alpha^2)
 .\end{aligned}
$$
This together with \eqref{Airy-de-2} yields $ \hbox{OrrAiry} (\psi_2) = O( \alpha^2)$.  Similarly, using (\ref{GreenAiry2}), we get 
$$
ErrorAiry(E_2)(z) \le C \eps \alpha\int_0^{+ \infty} | x - z_c |^{-3} O(\nu^{1/4}) dx = O(\alpha^3).
$$
Therefore, we have
$$
\OS_{\alpha,c}(\psi_0 + \psi_1 + \psi_2) = O(\alpha^2) .
$$
We define
$$
\phi_{s,-}^{app} = \psi_0 + \psi_1 + \psi_2 .
$$
To end this section we compute $\psi(0)$. By definition,
$$
\begin{aligned}
\psi_1(0) &= - RaySolver_\alpha(e_0) (0)
 = - \phi_{2,\alpha}(0) \int_0^{+\infty} e^{2\alpha x}\phi_{1,\alpha}(x) {e_0(x) \over U(x) - c} dx 
\\&=  -  2 \alpha \phi_{2,0}(0) \int_0^{+\infty}  U' (U-c)dz =
 \alpha  \phi_{2,0}(0) \Bigl[ (U - c)^2 \Bigr]_0^{+ \infty}
\\&= - \alpha \phi_{2,0}(0) \Bigl[ (U_+ - c)^2 - c^2 \Bigr]  
=   \alpha {U_+  \over U'(0)} (U_+ - 2 c) .
\end{aligned}$$
From the definition, we have 
$$
\phi_{s,-}(0) = U_0 -c + \psi_1(0) + O(\alpha^2).
$$
This proves the lemma,  using that $U_0 - c = O(z_c)$.


\subsection{Construction of $\phi^{app}_{s,+}$ \label{sec35}}


We first start with $\phi_{2,\alpha} = \phi_{2,0} e^{- \alpha z}$, which is an approximate solution of Rayleigh equation,
up to a $O(\alpha)$ error term.
Let 
$$
e_1 = Ray_\alpha (\phi_{2,\alpha}) = O(\alpha).
$$
We introduce
$$
\phi_3 = - RaySolver_\alpha(e_1) .
$$
Then 
$$
Ray_\alpha( \phi_{2,\alpha} + \phi_3) = -Err_{R,\alpha}(e_1) = O(\alpha^2).
$$
Let
$$
\phi_{s,+} = \phi_{2,\alpha} + \phi_3.
$$
Note that $\phi_{s,+}$ is bounded in $Y^\eta$, and thus behaves like $(z - z_c) \log (z - z_c)$ near $z_c$


We have
$$
\OS_{\alpha,c}(\phi_{s,+}) = - \eps (\partial_z^2 - \alpha^2)^2 \phi_{s,+},
$$
but, using Rayleigh equation,
$$
(\partial_z^2- \alpha^2) \phi_{s,+} = {U'' \over U - c} \phi_{s,+},
$$
hence
$$
(\partial_z^2- \alpha^2)^2 \phi_{s,+} = (\partial_z^2- \alpha^2) \Bigl( {U'' \over U - c} \phi_{s,+} \Bigr).
$$
The worst term in the right hand side is
$$
\Bigl[ \partial_z^2 \Bigl( { U'' \over U - c} \Bigr) \Bigr] \phi_{s,+} 
$$
which is of order $(\Im z_c)^{-3}$ for $\phi_{s,-}$. Hence $\OS_{\alpha,c}(\phi_{s,+})$ is of order
$$
{\eps \over (\Im z_c)^3}
\sim {\nu \over \alpha} {1 \over \nu^{3/4}} \sim 1 
$$
near $z = z_c$, which is $\alpha^{-1}$ larger than for $\phi_{s,-}^{app}$. 
As a consequence we loose a factor $\alpha^{-1}$ in the end of the construction with respect to $\phi_{s,-}^{app}$.
The construction is similar, up to a factor $\alpha^{-1}$.


\section{Green function for Orr-Sommerfeld equations}


We now construct an approximate Green function $G^{app}$ using the approximate solutions
$\phi_{s,\pm}^{app}$ and $\phi_{f,\pm}^{app}$. We will decompose this Green function into two components
$$
G^{app} = G_i^{app} + G_b^{app}
$$
where $G_i^{app}$ takes care of the source term $\delta_x$ and where $G_b^{app}$ takes care of the boundary conditions.


\subsection{Interior approximate Green function}


We look for $G_i^{app}(x,y)$ of the form
$$
G_i^{app}(x,y) = a_+(x)  \phi_{s,+}^{app}(y) 
+ {b_+(x) }  {\phi_{f,+}^{app}(y) \over \phi_{f,+}^{app}(x)} 
\quad \hbox{for} \quad y < x,
$$
$$
G_i^{app}(x,y) = a_-(x)  \phi_{s,-}^{app}(y)
+ {b_-(x)} {\phi_{f,-}^{app}(y) \over \phi_{f,-}^{app}(x)} 
\quad \hbox{for} \quad y  > x,
$$
where $\phi_{f,\pm}^{app}(x)$ play the role of normalization constants. Let
$$
F_\pm = \phi^{app}_{f,\pm}(x) 
$$
and let 
$$
v(x) = (- a_-(x), a_+(x), - b_-(x), b_+(x) ) .
$$
By definition of a Green function, $G^{app}$, $\partial_y G^{app}$ and $\partial_y^2 G^{app}$ are continuous at $x = y$,
whereas $- \eps \partial_y^3 G^{app}$ has a unit jump at $x = y$. 
Let
\beq \label{matriceM}
M = \left( \begin{array}{cccc} 
\phi_{s,-} & \phi_{s,+} & \phi_{f,-} / F_-  & \phi_{f,+} /F_+  \cr
\partial_y \phi_{s,-}  / \mu_f & \partial_y\phi_{s,+} /   \mu_f
& \partial_y\phi_{f,-} /  F_- \mu_f & \partial_y\phi_{f,+} / F_+ \mu_f  \cr
\partial_y^2 \phi_{s,-} / \mu_f^2 &\partial_y^2 \phi_{s,+} /   \mu_f^2 
& \partial_y^2 \phi_{f,-} /  F_- \mu_f^2  & \partial_y^2 \phi_{f,+} /  F_+ \mu_f^2 \cr
 \partial_y^3 \phi_{s,-} /  \mu_f^3 &  \partial_y^3 \phi_{s,+} /   \mu_f^3
& \partial_y^3 \phi_{f,-}  /  F_- \mu_f^3 &  \partial_y^3 \phi_{f,+} /  F_+ \mu_f^3 \cr 
\end{array} \right) ,
\eeq
where the functions $\phi_{s,\pm} = \phi_{s,\pm}^{app}$ and $\phi_{f,\pm} = \phi_{f,\pm}^{app}$ 
and their derivatives are evaluated at $y=x$, and where the various factors $\mu_f$ are introduced to renormalize the lines of $M$.
Then 
\beq \label{Mv}
M v = (0,0,0,- 1/ \eps \mu_f^3) .
\eeq
We will  evaluate $M^{-1}$ using the following block structure. 
Let $A$, $B$, $C$ and $D$ be the two by two matrices defined by
$$
M = \left( \begin{array}{cc} 
A & B \cr
C & D \cr \end{array} \right) .
$$
We will prove that $C$ is small, that $D$ is invertible and that $A$ is related to Rayleigh equations.
This will allow the construction of an explicit approximate inverse, and by iteration, of the inverse of $M$. 
Let us detail these points.

Let us first study $D$. Following  (\ref{fastapp1c}), for $z \gg \nu^{1/4}$,
$$
D = \left( \begin{array}{cc}
1 & 1 \cr
-1 & 1 \cr 
\end{array} \right) + o(1),
$$
hence $D$ is invertible and
$$
D^{-1} =  \left( \begin{array}{cc}
1 & -1 \cr
1 & 1 \cr 
\end{array} \right) + o(1).
$$
For $z$ of order $\nu^{1/4}$, we note that $F_+$ and $F_-$ are of order $O(1)$,
$$
\partial_y^2 \phi_{f,-} = \gamma^2 {1 \over (g')^2(x)} {Ai ( \gamma g(x) ) \over Ai(2, \gamma g(x)) } + O(\gamma)
$$
and similarly for $\partial_y \phi_{f,-}$ and $\partial_y \phi_{f,+}$. Note that $\gamma^2 / \mu_f^2$, $\gamma^3 / \mu_f^3$,
$Ai(2,\gamma g(x))$ and $Ci(2,\gamma g(x))$ are of order $O(1)$. As $g'(z_c) = 1$, up to normalization of
lines and columns, $D$ is close to
$$
\left( \begin{array}{cc}
Ai & Ci \cr
Ai' & Ci' \cr 
\end{array} \right) 
$$
which is invertible by definition of the special Airy functions $Ai$ and $Ci$.

Let us turn to $C$. The worst term in $C$ is those involving $\phi_{s,+}$ because of its logarithmic singularity.
More precisely, $\partial_y^k \phi_{s,+}$ behaves like $(z-z_c)^{k-1}$ and
is bounded by $| \Im c |^{k-1} \sim \nu^{(1 - k) / 4}$ for $k = 2$, $3$.
Hence, as $\mu_f^{-1} = O(\nu^{1/4})$, 
$$
C =  \left( \begin{array}{cc}
O(\nu^{1/2 }) & O(\nu^{1/2 }  (z-z_c)^{-1})  \cr
O(\nu^{3/4 }) & O(\nu^{3/4 } (z - z_c)^{-2}) \cr
\end{array} \right) 
$$
Note that $A = A_1 A_2$ with
$$
A_1 = \left( \begin{array}{cc}  
1 & 0 \cr
0 & \mu_f^{-1} \cr
\end{array} \right),
\quad 
A_2 = \left( \begin{array}{cc}  
\phi_{s,-}^{app} & \phi_{s,+}^{app} \cr
\partial_y \phi_{s,-}^{app}  & \partial_y \phi_{s,+}^{app}  \cr
\end{array} \right) .
$$
We have
$$
A_2^{-1} = {1 \over \det(A_2)} \left( \begin{array}{cc}  
\partial_y \phi_{s,+}^{app}  & - \phi_{s,+}^{app} \cr
- \partial_y \phi_{s,-}^{app}  & \phi_{s,-}^{app}  \cr
\end{array} \right) .
$$
The determinant  $A_2$ is the Wronskian of $\phi_{s,\pm}^{app}$ and hence a perturbation
of the Wronskian of $\phi_{1,\alpha}$ and $\phi_{2,\alpha}$ which equals to $e^{- \alpha x}$.
We distinguish between $x < \alpha^{1/2}$ and $x > \alpha^{1/2}$. In the second case, $\OS_{c,\alpha}$
is a small perturbation of a constant coefficient fourth order operator. The Green function may therefore be explicitly computed.
We will not detail the computations here and focus on the case where $x < \alpha^{1/2}$. In this case the Wronskian is of order $O(1)$.
As a consequence 
$$
A_2^{-1} = \left( \begin{array}{cc}  
O(\log | z - z_c | )  & O(1) \cr
O(1)  & O(z - z_c)  \cr
\end{array} \right) 
$$
and
$$
A^{-1} = \left( \begin{array}{cc}  
O(\log | z - z_c |)  & O(\mu_f) \cr
O(1)  & O(\mu_f (z - z_c))  \cr
\end{array} \right) 
$$
We now observe that the matrix $M$ has an approximate inverse
$$
\widetilde M = \left( \begin{array}{cc}
A^{-1} & - A^{-1} B D^{-1}  \cr
0 & D^{-1} \cr 
\end{array} \right) 
$$
in the sense that $M  \widetilde M = Id + N$ where
$$
N =  \left( \begin{array}{cc}
0 & 0 \cr 
 C A^{-1} &  - C A^{-1} B D^{-1} \cr 
\end{array} \right) .
$$
Now a direct calculation shows that 
$$
C A^{-1} = O(\nu^{1/4})
$$ 
since $\Im z_c = O(\nu^{1/4})$. 
As $D^{-1}$ and $B$ are uniformly bounded,
$N = O(\nu^{1/4})$. In particular, $(Id + N)^{-1}$
is well defined and 
$$
M^{-1} = \widetilde M (Id + N)^{-1} = \widetilde M \sum_n N^n.
$$
Note that the two first lines of $N^n$ vanish. The other lines are at most of order $O(\nu^{1/4})$.  Therefore
$$
(Id + N)^{-1} (0,0,0, 1 / \nu \mu_f^3) = \Bigl( 0, 0, O(1 / \nu \mu_f^4), 1 / \nu \mu_f^3 \Bigr) .
$$
As $D^{-1}$ is bounded and $A^{-1} B D^{-1}$ is of order $O(\mu_f)$, we obtain that 
$a_\pm$ and $b_\pm$ are respectively bounded by $C / \nu \mu_f^2$ and $C / \nu \mu_f^3$.


\subsection{Boundary approximate Green function}


We now add to $G_i^{app}$ another Green function $G_b^{app}$ to handle the boundary conditions.
We look for $G_b^{app}$ under the form
$$
G_b^{app}(y) = d_s \phi_{s,-}(y)  + d_f {\phi_{f,-}(y)  \over \phi_{f,-}(0)},
$$
where $\phi_{f,-}(0)$ in the denominator is a normalization constant, 
and look for $d_s$ and $d_f$ such that
\beq \label{Greenb1}
G_i^{app}(x,0) + G_b^{app}(0) = 
\partial_y G_i^{app}(x,0) + \partial_y G_b^{app}(0) =  0.
\eeq
Let
$$
M = \left( \begin{array}{cc} \phi_{s,-} & \phi_{f,-} / \phi_{f,-}(0) \cr
\partial_y \phi_{s,-} & \partial_y \phi_{f,-} / \phi_{f,-}(0) \cr \end{array} \right) ,
$$
the functions being evaluated at $y = 0$. Then (\ref{Greenb1}) can be rewritten as
$$
M d = - (G_i^{app}(x,0), \partial_y G_i^{app}(x,0)) 
$$
where $d = (d_s,d_f)$. Note that
$$
(G_i^{app}(x,0), \partial_y G_i^{app}(x,0))  = Q (a_+,b_+)
$$
where
$$
Q =  \left( \begin{array}{cc}
 \phi_{s,+}(0)  & 1 \cr
 \partial_y \phi_{s,+}(0)   & \partial_y \phi_{f,+}(0) / \phi_{f,+}(0) \cr
\end{array} \right) 
= \left( \begin{array}{cc}
 O(1)  & 1 \cr
 O( \log( \nu))   & O(\nu^{-1/4}) \cr
\end{array} \right) .
$$
By construction
\beq \label{defid}
d = -   M^{-1} Q (a_+,b_+) .
\eeq
We have
$$
M^{-1} = {1 \over \det(M)}  \left( \begin{array}{cc} \partial_y \phi_{f,-}(0) / \phi_{f,-}(0) & - 1 \cr
- \partial_y \phi_{s,-}(0)& \phi_{s,-}(0)  \end{array} \right) .
$$
The determinant of $M$ equals
$$
\det M = {W[\phi_{s,-},\phi_{f,-}](0) \over \phi_{f,-}(0)} 
$$
and does not vanish by assumption. Therefore
$$
M^{-1} = \left( \begin{array}{cc} O(\nu^{-1/4}) & - 1 \cr
O(1) & O(\nu^{1/4}) \end{array} \right) .
$$
As a consequence,
$$
M^{-1} Q =
\left( \begin{array}{cc} O(\nu^{-1/4}) & O(\nu^{-1/4}) \cr
O(1) & O(1) \end{array} \right).
$$


\subsection{Exact Green function}


Once we have an approximate Green function, we obtain the exact Green function by iteration, following the strategy developed in the previous part.

%
%
%

\chapter{Semi group of Navier-Stokes: unstable profiles}\label{chapter-green}





In this chapter, we study the linearized Navier-Stokes problem around a stationary boundary layer profile $U_s = (U(z),0)^{tr}$. More precisely,  
define the linear operator 
\beq \label{linearizedNS}
L  \omega = - U \partial_x \omega - U'' v_2 + \nu \Delta \omega,
\eeq
where $\omega = \partial_z v_1 - \partial_x v_2$ and $\nabla \cdot v = 0$, together with the boundary condition $v = 0$ on $z = 0$. Taking advantage of the divergence-free condition, we can recover velocity from vorticity through the Biot-Savart law: $v = \nabla^\perp \phi$ with $\Delta \phi =\omega$. The no-slip boundary conditions become 
\begin{equation}\label{BC-phi} \partial_x \phi_{\vert_{z=0}} = \partial_z \phi_{\vert_{z=0}} = 0. \end{equation} 
We will consider the periodic setting in the horizontal variable $x$. The linearized Navier Stokes equation near $U_s$ then reads
$$
\partial_t \omega - L \omega = 0
$$
with initial data $\omega_{\vert_{t=0}} = \omega^0(z)$. We are interested in deriving sharp bounds on the semigroup $e^{Lt}$ that are uniform in the inviscid limit. 

As the boundary layer profile $U = U(z)$ is independent of $x$, it is then convenient to introduce the Fourier transform in the $x$ variable.
Solutions to the linearized problem will be constructed in term of Fourier series
\beq \label{Fourier-w}
\omega(t,x,z) = \sum_{\alpha\in \ZZ} e^{i\alpha x} \hat \omega_\alpha(t,z)
\eeq
where the Fourier coefficients $\hat \omega_\alpha(t,z)$ solve the linear problem 
\begin{equation}\label{eqs-vorticitya} 
\partial_t \hat \omega_\alpha - L_\alpha \hat \omega_\alpha = 0, 
\end{equation}
with 
\begin{equation}\label{def-aLhat}
L_\alpha \hat \omega_\alpha 
: = - i\alpha U\hat \omega_\alpha + i\alpha \hat \phi_\alpha U'' +  \nu  \Delta_\alpha \hat \omega_\alpha, \qquad \Delta_\alpha \hat \phi_\alpha = \hat \omega_\alpha,
\end{equation}
recalling that $
\Delta_\alpha = \partial_z^2 - \alpha^2,
$ together with the boundary conditions
\begin{equation}\label{BC-ahat}
\alpha \hat \phi_\alpha = \partial_z\hat  \phi_\alpha =0, \qquad \mbox{on}\quad z=0.
\end{equation}
Observe that at $\alpha=0$, the linear problem \eqref{eqs-vorticitya}-\eqref{BC-ahat} becomes 
$$ 
\partial_t \hat \omega_0 -  \nu \partial_z^2 \hat \omega_0 = 0
$$
whose semigroup can be explicitly solved. In particular, $\hat \omega_0(t,z) =0$ for all positive times, if it is initially zero. In what follows, we shall thus focus on the case when $\alpha >0$; the case $\alpha <0$ is similar.

\section{Boundary layer spaces and main result}

In the inviscid case $\nu =0$, the boundary condition \eqref{BC-ahat} is reduced to $\alpha \hat \phi_\alpha =0$ on $z=0$. As a consequence of this discrepancy in boundary conditions, boundary layers arise in the inviscid limit. We thus need to introduce function spaces that capture boundary layer behavior of vorticity near the boundary. Precisely, we denote by ${\cal B}^{\beta,\gamma}$ the boundary layer function spaces with finite norm 
$$
\| f \|_{\beta,\gamma}  = \sup_{z \ge 0} | f(z) | e^{\beta  z} 
\Bigl( 1 + \delta^{-1} \phi_{P} (\delta^{-1} z)  \Bigr)^{-1}
$$
in which the boundary layer thickness $$\delta = \gamma \sqrt \nu$$ and the boundary layer weight function $$ \phi_P(z) = \frac{1}{1+z^P} $$
for some fixed constant $P>1$. We expect that the vorticity function $\hat \omega_\alpha(t,z)$, for each fixed $t,\alpha$, will be in ${\cal B}^{\beta,\gamma}$, precisely describing the behavior near the boundary and near infinity. However, the derivative of vorticity will not be in the function space: $\partial_z \omega \notin {\cal B}^{\beta,\gamma}$. Therefore, for $p\ge 1$, we are led to introduce the following one-dimensional boundary layer function spaces ${\cal B}^{\beta,\gamma,p}$, together with their finite corresponding norms
\begin{equation}\label{def-blnorm}
\| f \|_{\beta,\gamma, p}  = \sup_{z\ge 0} | f(z) | e^{\beta z} 
\Bigl( 1 + \sum_{q=1}^p \delta^{-q} \phi_{P-1+q} (\delta^{-1} z)  \Bigr)^{-1} .
\end{equation}
Note that
${\cal B}^{\beta,\gamma,1} = {\cal B}^{\beta,\gamma}$. We expect that $\partial_z \omega \in \cB^{\beta, \gamma,2}$, and more generally, 
\begin{equation}\label{expect-wbl} \partial_z^k \omega \in \cB^{\beta, \gamma,1+k}, \qquad k\ge 0.\end{equation}
By convention, $\cB^{\beta,\gamma,0} = \cA^{\beta}$, the function space with no boundary layer behavior, which is equipped with the norm 
$$\| f \|_{\beta}  = \sup_{z\ge 0} | f(z) | e^{\beta z} .$$

Our main result in this chapter is the following. 

\begin{theorem} \label{theo-main} 
Let $U(z)$ be a $C^\infty$ smooth boundary layer profile such that 
$U(0)=0$ and 
\begin{equation}\label{def-Ubl}
  |\partial_z^k (U(z) - U_+)| \le C_k e^{-\eta_0 z} , \qquad \forall~ z\ge 0, \quad k\ge 0,
  \end{equation}
for some constants $C_k, U_+, \eta_0$.
 Let $\lambda_0$ be the maximal unstable eigenvalue of the linearized Euler equations around $U$, 
namely the eigenvalue which has the largest real part $\Re\lambda_0$. 
 In case there is no unstable eigenvalue we set $\lambda_0=0$. 
 Let $\tau > 0$. Then there is a constant $C_{\tau}$,  so that  for any for $\alpha \le \nu^\zeta$, $\zeta<1/2$, 
 $\alpha \ne 0$ and
 $\nu \le 1$,  
\begin{equation}\label{main-bound}
\| e^{L_\alpha t} \omega_\alpha(0,.) \|_{\beta,\gamma}
 \le C_\tau e^{(\Re \lambda_0 + \tau) t }  \|\omega_\alpha(0,.) \|_{\beta,\gamma} ,  
\end{equation}
for any initial vorticity $\omega_\alpha(0,.)$  and for all $t\ge 0$, provided $\beta$ and $\gamma$ are small enough.  
\end{theorem}


 Theorem \ref{theo-main} provides a semigroup estimate for the linearized Navier-Stokes problem near an unstable 
 boundary layer profile, which is uniform in the vanishing viscosity limit. 
 Such an estimate is sharp without the knowledge of the multiplicity of the maximal unstable eigenvalue $\lambda_0$. 
The difficulty lies in the fact that the initial data $\omega$  
has a boundary layer type behavior and 
hence in order to propagate this boundary layer behavior, pointwise bounds on the Green 
function of linearized Navier-Stokes equations near a boundary layer are needed.

\section{The resolvent}


In order to study the semigroup $e^{L_\alpha t}$, it is convenient to 
take the Laplace transform of \eqref{eqs-vorticitya}. This leads to  the resolvent equation 
\begin{equation}\label{resolvent} 
(\lambda-L_\alpha)  \omega_\alpha =  f_\alpha
\end{equation} 
with $f_\alpha =  \omega_\alpha(0,z)$. As $L_\alpha$ is a compact perturbation of the Laplacian $\Delta_\alpha$, 
standard energy estimates yield that the operator $(\lambda - L_\alpha)^{-1}$ 
is well-defined and bounded from $H^{-1}$ to $H^{-1}$. Indeed, writing $\omega_\alpha = \Delta_\alpha \phi_\alpha$ and multiplying the equation \eqref{resolvent} by $\phi_\alpha$, we get 
 $$ 
\Re \lambda \| \nabla_\alpha \phi_\alpha\|^2_{L^2}  +  \nu \| \Delta_\alpha \phi_\alpha \|_{L^2}^2 \le \| f_\alpha \|_{H^{-1}} \| \phi_\alpha\|_{H^1} 
+ |\alpha| \int_0^\infty |U' \phi_\alpha \partial_z\phi_\alpha| \; dz ,
$$
in which the last term is bounded by $\gamma_0 \|\nabla_\alpha \phi_\alpha\|_{L^2}^2$. Since $\alpha \in \ZZ \setminus \{0\}$, we obtain 
\begin{equation}\label{est-smalla}
\| \phi_\alpha \|_{H^1} \le  \frac{ \| f_\alpha \|_{H^{-1}}}{\Re \lambda +\alpha^2\nu - \gamma_0 }
\end{equation}
for all complex $\lambda$ as long as the denominator on the right is positive. 
Here, $\gamma_0$ denotes some universal constant that is independent of $\alpha, \nu$ and $\lambda$. This proves that the operator $(\lambda - L_\alpha)^{-1}$ 
is well-defined and bounded from $H^{-1}$ to $H^{-1}$ by $|\Re \lambda - \gamma_0|^{-1}$,
 for any complex number $\lambda$ so that $\Re \lambda > \gamma_0$. Hence, the classical semigroup theory (see, for instance, \cite[Theorem 6.13]{Pazy} or \cite{Z2}) yields   
\beq \label{int11}
e^{L_\alpha t}  f_\alpha = {1 \over 2  \pi i} \int_{\Gamma_\alpha} e^{\lambda t}  (\lambda - L_\alpha)^{-1}  f_\alpha \, d \lambda
\eeq
where $\Gamma_\alpha$ is a contour lying on the right of the spectrum of $L_\alpha$. 
Writing $\omega_\alpha = \Delta_\alpha  \phi_\alpha$, the resolvent equation \eqref{resolvent} 
becomes the classical Orr-Sommerfeld equations for the stream function $\phi_\alpha$
\beq \label{OSphys}
\OS_{\alpha,\lambda}(\phi_\alpha) :=
(\lambda + i\alpha U) \Delta_\alpha \phi_\alpha - i\alpha U'' \phi_\alpha  - \nu  \Delta_\alpha^2 \phi_\alpha = f_\alpha,
\eeq
together with the boundary conditions: 
\begin{equation}\label{OSphys2}
 \phi_\alpha = \partial_z  \phi_\alpha =0, \qquad \mbox{on} \quad z=0.
\end{equation}


We then solve the Orr-Sommerfeld problem using the Green function $G_{\alpha, \lambda} (x,z) $, constructed in Chapter \ref{chapter-OS-unstable}. See Theorem \ref{theo-GreenOS-unstable}. 
The solution $\phi_\alpha$ to the Orr-Sommerfeld problem \eqref{OSphys}-\eqref{OSphys2} is then constructed by 
$$ 
\phi_\alpha(z) =  \int_0^\infty G_{\alpha, \lambda} (x,z)  f_\alpha(x) \; dx .
$$
Such a spectral formulation of the linearized Navier-Stokes equations near a boundary layer shear profile has been 
intensively studied  in the physical literature.  We in particular refer to
\cite{Reid, Sch, Lin0,LinBook} for the major works of Heisenberg, Tollmien, C.C. Lin, and Schlichting on the subject. 


To summarize, we have the following. 

\begin{lemma}\label{lem-integralrep}
Let $G_{\alpha,\lambda}(x,z)$ be the Green function of the Orr-Sommerfeld problem \eqref{OSphys}-\eqref{OSphys2} constructed in Theorem \ref{theo-GreenOS-unstable}. For each fixed $\alpha \not =0$,
there hold the integral representations 
\begin{equation}\label{int-La} 
 (\lambda - L_\alpha)^{-1}  f_\alpha (z) =  \int_0^\infty \Delta_\alpha G_{\alpha,\lambda} (x,z)  f_\alpha (x)\;dx
 \end{equation}
 and 
\begin{equation}\label{int-Ga}
e^{L_\alpha t}  f_\alpha (z) =  \frac{1}{2\pi i}  \int_{\Gamma_\alpha} \int_0^\infty e^{\lambda t}  
\Delta_\alpha  G_{\alpha,\lambda} (x,z)  f_\alpha (x) \; dxd \lambda,
\end{equation}
in which $\Gamma_\alpha$ is chosen depending on $\alpha$ and lying in the resolvent set of $L_\alpha$. 
 \end{lemma}

In what follows, we shall estimate the semigroup $e^{L_\alpha t}  f_\alpha$ relying on the pointwise bounds on the Green function $\Delta_\alpha  G_{\alpha,\lambda} (x,z)$ obtained from Theorem \ref{theo-GreenOS-unstable}.


\section{Semigroup bounds}\label{sec-semigroup}


In this section, we shall bound the semigroup $e^{L_\alpha t}$ for $| \alpha | \le\nu^{-\zeta}$ with $\zeta<1/2$. In view of Theorem \ref{theo-GreenOS-unstable} and Lemma \ref{lem-integralrep}, we decompose $e^{L_\alpha t}$ as follows: 
\begin{equation}\label{de-eLta}
e^{L_\alpha t}  = S_{\alpha,1} + S_{\alpha,2} 
\end{equation} 
with 
\begin{equation}\label{def-S1212}
\begin{aligned}
S_{\alpha,1}  \omega_\alpha (z): &=\frac{1}{2\pi i}  \int_{\Gamma_{\alpha}} \int_0^\infty e^{\lambda t}  
 \mathcal{S}_1(x,z) \omega_\alpha (x) \; dxd \lambda,
\\S_{\alpha,2}  \omega_\alpha (z): &= \frac{1}{2\pi i}  \int_{\Gamma_{\alpha}} \int_0^\infty e^{\lambda t}   
\mathcal{S}_2(x,z) \omega_\alpha (x) \; dxd \lambda,
\end{aligned}
\end{equation}
where the kernels $ \mathcal{S}_j(x,z)$ are meromorphic in $\lambda$ and satisfy 
\begin{equation}\label{est-decompGrOS}
\begin{aligned}
|\mathcal{S}_1(x,z) |
&\le  \frac{C_0}{   d(\alpha,\lambda)^2 } e^{-\theta_0 \mu_s |x-z|}
, \\
|\mathcal{S}_2(x,z)| &\le \frac{C_0 | \mu_f(z)|^2}{| \mu_f(x)| d(\alpha,\lambda) }   e^{- \theta_0 | \int_x^z \Re \mu_f \; dy|} ,
\end{aligned}\eeq
uniformly for $(\alpha,\lambda)\in \RR\setminus\{0\}\times \CC$ so that $| \alpha | \le\nu^{-\zeta}$, \eqref{range-c} holds, and
$
| E(\alpha,\lambda)| > \sigma_0,
$
where we recall 
$\mu_s = | \alpha |$ and $
 \mu_f (z)=  \nu^{-1/2} \sqrt{ \lambda + \nu \alpha^2+i \alpha U(z) }.$
 We shall estimate these operators, upon appropriately choosing the contour $\Gamma_\alpha$ of integration so that the estimates \eqref{est-decompGrOS} remain valid.


\subsection{Bounds on $S_{\alpha,1}$.}


In this section, we prove the following. 

\begin{proposition}\label{prop-S1a} 
Let $\beta \in (0,\frac12]$. 
For any positive $\tau$, there is a constant $C_\tau$ so that 
\begin{equation}\label{est-S1a}
\| S_{\alpha,1}\omega_\alpha \|_{\beta}
\le C_\tau e^{(\Re \lambda_0 +\tau)t} \| \omega_\alpha \|_{\beta,\gamma},
 \end{equation}
uniformly in $t\ge 0$, small $\nu >0$, and $|\alpha|\le \nu^{-\zeta}$ with $\zeta<1/2$. 
\end{proposition}
\begin{proof} 
We restrict ourselves to $\alpha > 0$.
Since the Green kernel $\mathcal{S}_1(x,z) $ is meromorphic in $\lambda$, we can apply the Cauchy's theory and take the contour $\Gamma_\alpha$ of integration to consist of $\lambda$ so that 
\begin{equation}\label{def-tau0}
\Re \lambda = \Re \lambda_0 +\tau
\end{equation}
for arbitrary small, but fixed, constant $\tau>0$. Since $\Gamma_\alpha$ remains in the resolvent set of $L_\alpha$, 
the (inviscid) Evans function $E(\alpha,\lambda)$ never vanishes. In addition, recalling the assumption \eqref{range-c} and writing $\lambda = \Re \lambda + i \Im \lambda$, we have
$$
d(\alpha,\lambda) =\inf_{z\in \RR_+}| \lambda + i \alpha U(z)| \ge \theta_0 (1 + \inf_z |\Im \lambda - \alpha U(z)|). $$ 
We thus obtain from \eqref{est-decompGrOS} that 
\begin{equation}\label{bd-S1X1}
|\mathcal{S}_1(x,z) |\le C_\tau (1 +  \inf_z |\Im \lambda - \alpha U(z)|)^{-2} e^{-\mu_s |x-z|},
\end{equation}
for all $\lambda\in \Gamma_\alpha$.

Let us now estimate the convolution $S_{\alpha,1}  \omega_\alpha (z)$ in the boundary layer norm $\| \cdot \|_{\beta,\gamma}$. First we recall from the definition that 
$$
| \omega_\alpha(x) | \le \|\omega_\alpha\|_{\beta,\gamma} e^{-\beta x} (1+ \delta^{-1}\phi_{P}(\delta^{-1}x)).
$$ 
Hence, recalling \eqref{def-S1212} and using \eqref{bd-S1X1}, we have 
$$\begin{aligned}
|S_{\alpha,1}  \omega_\alpha (z)| 
&\le C_\tau\| \omega\|_{\beta,\gamma} \int_{\RR} \int_0^\infty  (1 +  \inf_z |\Im \lambda - \alpha U(z)|)^{-2}  e^{(\Re \lambda_0 + \tau)t} 
\\&\quad \times e^{-\mu_s |x-z|}e^{-\beta |x|} (1+ \delta^{-1}\phi_{P}(\delta^{-1}x))\; dxd \Im \lambda 
.\end{aligned}$$
The integral in $\Im\lambda$ is clearly bounded, yielding 
$$\begin{aligned}
|S_{\alpha,1}  \omega_\alpha (z)| 
&\le C_\tau\| \omega\|_{\beta,\gamma} e^{(\Re \lambda_0 + \tau)t}  \int_0^\infty  
 e^{-\mu_s |x-z|}e^{-\beta |x|} (1+ \delta^{-1}\phi_{P}(\delta^{-1}x))\; dx
.\end{aligned}$$
Recall that $\mu_s = \alpha$. Using the triangle inequality $|z|\le |x|+ |x-z|$ and the fact that $\beta\le1/2 <\alpha$, we obtain 
$$ e^{-\mu_s |x-z|} 
e^{-\beta |x|} \le e^{-\beta z} e^{-\frac12\alpha |x-z|}.$$
Hence, we get 
$$\begin{aligned}
|S_{\alpha,1}  \omega_\alpha (z)| 
&\le C_\tau\| \omega\|_{\beta,\gamma}e^{(\Re \lambda_0 + \tau)t} 
 e^{-\beta z} \int_0^\infty  e^{-\frac12 \alpha |x-z|}(1+ \delta^{-1}\phi_{P}(\delta^{-1}x))\; dx
 \\
 &\le C_\tau\| \omega\|_{\beta,\gamma}e^{(\Re \lambda_0 + \tau)t} 
 e^{-\beta z} \Big( \alpha^{-1} + \delta^{-1}\int_0^\infty \phi_{P}(\delta^{-1}x)\; dx\Big).
\\
&\le C_\tau\| \omega\|_{\beta,\gamma}e^{(\Re \lambda_0 + \tau)t}  e^{-\beta z} ,
\end{aligned}$$
completing the proof of the proposition. 
\end{proof}


\subsection{Bounds on $\mathcal{S}_{\alpha,2} $.}


In this section, we give bounds on the semigroup $\mathcal{S}_{\alpha,2}$, defined as in \eqref{def-S1212}. Precisely, we have 

\begin{proposition}\label{prop-S2a} Let $\beta \in (0,\frac12]$. 
For any positive $\tau$, there is a constant $C_\tau$ so that 
\begin{equation}\label{est-S2a}
\| S_{\alpha,2}\omega_\alpha \|_{\beta,\gamma}
\le C_\tau e^{(\Re \lambda_0 +\tau)t} \| \omega_\alpha \|_{\beta,\gamma},
 \end{equation}
uniformly in $t\ge 0$, small $\nu >0$, and $|\alpha|\le \nu^{-\zeta}$ with $\zeta<1/2$.
\end{proposition}
To prove this proposition we will use the following Lemma

\begin{lemma}\label{lem-tGreen} 
Let $\mathcal{S}_{2}(x,z)$ be the Green kernel defined as in \eqref{def-S1212}. Introduce the temporal Green function 
\begin{equation}\label{def-tGreen2}G_2(t,x,z) :=  \frac{1}{2\pi i}  \int_{\Gamma_{\alpha}} e^{\lambda t}   
\mathcal{S}_2(x,z) \; d \lambda  .\end{equation}
Then, for any positive $\tau$, there are constants $C_\tau, \theta_\tau$ so that there holds 
\begin{equation}\label{ptw-G22}
\begin{aligned}
& |G_2(t,x,z) |
\\&\le C_\tau (\nu t)^{-1/2}e^{\tau t} e^{-\frac{|x-z|^2}{8 \nu t}} 
 +   C_\tau \sum_{\Re \lambda_\alpha\ge \tau}e^{(\Re \lambda_\alpha + \tau)t}\nu^{-1/2} e^{-\theta_{\tau}\nu^{-1/2}|x-z|} 
 \end{aligned}\end{equation}
uniformly in $t\ge 0$, small $\nu >0$, and $\alpha \in \ZZ^*$, in which 
$$
\theta_{\tau} = \frac12 \sqrt{\Re\lambda_\alpha + \tau + \alpha^2 \nu},
$$
and the summation is taken over finitely many unstable eigenvalues $\lambda_\alpha$ of $L_\alpha$ 
such that $\Re \lambda_\alpha \ge \tau$. 
\end{lemma}

%

\begin{proof} 
We shall apply the Cauchy's theory to move the contour of integration $\Gamma_\alpha$ in \eqref{def-tGreen2} 
so that the claimed bounds on $G_2(t,x,z)$ can be obtained. 
As $\Gamma_\alpha$ moves to the left in the complex plane, it may meet unstable eigenvalues of $L_\alpha$. 
We first treat the integral near these unstable eigenvalues. 

To proceed, let $\tau$ be an arbitrary positive number, and let $\lambda_\alpha = -i\alpha c$ be a zero of $E(\alpha,\lambda)$
 such that $\Re \lambda_\alpha \ge \tau$. As $E(\alpha,\lambda)$ is analytic in $c$, its zeros $\lambda_\alpha$ are isolated. 
 Taking $\tau$ smaller, if needed, we can assume that there is no other unstable eigenvalue
 in the ball $B(\lambda_\alpha, \frac12\tau)  = \{|\lambda - \lambda_\alpha |\le \frac12\tau\}$. 
 In particular, we have $|E(\alpha,\lambda)| \ge C_\tau$
for $\lambda \in \partial B(\lambda_\alpha, \frac12\tau) $. 
In addition, since $\Re \lambda \ge \frac14(\Re \lambda_\alpha + \tau)$ on $ \partial B(\lambda_\alpha, \frac12\tau) $, we have 
$$ 
\Re \mu_f = \nu^{-1/2} \Re \sqrt{\lambda + \alpha^2\nu+ i\alpha U } \ge \nu^{-1/2}\theta_\tau,
$$
with $\theta_\tau = \frac12\sqrt{\Re\lambda_\alpha + \tau + \alpha^2 \nu}.$  
Thus, on $\partial B(\lambda_\alpha, \frac12\tau)$, there holds
$$
\begin{aligned}
|\mathcal{S}_2(x,z)| &\le \frac{C_0 | \mu_f(y)|^2}{| \mu_f(x)| d(\alpha,\lambda) }   e^{- \theta_0 | \int_x^z \Re \mu_f \; dy|}  
\\
&\le C_\tau  \nu^{-1/2} e^{- \theta_\tau  \nu^{-1/2} |x-z| } ,
  \end{aligned}$$
upon recalling \eqref{range-c}. This yields
\begin{equation}\label{bd-GreenRes}
\Big| \int_{\partial B(\lambda_\alpha, \frac12\tau)} e^{\lambda t}   
\mathcal{S}_2(x,z) \; d \lambda \Big|\le C_\tau \nu^{-1/2}  e^{(\Re \lambda_\alpha + \tau)t}
   e^{- \theta_\tau \nu^{-1/2} |x-z| }  , 
\end{equation}  
which contributes to the last term in the Green function bound \eqref{ptw-G22}.  

We are now ready to choose a suitable contour of integration $\Gamma_\alpha$. 
Let us consider the case when $x<z$; the other case is similar.
Recall that 
$$
\mu_f(z) 
= \nu^{-1/2}\sqrt{\lambda + i\alpha U+ \alpha^2  \nu } .
$$ 
By construction, $\mathcal{S}_{2}(x,z)$ is holomorphic in $\lambda$, except on the complex half strip
$$ 
\mathcal{H}_\alpha(x,z): = \Big\{ \lambda = -k -\alpha^2  \nu + i \alpha U(y), \qquad k\in \RR_+, \quad y \in [x,z]\Big\}.
$$
In our choice of contour of integration below, we shall avoid to enter this complex strip. 
We set 
$$
\begin{aligned}
\Gamma_{\alpha,1} &:= \Big\{ \lambda =\gamma_1 -\alpha^2 \nu - i \alpha c, \qquad \min_{y\in [x,z]} U(y) \le c \le \max_{y\in [x,z]} U(y) \Big\} 
\\
\Gamma_{\alpha,2} &:=  \Big\{ \lambda =\gamma_1 -\alpha^2 \nu - k^2  \nu  - i \alpha \min_{[x,z]} U + 2  \nu i ak, \qquad k\ge 0 \Big\}
\\
\Gamma_{\alpha,3} &:= \Big\{ \lambda = \gamma_1 -\alpha^2 \nu - k^2  \nu  - i \alpha \max_{[x,z]} U + 2  \nu i ak, \qquad k\le 0 \Big\} 
\end{aligned}$$
where
\begin{equation}\label{def-aaa} 
\gamma_1 = \tau + a^2\nu + \alpha^2\nu, \qquad a = \frac{|x-z| + \sqrt{\nu t}}{4 \nu t} .
\end{equation}
The choice of the parabolic contours $\Gamma_{\alpha,2}$ and $\Gamma_{\alpha,3}$ 
is necessary to avoid singularities in small time \cite{ZumbrunHoward, GrN1}. 
We stress that they never meet the complex strip $\mathcal{H}_\alpha (x,z)$.
In addition, they may leave unstable eigenvalues to the right, in which case the contribution from unstable eigenvalues \eqref{bd-GreenRes} 
is added into the bounds on the Green function.

%
%

%
%
%
%
%
%
%


\subsubsection*{Bounds on $\Gamma_{\alpha,1}$.}


We start our computation with the integral on $\Gamma_{\alpha,1}$. 
We first note that for $\lambda \in \Gamma_{\alpha,1}$, there holds
$$  
\Re  \mu_f = \nu^{-1/2} \Re   \sqrt{ \gamma_1 + i\alpha (U-c)}  \ge \nu^{-1/2}\sqrt{\gamma_1} .
$$
Recalling \eqref{def-aaa}, we compute 
$$\begin{aligned}
| e^{\lambda t} |  = e^{\gamma_1 t} e^{-\alpha^2 \nu t}  = e^{\tau t} e^{a^2 \nu t} = e^{\tau t} e^{\frac{|x-z|^2}{16\nu t} + \frac{|x-z|}{8\sqrt{\nu t}} + \frac1{16}}.
\end{aligned}
  $$
On the other hand, using $\epsilon = \nu/i\alpha$ and $\Re \mu_f \ge \nu^{-1/2} \sqrt {\gamma_1}$, we have 
$$
\begin{aligned}
| \mathcal{S}_{2}(x,z) |  &\le C  \nu^{-1/2} \gamma_1^{-1/2}  e^{-\nu^{-1/2} \sqrt{\gamma_1} |x-z|} .\end{aligned}
$$
Using $\gamma_1 \ge a^2\nu$ and $\gamma_1 \ge \alpha^2 \nu$, we note that 
$$
\begin{aligned}
 e^{-\frac12\nu^{-1/2} \sqrt{\gamma_1} |x-z|} &\le e^{-\frac a2 |x-z|} = e^{-\frac{|x-z|^2}{8\nu t} - \frac{|x-z|}{8\sqrt{\nu t}}}
 \\
 e^{-\frac12\nu^{-1/2} \sqrt{\gamma_1} |x-z|} &\le e^{-\frac12 \alpha|x-z|}  .
  \end{aligned}$$
Thus, we obtain 
\begin{equation}
\begin{aligned}
|  e^{\lambda t} \mathcal{S}_{2}(x,z) |  &\le C  \nu^{-1/2} \gamma_1^{-1/2} e^{\tau t}  e^{-\frac{|x-z|^2}{8\nu t} - \frac12 \alpha|x-z|}
\end{aligned}\end{equation}
for any $\lambda \in \Gamma_{\alpha,1}$. 
Hence, we estimate 
$$\begin{aligned}
 \Big| \int_{\Gamma_{\alpha,1}} e^{\lambda t}   \mathcal{S}_{2}(x,z) \; d\lambda \Big|
&\le  C \alpha \nu^{-1/2} \gamma_1^{-1/2}e^{\tau t} e^{-\frac{|x-z|^2}{8\nu t} - \frac12 \alpha|x-z|}\int_{\min_{[x,z]} U}^{\max_{[x,z]}U} dc 
 \\&\le C \alpha \nu^{-1/2} \gamma_1^{-1/2}e^{\tau t}e^{-\frac{|x-z|^2}{8\nu t} - \frac12 \alpha|x-z|} |x-z| \| U'\|_{L^\infty}
   .      \end{aligned}
 $$ 
 Using the inequality $X e^{-X} \le C $ for $X\ge 0$, we have 
 $$ 
 e^{-\frac12 \alpha  |x-z|} \alpha |x-z| \le C.
 $$
We thus obtain 
\begin{equation}\label{bd-Ga1}\begin{aligned}
 \Big| \int_{\Gamma_{\alpha,1}} e^{\lambda t}   \mathcal{S}_{2}(x,z) \; d\lambda\Big|
&\le   C \nu^{-1/2} \gamma_1^{-1/2} e^{\tau t}e^{-\frac{|x-z|^2}{8\nu t}}
.\end{aligned}
  \end{equation}
This yields the claimed estimate on $\Gamma_{\alpha,1}$, upon noting that $\sqrt t \le C_\tau e^{\tau t}$ for any $t\ge 0$.

  
\subsubsection*{Bounds on $\Gamma_{\alpha,2}$ and $\Gamma_{\alpha,3}$.}


By symmetry, it suffices to give bounds on $\Gamma_{\alpha,2}$. For $\lambda \in \Gamma_{\alpha,2}$ and $y \in [x,z]$, we compute 
$$
\begin{aligned}
 \mu_f(y)  &= \nu^{-1/2}  \sqrt{ \gamma_1  - k^2 \nu+ i \alpha (U -  \min_{[x,z]} U) + 2 i  \nu  ak }
 \\&  =\nu^{-1/2} \sqrt{ \tau + \alpha^2 \nu + i \alpha (U -  \min_{[x,z]} U) + (a+ik)^2 \nu } .
 \end{aligned}
 $$
 Recalling that $k, \alpha, \tau\ge 0$, the above yields
  $$ 
  \Re \mu_f(y) \ge \Re \sqrt{(a+ik)^2} = a
  $$
for $y\in [x,z]$.  
So, using $a =\frac{|x-z| + \sqrt{\nu t}}{4\nu t}$, we get 
{red}{$$ 
\begin{aligned}
a^2 \nu t - a|x-z| 
&= \frac{|x-z|^2 + 2 \sqrt{\nu t}|x-z| + \nu t }{16 \nu t} - \frac{|x-z|^2 + \sqrt{\nu t}|x-z|}{4\nu t}
\\
&= \frac1{16}- \frac{3|x-z|^2 + 2\sqrt{\nu t}|x-z|}{16\nu t}
\end{aligned} $$
and hence }
\begin{equation}\label{bd-G2exp}\begin{aligned} 
e^{  \Re \lambda t} e^{- \int_x^z  \Re \mu_f(y) \; dy} 
&\le e^{\tau t} e^{a^2  \nu t- \nu k^2 t } e^{-a|x-z|} 
\le C e^{\tau t}e^{- \nu k^2 t } e^{-\frac{3|x-z|^2}{16 \nu t}} .
\end{aligned}\end{equation}
As for $\frac{1}{\mu_f(x)}$ in the Green kernel \eqref{est-decompGrOS}, we write 
$$
\begin{aligned}
 \mu_f(x)  =\nu^{-1/2} \sqrt{ \tau + \alpha^2 \nu + a^2 \nu - k^2 \nu  + i \alpha (U -  \min_{[x,z]} U) + 2 i\nu ak} .
 \end{aligned}$$
This implies that 
$$
|\mu_f(x)|\ge \nu^{-1/2}\sqrt{\alpha (U-\min_{[x,z]}U) + 2 \nu ak} \ge \sqrt{2 ak},
$$ 
recalling $k\ge 0$. We also have $|\mu_f(x)|\ge \Re \mu_f(x) \ge a$. Hence, 
$$
 |\mu_f(x)| \ge \frac12 (a + \sqrt {ak}).
$$
Hence, recalling $i\alpha \varepsilon  =  \nu$ and noting $ d\lambda = 2\nu i (a+ ik) dk $ on $\Gamma_{\alpha,2}$,
we can estimate 
$$
\begin{aligned}
 \Big| \int_{\Gamma_{\alpha,2}} e^{\lambda t}   \mathcal{S}_{2}(x,z) \; d\lambda\Big|
&\le  C_\tau e^{\tau t}e^{-\frac{|x-z|^2}{4 \nu t}}\int_{\RR_+} e^{- \nu k^2 t} { \frac{|a+ik|dk}{ a + \sqrt {ak}}} 
\\&\le C_\tau e^{\tau t}e^{-\frac{|x-z|^2}{4 \nu t}}\int_{\RR_+} e^{- \nu k^2 t}{( 1 + a^{-1/2}\sqrt k )} dk 
\\&\le  C_\tau (\nu t)^{-1/2}e^{\tau t}e^{-\frac{|x-z|^2}{4 \nu t}} {(1 + a^{-1/2} (\nu t)^{-1/4} )}
\end{aligned}
$$
up to the contribution from unstable eigenvalues \eqref{bd-GreenRes}. Note that $a\ge (\nu t)^{-1/2}$ and hence $a^{-1/2}(\nu t)^{-1/4}\le 1$.  
The Lemma follows.
\end{proof}

\begin{lemma}\label{lem-convolutionG2} 
Let 
$$
H(t,x,z): =(\nu t)^{-1/2} e^{-\frac{|x-z|^2}{M\nu t}} 
$$ for some positive $M$. 
For any positive $\beta$, there is a constant $C_0$ so that 
$$
\Big \| \int_0^\infty H(t,x,\cdot) \omega_\alpha(x)\; dx \Big\|_{ \beta, \gamma} 
\le C_0 e^{M\beta^2 \nu t } \| \omega_\alpha\|_{ \beta, \gamma}.
$$
\end{lemma}


\begin{proof} Let $\omega_\alpha (z)$ be a boundary layer function that satisfies
\begin{equation}\label{assmp-wbl}
|\omega_\alpha (z) | \le
 \| \omega_\alpha\|_{ \beta, \gamma} \Bigl( 1 +  \delta^{-1} \phi_{P} (\delta^{-1} z)  \Bigr) e^{-\beta z} .
 \end{equation}
We first show that the convolution has the right exponential decay at infinity. Indeed, in the case when $|x-z|\ge M\beta\nu t$, 
we have
$$
e^{-\frac{|x-z|^2}{M\nu t}} e^{-\beta  |x|} \le e^{-\beta  |z|} e^{-|x-z| \Big( \frac{|x-z|}{M\nu t} - \beta \Big)} \le e^{-\beta  |z|}.
$$
Whereas, for $|x-z| \le M \beta \nu t$, we note that 
$$ 
e^{-\frac{|x-z|^2}{M\nu t}} e^{-\beta  |x|} \le 
e^{- M \beta^2 \nu t} e^{-\beta  |x|} \le e^{-\beta |x-z|} e^{-\beta  |x|} \le e^{-\beta  |z|}.
$$
Combining, we have 
\begin{equation}\label{exp-beta13}
e^{-\frac{|x-z|^2}{M\nu t}} e^{-\beta x} \le 
e^{M \beta^2 \nu t} 
e^{-\beta |z|}, \qquad \forall x,z\in \RR
\end{equation}
{red}{which yields the spatial decay $e^{-\beta z}$ in the norm $\|\cdot \|_{\beta,\gamma}$.}


It remains to study the integral 
\begin{equation}\label{conv-heatbl1-stable}
\begin{aligned}
\int_0^\infty (\nu t)^{-1/2}  e^{-\frac{|x-z|^2}{M\nu t}} 
  \Bigl( 1 + \delta^{-1} \phi_{P} (\delta^{-1} x)  \Bigr) \; dx.
\end{aligned}\end{equation}
The integral without the boundary layer behavior is clearly bounded. Next, using the fact that $\phi_{P}(\delta^{-1}x)$ is decreasing in $x$, we have 
$$
\begin{aligned}
 \int_{z/2}^\infty & (\nu t)^{-1/2}   e^{-\frac{|x-z|^2}{M\nu t}} 
  \delta^{-1} \phi_{P} (\delta^{-1} x) \; dx 
 \\ &\le C_0   \delta^{-1} \phi_{P} (\delta^{-1} z)  \int_{z/2}^\infty  (\nu t)^{-1/2}   e^{-\frac{|x-z|^2}{M\nu t}}   \; dx
  \\&\le C_0   \delta^{-1} \phi_{P} (\delta^{-1} z) .  \end{aligned}$$
Whereas on $x\in (0,\frac z2)$, we have $|x-z|\ge \frac z2 $ and $\phi_{P} \le 1$. We have   
\begin{equation}\label{large-time-del}
\begin{aligned}
 \int_0^{z/2} &  (\nu t)^{-1/2}   e^{-\frac{|x-z|^2}{M\nu t}}  
  \delta^{-1} \phi_{P} (\delta^{-1} x) \; dx 
\\
&\le C_0 e^{-\frac{|z|^2}{8M\nu t}}  \delta^{-1}  \int_0^{z/2}   (\nu t)^{-1/2}   e^{-\frac{|x-z|^2}{2M \nu t}}  \; dx 
\\&\le C_0 e^{-\frac{|z|^2}{8M\nu t}}  \delta^{-1} .
  \end{aligned}\end{equation}
It remains to prove that 
\begin{equation}\label{X-bd}  e^{-\frac{|z|^2}{8M\nu t}}  \delta^{-1} \le C_0 \delta^{-1} e^{8M\nu t}  e^{-z/\sqrt \nu}\end{equation}
for some constant $C_0$. Indeed, the inequality is clear, for $|z| \ge 8M\nu^{3/2} t$, since $ e^{-\frac{|z|^2}{8M\nu t}} \le  e^{-z/\sqrt\nu }  $. Next, for $|z| \le 8M\nu^{3/2} t$, we have 
$$ 1\le e^{8M \nu t} e^{-z/\sqrt \nu}.$$ 
This proves \eqref{X-bd}, and so \eqref{large-time-del} is again bounded by $C_0  e^{8M\nu t}  \delta^{-1} \phi_{P} (\delta^{-1} z)$, upon recalling that the boundary layer thickness is of order $\delta = \gamma \sqrt \nu$. 
%
\end{proof}

\begin{lemma}\label{lem-convolutionG21} 
Let 
$$
R(t,x,z): = \nu^{-1/2} e^{-\theta_{\tau}\nu^{-1/2}|x-z|} 
$$ 
for some positive $\theta_\tau$. Then, for any positive $\beta$, there is a constant $C_\tau$, depending on $\theta_\tau$, so that 
$$
\Big \| \int_0^\infty R(t,x,\cdot) \omega_\alpha(x)\; dx \Big\|_{ \beta, \gamma} \le C_\tau \| \omega_\alpha\|_{ \beta, \gamma}.
$$
\end{lemma}

\begin{proof} We need to bound the integral 
$$
\begin{aligned}
\int_0^\infty \nu^{-1/2} e^{-\theta_{\tau}\nu^{-1/2}|x-z|} \Bigl( 1 + \delta^{-1} \phi_{P} (\delta^{-1} x)  \Bigr) e^{-\beta |x|}\; dx.
\end{aligned}
$$
First, taking $\nu$ smaller, if needed, we can assume that $\frac12\theta_\tau \nu^{-1/2}\ge \beta$, and thus by the triangle inequality $|z|\le |x| + |x-z|$, we have 
 $$ e^{-\frac12\theta_\tau \nu^{-1/2}|x-z|} e^{-\beta |x|} \le e^{-\beta |z|}.$$
Next, similarly as done in the previous lemma, we have 
$$
\begin{aligned}
 \int_{z/2}^\infty & \nu^{-1/2} e^{-\theta_{\tau}\nu^{-1/2}|x-z|} 
  \delta^{-1} \phi_{P} (\delta^{-1} x) \; dx 
 \\ &\le C_0   \delta^{-1} \phi_{P} (\delta^{-1} z)  \int_{z/2}^\infty \nu^{-1/2} e^{-\theta_{\tau}\nu^{-1/2}|x-z|} \; dx
  \\&\le C_\tau \delta^{-1} \phi_{P} (\delta^{-1} z) ,\end{aligned}$$
and 
$$
\begin{aligned}
 \int_0^{z/2} & \nu^{-1/2} e^{-\theta_{\tau}\nu^{-1/2}|x-z|} 
  \delta^{-1} \phi_{P} (\delta^{-1} x) \; dx 
  \\& \le  \nu^{-1/2} e^{-\frac12\theta_{\tau}\nu^{-1/2}|z|} 
   \int_0^{z/2}\delta^{-1} \phi_{P} (\delta^{-1} x) \; dx 
\\& \le C_0  \nu^{-1/2} e^{-\frac12\theta_{\tau}\nu^{-1/2}|z|} ,\end{aligned} 
  $$
which is again bounded by $C_\tau \delta^{-1} \phi_{P} (\delta^{-1} z) $, upon recalling that $\delta = \gamma \sqrt \nu$ and $\phi_P(Z) = (1+Z^P)^{-1}$. 
  \end{proof}

\begin{proof}[Proof of Proposition \ref{prop-S2a}] 
In view of \eqref{def-S1212} and \eqref{def-tGreen2}, we have 
$$ 
S_{\alpha,2} \omega_\alpha (z) = \int_0^\infty G_2(t,x,z) \omega_\alpha(x)\; dx.
$$
For each fixed positive $\tau$, we first show that the set of unstable eigenvalues $\lambda_\alpha$ of $L_\alpha$, 
for all $\alpha\in \ZZ^*$, such that $\Re \lambda_\alpha \ge \tau$ is finite. 
Since viscosity is sufficiently small, it suffices to study the Rayleigh problem
$$
 \Delta_\alpha \phi - \frac{U''}{U-c}\phi =0, \qquad c = -\frac{\lambda_\alpha}{i\alpha},
 $$
with the boundary condition $\phi_{\vert_{z=0}} =0$. For each $\alpha \in \ZZ^*$, 
it is clear that there are only finitely many unstable eigenvalues, since the Rayleigh operator 
is a compact perturbation of the Laplacian $\Delta_\alpha$. In addition, multiplying the Rayleigh equation by $\overline{\phi}$ 
and integrating by parts, we get   
$$ 
\int_0^\infty (|\partial_z\phi|^2 + \alpha^2 |\phi|^2) \; dz \le \frac{1}{|\Im c|} \int_0^\infty |U''| |\phi|^2\; dz.
$$ 
In particular, if $\alpha^2 |\Im c| \ge \|U''\|_{L^\infty}$, there is no nontrivial solution to the Rayleigh problem. 
This implies that there is no unstable eigenvalue $\lambda_\alpha$, whenever $\alpha \Re \lambda_\alpha \ge \|U''\|_{L^\infty}$. 
In particular, there are no unstable eigenvalues $\Re \lambda_\alpha \ge \tau$, whenever $\alpha\ge \tau^{-1}\|U''\|_{L^\infty}.$ 

This proves that there are finitely many unstable eigenvalues of $L_\alpha$ so that $\Re \lambda_\alpha \ge \tau$ for all $\alpha \in \ZZ^*$.
 In particular, the summation in \eqref{ptw-G22} from Lemma \ref{lem-tGreen} is finite, independent of $\alpha\in \ZZ^*$, 
 yielding 
$$ 
|G_2(t,x,z) |\le C_\tau (\nu t)^{-1/2}e^{\tau t} e^{-\frac{|x-z|^2}{8 \nu t}}  
+  C_\tau e^{(\Re \lambda_0 + \tau)t}\nu^{-1/2} e^{-\theta_{\tau}\nu^{-1/2}|x-z|} 
$$
with $\theta_{\tau} = \frac12 \sqrt{\Re \lambda_0 + \tau + \alpha^2 \nu}$, 
where $\lambda_0$ denotes the maximal unstable eigenvalue.
Finally, applying Lemmas \ref{lem-convolutionG2} and \ref{lem-convolutionG21}, 
respectively, to the above pointwise bounds, we complete the proof of the Proposition.
\end{proof}

\part{Nonlinear analysis}

%
%
%
%

\chapter{The unstable case}\label{chapter-nonlinear-unstable}


\section{Introduction}


We now detail our main nonlinear result for the unstable case. Precisely, we prove 

\begin{theo} \label{theoinstable}
Let $U(y)$ be a smooth and analytic function which converges exponentially fast at infinity to a constant, 
with $U(0) = 0$ and assume that it is spectrally unstable for linearized
Euler equations, with a simple eigenvalue. 
Then it is nonlinearly unstable for Navier Stokes equations in $L^\infty$, provided $\nu$ is small enough,
in the following sense.
 For any $s$ arbitrarily large, there exist $\sigma_0 > 0$, $C_0 > 0$ 
and a sequence of solutions $u^\delta$ of Navier Stokes equations with forcing terms $f^\delta$,
 on some interval $[0,T^\delta]$, such that, as $\delta \to 0$,
$$
\| u^\delta(0) - U^\nu(0) \|_{H^s} \le \delta,
$$
$$
\| f^\delta \|_{L^\infty([0,T^\nu],H^s)} \le \delta,
$$
but
$$
\| u^\delta(T^\delta) - U^\nu(T^\delta) \|_{L^p} \ge \sigma_0, \qquad \forall p\in [1,\infty]
$$
and
$$
T^\delta = O ( \log \delta^{-1} ) ,
$$
where $U^\nu(t,y)$ is the solution of heat equation with diffusivity $\nu$ and initial data $U(y)$.
\end{theo}


Let us now discuss this result. Up to the best of our knowledge it is the first rigorous result of instability of a shear layer
profile near a boundary for Navier Stokes equations. According to Rayleigh's criterium 
the profile $U^\nu$ will have an inflection point.
Physically, this may correspond to a reverse flow and thus rules out the exponential profile $U_\infty (1 - e^{-y / C})$.


\subsection*{Notations}


For $\alpha \in \rit$ we define
$$
\Delta_\alpha = \partial_y^2 - \alpha^2 
$$
and
$$
\nabla_\alpha = (i \alpha, \partial_y) .
$$
In particular $\nabla_\alpha^2 = (- \alpha^2, \partial_y^2)$.
The three dimensional case is exactly similar to the two dimensional one, therefore we restrict ourselves to the two dimensional case.


\section{General strategy}\label{sec-strategy}


The starting point is the choice of the shear layer profile (\ref{shear}). We will choose a shear profile $U_0 = (U(y),0)$ 
which is unstable with respect to linearized Euler equations. More precisely, we start from $U_0$, such that
there exists an exponentially growing solution to the following linearized Euler equations
\beq \label{linE1}
\partial_t v + (U_0 \cdot \nabla) v + (v \cdot \nabla) U_0 + \nabla p = 0,
\eeq
\beq \label{linE2}
\nabla \cdot v = 0,
\eeq
\beq \label{linE3} 
v_2 = 0 \qquad \hbox{on} \qquad y = 0 .
\eeq
The study of the linear stability of a shear layer profile is a classical issue in fluid mechanics. The classical strategy to address
this question is to introduce the stream function of $v$ and to take its Fourier transform in the tangential variable $x$ (with dual
Fourier variable $\alpha$) and the Laplace transform in time (with dual variable $\lambda = -i\alpha c$). Precisely, we look for $v$ of the form
\beq \label{stream1}
v = \nabla^\perp \Bigl( e^{i \alpha (x - c t) } \psi(y) \Bigr) + \mbox{complex conjugate.}
\eeq
Putting (\ref{stream1}) in (\ref{linE1}), we get the classical Rayleigh equation for the stream function $\psi$
\beq \label{Ray1}
(U - c) (\partial_y^2 - \alpha^2) \psi = U'' \psi,
\eeq
\beq \label{Ray2}
\alpha  \psi(0) = \lim_{y\to + \infty}  \psi(y) = 0 .
 \eeq
 The study of the linear stability of $U$ reduces to a spectral problem: find $c$ and $\psi$, solutions of Rayleigh equations, with 
 $\Im(\alpha c)  > 0$. Following the classical Rayleigh criterium, if such an instability exists, then $U$ must have an inflection point.
Such smooth unstable profiles do exist (see, for instance, \cite{Grenier00CPAM}).
We choose the most unstable mode, namely the largest $|\alpha \Im c|$.
Starting with such an instability, we can construct an instability for the following linearized Navier Stokes equations
\beq \label{linNS1}
\partial_t v + (U_0 \cdot \nabla) v + (v \cdot \nabla) U_0 - \nu\Delta v + \nabla p = 0,
\eeq
\beq \label{linNS2}
\nabla \cdot v = 0,
\eeq
\beq \label{linNS3} 
v = 0 \qquad \hbox{on} \qquad y = 0 .
\eeq
The analogs of the Rayleigh equations \eqref{Ray1}-\eqref{Ray2} are the Orr Sommerfeld equations which read
\beq \label{inOS1}
-\eps (\partial_y^2 - \alpha^2)^2 \psi + (U - c) (\partial_y^2 - \alpha^2) \psi = U'' \psi,
\eeq
\beq \label{inOS2}
\alpha  \psi(0) = \psi'(0) = \lim_{y\to + \infty}  \psi(y) = 0,
 \eeq
 where
$$
\eps = {\nu \over i \alpha} .
$$
Such a spectral formulation of the linearized Navier-Stokes equations near a boundary layer shear profile has been 
intensively studied  in the physical literature.  We in particular refer to
\cite{Reid,Sch} for the major works of Heisenberg, Tollmien, C.C. Lin, and Schlichting on the subject. 
We also refer to \cite{GGN1,GGN3,GrN1} for the rigorous spectral analysis of the Orr-Sommerfeld equations.  

Now starting from an unstable mode $(\psi^0,c^0)$ of the Rayleigh equation for some positive $\alpha$, it is possible to construct an unstable mode
$(\psi^\nu,c^\nu)$ for Navier Stokes equations, provided $\nu$ is small enough, such that
\beq \label{mode1}
\psi^\nu - \psi^0 = O(\nu^{1/2}),
\eeq
\beq \label{mode2}
c^\nu - c^0 = O(\nu^{1/2}).
\eeq
Let
$$
\lambda_0 = \alpha c^\nu .
$$
This has been proved rigorously in \cite{GrN1} through a complete analysis of the Green function of Orr Sommerfeld equation.
More precisely, the proof of (\ref{mode1})-(\ref{mode2})
relies on the complete description of all four independent solutions of the fourth order differential
equation (\ref{inOS1}). It can be proven that two of them go to $+ \infty$ as
$y \to + \infty$. These two solutions can be forgotten in the construction of an unstable mode. The other two
converge to $0$ as $y \to + \infty$. One, called $\psi_f$, has a "fast" behavior, namely behaves
like $\exp( - C y / \sqrt{\eps})$ for large $y$. The other one, called $\psi_s$, has a "slow" behavior and behaves
like $\exp(- C | \alpha | y)$. The second one, $\psi_s$, comes from the Rayleigh mode, 
and is a small perturbation of $\psi^0$. Then the unstable mode $\psi^\nu$ is a combination of $\psi_f$ and $\psi_s$
and is of the form
\beq \label{structurepsi}
\psi^\nu = \alpha^\nu \psi_f + \beta^\nu \psi_s.
\eeq
The two relations $\psi^\nu(0) = \partial_y \psi^\nu(0) = 0$ give the dispersion relation. Using the fact
that $\psi_s$ is an approximate eigenmode for Rayleigh equation, (\ref{mode1}) and (\ref{mode2}) can be proved using an implicit function
theorem (see \cite{GrN1} for complete details). We also get, for every positive $k$, that
\beq \label{structurepsi2}
| \partial_y^k \psi^\nu (y) | \le {C_k \over | \eps^{(k - 1) / 2} |}  e^{- C y / | \sqrt{\eps}|} + C_k e^{- \beta y} 
\eeq
for some positive $\beta$.
Note that  $\psi^\nu$ has a boundary layer behavior. This is natural since there is a change of boundary conditions 
between Rayleigh and Orr Sommerfeld equations. The size of the boundary layer is of order $\eps^{1/2} \approx \nu^{1/2}$ (for fixed $\alpha$), which is introduced to
balance $\eps \partial_y^4$ and $\partial_y^2$. This sublayer is known as "viscous sublayer" in the physical literature \cite{Reid}.
Note that $\psi_s$ and $\psi_f$ are analytic on a strip $| \Im y | \le \sigma_0$ for some $\sigma_0 > 0$.

Once the linear instability is constructed, we may construct an approximate solution of the form
\beq \label{uapp}
u^{app}(t,x,y) = \sum_{j=1}^M \nu^{ Nj} u^{j}(t,x,y) 
\eeq
starting from the maximal unstable eigenmode
$$
u^1(t,x,y) = \Re \Bigl( \psi^\nu e^{i \alpha (x - c^\nu t)} \Bigr) .
$$
The construction of such an approximate solution is routine work, and involves successive resolutions of linearized Navier
Stokes equations
\begin{equation}\label{eqs-un}
\begin{aligned}
\partial_t u^n + (U_0 \cdot \nabla) u^n + (u^n \cdot \nabla) U_0 -\nu \Delta u^n + \nabla p^n &= F_n,
\\
\nabla \cdot  u^n &= 0,
\end{aligned} \end{equation}
together with the zero initial data and zero Dirichlet boundary conditions, with
$$
F_n = - \sum_{1 \le j \le n-1} (u^j \cdot \nabla ) u^{n-j} .
$$
Note that by construction \cite{Grenier00CPAM}, $u^{app}$ solves Navier Stokes equations, up to a very small term $R^M$, of order
$\nu^{N(M+1)} e^{(M+1) \Re \lambda_\nu t}$, with $\lambda_\nu = -i\alpha c^\nu$, the maximal unstable eigenvalue. 

Let  $u^\nu$ be the solution of Navier Stokes equations
with initial data $u^{app}(0)$. A natural next step is to try to bound the difference $v : = u^\nu - u^{app}$ in $L^2$ norm. However,
we only get
$$
\frac{d}{dt} \| v \|_{L^2}^2 \le C (1+ \| \nabla u^{app} \|_{L^\infty}) \| v \|_{L^2}^2 + \| R^M \|_{L^2}^2 .
$$
As there is a boundary layer in $u^{app}$, $\| \nabla u^{app} \|_{L^\infty}$ is unbounded as $\nu \to 0$, and thus, this energy inequality
is useless when $u^{app} - U$ is of order greater than $\sqrt \nu$. Using only energy estimates, we cannot obtain $O(1)$ instability
in $L^\infty$, and are limited to $O(\sqrt \nu)$ instability (this is the main limitation of \cite{Grenier00CPAM}, noting there the instability is of order $\nu^{1/4}$, since viscosity is of order $\sqrt \nu$ after the hyperbolic scaling in the boundary layers).

The reason of this failure is that the viscous sublayer becomes linearly unstable in the inviscid limit \cite{Reid,GGN3}. 
The next natural idea is to work with analytic initial
data and to hope that analyticity will kill sublayer instabilities, exactly as in Caflisch and Sammartino work \cite{SammartinoCaflisch2}, 
where the authors used
analyticity to kill any instability of Prandtl's layers. However in the current setting, we want to get control over time intervals of order
$\log \nu^{-1}$, namely on unbounded time intervals. As the analyticity radius decreases with time, it becomes small, of order
$1 / t$ as $t$ increases, and is therefore too small to control instabilities in large times. This strategy therefore fails.

In this paper, we will directly prove that the series (\ref{uapp}) converges as $M$ goes to $+ \infty$, in analytic spaces.
This leads to a direct construction of a genuine solution of Navier Stokes equations, defined by
\beq \label{utrue}
u^\nu(t,x,y) = U +  \sum_{j=1}^{+ \infty} \nu^{ Nj} u^{j}(t,x,y) 
\eeq
The underlying idea is the following: if we try to control the difference between the true solution and an approximate one, we have
to bound solutions of linearized Navier Stokes equations. However because of the shear, vertical derivatives of such solutions
increase polynomially in time, simply because of the term $\partial_t + U(y) \partial_x$, which generates high normal derivatives.
This polynomial growth can not be avoided, except if we are working with a finite sum of eigenmodes. For eigenmodes, we simply
have an exponential growth, without polynomial disturbances. As a matter of fact, all the terms appearing in
(\ref{utrue}) are driven by eigenmodes through Orr Sommerfeld equations.

The proof of the convergence of (\ref{utrue}) relies on the accurate description of the Green function of Orr Sommerfeld equations,
detailed in \cite{GrN1}, and on the introduction of so called generator functions. Generator functions combine all the norms
of all the $u^j$, and can be seen as a time and space depending norm. We prove that these generator functions satisfy
a Hopf inequality, which allows us to get analytic bounds which are uniform in $M$.

The plan of this paper is the following. We begin with the definition of generator functions. We then study the generator function of
solutions of Laplace equations, and then of Orr Sommerfeld equations. We then detail the construction of $u^j$ and derive uniform
bounds on the generator functions, which ends the proof.



\section{Generator functions}



\subsection{Definition}\label{sec-Gen}


Let $f(x,y)$ be a smooth function. 
For $z_1,z_2\ge 0$, we define the following two functions, called in this paper "generator functions"
\beq \label{Genbl}
\begin{aligned}
Gen_0(f)(z_1,z_2) &= \sum_{\alpha \in \ZZ} \sum_{\ell \ge 0}  e^{z_1 |\alpha|}  \|  \partial_y^\ell f_\alpha \|_{\ell,0}{z_2^\ell \over \ell ! } ,
\\Gen_\delta(f)(z_1,z_2) &= \sum_{\alpha \in \ZZ}\sum_{\ell \ge 0}  e^{z_1 |\alpha|} \| \partial_y^\ell f_\alpha \|_{\ell,\delta} {z_2^\ell \over \ell ! } ,
\end{aligned}\eeq
in which $f_\alpha(y)$ denotes the Fourier transform of $f(x,y)$ with respect to the $x$ variable. In these sums, 
$$
\begin{aligned}
\| f_\alpha \|_{\ell,0} &= \sup_{y} \varphi(y)^\ell| f_\alpha(y) | ,
\\\| f_\alpha \|_{\ell,\delta} &= \sup_{y} \varphi(y)^{\ell}| f_\alpha(y) | \Bigl( \delta^{- 1} e^{-y / \delta} + 1 \Bigr)^{-1} ,
\end{aligned}$$
where 
$$
\varphi(y) = \frac{y}{1+y}
$$ 
and where the boundary layer thickness $\delta$ is equal to
$$
\delta = \gamma_0 \sqrt \nu
$$
for some sufficiently large $\gamma_0>0$. 
More precisely, $\gamma_0 $ will be chosen so that $\gamma_0^{-1} \le \sqrt{\Re \lambda_0/2}$, 
where $\lambda_0$ is the maximal unstable eigenvalue of the linearized Euler equations around $U$.  

Note that $Gen_0$, $Gen_\delta$ and all their derivatives are non negative for positive $z_1$ and $z_2$.
These generator functions $Gen_0(\cdot)$ and $Gen_\delta(\cdot)$ will respectively control the velocity and the vorticity of the solutions
of Navier Stokes equations.

For convenience, we introduce the following generator functions of one-dimensional functions $f = f(y)$:
\beq \label{Genbl-a}
\begin{aligned}
Gen_{0,\alpha}(f)(z_2) &= \sum_{\ell \ge 0}  \|  \partial_y^\ell f \|_{\ell,0}{z_2^\ell \over \ell ! } ,
\\Gen_{\delta,\alpha}(f)(z_2) &= \sum_{\ell \ge 0}   \| \partial_y^\ell f \|_{\ell,\delta} {z_2^\ell \over \ell ! } .
\end{aligned}\eeq
Of course, it follows that 
$$
Gen_0(f) = \sum_{\alpha\in \ZZ} e^{z_1|\alpha|} Gen_{0,\alpha} (f_\alpha)
$$ 
for functions of two variables $f = f(x,y)$, and similarly for $Gen_\delta$.

\section{Orr-Sommerfeld equations}



\subsection{Introduction}


In this section, we study the Orr-Sommerfeld equations 
\begin{equation}\label{OS1} 
Orr_{\alpha,c}(\phi) :=  - \epsilon \Delta_\alpha^2\phi + (U-c) \Delta_\alpha \phi - U'' \phi = f,
\end{equation}
together with the boundary conditions 
\begin{equation}\label{OS2} \phi_{\vert_{y=0}} = 0, \qquad \partial_y \phi_{\vert_{y=0}} =0,\end{equation}
and $\phi \to 0$ as $y \to + \infty$,
with $\Delta_\alpha = \partial_y^2 - \alpha^2$. We shall focus on the case when $\alpha\not =0$; the $\alpha=0$ case will be treated in Section \ref{sec-azero}. 
 The Orr-Sommerfeld problem is the resolvent problem of the linearized Navier-Stokes equations around
 a shear profile $U$, 
 written in terms of the stream function $\phi$. 
We shall study the generator functions of Orr-Sommerfeld solutions.

Throughout this paper, $| \Im c|$ will always be larger than $3\Re \lambda_0 /2$, 
where $\lambda_0$ is the speed of growth of the linear instability. In particular,
$\Im c$ will be bounded away from $0$. Moreover we will restrict ourselves to $\alpha \ll \epsilon^{-1/3}$, or precisely
\begin{equation}\label{range-a}
| \eps \alpha^3\log \nu | \le 1.
\end{equation}
Let us first describe Orr Sommerfeld equations in an informal way.
For small $\epsilon$, Orr Sommerfeld equations are a viscous perturbation of Rayleigh equations
\begin{equation}\label{def-reRay1} 
Ray_{\alpha,c}(\phi) := (U-c) \Delta_\alpha \phi - U'' \phi = f,
\end{equation}
with boundary conditions $\phi(0) = 0$ and $\lim_{y \to + \infty} \phi(y) = 0$.
Note that the equation $Ray_{\alpha,c}(\phi) = 0$ may be rewritten as
$$
(\partial_y^2 - \alpha^2) \phi - {U'' \over U - c} \phi = 0.
$$
As $y \to + \infty$, it therefore simplifies into $\partial_y^2 \phi - \alpha^2 \phi = 0$. Hence this equation
has two independent solutions $\phi_{s,\pm}$, with respective asymptotic behaviors
$e^{\pm |\alpha| y}$. Note that $c$ is an eigenvalue if and only if $\phi_{s,-}(0) = 0$.
We have to bound these solutions uniformly in $| \alpha | \ge 1$ and $c$, with $| \Im c| \ge 3\Re \lambda_0 / 2$.
As $\alpha$ goes to $\infty$, $\phi_{s,\pm}$ converge to $e^{\pm | \alpha | y}$. Moreover,
$ U'' / (U-c)$ is bounded since $|\Im c| \ge 3 \Re \lambda_0/2$ is bounded away from $0$, therefore this term may be
handled as a regular perturbation. 

When we add the viscous term $\epsilon \Delta_\alpha^2\phi$, these two solutions are slightly perturbed, but give birth to two independent
solutions of Orr Sommerfeld with a "slow" behavior, which behave like $e^{\pm | \alpha | y}$. 
Two additional solutions, called $ \phi_{f,\pm}$, appear, with a fast behavior.
For these solutions the viscous term is no longer negligible and is of the same order as the Rayleigh one. At leading order,
$ \phi_{f,\pm}$ are solutions to
\beq \label{AiryA}
- \eps \partial_y^4 \phi + (U-c + \alpha^2 \epsilon ) \partial_y^2 \phi = 0.
\eeq
Let
$$
\mu_f(y) =  \sqrt{\alpha^2 + \frac{U-c}{\epsilon}},
$$
taking the positive real part.
Then at first order  $\phi_{f,\pm}$ behaves like $e^{\pm \int_0^y \mu_f(z) dz}$.

%
%
%
%
Let $G_{\alpha,c}(x,y)$ be the Green function of the Orr-Sommerfeld problem. This Green function may be decomposed in
a "slow part" $G_s$ and a "fast part" $G_f$, such that
$$
G_{\alpha,c}(x,y) = G_s(x,y) + G_f(x,y) .
$$
We recall the following theorem, which is the main result of \cite[Theorem 2.1]{GrN1}.

\begin{theorem}\label{theo-GreenOS} 
Let $\alpha, c$ be arbitrary, so that  $| \Im c |$ is bounded away from $0$ and $| \alpha \Im c| > 3 \Re \lambda_0 / 2$.
Then, there are universal positive constants $C_0, \theta_0$ so that 
\begin{equation}\label{est-GrOS1}
 |G_s(x,y)|  \le  \frac{C_0}{\mu_s (1 + |\Im c|)}  \Big( e^{-\theta_0\mu_s |x-y|}  +  e^{-\theta_0 \mu_s |x+y|} \Big) 
 \eeq
 \beq \label{est-GrOS1b}
 | G_f(x,y) | \le  \frac{C_0}{m_f (1 + |\Im c|) } \Big( e^{- \theta_0 m_f|x-y|}  + e^{- \theta_0 m_f|x+y|} \Big) 
\end{equation}
for all $x,y\ge 0$, in which 
 \begin{equation}\label{def-mMf}
\mu_s = |\alpha|, \qquad  m_f = \inf_{y} \Re  \mu_f (y) , 
\end{equation}
with 
$$
\mu_f(y) =  \sqrt{\alpha^2 + \frac{U-c}{\epsilon}},
$$ 
taking the positive real part. Similar bounds hold for derivatives, namely for $k \ge 0$ and $l \ge 0$ with $k + l \le 3$,
\begin{equation}\label{est-GrOS2}
 |\partial_x^k \partial_y^l G_s(x,y)|  
 \le  \frac{C_{k,l}}{\mu_s^{1 - k - l}( 1 + |\Im c|)}  \Big( e^{-\theta_0\mu_s |x-y|}  +  e^{-\theta_0 \mu_s (|x+y|)} \Big) 
\eeq
\beq \label{est_GrOS2b}
| \partial_x^k \partial_y^l G_f (x,y) |
\le    \frac{C_{k,l}}{m_f^{1 - k - l} ( 1 + |\Im c| )} \Big( e^{- \theta_0 m_f|x-y|}  + e^{- \theta_0 m_f|x+y|} \Big) 
\end{equation}
\end{theorem}


In the sequel, $\Im c$ will always be bounded away from $0$, but can be very large.
This theorem is the main result of \cite[Theorem 2.1]{GrN1}. However, for the sake of completeness
we sketch in the following lines the computation of the Green function, at a formal level.
The Green function $G_{\alpha,c}(x,y)$ is constructed through the representation 
$$
G_{\alpha,c}(x,y) = \left \{ \begin{aligned} 
\sum_{k = s,f} d_{k}(x) \phi_{k,-}(y)  + \sum_{k = s,f} e_k(x) \phi_{k,-}(y) , \quad & y>x>0\\
\sum_{k = s,f} d_{k}(x) \phi_{k,-}(y)  + \sum_{k = s,f} f_k(x) \phi_{k,+}(y), \quad & 0<y<x\\
\end{aligned} \right.  
$$
where the first sum takes care of the boundary condition and the second one of the singularity of the Green function near $x = y$.

It remains to compute $d_k$, $e_k$ and $f_k$, so that $G_{\alpha,c}$ is continuous, together with its first two derivatives, so that
$\eps \partial^3_y G_{\alpha,c}$ has a unit jump at $y = x$ and so that $G_{\alpha,c}$ satisfies Dirichlet boundary condition together
with its first derivative.
The main contribution in $\eps  \partial^3_y G_{\alpha,c}$ comes from fast modes since $\mu_f \gg \mu_s$. 
In order to get a unit jump at $y = x$
we have to choose, at leading order,
$$
e_f(x)  \sim - {\phi_{f,-}(x)^{-1} \over 2 \eps \mu_f(x)^3}, \qquad
f_f(x) \sim - {\phi_{f,+}(x)^{-1} \over 2 \eps \mu_f(x)^3} 
$$
Note that
$$
{1 \over \eps \mu_f^2} = {1 \over \eps  \alpha^2 + U - c}
$$
and is therefore bounded by $C/ (1+| \Im c|)$ for large $| \Im c |$.
With this choice of $e_f$ and $f_f$, at leading order, $G_{\alpha,c}$ and its second derivatives
are equal at $x^+$ and $x^-$. To match the first derivative we use
$\phi_s(y)$. Note that
$$
\partial_y \Bigl( e_f(x) \phi_{f,-}(y) \Bigr) \sim {1 \over 2 \eps \mu_f^2} = {1 \over 2} {1 \over  \eps \alpha^2 + U - c}
$$
The jump of the first derivative of fast modes is therefore bounded by $C/ (1 + | \Im c|)^{-1}$ for large $| \Im c|$. 
This jump is compensated by the slow modes. 
As a consequence, as the jump in the first derivatives of $e^{\pm | \alpha | y}$ is $|\alpha|$,
 $e_s(x)$ and $f_s(x)$ are of order $1 / |\alpha| (1+| \Im c|)$ for large $| \Im c |$.
The bounds on $d_k$ can be obtained in a similar way.


\subsection{Pseudoinverse}


We will also be interested in the case when $Orr_{\alpha,c}$ is not invertible. In this case,
$Im(Orr_{\alpha,c})$ is of codimension $1$, and the equation $Orr_{\alpha,c}(\phi) = f$
may be solved only if $\langle f,\phi_{\alpha,c}\rangle_{L^2} = 0$, where $\phi_{\alpha,c}$ spans the orthogonal
of $Im(Orr_{\alpha,c})$. In \cite{GrN1} we show that we can construct an inverse of $Orr_{\alpha,c}$
on $Im(Orr_{\alpha,c})$ through a kernel $G = G_s + G_f$ which satisfies the same bounds as in Theorem
\ref{theo-GreenOS}. Moreover, the eigenmode of $Orr_{\alpha,c}$ does not lie in
$Im(Orr_{\alpha,c})$. More precisely

\begin{theo} \label{theopseudo} \cite{GrN1}
Let $\alpha$ be fixed.
Let $c_0$ be a simple eigenvalue of $Orr_{\alpha,c}$ with corresponding eigenmode
$\phi_{\alpha,c_0}$. Then there exists a bounded family of linear forms $l^\nu$
and a family of pseudoinverse operators $Orr^{-1}$ such that for any stream function $\phi$,
$$
Orr_{\alpha,c}\Bigl(Orr^{-1}(\phi) \Bigr)  = \phi - l^\nu(\phi) \phi_{\alpha,c_0} .
$$
Moreover, $Orr^{-1}$ may be defined through a Green function $G = G_s + G_f$
which satisfies (\ref{est-GrOS1}) and (\ref{est-GrOS1b}).
\end{theo}


\subsection{Bounds on solutions of Orr Sommerfeld equations}


\begin{proposition}\label{prop-OS1} ($L^\infty$ norms) \\ 
Let $\phi$ solve the Orr-Sommerfeld problem \eqref{OS1}-\eqref{OS2}. 
For $| \epsilon \alpha^3 | \le 1$, $|\alpha \Im c| >  3 \Re \lambda_0 / 2$ and $| \Im c|$ bounded away from $0$, there hold
\beq \label{LLa1}
\begin{aligned}
 | \alpha |^2  \|\phi \|_{0,0} + 
 | \alpha |  \| \nabla_\alpha \phi \|_{0,0}  +  \|\nabla_\alpha^2 \phi \|_{0,0}
&\le {C_0 \over 1+ | \Im c |} \|f\|_{0,0}
\end{aligned}
\eeq
and
\beq \label{LLa2}
\begin{aligned}
 \|\sqrt \epsilon \nabla_\alpha^3\phi \|_{0,0}  + \| \epsilon \nabla_\alpha^4\phi \|_{0,0}
&\le C_0 \|f\|_{0,0}.
\end{aligned}
\eeq
\end{proposition}
Equations (\ref{LLa1}) and (\ref{LLa2}) express a classical regularity result: Orr Sommerfeld equation is a small fourth order 
elliptic perturbation
of a second order elliptic equation. Therefore we gain the full control on two derivatives of the solution, and partial controls on
third and fourth derivatives, with prefactors $\sqrt{\eps}$ and $\eps$. 
\begin{proof} By construction, the solution $\phi$ is of the form 
\beq \label{expressint} 
\phi(y) = \int_0^\infty G_{\alpha,c} (x,y) f(x) \; dx .
\eeq
Hence, 
$$
\begin{aligned} 
|\phi(y)| 
&\le  C_0 \int_0^\infty  \Big( \frac{e^{-\theta_0\mu_s |x-y|} }{\mu_s (1+|\Im c|)}
+ \frac{e^{- \theta_0 m_f|x-y|}}{ m_f (1+|\Im c|)}  \Big) f(x)\; dx 
\\
& +  C_0 \int_0^\infty  \Big( \frac{e^{-\theta_0\mu_s |x+y|}}{\mu_s (1+|\Im c|)} 
+ \frac{e^{- \theta_0 m_f|x+y|} }{ m_f (1+|\Im c|)} \Big) f(x)\; dx 
\\
&\le  C_0 \| f \|_{0,0} (\mu_s^{-2} + m_f^{-2}) (1+|\Im c|)^{-1} 
.\end{aligned}$$
We recall that $\mu_s = |\alpha|$ and that  $\epsilon \mu_f^2 = \eps \alpha^2 + (U-c)$. Hence,
$$
{\alpha^2 \over \mu_f^2} = {2 \over 1 + \alpha (U - c)/ (\eps \alpha^3) }.
$$
which is bounded since $| \alpha \Im c | > \Re \lambda_0$ and $| \eps \alpha^3 | \le 1$.
This proves that 
$$
\| \alpha^2 \phi\|_{0,0} \le C_0 (1+|\Im c|)^{-1} \|f\|_{0,0}.
$$ 
To get the bounds on  $\alpha \partial_y\phi$ and $\partial_y^2 \phi$,
 we differentiate (\ref{expressint}) with respect to $y$, splitting the integral in $x < y$ and $x > y$, and fulfill similar computations.

Similarly, we compute 
$$
\begin{aligned} 
|\partial_y^3\phi(y)| 
&\le C_0  (1+|\Im c|)^{-1} \int_0^\infty  \Big( \mu_s^2e^{-\theta_0\mu_s |x-y|} + m_f^2 e^{- \theta_0 m_f|x-y|} \Big) f(x)\; dx 
\\
& + C_0  (1+|\Im c|)^{-1} \int_0^\infty  \Big( \mu_s^2e^{-\theta_0\mu_s |x+y|} + m_f^2 e^{- \theta_0 m_f|x+y|} \Big) f(x)\; dx 
\\
&\le C_0( \mu_s  + m_f) (1+|\Im c|)^{-1} \| f\|_{0,0} 
.\end{aligned}$$
As $\sqrt{| \eps |} |\alpha| \le 1$, $\sqrt{| \eps |} \mu_s \le C$ and
$$
\sqrt{| \eps | } \mu_f = | \sqrt{ \eps \alpha^2 + U - c} | \le C (1 + | \Im c |),
$$ 
which yields the estimate for $\sqrt \epsilon \partial_y^3\phi$.
For $\epsilon \partial_y^4 \phi$, we directly use the Orr-Sommerfeld equation $Orr_{\alpha,c}(\phi) = f$. 
\end{proof}

\begin{proposition}\label{prop-OS2} (Boundary layer norms)\\
Let $\phi$ solve the Orr-Sommerfeld problem \eqref{OS1}-\eqref{OS2}, with source $f$ having a boundary layer behavior.
For $| \epsilon \alpha^3 \log \nu | \le 1$, $| \alpha \Im c | > 3 \Re \lambda_0 / 2$ and $| \Im c|$ bounded away from $0$, there holds
\beq \label{propOS2bis}
 (1 + | \Im c |) \Big( \|  \nabla_\alpha \phi \|_{0,0}  + \| \nabla_\alpha^2 \phi\|_{0,\delta}\Big)
+ | \epsilon | \ \|\nabla_\alpha^4 \phi\|_{0,\delta} \le C_0 \|f\|_{0,\delta}.
\eeq
\end{proposition}
\begin{proof} 
Let $\chi(y)$ be a non negative function which equals $1$ for $0 \le y \le 1$ and $0$ for $y > 1$. Let us split the forcing
term $f$ in its boundary layer term and in its "inner term"
$$
f = f_b + f_i
$$
with 
$$
f_b(y)=   \chi\Bigl( { y \over \delta \log \delta^{-1}} \Bigr) f(y) .
$$
Note that $\| f_i \|_{0,0} \le C \| f \|_{0,\delta}$ and 
$$
| f_b (x) | \le C \| f \|_{0,\delta} \delta^{-1} e^{- y / \delta} .
$$
Let $\phi_b$ and $\phi_i$ be the solutions of $Orr(\phi_b) = f_b$ and $Orr(\phi_i) = f_i$.
Note that $\phi_i$ satisfies (\ref{propOS2bis}), thanks to the previous Proposition. 

It remains to bound $\phi_b$. For this
we split the Green function in its fast part $G_f$ and its slow part $G_s$.
For the fast part we have to bound $G_f \star f_b$, which is a convolution between an exponentially decreasing
kernel and a exponentially decreasing source. It is therefore bounded by
$C \| f \|_{0,\delta} \delta^{-1} e^{- y / \delta}$ provided $m_f > 2 \delta^{-1}$, which is the case provided $\gamma_0$ is large enough.

Let us turn to the slow part $G_s$. Let us first assume that $G_s(0,y) = 0$ for any positive $y$.
Then $\partial_y^2 G_s(0,y) = 0$ for any positive $y$. As
$$
| \partial_y^2 \partial_x G_s(x,y) | \le C {\mu_s^2 \over (1 + | \Im c |)}
$$ 
we have 
$$
| \partial_y^2 G_s(x,y) |\le {C \mu_s^2 x  \over 1 + | \Im c |} ,
$$  
noting the $x$ factor on the right hand side. By convolution between $G_s$ and $f_b$ we have
$$
|\partial_y^2 \phi_b(y) | \le C {| \delta \alpha^2 | \over 1 + | \Im c|}  \| f \|_{0,\delta},
$$
which leads to the desired bound, taking into account that $| \delta \alpha^2 | \le C$, provided that $G_s(0,y) = 0$ for any positive $y$. 

However, it is not the case that $G_s(0,y)=0$ for $y>0$, but we rather have 
$$
G(0,y) = G_f(0,y) + G_s(0,y) =0.
$$ 
Therefore
$\partial_y^2 G_s(0,y) = - \partial_y^2 G_f(0,y)$. For $y \ge \frac{1}{\theta_0 m_f} \log m_f$, we get
$$
|\partial_y^2 G_s(0,y) | \le {C_0  m_f \over  1 + | \Im c | } e^{- \theta_0 m_f y}  \le {C_0  \over  1 + | \Im c | } .
$$
On the other hand, for $y \le  \frac{1}{\theta_0 m_f} \log m_f$, we use
$$
|\partial_y^3 G_s(0,y) | \le {C_0 \mu_s^2 \over 1 + | \Im c | } 
$$
to get
$$
|\partial_y^2 G_s(0,y) | \le {C_0 \over 1 + | \Im c |  } 
+  {C_0 m_f^{-1} \mu_s^2 \log m_f \over 1 + | \Im c | } 
$$
which is bounded by a constant divided by $(1 + | \Im c |)$, upon recalling $m_f > 2 \delta^{-1}$ and using the assumption $\alpha^2 \nu \log \frac1\nu \le 1$. This ends the proof of the bound on $\partial_y^2 \phi_b$.
The bounds on $\phi_b$ and $\partial_y \phi_b$ are similar.
\end{proof}


\subsection{Generator functions}


In this section, we study the generator of solutions to the Orr-Sommerfeld problem.

\begin{proposition}\label{prop-OS3} Let $\phi$ solve the Orr-Sommerfeld problem \eqref{OS1}-\eqref{OS2}, 
with source term $f$.
For $| \epsilon \alpha^3 \log \nu | \le  1$, $| \alpha \Im c | > 3 \Re \lambda_0 / 2$ and for $| \Im c|$ bounded away from $0$, 
there are positive constants $C_0, \theta_0$ (independent
on $\epsilon$ and $\alpha$) so that 
\beq \label{prop-OS3-1}
Gen_{0,\alpha}(\nabla_\alpha \phi) + Gen_{\delta,\alpha}(\nabla^2_\alpha \phi) \le  \frac{C_0}{1+|\Im c|} Gen_{\delta,\alpha} (f),
\eeq
for all $z_2$ so that $0 \le z_2  \le  \theta_0$. Moreover, provided $f_\alpha = 0$ if $| \epsilon \alpha^3  \log \nu | \ge 1$,
\beq \label{prop-OS3-2}
Gen_0(u) + Gen_\delta(\omega)  \le  \frac{C_0}{1+|\Im c|} Gen_\delta (f),
\eeq
\beq \label{prop-OS3-2}
\partial_{z_1} Gen_0(u) + \partial_{z_1} Gen_\delta(\omega) 
\le  \frac{C_0}{1+|\Im c|} \partial_{z_1} Gen_\delta (f),
\eeq
and
\beq \label{prop-OS3-3}
\partial_{z_2} Gen_0(u) + \partial_{z_2} Gen_\delta(\omega) 
 \le  \frac{C_0}{1+|\Im c|} \Bigl[ \partial_{z_2} Gen_\delta(f) + Gen_\delta(f) \Bigr].
\eeq
\end{proposition}
\begin{proof} We estimate each term in the generator functions. 
The term  $n=0$ is already treated in Proposition  \ref{prop-OS2}. 
For $n\ge 1$, we compute 
\begin{equation}\label{OS-yndy}
\begin{aligned}
 Orr_{\alpha,c} (\varphi^n \partial_y^n \phi) 
 &= \varphi^n \partial_y^n f 
- 3 \eps \partial_y \varphi^n \partial_y^{n+3} \phi 
 - 6 \eps \partial_y^2 \varphi^n \partial_y^{n+2} \phi 
\\&\quad 
- 3 \eps \partial_y^3 \varphi^n \partial_y^{n+1}\phi  - \epsilon \partial_y^4\varphi^n \partial_y^n \phi 
\\&\quad + 4 \eps \alpha^2 \partial_y \varphi^n \partial_y^{n+1} \phi
+ 2 \eps \alpha^2 \partial_y^2 \varphi^n \partial_y^n \phi
\\&\quad + 
(U-c)\partial_y^2 \varphi^n \partial_y^n \phi 
+ 2 (U-c)\partial_y \varphi^n \partial_y^{n+1}\phi  
\\&\quad + 
 \sum_{1\le k\le n} \frac{n!}{k! (n-k)!} \varphi^n\Big(\partial_y^k U \partial_y^{n-k}
  \Delta_\alpha \phi - \partial_y^k U'' \partial_y^{n-k} \phi \Big) .
\end{aligned}
\end{equation}
Let us estimate each term on the right. For convenience, we set 
$$ \cA_n: =\| \varphi^n \partial_y^n \nabla_\alpha \phi \|_{0,0}  + \| \varphi^n \partial_y^n \nabla_\alpha^2 \phi\|_{0,\delta} 
+ | \epsilon |  (1 + | \Im c |)^{-1} \| \varphi^n \partial_y^n\nabla_\alpha^4 \phi\|_{0,\delta} ,
$$ 
for $n\ge 0$, and $\cA_n=0$ for negative $n$. As $\varphi = y/(1+y)$, we compute 
$$
\| 3 \eps \partial_y \varphi^n \partial_y^{n+3} \phi 
 + 6 \eps \partial_y^2 \varphi^n \partial_y^{n+2} \phi 
 + 3 \eps \partial_y^3 \varphi^n \partial_y^{n+1}\phi  + \epsilon \partial_y^4\varphi^n \partial_y^n \phi \|_{0,\delta}
 $$
 $$
\le 
C_0(1 + | \Im c |) \sum_{k=1}^4 \frac{n!}{(n-k)!}\cA_{n-k}, 
$$
$$
\| 4 \eps \alpha^2 \partial_y \varphi^n \partial_y^{n+1} \phi
+ 2 \eps \alpha^2 \partial_y^2 \varphi^n \partial_y^n \phi \|_{0,\delta} \le
C \eps \alpha^2 \Bigl[ n \cA_{n-1} + n (n-1) \cA_{n-2}\Big] ,
$$
and 
$$
\begin{aligned}
\|(U-c)\partial_y \varphi^n \partial_y^{n+1}\phi \|_{0,\delta} 
& \le  C_0(1 + | \Im c |) n \cA_{n-1}
\\
\|(U-c)\partial_y^2 \varphi^n \partial_y^n \phi \|_{0,\delta}
& \le  C_0(1 + | \Im c |) \Big[ n \cA_{n-1} + n (n-1) \cA_{n-2}\Big] .
\end{aligned}$$
Finally, we treat the summation in \eqref{OS-yndy}. Set 
$$
\cB_n =  \| \partial_y^nU \|_{0,0} + \|  \partial_y^n U'' \|_{0,0} .
$$
We estimate 
$$
 \sum_{1\le k\le n} \frac{n!}{k! (n-k)!} \|\varphi^n(\partial_y^k U \partial_y^{n-k}
  \Delta_\alpha \phi - \partial_y^k U'' \partial_y^{n-k} \phi ) \|_{0,\delta}
$$
$$
\le C_0 \sum_{1\le k\le n} \frac{n!}{k! (n-k)!} \Big( \|\partial_y^k U\|_{0,0} \|\varphi^{n-k} \partial_y^{n-k}
\Delta_\alpha \phi \|_{0,\delta} 
$$
$$
\qquad  + \| \partial_y^k U'' \|_{0,0}\|\varphi^{n-k}\partial_y^{n-k} \phi  \|_{0,\delta}\Big)
$$
$$
\le C_0 \sum_{1\le k\le n} \frac{n!}{k! (n-k)!} \cB_k \cA_{n-k} .
$$  
Thus, applying Proposition \ref{prop-OS2} to \eqref{OS-yndy}, we obtain, using $| \eps \alpha^2 | \le 1$,  
$$ 
\begin{aligned}
&\|  \nabla_\alpha (\varphi^n \partial_y^n \phi) \|_{0,0}  + \| \nabla_\alpha^2 (\varphi^n \partial_y^n \phi)\|_{0,\delta}
+ | \epsilon | (1 + | \Im c |) ^{-1}\|\nabla_\alpha^4  (\varphi^n \partial_y^n \phi)\|_{0,\delta} 
\\&\le \frac{C_0 }{1 + | \Im c |}\|\varphi^n \partial_y^nf\|_{0,\delta} + C_0 \sum_{k=1}^4 \frac{n!}{(n-k)!}\cA_{n-k} 
+ C_0 \sum_{1\le k\le n} \frac{n!}{k! (n-k)!} \cB_k \cA_{n-k} .
\end{aligned}$$
Expanding the left hand side, we thus have 
$$
\cA_n \le  \frac{C_0 }{1 + | \Im c |}\|\varphi^n \partial_y^nf\|_{0,\delta} + C_0 \sum_{k=1}^4 \frac{n!}{(n-k)!}\cA_{n-k} +
C_0 \sum_{1\le k\le n} \frac{n!}{k! (n-k)!} \cB_k \cA_{n-k} 
$$
for all $n\ge 0$. Multiplying the above equation by $z_2^n / n!$ and summing up the result in $n\ge 0$, we obtain 
\begin{equation}\label{bd-Ann} \begin{aligned}
 \sum_{n\ge 0}\cA_n \frac{z_2^n}{n!}
&\le  \frac{C_0}{1+|\Im c|} Gen_{\delta,\alpha}(f) + C_0  \sum_{n\ge 0}\sum_{k=1}^4 \frac{n!}{(n-k)!}\cA_{n-k}  \frac{z_2^n}{n!}
\\
&\quad + 
C_0  \sum_{n\ge 0} \sum_{k=1}^{n} \frac{n!}{k! (n-k)!} \cB_k \cA_{n-k} \frac{z_2^n}{n!}
.\end{aligned}\end{equation}
Since $|z_2| \le 1$, we compute 
$$ \sum_{n\ge 0}\sum_{k=1}^4 \frac{n!}{(n-k)!}\cA_{n-k}  \frac{z_2^n}{n!} \le 4 z_2  \sum_{n\ge 0}\cA_n \frac{z_2^n}{n!},$$
which can be absorbed into the left hand side of \eqref{bd-Ann}, for sufficiently small $z_2$. Similarly, 
$$
\sum_{n\ge 0} \sum_{k=1}^{n} \frac{n!}{k! (n-k)!} \cB_k \cA_{n-k} \frac{z_2^n}{n!} \le 
C_0   \sum_{n\ge 0}\cA_n \frac{z_2^n}{n!} \sum_{n\ge 1}\cB_{n} \frac{z_2^n}{n!}
$$
which is again absorbed into the left of \eqref{bd-Ann}, upon using the assumption that $U$ is analytic, and
that the sum of $\cB_n$ begins on $n\ge 1$. This ends the proof of (\ref{prop-OS3-1}).

Summing (\ref{prop-OS3-1}) gives (\ref{prop-OS3-2}). Multiplying by $| \alpha |$ before summing
(\ref{prop-OS3-1}) we get (\ref{prop-OS3-3}). Now multiplying by $z_1^{n-1} / (n-1)!$ instead of 
$z_1^n \over n!$ we get
$$
\begin{aligned}
 \sum_{n\ge 1}\cA_n \frac{z_2^{n-1}}{(n-1)!}
&\le  \frac{C_0}{1+|\Im c|} \partial_{z_2} Gen_{\delta,\alpha}(f) 
+ C_0  \sum_{n\ge 1}\sum_{k=1}^4 \frac{n!}{(n-k)!}\cA_{n-k}  \frac{z_2^{n-1}}{(n-1)!}
\\
&\quad + 
C_0  \sum_{n\ge 1} \sum_{k=1}^{n} \frac{n!}{k! (n-k)!} \cB_k \cA_{n-k} \frac{z_2^n}{n!}
.\end{aligned}
$$
If $z_2$ is small enough, the right hand side may be absorbed in the left hand side, excepted 
the terms involving $\cA_0$, which are bounded by $Gen_\delta(f)$. This ends the proof of the Proposition.
\end{proof}


\subsection{The $\alpha=0$ case}\label{sec-azero}
In this section, we treat the case when $\alpha=0$. In this case, the resolvent equation of the linearized Navier-Stokes equations simply becomes the resolvent equation for the heat equation  
\begin{equation}
\label{eqs-azero}
(\lambda -\nu \partial_y^2 ) u_0 = F_0,
\end{equation} 
with $u_0=0$ at $y=0$, where $u_0$ denotes the Fourier mode of the first component of velocity $u$ at $\alpha =0$. We have the following simple proposition. 

\begin{proposition}\label{prop-azero} Let $u_0$ solve \eqref{eqs-azero} and let $\omega_0 = \partial_y u_0$ be the corresponding vorticity. For $\Re \lambda > 3 \Re \lambda_0 / 2$, we have  
	\beq \label{prop-azero1}
	\begin{aligned}
	Gen_{0,0}(u_0) &\le  \frac{C_0}{1+\Re \lambda} Gen_{0,0} (F_0), \\
	Gen_{\delta,0}(\omega_0) &\le  \frac{C_0}{1+\Re \lambda} \Big( Gen_{0,0} (F_0) + Gen_{\delta,0} (\partial_y F_0) \Big) 
	,	\end{aligned}
	\eeq
	for all $z_2\ge 0$.\end{proposition}
\begin{proof} The solution $u_0$ to \eqref{eqs-azero} with the zero boundary condition satisfies 
	$$ u_0(y) = \int_0^\infty G_0(y,z) F_0(z) \; dz$$
	where $G_0(y,z)$ denotes the Green function for $(\lambda - \nu \partial_y^2)$ with the Dirichlet boundary condition. In particular, we have 
	$$ |G_0(y,z)| \le \nu^{-1/2}|\lambda|^{-1/2} e^{- \nu^{-1/2} \Re \sqrt \lambda |y-z|} .$$  
In particular, for $\Re \lambda > 3\Re \lambda_0/2$, we estimate 
	$$\begin{aligned}
	 | u_0(y)| 
	 &\le \int_0^\infty \nu^{-1/2}|\lambda|^{-1/2} e^{- \nu^{-1/2} \Re \sqrt \lambda |y-z|} |F_0(z)| \; dz
	 \\
	 &\le C_0 (1+\Re \lambda)^{-1} \sup_y |F_0(y)| .
	\end{aligned}
	$$
The estimates for derivatives are obtained in the same way as done in the previous section for the Orr-Sommerfeld equations. The proposition follows.   
\end{proof}

\section{Construction of the instability}



\subsection{Iterative construction}


Let us now describe the iterative construction of $u^n$ and $\omega^n$ 
and of the infinite series which defines  the solution \eqref{utrue}. 
We start with the most unstable eigenmode $(\psi_0,\alpha_0,c_0)$ to the Orr-Sommerfeld problem; namely, we start with a solution of 
$$
Orr_{\alpha_0,c_0} (\psi_0) = 0,
$$
with the zero boundary conditions on $y=0$, such that $\alpha_0 \Im c_0$ is maximum. In fact we just need to start
from a mode such that $\alpha_0 \Im c_0$ is strictly larger than half of this maximum. 
Up to a change of sign we may assume
that $\alpha_0 > 0$. Up to a rescaling we may also assume that $\alpha_0 = 1$.

This mode corresponds to a complex solution $\nabla^\perp(e^{i \alpha_0 (x - c_0t)} \psi_0(y))$ of the linearized Navier Stokes equations.
We have to take the real part of this solution in order to deal with real valued solutions. Note that
$(\bar \psi_0,-\alpha_0,\bar c_0)$ is also an eigenmode. We therefore sum up the two unstable eigenmodes
corresponding to $\alpha_0$ and $- \alpha_0$ and define 
$$
\psi^1(t,x,y) =  e^{i \alpha_0 ( x - \Re c_0  t) + \alpha_0 \Im c_0 t} \psi_0(y) 
+  e^{ - i \alpha_0 ( x - \Re c_0  t) + \alpha_0 \Im c_0 t} \bar \psi_0(y)  .
$$
Let $\psi^1_\alpha$ be the Fourier transform of $\psi^1$ in $x$ variable. 
Then all the $\psi^1_\alpha$ vanish, except two of them, namely $\alpha = \pm \alpha_0 = \pm 1$. 
We then iteratively solve the resolvent equation of the linearized Navier-Stokes problem \eqref{eqs-un} for $u^n = \nabla^\perp \psi^n$. In term of vorticity $\omega^n = \Delta \psi^n$, the problem reads
\begin{equation}\label{eqs-wn}
\begin{aligned}
\partial_t \omega^n + U \partial_x \omega^n - U'' \partial_x \psi^n -\nu \Delta \omega^n &= - \sum_{1 \le j \le n-1} (u^j \cdot \nabla ) \omega^{n-j},
\end{aligned} \end{equation}
together with the zero boundary condition on $u^n = \nabla^\perp \psi^n$. Precisely, we search for $\psi^n$ 
under the form
\beq \label{firstpsin}
\psi^n = \sum_{|\alpha| \le n} \psi_\alpha^n e^{i \alpha ( x - \Re c_0 t)} e^{n \Im c_0 t},
\eeq
where the sum runs on all the $\alpha$ which are multiples of $\alpha_0$. This yields, for $n\ge 2$,
\begin{equation}\label{OS-un}
Orr_{\alpha,c} (\psi^n_\alpha) = {1 \over i \alpha} \sum_{\alpha'} 
\sum_{1 \le j \le n-1} (u^j_{\alpha - \alpha'} \cdot \nabla_{\alpha'})  \omega_{\alpha'}^{n-j},
\end{equation}
in which $u^j_\alpha = \nabla_\alpha^\perp \psi^j_\alpha$ and $\omega^j_\alpha = \Delta_\alpha \psi^j_\alpha$, 
together with the zero boundary conditions on $\psi^n$ and $\partial_y \psi^n$, and in which
$$
c = \Re c_0  + i n {\alpha_0 \over \alpha} \Im c_0 .
$$
Note that a $\alpha^{-1}$ factor appears in front
of the source term, since Orr Sommerfeld is obtained by taking the vorticity of Navier Stokes equations and dividing by
$\alpha$.
Note also that all but a finite number of $\psi^n_\alpha$ vanish.
Again the sum runs on all the $\alpha$ which are multiple of $\alpha_0$. Note also that
$| \Im c | \ge | \Im c_0|$ and is thus bounded away from $0$.

As Proposition \ref{prop-OS3} only holds for $| \alpha^3 \epsilon |\le 1$ or equivalently
$|\alpha| \le \nu^{-1/2}$, we will only retain the $| \alpha | \le \nu^{-1/2}$ in the construction
of $\psi^n$ and restrict (\ref{firstpsin}) to 
$$
\psi^n = \sum_{| \alpha | \le \nu^{-1/2}} \psi_\alpha^n e^{i \alpha (x - \Re c_0 t)} e^{n \Im c_0 t} .
$$
This leads to the introduction of the force
$$
f^n = \sum_{| \alpha | > \nu^{-1/2}} \sum_{\alpha'} \sum_{1 \le j \le n-1}
(u^j_{\alpha - \alpha'} . \nabla_{\alpha'} ) \omega_{\alpha'}^{n-j} 
$$
that will be estimated below.


\subsection{Bounds on $\psi^n$}


 We prove the following. 

\begin{proposition}
Introduce the iterative norm 
\begin{equation}\label{def-Gn}
G^n =  Gen_0 (u^n) + Gen_0 (v^n) + \partial_{z_1} Gen_0(u^n) 
\eeq
$$
\qquad \qquad + Gen_\delta(\omega^n) + \partial_{z_1} Gen_\delta(\omega^n) + \partial_{z_2} Gen_\delta(\omega^n),
$$
for $n \ge 1$. Then, $G^n(z_1,z_2)$ are well-defined for sufficiently small $z_1,z_2$, and in addition, 
there exists some universal constant $C_0$ so that 
\beq \label{boundGn}
G^n \le {C \over n} \sum_{1 \le j \le n-1} \Big( G^j \partial_{z_1} G^{n-j} + G^j \partial_{z_2} G^{n-j} \Big) .
\eeq
\end{proposition}
Note that the derivatives appearing in (\ref{def-Gn}) are non negative.

\begin{proof} Applying Proposition \ref{prop-OS3} to the Orr-Sommerfeld equation \eqref{OS-un}, 
using $|\alpha | \le \nu^{-1/2}$,  and summing over $\alpha$, we get 
\begin{equation}\label{iter-un}
{\cal A} (\omega^n) \le {C_0 \over n} {\cal A} \Bigl( \sum_{1 \le j \le n-1}  (u^j \cdot \nabla) \omega^{n-j}  \Bigr) .
\end{equation}
Moreover, using Proposition \ref{prop-elliptic}
$$
Gen_0(u^n) + Gen_0(v^n) \le C Gen_\delta(\omega^n)
$$
and
$$
\partial_{z_1} Gen_0(u^n) \le C \partial_{z_1} Gen_\delta(\omega^n).
$$
Proposition \ref{prop-Gendy} then gives the desired bound.
\end{proof}


\subsection{Bounds on the generator function}


\begin{theo} For $n\ge 1$, let $G^n(z_1,z_2)$ be defined as in \eqref{def-Gn}. Then, the series
$$
G(\tau,z_1,z_2) = \sum_{n=1}^{+\infty} G^n(z_1,z_2) \tau^{n-1}  
$$
converges, for sufficiently small $\tau$, $z_1$, and $z_2$. 
\end{theo}
\begin{proof}
For $N\ge 1$, let us introduce the partial sum 
$$
G_N(\tau,z_1,z_2)  := \sum_{n=1}^N G^n(z_1,z_2) \tau^{n-1} ,
$$
for $\tau,z_1,z_2\ge 0$. 
Note that $G_N$ is a polynomial in $\tau$, and thus well-defined for all times $\tau\ge 0$. 
We also note that all the coefficients $G^n(z_1,z_2)$ are positive. 
In particular, $G_N(\tau, z_1,z_2)$ is positive, and so are all its time derivatives (when $z_1 > 0$ and $z_2 > 0$). 
Moreover, $G_N(\tau,z_1,z_2)$, and all its derivatives, are increasing in $N$. We also observe that, at $\tau =0$,
$$
G_N(0,z_1,z_2) = G^1(z_1,z_2),
$$
for all $N\ge 1$, and hence, 
$$
G(0,z_1,z_2) = \lim_{N\to \infty} G_N(0,z_1,z_2) = G^1(z_1,z_2).
$$ 
Next, multiplying (\ref{boundGn}) by $\tau^{n-2}$ and summing up the result, we obtain 
the following partial differential inequality
$$
\partial_\tau G_N \le C G_N \partial_{z_1} G_N + C G_N \partial_{z_2} G_N ,
$$
for all $N\ge 1$. Therefore the generator function satisfies an Hopf-type equation, or more precisely an Hopf inequality.

As $G_N$ is increasing in $z_1$ and $z_2$, we focus on the diagonal $z_1 = z_2$, and introduce
$$
F_N(\tau,z) = G_N(\tau,\theta(\tau) z, \theta(\tau) z) 
$$
for $\tau, z\ge 0$, where $\theta(\cdot)$ will be chosen later, with $\theta(0)=1$.
It follows that $F_N$ satisfies
\begin{equation}\label{ode-FN}
\partial_\tau F_N \le (2 C F_N  + \theta'(\tau) z) \partial_z F_N .
\end{equation}
Note that $F_N$ is increasing in $N$. At $\tau= 0$, $F_N(0,z) = G_N(0,z,z) = G^1(z,z)$, which is independent on $N$.
Let  $\rho > 0$ be small enough such that
$$
M_0 = \sup_{0 \le z \le \rho} G^1(z,z)
$$
is well defined. 
We now define $\theta(\tau)$ in such a way that 
$$
4 C M_0 + \theta'(\tau) \rho < 0,
$$
with $\theta(0)=1$. For instance, we take $$
\theta(\tau) = 1 - 6 C M_0 \rho^{-1} \tau .
$$
We will work on a time interval where $\theta(\tau) \ge 1/2$, namely on $[0,T_0]$ where $T_0 = \rho / 12 C M_0$.
Let $T_N$ be the largest time $\le T_0$ such that $F_N(\tau ,z)  \le 2 M_0$, for $0 \le \tau \le T_N$ and $0 \le z \le \rho$.
Note that $T_N$ exists and is strictly positive, since $F_N$ is well defined for all the positive times, and continuous in time. It remains to prove that $\inf_{N\ge 1}T_N$ is positive. 

Let us define the characteristics curves $X_N(\tau ,z)$ by solving
$$
\partial_\tau X_N(\tau ,z) = - 2 C F_N(\tau ,X_N(\tau ,z)) - \theta'(\tau) X_N(\tau ,z),
$$
together with $
X_N(0,z) = z
$. Observe that the characteristics are outgoing at $z = 0$ and $z = \rho$. Therefore the characteristics completely fill 
$[0,T_N] \times [0,\rho]$. Let us now introduce 
$$
\widetilde F_N(\tau ,z) = F_N(\tau ,X_N(\tau ,z)) .
$$
It follows from \eqref{ode-FN} that 
$$
\partial_\tau \widetilde F_N(\tau ,z) = \partial_\tau F_N + \partial_\tau X_N \partial_z F_N \le 0.
$$
As a consequence, 
$$
\sup_{0 \le z \le \rho} F_N(\tau ,z) \le \sup_{0 \le z \le \rho} F_N(0,z)  = \sup_{0 \le z \le \rho} G^1(z,z) = M_0.
$$
Therefore $T_N \ge T_0$, for all $N\ge 1$, and $F_N$ is bounded uniformly on $[0,T_0] \times [0,\rho]$. We can therefore take the limit $N \to + \infty$. This leads to the 
convergence of $G_N(\tau,z_1,z_1)$ as $N \to \infty$ for $0 \le \tau \le T_0$ and for $0 \le z_1 \le \rho$. Since $G_N(\tau,z_1,z_2)$ is increasing in $z_1,z_2$, the convergence of $G_N(\tau,z_1,z_2)$ as $N\to \infty$ follows.
\end{proof}


\subsection{End of proof}


It remains to bound the force term $f^n$.
For this we note that the cut off occurs for $| \alpha | \ge \nu^{-1/2}$ where the corresponding modes are exponentially
small. The force term is therefore exponentially small itself, and therefore arbitrary small in any Sobolev space.

Now $\sum_n \omega^n \tau^n$ converges for $\tau$ small enough. Let 
$$
\tau = \nu^N e^{\alpha_0 \Im c_0 t} .
$$
Then as long as $\tau$ remains small, namely as long as $t$ remains small than 
$C \log \nu^{-1}$ for some constant $C$, this series converges and defines a solution of the full incompressible
Navier Stokes equations. Note that the series defining $u$ and $v$ also converges in the same way.
This prove Theorem \ref{theoinstable} for stationary boundary layers. 



\section{Time dependent shear flow}


We now turn to the case where the shear flow $U_s$ depends on time. Let $\Omega_s$ be the corresponding vorticity.
Note that $U_s$ is a solution of 
$$
\partial_t U_s - \nu \partial_{yy} U_s = 0
$$ 
and hence depends on $t$ through
$\nu t$. We put this dependency in the notation $U_s(\nu t,y)$.
The perturbation satisfies
\beq \label{NSnon1}
\begin{aligned}
\partial_t \omega &+ (U_s(0) . \nabla) \omega + (u \cdot \nabla) \Omega_s(0) - \nu \Delta \omega 
\\&= - (u \cdot \nabla) \omega
- Q_1(\nu t) \omega - Q_2(\nu t) u
\end{aligned}\eeq
where
$$
Q_1(\nu t) \omega = \Bigl(U_s(\nu t) - U_s(0) \Bigr) \partial_x \omega
$$
and
$$
Q_2(\nu t) u = (u . \nabla) \Bigl(\Omega_s(\nu t) - \Omega_s(0) \Bigr) .
$$
Note that
$$
U_s(\nu t,y) = U_s(0) + \sum_{k = 1}^M (\nu t)^k U_s^k(y) + O((\nu t)^{M+1}).
$$
As we are interested in times of order $\log \nu$, and keeping in mind that $\partial_x \omega$ is always bounded
by $\nu^{-1/2}$, we can put the $O()$ in the forcing term.  

We will fulfill a perturbative analysis and look for solutions of (\ref{NSnon1}) of the form
\beq \label{perturb1}
\psi(t,x,y) = \sum_{n \ge 1} \sum_{p \ge 0} \sum_{q \ge 0} \sum_{\alpha \in \zit} e^{n \Im c_0 t} t^p \nu^q 
e^{i \alpha (x - \Re c_0 t)} \psi_{\alpha}^{n,p,q}(y) .
\eeq
In fact it is sufficient to bound $q$ by some large integer $M$, since we allow a small forcing term.

Putting (\ref{perturb1}) in (\ref{NSnon1}) we get
\beq \label{perturb2}
i\alpha~ Orr_{\alpha,c} (\psi_\alpha^{n,p,q}) + (p+1)  \psi_{\alpha}^{n,p+1,q} = Q^{n,p,q} + L^{n,p,q}
\eeq
where
$$
Q^{n,p,q} = 
\sum_{\alpha'} \sum_{1 \le j \le n-1}\sum_{0 \le k \le p} \sum_{0 \le l \le q}
(u_{\alpha - \alpha'}^{j,k,l} . \nabla_{\alpha'} ) \omega_{\alpha'}^{n-j,p-k,q-l}
$$
and 
$$
L^{n,p,q} = i\alpha
 \sum_{k=1}^M U_s^k \omega_{\alpha}^{n,p-k,q-k}
- i \alpha  \sum_{k=1}^M \psi_\alpha^{n,p-k,q-k} \partial_y \Omega_s^k,
$$
with the convention that a quantity vanishes if one of its indices is negative.
Note that $Q^{n,p,q}$ only involves
$\psi_{\alpha'}^{n',p',q'}$ with $n' < n$. Next, $L^{n,p,q}$ involves $\psi_{\alpha'}^{n',p',q'}$ with $n' \le n$ and $p' < p$ and $q' < q$.

We will solve this equation by recurrence on the power of $\nu$, namely on $q$.
We begin with the leading order $q = 0$. All the $\psi_\alpha^{n,p,0}$ vanish, except when $p = 0$.
System (\ref{perturb1}) then reduces to (\ref{OS-un}) 
$$
\psi_\alpha^{n,0,0} = \psi_\alpha^n 
$$
which are constructed in the previous section, up to any arbitrarily large $n$. 

We then turn to $q = 1$.
The first terms $\psi_\alpha^{n,0,0}$ create an "error term" $L$ involving $(\nu t)^k$ for $1 \le k \le M$. 
Let us first focus on the case $k = 1$.
For $k = 1$, the corresponding $L^{n,1,1}$ term is
$$
i\alpha \Bigl( U_s^1 \omega_\alpha^{n,0,0} -\psi_\alpha^{n,0,0} \partial_y \Omega_s^1  \Bigr) . 
$$
For $n = 1$ and $\alpha = \pm 1$ we note that $Orr_{\alpha,c}$ is not invertible.
This operator may also be non invertible for other values of $\alpha$ (in finite number).
To simplify the discussion we assume that this does not occur (the general case is similar). 

If $(n,\alpha) \ne (1, \pm 1)$, we take $\psi_\alpha^{n,p,1} = 0$ for $p > 1$. This leads to
\beq \label{perturb10}
i\alpha Orr_{\alpha,c}(\psi_\alpha^{n,1,1}) = 
\sum_{\alpha'} \sum_{1 \le j \le n-1} \sum_{k=0,1}
(u^{j,k,k}_{\alpha - \alpha'} . \nabla_{\alpha'}) \omega_{\alpha'}^{n-j,1-k,1-k} + L^{n,1,1}
\eeq
which is a linearized version of (\ref{OS-un}).

For $(n,\alpha) = (1,\pm 1)$, we note that $Orr_{\alpha,c}$ is not invertible. 
Let $A$ be the right hand side of (\ref{perturb10}). Thanks to Theorem \ref{theopseudo} 
we define 
$$
\psi_{\pm 1}^{1,1,1} = Orr^{-1} (A),
$$ 
and we use $\psi_{\pm 1}^{1,2,1}$ to handle the remainder
$$
2 \omega_{\pm 1}^{1,2,1} =  l^{\nu} (A) \phi_{1,c_0}.
$$
Now to bound $\psi_\alpha^{n,1,1}$ we introduce the corresponding generator function $G^1$ and proceed as
in the previous section. This leads to the following inequality
$$
\partial_t G^1 \le G^0 \partial_{z_1} G^1 + G^0 \partial_{z_2} G^1 
+ G^1 \partial_{z_1} G^0 + G^1 \partial_{z_2} G^0 + C \partial_x G^0 + C G^0 .
$$
Using the same arguments as in the previous section, we obtain bound uniform bounds on $G^1$.
The recurrence can be continued using similar arguments. This ends the proof of Theorem \ref{theoinstable}.

\chapter{The stable case}\label{chapter-nonlinear-stable}

\def\app{\mathrm{app}}


\section{Introduction}


In this chapter, we prove the nonlinear instability of monotonic shear profiles. 
Roughly speaking, the two main results are
\begin{itemize} 

\item in the case of time-dependent shear flows, we construct Navier-Stokes solutions, with arbitrarily small forcing, of order $\mathcal{O}(\nu^P)$, with $P$ as large as we want,
so that the shear flows are nonlinear unstable up to order $\nu^{1/2+\epsilon}$ in $L^\infty$ and $L^2$, with $\epsilon$ being arbitrarily small. 

\item in the case of stationary shear flows, we construct Navier-Stokes solutions, without  forcing term, 
so that the shear flows are nonlinear unstable up to order $\nu^{5/4}$ in $L^\infty$ and $L^2$. 

\end{itemize}

%
%
%
%


\subsection{Stable shear profiles}\label{sec-BL}



%

We consider {\em stable shear profiles}; namely, those that are spectrally stable to the Euler equations. This includes, for instance, boundary layer profiles without an inflection point 
by view of the classical Rayleigh's inflection point theorem. We also assume that $U(z)$ is strictly monotonic, real analytic, that $U(0) =0$ 
and that $U(z)$ converges exponentially fast at infinity to a finite constant $U_+$. 
By a slight abuse of language, such profiles will be referred to
as stable profiles throughout this book.

In order to study the instability of such shear layers, we first analyze the spectrum of the corresponding linearized problem 
around initial profiles $U(z)$. This leads to the following linearized problem for vorticity $\omega = \partial_z v_1 - \partial_x v_2$, which reads 
\begin{equation}\label{EE-vort} 
(\partial_t - L)\omega = 0, \qquad L\omega : = \nu \Delta \omega - U\partial_x \omega - v_2 U'' ,
\end{equation}
together with $v = \nabla^\perp \phi$ and $\Delta \phi = \omega$, satisfying the no-slip boundary conditions 
$\phi = \partial_z \phi =0$ on $\{z=0\}$. 

We then take the the Fourier transform in the $x$ variable only, denoting by $\alpha$ the corresponding wavenumber,
which leads to
\begin{equation}\label{EE-vort-a}
 (\partial_t - L_\alpha)\omega = 0, \qquad L_\alpha\omega : = \nu \Delta_\alpha \omega - i\alpha U \omega - i\alpha \phi U'' ,
 \end{equation}
where
$$
\omega_\alpha = \Delta_\alpha \phi_\alpha,
$$
together with the zero boundary conditions 
$\phi_\alpha = \phi'_\alpha =0$ on $z=0$. Here, 
$$
\Delta_\alpha = \partial_z^2 - \alpha^2. 
$$
Together with Y. Guo, we proved in \cite{GGN1,GGN3} that, even for profiles $U$ which are stable as $\nu = 0$, 
there are unstable eigenvalues to the Navier-Stokes problem \eqref{EE-vort-a} for sufficiently small viscosity $\nu$ 
and for a range of wavenumber $\alpha \in [\alpha_1,\alpha_2]$, with $\alpha_1\sim \nu^{1/4}$ and $\alpha_2\sim \nu^{1/6}$. 
The unstable eigenvalues $\lambda_*$ of $L_\alpha$, found in \cite{GGN3}, satisfy 
\begin{equation}\label{lambda-GGN} 
\Re \lambda_* \sim \nu^{1/2} .
\end{equation}
Such an instability was first observed by Heisenberg \cite{Hei,HeiICM}, then Tollmien and C. C. Lin \cite{Lin0,LinBook}; 
see also Drazin and Reid \cite{Reid,Schlichting} for a complete account of the physical literature on the subject. 
See also Theorem \ref{theo-spectralinstablity} below for precise details. In coherence with the physical literature \cite{Reid},
we believe that, $\alpha$ being fixed,
 this eigenvalue is the most unstable one. However, this point is an open question from the mathematical point of view.
 
Next, we observe that $L_\alpha$ is a compact perturbation of the Laplacian $ \nu \Delta_\alpha$, and hence its unstable spectrum in the usual $L^2$ space is discrete. Thus, for each $\alpha,\nu$, we can define the maximal unstable eigenvalue $\lambda_{\alpha,\nu}$ so that $\Re \lambda_{\alpha,\nu}$ is maximum. We set $\lambda_{\alpha,\nu} =0$, if no unstable eigenvalues exist. 

In this paper, we assume that the unstable eigenvalues found in the spectral instability result, Theorem \ref{theo-spectralinstablity}, 
are maximal eigenvalues. Precisely, we introduce 
\begin{equation}\label{def-ga0max0}
\gamma_0 : = \lim_{\nu \to 0} \sup_{\alpha \in \RR}  \nu^{-1/2}\Re \lambda_{\alpha,\nu}  .
\end{equation}
The existence of unstable eigenvalues in Theorem \ref{theo-spectralinstablity} implies that $\gamma_0$ is positive. 
Our spectral assumption is that $\gamma_0$ is finite (that is, the eigenvalues in Theorem \ref{theo-spectralinstablity} are maximal).


\subsection{Main results}


We are ready to state the main instability results for stable shear layers.

\begin{theorem}\label{theo-approx} 
Let $U_\bl(t,z)$ be a time-dependent stable shear layer profile as described in Section \ref{sec-BL}. 
Then, for arbitrarily large $s,N$ and arbitrarily small positive $\epsilon$, there exists
a sequence of functions $u^\nu$ that solves the Navier-Stokes equations, with some forcing $f^\nu$, so that 
$$
\| u^\nu(0) - U_\bl(0) \|_{H^s} + \sup_{t\in [0,T^\nu]} \| f^\nu (t)\|_{H^s} \le \nu^N,
$$
but 
$$
\| u^\nu(T^\nu) - U_\bl(T^\nu) \|_{L^\infty} \gtrsim \nu^{\frac12+\epsilon},
$$
$$
\| u^\nu(T^\nu) - U_\bl(T^\nu) \|_{L^2} \gtrsim \nu^{\frac12+\epsilon},
$$
$$\| \omega^\nu(T^\nu) - \omega_\bl(T^\nu) \|_{L^\infty} \to \infty,$$
for time sequences $T^\nu$ of order $\nu^{-1/2}\log \nu^{-1}$. Here, $\omega^\nu = \nabla\times u^\nu$ denotes the vorticity of fluids. 
\end{theorem}


\begin{theorem}[Instability result for stable profiles] \label{firstinstability-stable}
Let $U_\bl = U(z)$ be a stable stationary shear layer profile as described in Section \ref{sec-BL}. 
Then, for any $s,N$ arbitrarily large, there exists a sequence of solutions $u^\nu$ to the Navier-Stokes equations, 
with forcing $f^\nu = f^P$ (boundary layer forcing), so that $u^\nu$ satisfy 
$$
\| u^\nu(0) - U_\bl \|_{H^s} \le \nu^N,
$$
but 
$$\| u^\nu (T^\nu) - U_\bl\|_{L^\infty} \gtrsim \nu^{5/4},$$
$$\| u^\nu (T^\nu) - U_\bl\|_{L^2} \gtrsim \nu^{5/4},$$
$$\| \omega^\nu (T^\nu) - \omega_\bl \|_{L^\infty} \gtrsim 1,$$
for some time sequences  $T^\nu$ of order $\nu^{-1/2}\log \nu^{-1}$. 
\end{theorem}

The spectral instability for stable profiles gives rise to sublayers (or critical layers) whose thickness is of order $\nu^{3/4}$. 
The velocity gradient in this sublayer grows like $\nu^{-3/4} e^{t/\nu^{1/2} }$, and becomes large when $t$ is of order $T^\nu$.
As a consequence, they may in turn become unstable after the instability time $T^\nu$ obtained in the above theorem. 
Thus, in order to improve the $\nu^{5/4}$ instability, one needs to further examine the stability properties 
of this sublayer itself (see \cite{GrN4}).


\subsection{Boundary layer norms}


We end the introduction by introducing the boundary layer norms to be used throughout the chapter. Precisely, for each vorticity function $\omega_\alpha = \omega_\alpha(z)$, we introduce 
the following boundary layer norms
\begin{equation}\label{assmp-wbl-stable} 
\| \omega_\alpha\|_{ \beta, \gamma, 1} : = \sup_{z\ge 0} 
\Big [ \Bigl( 1 +  \delta^{-1} \phi_{p} (\delta^{-1} z)  \Bigr)^{-1} e^{\beta z} |\omega_\alpha (z)| \Big],
\end{equation}
where $\beta > 0$, $p$ is a large, fixed number, 
$$
\phi_p(z) = {1 \over 1 + z^p},
$$
and with the boundary layer thickness 
$$
\delta = \gamma \nu^{1/4}
$$
for some $\gamma>0$. We introduce the boundary layer space ${\cal B}^{\beta,\gamma,1}$ 
  to consist of functions whose $\|\cdot \|_{ \beta, \gamma, 1} $ norm is finite. 
  We also denote by $L^\infty_\beta$ the function spaces equipped with the finite norm 
  $$
  \|\omega \|_\beta = \sup_{z\ge 0} e^{\beta z} |\omega(z)|.
  $$ 
  When there is no weight $e^{\beta z}$, we simply write $L^\infty$ for the usual bounded function spaces. 
  Clearly, 
  $$
  L^\infty_\beta \subset {\cal B}^{\beta,\gamma,1}
.  $$ 
 In addition, it is straightforward to check that 
\begin{equation}\label{al-norm}
\| f g \|_{\beta,\gamma, 1} \le \| f \|_{L^\infty} \| g \|_{ \beta,\gamma,1}.
\end{equation}
Finally, for functions $\omega(x,z)$, we introduce 
$$\| \omega\|_{ \sigma,\beta, \gamma, 1} : = \sup_{\alpha \in \RR} (1+|\alpha|)^{\sigma}\| \omega_\alpha\|_{ \beta, \gamma, 1} ,$$
for $\sigma> 1$, in which $\omega_\alpha$ is the Fourier transform of $\omega$ in the variable $x$. Combining with \eqref{al-norm}, we have 
\begin{equation}\label{al-norm1}
\| f g \|_{\sigma,\beta,\gamma, 1} \le \| f \|_{\sigma,0} \| g \|_{\sigma,\beta,\gamma,1},
\end{equation}
where $\| f \|_{\sigma,0} = \sup_{\alpha \in \RR} (1+|\alpha|)^{\sigma}\| f_\alpha\|_{L^\infty}$. 


\section{Linear instability}\label{sec-linear}


In this section, we shall recall the spectral instability of stable boundary layer profiles \cite{GGN3} and the semigroup estimates 
on the corresponding linearized Navier-Stokes equation \cite{GrN2,GrN3}.
 

\subsection{Spectral instability}\label{sec-grmode}


The following theorem, proved in \cite{GGN3}, provides an unstable eigenvalue of $L$ for generic shear flows.  

\begin{theorem}[Spectral instability; \cite{GGN3}]\label{theo-spectralinstablity}
Let $U(z)$ be an arbitrary shear profile with $U(0)=0$ and $U'(0) > 0$ and satisfy 
$$
\sup_{z \ge 0} | \partial^k_z (U(z) - U_+) e^{\eta_0 z} | < + \infty, \qquad k=0,\cdots ,4,
$$ 
for some constants $U_+$ and $\eta_0 > 0$. Let $R = \nu^{-1}$ be the Reynolds number, 
and set $\alpha_\mathrm{low}(R)\sim R^{-1/4}$ and  $ \alpha_\mathrm{up}(R)\sim R^{-1/6}$ be the lower and upper stability branches. 

Then, there is a critical Reynolds number $R_c$ so that for all $R\ge R_c$ and all $\alpha \in (\alpha_\mathrm{low}(R), 
\alpha_\mathrm{up}(R))$, there exist
a nontrivial triple $c(R), \hat v(z; R), \hat p(z;R)$, with $\mathrm{Im} ~c(R) >0$, 
such that $v_R: = e^{i\alpha(x-ct) }\hat v(z;R)$ and $p_R: = e^{i\alpha(x-ct)} \hat p(z;R)$ solve the linearized Navier-Stokes problem 
\eqref{EE-vort}. Moreover there holds the following estimate for the growth rate of the unstable solutions:
$$ 
\alpha \Im c(R) \quad \approx\quad  R^{-1/2}
$$
as $R \to \infty$. 
\end{theorem}

The proof of the previous Theorem, which can be found in \cite{GGN3}, gives a detailed description of the unstable mode.
In this paper we focus on the lower branch of instability.  In this case
$$
\alpha_\nu \approx R^{-1/4} = \nu^{1/4}, \qquad \Re \lambda_\nu\approx R^{-1/2} = \nu^{1/2},
$$
The vorticity of the unstable mode is of the form 
\begin{equation}\label{w0-gr-stable} 
\omega_0 = e^{\lambda_\nu t} \Delta ( e^{i \alpha_\nu x} \phi_0(z) ) \quad + \quad \mbox{complex conjugate} 
\end{equation}
The stream function $\phi_0$ is constructed through asymptotic expansions, and is of the form 
$$
 \phi_0: =  \phi_{in,0}(z) +\delta_\bl \phi_{bl,0} ( \delta_\bl^{-1} z) 
  $$ 
for some boundary layer function $\phi_{bl,0}$, where $\delta_\bl = \nu^{1/4}$.


By construction,  derivatives of $\phi_{\bl,0}$ satisfy 
$$ 
|\partial_z^k \phi_{\bl,0}(\delta_\bl^{-1} z)| \le C_k \delta_\bl^{-k} e^{-\eta_0 z / \delta_\bl} .
$$
%
%
In addition, it is clear that each $x$-derivative of $\omega_0$ gains a small factor of $\alpha_\nu \approx \nu^{1/4}$. 
We therefore have an accurate description of the linear unstable mode.


\subsection{Linear estimates}


The corresponding semigroup $e^{Lt}$ of the linear problem \eqref{EE-vort} is constructed through the path integral 
\beq \label{int2}
e^{L t} \omega = \int_\RR e^{i\alpha x} e^{L_\alpha t} \omega_\alpha \; d\alpha 
\eeq
in which $\omega_\alpha$ is the Fourier transform of $\omega$ in tangential variables and
$L_\alpha$, defined as in \eqref{EE-vort-a}, is the Fourier transform of $L$. 
One of the main results proved in \cite{GrN3} is the following theorem.

\begin{theorem} \cite{GrN3}\label{theo-eLt-stable} 
Let $\alpha \lesssim \nu^{1/4}$.
Let $\omega_\alpha \in {\cal B}^{\beta,\gamma,1}$ for some positive $\beta$, $\gamma_0$ be defined as in \eqref{def-ga0max0}. 
Assume that $\gamma_0$ is finite. 
Then, for any $\gamma_1>\gamma_0$, there is some positive constant $C_\gamma$ so that 
$$\begin{aligned}
\| e^{L_\alpha t}\omega_\alpha\|_{ \beta, \gamma, 1} &\le C_\gamma  
e^{\gamma_1 \sqrt{\nu} t }  e^{- \frac14 \alpha^2 \nu t}  \| \omega_\alpha\|_{ \beta, \gamma, 1},
\\
\| \partial_ze^{L_\alpha t}\omega_\alpha\|_{ \beta, \gamma, 1} &\le C_\gamma \Big( \nu^{-1/4}+
 (\nu t)^{-1/2} \Big) e^{\gamma_1 \sqrt{\nu} t }  e^{- \frac14 \alpha^2 \nu t}  \| \omega_\alpha\|_{ \beta, \gamma, 1}. 
\end{aligned}$$
\end{theorem}


\section{Approximate solutions}\label{sec-approx}



Theorem \ref{theo-approx} follows at once from the following theorem. 

\begin{theorem}\label{theo-approx-stable} 
Let $U(z)$ be a stable shear layer profile, and let $U_\bl(\nu t,z)$ be the corresponding time-dependent shear profiles. 
Then, there exist an approximate solution $\widetilde u_\app$ that approximately solves Navier-Stokes equations in the following sense: 
for arbitrarily large numbers $p,M$ and for any $\epsilon>0$, the functions $\widetilde u_\app$ 
solve
\begin{equation}\label{NS-Uapp}
\begin{aligned}
\partial_t \widetilde u_\app + (\widetilde u_\app \cdot \nabla) \widetilde u_\app + \nabla \widetilde p_\app   &=
\nu \Delta \widetilde u_\app + {\cal E}_\app,
\\
\nabla \cdot \widetilde u_\app &= 0,
\end{aligned}\end{equation}
for some remainder $ {\cal E}_\app$ and for time $t \le T_\nu$, with $T_\nu$ being defined through
$$
\nu^p e^{\Re \lambda_0 T_\nu} = \nu^{\frac12 + \epsilon}.
$$
In addition, for all $t \in [0,T_\nu ]$, there hold
$$
\begin{aligned}\| \mathrm{curl} (\widetilde u_\app - U_\bl(\nu t,z) )\|_{\beta,\gamma,1} &\lesssim \nu^{\frac12 + \epsilon},
\\
\|\mathrm{curl} \mathcal{E}_\app (t)\|_{\beta,\gamma,1} 
&\lesssim   \nu^{M} .
\end{aligned}$$
Furthermore, there are positive constants $\theta_0,\theta_1,\theta_2$ independent of $\nu$ so that there holds 
$$  
\theta_1 \nu^p e^{\Re \lambda_0 t} \le  \| (\widetilde u_\app - U_\bl) (t)\|_{L^\infty} \le \theta_2 \nu^p e^{\Re \lambda_0 t} 
$$
for all $t\in [0,T_\nu]$. In particular, 
$$
\| (\widetilde u_\app - U_\bl) (T_\nu)\|_{L^\infty}\gtrsim \nu^{\frac12 + \epsilon}.
$$ 
\end{theorem}


\subsection{Formal construction}


The construction is classical, following \cite{Gr1}. Indeed, the approximate solutions are constructed in the following form
\beq \label{def-Uapp}
\widetilde u_\app(t,x,z) = U_\bl( \nu t,z) + \nu^p \sum_{j = 0}^M \nu^{j/4} u_j(t,x,z).
\eeq
For convenience, let us set $v = u -U_\bl $, where $u$ denotes the genuine solution to the Navier-Stokes equations. 
Then, the vorticity $\omega = \nabla \times v$ solves 
$$
 \partial_t \omega + (U_\bl  (\nu t, y)  + v)\cdot \nabla \omega + v_2 \partial_y^2U_s(\nu t,y) - \nu \Delta \omega =0 
$$
in which $v = \nabla^\perp \Delta^{-1} \omega$ and $v_2$ denotes the vertical component of velocity. Here and in what follows, 
$\Delta^{-1}$ is computed with the zero Dirichlet boundary condition. As $U_\bl $ depends slowly on time, 
we can rewrite the vorticity equation as follows: 
\begin{equation}\label{vort-Phi}
 (\partial_t - L) \omega  + \nu^{1/4} S \omega + Q(\omega, \omega) =0.
\end{equation}
In \eqref{vort-Phi}, $L$ denotes the linearized Navier-Stokes operator around the stationary boundary layer $U  = U_s(0,z)$:  
$$ L \omega : = \nu \Delta \omega  - U \partial_x \omega - u_2 U'' ,$$
$Q(\omega, \tilde \omega)$ denotes the quadratic nonlinear term $u \cdot \nabla \tilde \omega$, 
with $v = \nabla^\perp \Delta^{-1} \omega$, and $S$ denotes the perturbed operator defined by 
$$
\begin{aligned}
 S\omega: &= \nu^{-1/4}[U_s ( \nu t, z)  - U (z)] \partial_x \omega + \nu^{-1/4} u_2 [\partial_y^2 U_s ( \nu t, z)  - U''(z)] 
.\end{aligned}$$
Recalling that $U_s$ solves the heat equation with initial data $U(z)$, we have 
$$| U_s ( \nu t, z)  - U (z) | \le C \| U''\|_{L^\infty} e^{-\eta_0 z} \nu t$$ 
and 
$$ | \partial_z^2 U_s ( \nu t, z)  - U''(z)| \le C \| U''\|_{W^{2,\infty}} e^{-\eta_0 z} \nu t.$$
Hence, 
\begin{equation}\label{def-Sop}
\begin{aligned}
 S\omega  = \nu^{-1/4}\cO( \nu t e^{-\eta_0 z})\Big[  |\partial_x \omega |+ |\partial_x \Delta^{-1}\omega| \Big]
\end{aligned}\end{equation}
in which $\Delta^{-1} \omega$ satisfies the zero boundary condition on $z=0$. 
The approximate solutions are then constructed via the asymptotic expansion: 
\begin{equation}\label{def-omegaAAA} \omega_\app = \nu^p \sum_{j=0}^M \nu^{j/4} \omega_j,\end{equation}
in which $p$ is an arbitrarily large integer. Plugging this Ansatz into \eqref{vort-Phi} and matching order in $\nu$,
 we are led to solve 
 \begin{itemize}

\item for $j =0$: 
$$ (\partial_t - L) \omega_0 = 0$$
with zero boundary conditions on $v_0= \nabla^\perp (\Delta)^{-1}\omega_0$ on $z=0$; 

\item for $0<j\le M$: 
\begin{equation}\label{eqs-omegaj} (\partial_t - L) \omega_j  = R_j, \qquad {\omega_j}_{\vert_{t=0}}=0,\end{equation}
with zero boundary condition on $v_j = \nabla^\perp (\Delta)^{-1}\omega_j$ on $z=0$. Here, the remainders $R_j$ are defined by 
$$ 
R_j = S \omega_{j-1} + \sum_{k + \ell + 4p = j}Q(\omega_k, \omega_\ell).
$$

\end{itemize}
As a consequence, the approximate vorticity $\omega_\app$ solves \eqref{vort-Phi}, leaving the error $R_\app$ defined by
\begin{equation}
\label{error-app}
\begin{aligned}
R_\app
&= \nu^{p+\frac{M+1}{4}} S \omega_{M} +  \sum_{k+ \ell> M+1 -4 p; 1\le k,\ell \le M} \nu^{2p+ \frac{k+\ell}{4}}Q(\omega_k, \omega_\ell) 
\end{aligned}\end{equation}
which formally is of order $ \nu^{p+\frac{M+1}{4}}$, for arbitrary $p$ and $M$.

 
\subsection{Estimates}


We start the construction with $\omega_0$ being the maximal growing mode, constructed in Section \ref{sec-grmode}. We recall 
\begin{equation}\label{def-omega000}
\omega_0 = e^ {\lambda_\nu t} e^{i\alpha_\nu x} \Delta_{\alpha_\nu} \Big( \phi_{in,0}(z) 
+\nu^{1/4} \phi_{\bl,0} ( \nu^{-1/4} z) \Big)  \quad+\quad \mbox{c.c.}
\end{equation}
with $\alpha_\nu \sim \nu^{1/4}$ and $\Re \lambda_\nu \sim \sqrt \nu$. In what follows, $\alpha_\nu$ and $\lambda_\nu$ are fixed. 
We obtain the following lemma.  

\begin{lemma}\label{lem-omegaj} 
Let $\omega_0$ be the maximal growing mode \eqref{def-omega000}, and let $\omega_j$ 
be inductively constructed by \eqref{eqs-omegaj}. Then, there hold the following uniform bounds:
\begin{equation}\label{induction-jn}
\begin{aligned}
\| \partial_x^a\partial_z^b\omega_j \|_{\sigma,\beta,\gamma,1} & \le C_0 \nu^{a/4}\nu^{-b/4} \nu^{-\frac12[\frac{j}{4p}]}
e^{\gamma_0 (1+\frac{j}{4p})\nu^{1/2}t } 
\end{aligned} \end{equation}
for all $a,j\ge 0$ and for $b=0,1$. 
In addition, the approximate solution $\omega_\app$ defined as in \eqref{def-omegaAAA} satisfies 
\begin{equation}\label{est-omegaAAA}
\| \partial_x^a \partial_z^b \omega_\app\|_{\sigma,\beta,\gamma,1} 
\lesssim \nu^{a/4} \nu^{-b/4} \sum_{j=0}^M \nu^{-\frac12[\frac{j}{4p}]}  \Big( \nu^p e^{\gamma_0 \nu^{1/2} t}\Big)^{1+ \frac{j}{4p}} , 
\end{equation}
for $a\ge 0$ and $b=0,1$. 
 Here, $[k]$ denotes the largest integer so that $[k]\le k$. 
\end{lemma}

\begin{proof} 
For $j\ge 1$, we construct $\omega_j$ having the form 
$$ \omega_j = \sum_{n \in \ZZ} e^{i n \alpha_\nu x} \omega_{j,n}$$
It follows that $\omega_{j,n}$ solves  
$$
(\partial_t - L_{\alpha_n}) \omega_{j,n}  = R_{j,n}, \qquad {\omega_{j,n}}_{\vert_{t=0}}=0 
$$
with $\alpha_n = n \alpha_\nu$ and $R_{j,n}$ the Fourier transform of $R_j$ 
evaluated at the Fourier frequency $\alpha_n$. Precisely, we have 
$$ 
R_{j,n} = S_{\alpha_n} \omega_{j-1,n} + \sum_{k + \ell + 8p = j}\sum_{n_1+n_2 = n}Q_{\alpha_n}(\omega_{k,n_1}, \omega_{\ell,n_2}),
$$
in which $S_{\alpha_n}$ and $Q_{\alpha_n}$ denote the corresponding operator $S$ and $Q$ in the Fourier space. 
The Duhamel's integral reads 
\begin{equation}\label{Duh-omegajn}
\omega_{j,n}(t) = \int_0^t e^{L_{\alpha_n} (t-s)} R_{j,n}(s)\; ds \end{equation}
for all $j\ge 1$ and $n \in \ZZ$. 

It follows directly from an inductive argument and the quadratic nonlinearity of $Q(\cdot,\cdot)$ that for all $0\le j\le M$, 
$\omega_{j,n} = 0$ for all $|n|\ge 2^{j+1}$. 
This proves that $|\alpha_n| \le 2^{M+1} \alpha_\nu \lesssim \nu^{1/4}$, for all $|n| \le 2^{M+1}$.  
Since $\alpha_n \lesssim \nu^{1/4}$, the semigroup bounds from Theorem \ref{theo-eLt-stable} read 
\begin{equation}\label{eLt-jn}\begin{aligned}
\| e^{L_\alpha t}\omega_\alpha\|_{ \beta, \gamma, 1} &\lesssim   e^{\gamma_1 \nu^{1/2} t }  
e^{- \frac14 \alpha^2 \nu t}  \| \omega_\alpha\|_{ \beta, \gamma, 1},
\\\| \partial_ze^{L_\alpha t}\omega_\alpha\|_{ \beta, \gamma, 1} &\lesssim \Big( \nu^{-1/4}
+ ( \nu t)^{-1/2} \Big) e^{\gamma_1 \nu^{1/2} t }  e^{- \frac14 \alpha^2 \nu t}  \| \omega_\alpha\|_{ \beta, \gamma, 1}. 
\end{aligned}\end{equation}
In addition, since $\alpha_n\lesssim \nu^{1/4}$, from \eqref{def-Sop}, 
we compute 
$$
\begin{aligned}
 S_{\alpha_n}\omega_{j-1,n}  = \cO( \nu t e^{-\eta_0 z})\Big[  |\omega_{j-1,n}|+ |\Delta_{\alpha_n}^{-1}\omega_{j-1,n}| \Big]
\end{aligned}
$$
and hence by induction we obtain 
\begin{equation}\label{bd-Sjn}
\begin{aligned}
\| S_{\alpha_n}\omega_{j-1,n}\|_{\beta,\gamma,1} 
&\lesssim  \nu t \Big[ \|\omega_{j-1,n}\|_{\beta,\gamma,1}+ \| e^{-\eta_0 z}\Delta_{\alpha_n}^{-1}\omega_{j-1,n}\|_{\beta,\gamma,1} \Big]
\\
&\lesssim  \nu t \Big[ \|\omega_{j-1,n}\|_{\beta,\gamma,1}+ \| \Delta_{\alpha_n}^{-1}\omega_{j-1,n}\|_{L^\infty} \Big]
\\
&\lesssim  \nu t \nu^{-\frac12[\frac{j-1}{4p}]} e^{\gamma_0 (1+\frac{j-1}{4p})\nu^{1/2}t } ,
\end{aligned}
\end{equation}
 where we used $\|e^{-\eta_0 z} \cdot \|_{\beta,\gamma,1} \le \|\cdot \|_{L^\infty}$ for $\beta < \eta_0$, and 
 $$
 \| \Delta_{\alpha}^{-1} \omega\|_{L^\infty} 
 \le C \| \omega \|_{\beta,\gamma,1},
 $$
uniformly in small $\alpha$; we shall prove this inequality in the Appendix. Let us first consider the case when $1\le j\le 4p-1$, for which $R_{j,n} = S_{\alpha_n} \omega_{j-1,n} $. 
That is, there is no nonlinearity in the remainder. 
Using the above estimate on $S_{\alpha_n}$ and the semigroup estimate \eqref{eLt-jn} into \eqref{Duh-omegajn}, 
we obtain, for $1\le j\le 4p-1$,
$$\begin{aligned} \|\omega_{j,n}(t)\|_{\beta,\gamma,1} 
&\le \int_0^t \| e^{L_{\alpha_n} (t-s)} S_{\alpha_n} \omega_{j-1,n} (s) \|_{\beta,\gamma,1}\; ds 
\\
&\le C\int_0^t e^{\gamma_1 \nu^{1/2} (t-s) }  \| S_{\alpha_n} \omega_{j-1,n} (s) \|_{\beta,\gamma,1}\; ds 
\\
&\le C\int_0^t e^{\gamma_1 \nu^{1/2} (t-s) }  \nu s e^{\gamma_0 (1+\frac{j-1}{4p})\nu^{1/2}s } \; ds .
\end{aligned}$$
We choose
$$
\gamma_1 = \gamma_0 (1+ \frac{j-1}{4p} + \frac{1}{8p})
$$
in (\ref{eLt-jn}) and use the inequality 
$$
\nu^{1/2} t \le C e^{\frac{\gamma_0}{8p}\nu^{1/2}t}.
$$
and obtain
\begin{equation}\label{est-ojn1}\begin{aligned} \|\omega_{j,n}(t)\|_{\beta,\gamma,1} 
&\le C\int_0^t e^{\gamma_1 \nu^{1/2} (t-s) } \nu^{1/2} e^{\gamma_0 (1+\frac{j-1}{4p} + \frac{1}{8p})\nu^{1/2}s } \; ds 
\\
&\le C\nu^{1/2} e^{\gamma_0 (1+\frac{j-1}{4p} + \frac{1}{8p})\nu^{1/2} t} \int_0^t  \; ds 
\\
&\le C\nu^{1/2} t e^{\gamma_0 (1+\frac{j-1}{4p} + \frac{1}{8p})\nu^{1/2} t} 
\\& \le C e^{\gamma_0 (1+\frac{j}{4p})\nu^{1/2}t } .
\end{aligned}\end{equation}
Similarly, as for derivatives, we obtain 
$$\begin{aligned} 
&
\|\partial_z\omega_{j,n}(t)\|_{\beta,\gamma,1} 
\\&\le \int_0^t \| e^{L_{\alpha_n} (t-s)} S_{\alpha_n} \omega_{j-1,n} (s) \|_{\beta,\gamma,1}\; ds 
\\
&\le C \int_0^t \Big( \nu^{-1/4}+ ( \nu (t-s))^{-1/2} \Big) e^{\gamma_1 \nu^{1/2} (t-s) }  \| S_{\alpha_n} \omega_{j-1,n} (s) \|_{\beta,\gamma,1}\; ds 
\\
&\le C 
\int_0^t \Big( \nu^{-1/4}+ ( \nu (t-s))^{-1/2} \Big) e^{\gamma_1 \nu^{1/2} (t-s) }  \nu s e^{\gamma_0 (1+\frac{j-1}{4p})\nu^{1/2} s } \; ds ,
\end{aligned}$$
in which the integral involving $\nu^{-1/4}$ is already treated in \eqref{est-ojn1} and bounded by $C\nu^{-1/4}e^{\gamma_0 (1+\frac{j}{4p})\nu^{1/2}t }.$ As for the second integral, we estimate 
\begin{equation}\label{t12-bd} 
\begin{aligned}
\int_0^t &( \nu (t-s))^{-1/2} e^{\gamma_1 \nu^{1/2} (t-s) }  \nu s e^{\gamma_0 (1+\frac{j-1}{4p})\nu^{1/2} s } \; ds
\\ 
&\le  \int_0^t ( \nu (t-s))^{-1/2} e^{\gamma_1 \nu^{1/2} (t-s) }  \nu^{1/2} e^{\gamma_0 (1+\frac{j-1}{4p} + \frac{1}{8p})\nu^{1/2} s } \; ds
\\&\le \nu^{1/2} e^{\gamma_0 (1+\frac{j-1}{4p} + \frac{1}{8p})\nu^{1/2} t }  \int_0^t ( \nu (t-s))^{-1/2} \; ds 
\\&\le C \sqrt t e^{\gamma_0 (1+\frac{j-1}{4p} + \frac{1}{8p})\nu^{1/2} t } 
\\&\le C \nu^{-1/4} e^{\gamma_0 (1+\frac{j}{4p} )\nu^{1/2} t } .
\end{aligned}\end{equation}
Thus, 
$$\|\partial_z\omega_{j,n}(t)\|_{\beta,\gamma,1} \le C \nu^{-1/4} e^{\gamma_0 (1+\frac{j}{4p} )\nu^{1/2} t }.$$ 
This and \eqref{est-ojn1} prove the inductive bound \eqref{induction-jn} for $j\le 4p-1$. 

For $j \ge 4p$, the quadratic nonlinearity starts to play a role. For $k+\ell = j-4p$, we compute 
\begin{equation}\label{est-Q2w}
Q_{\alpha_n}(\omega_{k,n_1}, \omega_{\ell,n_2}) = i \alpha_\nu \Big( n_2 \partial_z \Delta_{\alpha_n}^{-1} \omega_{k,n_1} 
\omega_{\ell,n_2}  - n_1 \Delta_{\alpha_n}^{-1} \omega_{k,n_1} \partial_z \omega_{\ell,n_2}\Big).
\end{equation}
Using the algebra structure of the boundary layer norm (see \eqref{al-norm}), we have 
$$\begin{aligned}
\alpha_\nu\| \partial_z \Delta_{\alpha_n}^{-1} \omega_{k,n_1} \omega_{\ell,n_2} \|_{\beta,\gamma,1} 
&\lesssim 
\nu^{1/4}\| \partial_z \Delta_{\alpha_n}^{-1} \omega_{k,n_1}\|_{L^\infty} \|\omega_{\ell,n_2} \|_{\beta,\gamma,1}
\\&\lesssim 
\nu^{1/4}\| \omega_{k,n_1}\|_{\beta,\gamma,1} \|\omega_{\ell,n_2} \|_{\beta,\gamma,1}
\\&\lesssim \nu^{1/4}\nu^{-\frac12[\frac{k}{4p}]}  \nu^{-\frac12[\frac{\ell}{4p}]}  e^{\gamma_0 (2+\frac{k+\ell}{4p})\nu^{1/2}t }
 \end{aligned}$$
 where we used
 $$
 \| \partial_z \Delta_{\alpha_n}^{-1} \omega_{k,n_1} \|_{L^\infty} 
 \le C \| \omega_{k,n_1}\|_{\beta,\gamma,1},
 $$
 an inequality which is proven in the Appendix.
Moreover,
$$\begin{aligned}
\alpha_\nu\| \Delta_{\alpha_n}^{-1} \omega_{k,n_1} \partial_z \omega_{\ell,n_2} \|_{\beta,\gamma,1} 
&\lesssim \nu^{1/4}
\|  \Delta_{\alpha_n}^{-1} \omega_{k,n_1}\|_{L^\infty} \| \partial_z\omega_{\ell,n_2} \|_{\beta,\gamma,1}
\\&\lesssim 
\nu^{1/4}\| \omega_{k,n_1}\|_{\beta,\gamma,1} \|\partial_z\omega_{\ell,n_2} \|_{\beta,\gamma,1}
\\&\lesssim \nu^{-\frac12[\frac{k}{4p}]}  \nu^{-\frac12[\frac{\ell}{4p}]}  e^{\gamma_0 (2+\frac{k+\ell}{4p})\nu^{1/2}t },
 \end{aligned}$$
 in which the derivative estimate \eqref{induction-jn} was used. 
We note that 
$$
[\frac{k}{4p}]+ [\frac{\ell}{4p}] \le [\frac{k+\ell}{4p}] = [\frac{j}{4p}] - 1.
$$
 This proves 
$$\begin{aligned}
\|& Q_{\alpha_n}(\omega_{k,n_1}, \omega_{\ell,n_2}) \|_{\beta,\gamma,1} 
\lesssim \nu^{1/2} \nu^{-\frac12[\frac{j}{4p}]} e^{\gamma_0 (1+\frac{j}{4p})\nu^{1/2}t }
\end{aligned}$$ 
for all $k+\ell = j-4p$. This, together with the estimate \eqref{bd-Sjn} on $S_{\alpha_n}$, yields 
$$
\begin{aligned}
\| R_{j,n} (t)\|_{\beta,\gamma,1} 
&\lesssim \nu t \nu^{-\frac12[\frac{j-1}{4p}]} e^{\gamma_0 (1+\frac{j-1}{4p})\nu^{1/2}t } + \nu^{1/2} \nu^{-\frac12[\frac{j}{4p}]} e^{\gamma_0 (1+\frac{j}{4p})\nu^{1/2}t }
\\
&\lesssim \nu^{1/2} \nu^{-\frac12[\frac{j}{4p}]} e^{\gamma_0 (1+\frac{j}{4p})\nu^{1/2}t },
\end{aligned}
$$
for all $j\ge 4p$ and $n\in \ZZ$, in which we used $\nu^{1/2} t \le e^{\gamma_0 t/4p}$. 

Putting these estimates into the Duhamel's integral formula \eqref{Duh-omegajn}, we obtain, for $j\ge 4p$,
$$
\begin{aligned}
\|\omega_{j,n} (t)\|_{\beta,\gamma,1}  
&\le C \int_0^t  e^{\gamma_1 \nu^{1/2} (t-s) } \| R_{j,n}(s)\|_{\beta,\gamma,1} \; ds
\\
&\le C \int_0^t  e^{\gamma_1 \nu^{1/2} (t-s) } \nu^{1/2} \nu^{-\frac12[\frac{j}{4p}]} e^{\gamma_0 (1+\frac{j}{4p})\nu^{1/2}s } \; ds  
\\
&\lesssim \nu^{-\frac12[\frac{j}{4p}]}  e^{\gamma_0 (1+\frac{j}{4p})\nu^{1/2}s } 
 \end{aligned}$$
and 
$$\begin{aligned} 
&
\|\partial_z\omega_{j,n}(t)\|_{\beta,\gamma,1} 
\\&\le C \int_0^t \Big( \nu^{-1/4}+ (\nu (t-s))^{-1/2} \Big) e^{\gamma_1 \nu^{1/2} (t-s) }  \| R_{j,n}(s)\|_{\beta,\gamma,1} \; ds
\\
&\le C 
\int_0^t \Big( \nu^{-1/4}+ (\nu (t-s))^{-1/2} \Big) e^{\gamma_1 \nu^{1/2} (t-s) } \nu^{1/2} \nu^{-\frac12[\frac{j}{4p}]} e^{\gamma_0 (1+\frac{j}{4p})\nu^{1/2}s } \; ds .
\end{aligned}$$
Using \eqref{t12-bd}, we obtain 
$$
\begin{aligned}
\|\partial_z \omega_{j,n} (t)\|_{\beta,\gamma,1}  
\lesssim \nu^{-1/4}\nu^{-\frac12[\frac{j}{4p}]}  e^{\gamma_0 (1+\frac{j}{4p})\nu^{1/2}s } ,
 \end{aligned}$$
which completes the proof of \eqref{induction-jn}. The lemma follows. 
\end{proof}

 
 \subsection{The remainder}
 

We recall that the approximate vorticity $\omega_\app$, constructed as in \eqref{def-omegaAAA}, approximately solves \eqref{vort-Phi}, 
leaving the error $R_\app$ defined by
$$
\begin{aligned}
R_\app
&= \nu^{p+\frac{M+1}{4}} S \omega_{M} +  \sum_{k+ \ell> M+1 -4 p; 1\le k,\ell \le M} \nu^{2p+ \frac{k+\ell}{4}}Q(\omega_k, \omega_\ell) .
\end{aligned}$$
Using the estimates in Lemma \ref{lem-omegaj}, we obtain  
$$\begin{aligned}
\|S \omega_{M} \|_{\sigma,\beta,\gamma,1} &\lesssim  \nu^{1/2} \nu^{-\frac12[\frac{M+1}{4p}]} 
e^{\gamma_0 (1+\frac{M+1}{4p})\nu^{1/2}t }
 \\
 \|Q(\omega_k, \omega_\ell) \|_{\sigma,\beta,\gamma,1} &\lesssim  \nu^{1/2} 
 \nu^{-\frac12[\frac{k+\ell}{4p}]} e^{\gamma_0 (2+\frac{k+\ell}{4p})\nu^{1/2}t }.\end{aligned}$$
This yields 
\begin{equation}\label{est-RAA}
\begin{aligned}
\|\ R_\app\|_{\sigma,\beta,\gamma,1} 
&\lesssim  \nu^{1/2} \sum_{j=M+1}^{2M} \nu^{-\frac12[\frac{j}{4p}]}  \Big( \nu^p e^{\gamma_0 \nu^{1/2} t}\Big)^{1+ \frac{j}{4p}} .
\end{aligned}\end{equation}


\subsection{Proof of Theorem \ref{theo-approx-stable}}


The proof of the Theorem now  straightforwardly follows from the estimates from Lemma \ref{lem-omegaj} 
and the estimate \eqref{est-RAA} on the remainder. Indeed, we choose the time $T_*$ so that 
\begin{equation}\label{def-Tstar} \nu^p e^{\gamma_0 \nu^{1/4} T_*} = \nu^\tau\end{equation}
for some fixed $\tau>\frac12$. It then follows that for all $t\le T_*$ and $j\ge 0$, there holds 
$$
\begin{aligned}
 \nu^{-\frac12[\frac{j}{4p}]}  \Big( \nu^p e^{\gamma_0 \nu^{1/2} t}\Big)^{1+ \frac{j}{4p}}  \lesssim \nu^\tau  \nu^{(\tau - \frac12) \frac{j}{4p}}.
\end{aligned}
$$
Using this into the estimates \eqref{est-omegaAAA} and \eqref{est-RAA}, respectively, we obtain 
\begin{equation}\label{est-wRapp}
\begin{aligned}
\| \partial_z^b\omega_\app(t)\|_{\sigma,\beta,\gamma,1} &\lesssim \nu^{p-b/4} e^{\gamma_0 \nu^{1/2} t}  \lesssim \nu^{\tau- b/4},
\\
\|R_\app (t)\|_{\sigma,\beta,\gamma,1} 
&\lesssim  \nu^{1/2} \nu^{-\frac12[\frac{M}{4p}]}  \Big( \nu^p e^{\gamma_0 \nu^{1/2} t}\Big)^{1+ \frac{M}{4p}} 
\\&\lesssim \nu^{\tau+1/2}  \nu^{(\tau - \frac12) \frac{M}{4p}},
\end{aligned}\end{equation}
for all $t \le T_*$. Since $\tau>\frac12$ and $M$ is arbitrarily large (and fixed), the remainder is of order $\nu^P$ 
for arbitrarily large number $P$. The theorem is proved.


\section{Nonlinear instability}\label{sec-nonlinear}


We are now ready to give the proof of Theorem \ref{firstinstability-stable}. Let $\widetilde u_\app$ 
be the approximate solution constructed in Theorem \ref{theo-approx-stable} and let  
$$v = u - \widetilde u_\app,$$
with $u$ being the genuine solution to the nonlinear Navier-Stokes equations. 
The corresponding vorticity $\omega = \nabla \times v$ solves 
$$ 
\partial_t   \omega + (\widetilde u_\mathrm{app} + v) \cdot \nabla \omega + v \cdot \nabla \widetilde \omega_\mathrm{app} 
= \Delta \omega +  R_\app
$$
for the remainder $R_\app  = \mathrm{curl } \,  \, {\cal E}_\app$ satisfying the estimate \eqref{est-RAA}. Let us write 
$$
u_\app = \widetilde u_\app - U_\bl .
$$
To make use of the semigroup bound for the linearized operator $\partial_t - L$, we rewrite the vorticity equation as  
$$ 
(\partial_t - L) \omega + (u_\app + v) \cdot \nabla \omega + v \cdot \nabla \omega_\app  = R_\app
$$
with $\omega_{\vert_{t=0}} = 0$. We note that since the boundary layer profile is stationary, 
the perturbative operator $S$ defined as in \eqref{def-Sop} is in fact zero. The Duhamel's principle then yields 
\begin{equation}\label{Duh-omega} 
\omega (t) = \int_0^t e^{L(t-s)} \Big( R_\app - (u_\app + v) \cdot \nabla \omega - v \cdot \nabla \omega_\app  \Big) \; ds.
\end{equation}
Using the representation \eqref{Duh-omega}, we shall prove the existence and give estimates on $\omega$. 
We shall work with the following norm 
\begin{equation}\label{def-123norm} 
||| \omega(t) |||: = \| \omega(t) \|_{\sigma,\beta,\gamma,1}+ \nu^{1/4}\| \partial_x \omega(t) \|_{\sigma,\beta,\gamma,1} 
+ \nu^{1/4}\| \partial_z \omega (t)\|_{\sigma,\beta,\gamma,1} 
\end{equation}
in which the factor $\nu^{1/4}$ added in the norm is to overcome the loss of $\nu^{-1/4}$ for derivatives
(see \eqref{semi-bounds} for more details).

Let $p$ be an arbitrary large number. We introduce the maximal time $T_\nu$ of existence, defined by 
\begin{equation}\label{claim-www}
\begin{aligned}
T_\nu: = \max\Big \{ t\in [0,T_*]~:~\sup_{0\le s\le t} ||| \omega(s) |||  \le \nu^p e^{\gamma_0 \nu^{1/2} t}\Big\}
 \end{aligned}\end{equation}
in which $T_*$ is defined as in \eqref{def-Tstar}. By the short time existence theory, with zero initial data, $T_\nu$ exists and is positive. 
It remains to give a lower bound estimate on $T_\nu$.  First, we obtain the following lemmas. 
\begin{lemma} For $t \in [0,T_*]$, there hold
$$
\begin{aligned}
\| \partial_x^a \partial_z^b \omega_\app(t)\|_{\sigma,\beta,\gamma,1} &\lesssim \nu^{a/4-b/4} \Big( \nu^p e^{\gamma_0 \nu^{1/2} t}\Big)
\\
\|\partial_x^a \partial_z^bR_\app (t)\|_{\sigma,\beta,\gamma,1} 
&\lesssim  \nu^{1/2+a/4-b/4} \nu^{-\frac12[\frac{M}{4p}]}  \Big( \nu^p e^{\gamma_0 \nu^{1/2} t}\Big)^{1+ \frac{M}{4p}} .
\end{aligned}$$
\end{lemma}
\begin{proof} This follows directly from Lemma \ref{lem-omegaj} and the estimate \eqref{est-RAA} on the remainder $R_\app$, 
upon noting the fact that for $t\in [0,T_*]$, $\nu^p e^{\gamma_0 \nu^{1/2} t}$ remains sufficiently small.
\end{proof}

\begin{lemma} There holds 
$$
\Big\|( u_\app + v) \cdot \nabla \omega + v \cdot \nabla  \omega_\app \Big\|_{\sigma,\beta,\gamma,1} 
\lesssim \nu^{-\frac14}\Big( \nu^p e^{\gamma_0 \nu^{1/2} t}\Big)^2.
$$
for $t\in [0,T_\nu]$. 
\end{lemma}
\begin{proof}
We first recall the elliptic estimate 
$$
\| u\|_{\sigma,0} \lesssim \|\omega\|_{\sigma,\beta,\gamma,1}
$$ 
which is proven in the Appendix \ref{sec-elliptic}), and the following uniform bounds (see \eqref{al-norm1})
$$
\begin{aligned}
\| u \cdot \nabla \tilde \omega \|_{\sigma,\beta,\gamma,1}
&\le \| u \|_{\sigma,0} \| \nabla \tilde \omega \|_{\sigma,\beta,\gamma,1}
\\&\le \| \omega \|_{\sigma,\beta,\gamma,1} \| \nabla\tilde \omega \|_{\sigma,\beta,\gamma,1} .
 \end{aligned} $$
Using this and the bounds on $\omega_\app$, we obtain 
$$
\begin{aligned}
 \| v \cdot \nabla \omega_\app \|_{\sigma,\beta,\gamma,1} &\lesssim 
 \nu^{-1/4}\Big( \nu^p e^{\gamma_0 \nu^{1/2} t}\Big) \| \omega \|_{\sigma,\beta,\gamma,1} 
 \lesssim\nu^{-\frac14} \Big( \nu^p e^{\gamma_0 \nu^{1/2} t}\Big)^2
\end{aligned}
$$
and 
$$
\begin{aligned}
 \| ( u_{\app} + v)\cdot \nabla \omega \|_{\sigma,\beta,\gamma,1}  &
 \lesssim \Big( \nu^p e^{\gamma_0 \nu^{1/4} t} + \| \omega\|_{\sigma,\beta,\gamma,1}\Big) \| \nabla \omega \|_{\sigma,\beta,\gamma,1} 
\\&\lesssim \nu^{-\frac14}\Big( \nu^p e^{\gamma_0 \nu^{1/4} t}\Big)^2
 .\end{aligned}
 $$
This proves the lemma. \end{proof}

Next, using Theorem \ref{theo-eLt-stable} and noting that $\alpha e^{-\alpha^2 \nu t} \lesssim 1+ (\nu t)^{-1/2}$, we obtain the following uniform semigroup bounds:
\begin{equation}\label{semi-bounds}\begin{aligned}
\| e^{L t}\omega\|_{ \sigma,\beta, \gamma, 1} &\le C_0 \nu^{-1/2} e^{\gamma_1 \nu^{1/2} t }  \| \omega\|_{\sigma, \beta, \gamma, 1} 
\\
\| \partial_x e^{L t}\omega\|_{ \sigma,\beta, \gamma, 1} &\le
 C_0 \nu^{-1/2}\Big( 1+ (\nu t)^{-1/2} \Big) e^{\gamma_1 \nu^{1/2} t }  \| \omega\|_{\sigma, \beta, \gamma, 1} 
\\
\| \partial_z e^{L t}\omega\|_{ \sigma,\beta, \gamma, 1} &\le
 C_0 \nu^{-1/2} \Big( \nu^{-1/4}+ (\nu t)^{-1/2} \Big) e^{\gamma_1 \nu^{1/2} t }  \| \omega\|_{\sigma, \beta, \gamma, 1} .
\end{aligned}\end{equation}
We are now ready to apply the above estimates into the Duhamel's integral formula \eqref{Duh-omega}. 
We obtain 
$$
\begin{aligned}
 \|  \omega (t)\|_{\sigma,\beta,\gamma,1} 
 &\lesssim \nu^{-1/2}\int_0^t e^{\gamma_1 \nu^{1/2} (t-s)} \nu^{-\frac14}\Big (\nu^{p} e^{\gamma_0 \nu^{1/2} s}\Big)^2\; ds
 \\&\quad + \nu^{-1/2}\int_0^t e^{\gamma_1 \nu^{1/2} (t-s)}  
 \nu^{1/2} \nu^{-\frac12[\frac{M}{4p}]}  \Big( \nu^p e^{\gamma_0 \nu^{1/2} s}\Big)^{1+ \frac{M}{4p}}  
 \; ds 
\\
 &\lesssim \nu^{-5/4} \Big (\nu^{p} e^{\gamma_0 \nu^{1/2} t}\Big)^2 + \nu^P\Big( \nu^p e^{\gamma_0 \nu^{1/2} t} \Big),
 \end{aligned}$$
upon taking $\gamma_1$ sufficiently close to $\gamma_0$. Set $T_1$ so that 
\begin{equation}\label{def-T11} \nu^p e^{\gamma_0 \nu^{1/2} T_1} =\theta_0 \nu^{\frac54},\end{equation}
for some sufficiently small and positive constant $\theta_0$. 
Then, for all $t\le T_1$, there holds 
$$
 \begin{aligned}\|  \omega (t)\|_{\sigma,\beta,\gamma,1}  
 &\lesssim \nu^{p} e^{\gamma_0 \nu^{1/2}t} \Big[ \theta_0
 + \nu^{P}\Big]
\end{aligned} $$
Similarly, we estimate the derivatives of $\omega$. The Duhamel integral and the semigroup bounds yield 
$$
\begin{aligned}
\| \nabla \omega (t)\|_{\sigma,\beta,\gamma,1} 
 &\lesssim \nu^{-1/2}\int_0^t e^{\gamma_1 \nu^{1/2} (t-s)}  \Big( \nu^{-1/4}+ (\nu (t-s))^{-1/2} \Big) 
 \\&\quad \times  
 \Big[ 
 \nu^{-\frac14}\Big (\nu^{p} e^{\gamma_0 \nu^{1/2} s}\Big)^2 
+  \nu^{-\frac12[\frac{M}{4p}]}  \Big( \nu^p e^{\gamma_0 \nu^{1/2} s}\Big)^{1+ \frac{M}{4p}} 
\Big] \; ds 
\\&\lesssim \nu^{-5/4} \Big[ 
 \nu^{-\frac14}\Big (\nu^{p} e^{\gamma_0 \nu^{1/2} s}\Big)^2 
+  \nu^{-\frac12[\frac{M}{4p}]}  \Big( \nu^p e^{\gamma_0 \nu^{1/2} s}\Big)^{1+ \frac{M}{4p}} 
\Big] 
 \end{aligned}$$
By view of \eqref{def-T11} and the estimate \eqref{t12-bd}, 
the above yields 
$$
 \begin{aligned}
 \| \nabla \omega (t)\|_{\sigma,\beta,\gamma,1}  &\lesssim \nu^{p-\frac14} e^{\gamma_0 \nu^{1/2}t} \Big[ \theta_0 
+ \nu^P\Big] .
\end{aligned} 
$$
To summarize, for $t\le \min\{T_*, T_1,T_\nu\}$, with the times defined as in \eqref{def-Tstar},
 \eqref{claim-www}, and \eqref{def-T11}, we obtain 
$$ ||| w(t)||| \lesssim  \nu^{p} e^{\gamma_0 \nu^{1/2}t} \Big[ \theta_0+ \nu^P\Big].$$
Taking $\theta_0$ sufficiently small, we obtain 
$$
 ||| w(t)||| \ll\nu^{p} e^{\gamma_0 \nu^{1/2}t} 
$$
for all time $t\le \min\{T_*, T_1,T_\nu\}$. 
In particular, this proves that the maximal time of existence $T_\nu$ is greater than $ T_1$, defined as in \eqref{def-T11}.  
This proves that at the time $t=T_*$, the approximate solution grows to order of $\nu^{5/8}$ in the $L^\infty$ norm. 
Theorem \ref{firstinstability-stable} is proved.



\section{Elliptic estimates}\label{sec-elliptic}


In this section, for sake of completeness, we recall the elliptic estimates with respect to the boundary layer norms. 
These estimates are proven in \cite[Section 3]{GrN2}. 

First, we consider the classical one-dimensional Laplace equation
\beq \label{Lap1}
\Delta_\alpha \phi = \partial_z^2 \phi - \alpha^2 \phi = f
\eeq
on the half line $z \ge 0$, with the Dirichlet boundary condition $ \phi(0) = 0$. We recall the function space $L^\infty_\beta$ defined by the finite norm $\| f\|_\beta  =\sup_{z \ge 0}|f(z) | e^{\beta z}$. We will prove
\begin{prop}
If $f \in L^\infty_\beta$ for some $\beta >0$, then $\phi \in L^\infty$. In addition, there holds
\beq \label{Lap3}
(1+\alpha^2) \| \phi \|_{L^\infty} + (1+ | \alpha |) \,  \| \partial_z \phi \|_{L^\infty}
+ \| \partial_z^2 \phi \|_{L^\infty}  \le C \| f \|_{\beta},
\eeq
where the constant $C$ is independent of $\alpha \in \RR$.
\end{prop}
\begin{proof} The solution $\phi$ of (\ref{Lap1}) is explicitly given by
\begin{equation}\label{laplacephi1}
\phi(z) =\int_0^\infty G_\alpha (x,z)f(x) dx 
\end{equation}
where $G_\alpha(x,z) =  - {1 \over 2 \alpha} \Bigl( e^{- \alpha | x-z | }  - e^{-\alpha | x + z |} \Bigr) $. A direct bound leads to
$$
\| \phi \|_{L^\infty} \le {C \over \alpha^2} \| f \|_{\beta}
$$
in which the extra factor of $\alpha^{-1}$ is due to the $x$-integration. 
Differentiating the integral formula, we get
$$
 \|\partial_z \phi \|_{L^\infty} \le {C \over \alpha} \| f \|_{\beta} .
$$
The estimate for $\partial_z^2 \phi$ follows by using directly the equation $\partial_z^2 \phi = \alpha^2 \phi + f$. This yields the lemma for the case when $\alpha$ is bounded away from zero. 

As for small $\alpha$, we note that $G_\alpha(0,z) = 0$ and $|\partial_x G_\alpha(x,z)|\le 1$. Hence, $|G_\alpha(x,z)|\le |x|$ and so    
$$ |\phi(z)| \le \int_0^\infty |G_\alpha (x,z)f(x)| dx  \le \|f\|_\beta \int_0^\infty |x| e^{-\beta x} \; dx \le C \| f\|_\beta.$$
Similarly, since $|\partial_z G_\alpha(x,z)|\le 1$, we get 
$$ |\partial_z\phi(z)| \le \int_0^\infty |\partial_zG_\alpha (x,z)f(x)| dx  \le \|f\|_\beta \int_0^\infty e^{-\beta x} \; dx \le C \| f\|_\beta.$$
The lemma follows. 
\end{proof}
We now establish a similar property for ${\cal B}^{\beta,\gamma,1}$ norms:
\begin{prop} \label{proplaplace3}
If $f \in {\cal B}^{\beta,\gamma,1}$ for some $\beta>0$, then $\phi \in L^\infty$. In addition, there holds
\beq \label{Lap4}
(1+| \alpha |) \, \| \phi \|_{L^\infty} 
+  \| \partial_z \phi \|_{L^\infty}   \le C \| f \|_{\beta,\gamma,1},
\eeq
where the constant $C$ is independent of $\alpha \in \RR$. 
\end{prop}
\begin{proof}
We will only consider the case $\alpha > 0$, the opposite case being similar. As above, since $G_\alpha(x,z)$ is bounded by $\alpha^{-1}$, using (\ref{laplacephi1}), we have 
$$
| \phi (z) | \le \alpha^{-1} \| f \|_{\beta,\gamma,1} \int_0^\infty
e^{- \alpha |   z -   x |} e^{-\beta   x} 
\Bigl( 1 + \delta^{-1} \phi_P(\delta^{-1} x) \Bigr) d  x
$$
$$
\le  \alpha^{-1} \| f \|_{\beta,\gamma,1} 
\Bigl( \alpha^{-1} + \delta^{-1} \int_0^\infty \phi_P(\delta^{-1} x) d   x \Bigr) 
$$
which yields the claimed bound for $\phi$ since $P >1$. A similar proof applies for $\partial_z \phi$. 
\end{proof}

%

Next, let us now turn to the two dimensional Laplace operator.
\begin{prop} \label{inverseLaplace}
Let $\phi$ be the solution of
$$
- \Delta \phi = \omega
$$
with the zero Dirichlet boundary condition, and let
$$
v = \nabla^\perp \phi. 
$$
If $\omega \in {\cal B}^{\sigma,\beta,\gamma,1}$, then $\phi  \in {\cal B}^{\sigma,0}$ and
$v  = (v_1,v_2) \in {\cal B}^{\sigma,0}$. Moreover, there hold the following elliptic estimates 
\beq \label{Lap4}
\| \phi \|_{\sigma,0} + \|  v_1 \|_{\sigma,0} + \|  v_2 \|_{\sigma,0}   \le C \| \omega \|_{\sigma,\beta,\gamma,1},
\eeq
\end{prop}
\begin{proof}
The proof follows directly from taking the Fourier transform in the $x$ variable, with dual integer Fourier component $\alpha$, and using Proposition \ref{proplaplace3}. 
\end{proof}

\chapter{Prandtl boundary layer}\label{chapter-prandtl}


In this chapter, we shall discuss the classical inviscid limit problem and present new results on the classical Prandtl's boundary layers, which are direct consequences of the results established in the previous chapters. Precisely, consider the incompressible Navier Stokes equations 
\begin{equation}\label{NSE-B1} 
\left \{ \begin{aligned}
v_t + v \cdot \nabla v + \nabla p &= \nu \Delta v 
\\
\nabla \cdot v &=0
\end{aligned}\right. \end{equation} 
posed in a spatial domain $\Omega \subset \mathbb{R}^n$ with $n \ge 2$. When there is a boundary, we impose the classical no-slip boundary condition
\begin{equation}\label{NSE-B2}
v_{\vert_{\partial \Omega}} =0 .
\end{equation}
The inviscid limit problem of $\nu\to 0$ is very classical and challenging. As $\nu\to 0$, we expect to obtain Euler solutions that solve 
\begin{equation}\label{Euler-B1} 
\left \{ \begin{aligned}
v^E_t + v^E \cdot \nabla v^E + \nabla p^E &= 
0
\\
\nabla \cdot v^E &=0
\end{aligned}\right. \end{equation} 
together with the boundary condition 
\begin{equation}\label{Euler-B2}
v^E \cdot n_{\vert_{\partial \Omega}} =0
\end{equation}
where $n$ denotes the outer unit vector normal to the boundary $\partial \Omega$. Due to the mismatch of the boundary conditions between Euler and Navier-Stokes problems, boundary layers appear, which are the main source of obstacles to understand the inviscid limit. 

We start with the inviscid limit problem on the whole space or domains without a boundary. We then focus on the case with a boundary. In particular, we recall the classical result of Kato and of Sammartino-Calfisch. We end the chapter with new instability results on the classical Prandtl's boundary layers.


\section{Whole space case}


Let us consider in this section the whole space case: $\Omega = \mathbb{R}^n$ (or domains without a boundary).  We have the following. 

\begin{theorem} 
Let $v^E$ be a smooth solution of Euler equations on $[0,T]$, 
and let $v^\nu$ be a sequence of weak solutions of Navier Stokes equations on $[0,T]$, with smooth initial data $v_0^E$ and $v_0$, respectively. 
Then, in the energy norm, solutions to Navier-Stokes converge to those of Euler as $\nu \to 0$, 
if the convergence holds at initial time. More precisely, there holds 
$$ 
\sup_{t\in [0,T]} \|v(t) - v^E(t)\|_{L^2}  \le e^{C_0 t}\|v_0 - v_0^E\|_{L^2}+ C_0 \nu, 
$$
as $\nu \to 0$, for some positive constant $C_0$ that depends on certain Sobolev norms of Euler solution and on time $T$. 
\end{theorem}    
\begin{proof} The difference $w = v - v^E$ satisfies $\nabla \cdot w =0$ and solves 
$$ w_t + (v^E + w)\cdot \nabla w + w\cdot \nabla v^E + \nabla (p - p^E) = \nu \Delta w + \nu \Delta v^E.$$
We multiply the equation by $w$ and integrate over $\Omega$. Since both $v^E$ and $w$ are divergence-free, we can compute 
$$
 \int_\Omega ((v^E + w) \cdot \nabla w) \cdot w\; dx = \int_\Omega ((v^E + w)\cdot \nabla ) \frac{|w|^2}{2} \; dx   = 0.
 $$  
This  yields 
$$ 
\frac{1}{2}\frac{d}{dt} \int_\Omega |w|^2\; dx  + \nu \int_\Omega |\nabla w|^2 \; dx 
= - \int_\Omega (w\cdot \nabla v^E) \cdot w \; dx + \nu \int_\Omega \Delta v^E \cdot w\; dx. 
$$  
Now, the first term on the right is bounded by $C_0 \| w\|_{L^2}^2$, assuming that the Euler solution has bounded derivatives. 
One can use  Hold\"er's and  Young's inequalities to treat the last term as
 $$
  \nu \int_\Omega \Delta v^E \cdot w\; dx \le \| w\|_{L^2}^2  + \nu^2 \|\Delta v^E\|_{L^2}^2.
  $$
Putting this into the above energy equality, one gets 
$$
 \frac{1}{2}\frac{d}{dt} \int_\Omega |w|^2\; dx  + \nu \int_\Omega |\nabla w|^2 \; dx \le C_0 \|w\|_{L^2}^2 + C_0 \nu^2.  
 $$ 
The standard Gronwall's inequality finishes the theorem.  
\end{proof}

\begin{remark} One could in fact obtain a strong convergence of Navier-Stokes to Euler in $H^s$ space with 
$s>\frac n2+1$ and lower the regularity of the initial data in the same $H^s$ space. 
The original proof was due to Swann \cite{Swann71} and Kato \cite{Kato72}, and later improved by Constantin \cite{Constantin86} and Masmoudi \cite{Masmoudi07}. 
Many authors study the inviscid limit problem and obtain the convergence for even less regular data; 
see, for instance, \cite{Chemin93,Chemin95,Chemin96, ConstantinWu95, ConstantinWu96} on the study of vortex patches (vorticity is a characteristic function of a smooth domain).
\end{remark}


\section{$L^2$ description: Kato's criteria}\label{sec-K}


We now study the inviscid limit problem in domains with a boundary. Recall that there is a mismatch between the boundary conditions for Euler and Navier-Stokes problems, \eqref{NSE-B2} and \eqref{Euler-B2}. In particular, the tangential component of Euler velocity $v^E$ may not vanish on the boundary, and thus a boundary layer corrector is needed. Classically, Prandtl corrected it with a boundary layer whose size is of order $\sqrt\nu$, which arises from balancing the viscous and inertial forces, leading to the celebrated Prandtl boundary layer equation. We discuss this in details in the next section. On the other hand, in order to correct the right boundary condition, Kato \cite{Kato84} introduce a  so-called ``fake''  boundary layer, with size of order $\nu$, much smaller than the classical layer size, and perform the standard energy estimates, leading to the well-known Kato's criteria for the inviscid limit to hold. Precisely, we have the following.

\begin{theorem}[Kato's criteria; \cite{Kato84}] Let $T$ be a positive time. The followings are equivalent in the inviscid limit $\nu \to 0$:

(i) $v \to v^E$ strongly in $L^2(\Omega)$, uniformly in time in $[0,T]$;

(ii) $v \to v^E$ weakly in $L^2(\Omega)$, a.e. in time in $[0,T]$;

(iii) the energy dissipation rate in $\Omega$ is vanishing: $$\nu \int_0^T \int_{\Omega} |\nabla v|^2\; dx dt \to 0;$$

(iv) the energy dissipation rate near $\partial \Omega$ is vanishing: 
$$\nu \int_0^T \int_{\{ d(x,\partial\Omega)\le \nu\}} |\nabla v|^2\; dx dt \to 0.$$

\end{theorem}

\begin{proof}
By definition, it is clear that $(i)$ implies $(ii)$ and $(iii)$ implies $(iv)$. Moreover, the energy inequality for solutions of Navier Stokes equations gives
$$
\| v(t) \|_{L^2}^2  + \nu \int_0^t \int_\Omega | \nabla v (\tau,x) |^2 dx d\tau \le \| v(0) \|_{L^2}^2.
$$
As smooth solutions of Euler equations have a constant norm, namely 
$$
\| v^E(t) \|_{L^2} = \| v^E_0 \|_{L^2},
$$
it follows that $(ii)$ implies $(iii)$.
It therefore remains to prove that $(iv)$ implies $(i)$.
For this we introduce Kato's construction of boundary layers. Kato's boundary layer is supported near the boundary: 
  $$ 
  \mathrm{supp}(v^K) \subset \Gamma_\nu : = \{ x\in \Omega ~:~ d(x,\partial\Omega) \le \nu\} .
  $$
  It may even  be chosen explicitly and somehow arbitrarly.
For instance, in the half-space domain with boundary $\{y=0\}$, one might take 
$$
v_K = - \nabla^\perp \Big( \chi (\frac{y}{\nu}) y v^E_t(x,0) \Big),
$$
where $\chi(\cdot)$ is some smooth cut-off function with support within $[0,1]$ 
and with $\chi(0) = 1$.
Note that $v^K$ is divergence-free and satisfies 
$v^K = - v^E$ on the boundary $\{y=0\}.$ 

For general domains, the construction is similar: 
$y$ is replaced by the distance function $d(x,\partial \Omega)$ to the boundary 
and $v_t^E$ is replaced by the value of the tangential component of velocity on the boundary.
Since  Kato's layer is supported in a thin strip of order $\nu$ near the boundary, we have
$$
\|v^K \|_{L^\infty}\lesssim 1, \quad \| v^K \|_{L^2} \lesssim  \nu^{1/2},
 \quad \|v^K_t\|_{L^2} \lesssim  \nu^{1/2}, \quad \|\nabla v^K \|_{L^2} \lesssim \nu^{-1/2},
 $$   
 where $v^K_t$ is the tangential component of the velocity.

Having defined $v_\mathrm{app} = v^E + v^K$, we now apply the above energy estimates
on $w = v - v^E - v^K$, yielding 
\begin{equation}\label{EE1}
 \frac 12 \frac{d}{dt} \| w\|_{L^2}^2 + \nu \| \nabla w\|_{L^2}^2  = - \int_\Omega (w\cdot \nabla v_\mathrm{app})\cdot w \; dx +  \int_\Omega R_{\mathrm{app}} \cdot w\; dx,
 \end{equation}
in which $R_\mathrm{app}$ denotes the error of "approximation", which in our case is simply 
$$ 
R_\mathrm{app}: = NS( v^E + v^K) .
$$
 Let us begin by the study of the second term of
 the right-hand side of \eqref{EE1}. Note that $R_{app}$ equals
 $$
 R_{app} = NS(v^E) + v^E . \nabla v^K + v^K . \nabla v^E + v^K . \nabla v^K - \nu \Delta v^K,
 $$ 
 with $NS(v^E) = - \nu \Delta v^E$.
 We will only bound the worst one, namely $v^K . \nabla v^K$,
 leaving the other terms to the reader. 
The most difficult term in the right-hand side of \eqref{EE1} is 
$$\begin{aligned}
 \int_\Omega (v^K \cdot \nabla v^K) \cdot w \; dx &= \int_\Omega (v^K \cdot \nabla v^K) \cdot (v - v^E - v^K) \; dx  \\
 &= \int_\Omega (v^K \cdot \nabla v^K )\cdot (v-v^E) \; dx
 \\
 &\le  \int_{\Gamma_\nu} \nu |v^K \cdot \nabla v^K| \frac{|v-v^E|}{d(x,\partial\Omega)} \; dx 
 \\
 &\lesssim \nu \| v^K \cdot \nabla v^K \|_{L^2(\Gamma_\nu)} \| \nabla (v-v^E)\|_{L^2(\Gamma_\nu)} 
 \\
 &\lesssim \sqrt \nu \|\nabla (v-v^E)\|_{L^2(\Gamma_\nu)},
  \end{aligned}$$
in which the Hardy's inequality was used. The other  terms  appear to
be easier. Namely $NS(v^E) = \nu \Delta v^E$ is straightforward to bound. 
Let us give bounds on the convection term
$$\begin{aligned}
 \int_\Omega & (w\cdot \nabla v_\mathrm{app})\cdot w \; dx 
\\ &\le \| \nabla v^E\|_{L^2} \| w\|_{L^2}^2 +  \int_\Omega (w\cdot \nabla v^K)\cdot  (v - v^E) \; dx 
 \\
 &\le \| \nabla v^E\|_{L^2} \| w\|_{L^2}^2 + \nu^2 \|\nabla v^K \|_{L^\infty} \int_{\Gamma_\nu} \Big| \frac{v-v^E}{d(x,\partial \Omega)}\Big|^2\; dx 
 \\&\quad +  \int_\Omega (v^K\cdot \nabla v^K)\cdot  (v - v^E) \; dx
 \\
 &\lesssim
  \| \nabla v^E\|_{L^2} \| w\|_{L^2}^2 +\nu 
\|\nabla (v-v^E)\|_{L^2(\Gamma_\nu)}^2 +  \sqrt \nu \|\nabla (v-v^E)\|_{L^2(\Gamma_\nu)},\end{aligned}$$ 
in which again the Hardy's inequality and the previous estimate were used. Putting all these estimates
together, one gets at once the energy inequality 
\begin{equation}\label{EE2} \frac 12 \frac{d}{dt} \| w\|_{L^2}^2 + \nu \| \nabla w\|_{L^2}^2 
\lesssim  \| w\|_{L^2}^2 
+ \sqrt \nu  \|\nabla v\|_{L^2(\Gamma_\nu)} 
+ \nu  \|\nabla v\|_{L^2(\Gamma_\nu)}^2 
+ \delta(\nu),
\end{equation}
for some constant $\delta(\nu)\to 0$ as $\nu \to 0$. 

It follows that if we assume that the anomalous dissipation {within a very thin layer near the boundary} tends to zero, namely  
\begin{equation}\label{Kato}
\nu 
\int_0^T \int_{\Gamma_\nu} |\nabla v|^2\; dx dt \to 0,\end{equation}
then inviscid limit to hold. Indeed, if \eqref{Kato} holds, the standard Gronwall's inequality yields that the $L^2$ energy norm of 
$w = v - v^E - v^K$ tends to zero, uniformly in time. As the $L^2$ energy norm of Kato layer is of order $\nu^{1/2}$, 
this proves the $L^2$ convergence of NSE to Euler solutions, under the assumption \eqref{Kato}. 
Hence $(iv)$ implies $(i)$, which ends the proof.  
\end{proof}

\bigskip

\begin{remark}
Since Kato's work, there are many variants of his criteria for the inviscid limit to hold. This includes (i) several equivalent formulations in term of vorticity  \cite{Kelliher07, Kelliher17}, (ii) the convergence to Lions' dissipative solutions of Euler equations is equivalent to the weak convergence of the vorticity component $\nu (\omega \times n)\cdot \tau $ on the boundary \cite{BardosTiti}, among many others.

\end{remark}

%

%
%
%
%
%
%
%

 Kato's criteria and its variants appear equally hard to verify. 
 They do however indicate that there would be violent behaviors occurring in the thin layer of size $\nu$ near the boundary, 
 should the inviscid limit fail to hold. The size $\nu$ is much smaller than the size $\sqrt \nu$ of the classical boundary layer, 
 predicted by Prandtl's theory. For instance, Prandtl predicted that a vorticity
of order $\nu^{-1/2}$ is created near the boundary, whereas the mentioned criteria indicates 
that the vorticity would be much larger, of order $\nu^{-1}$ near the boundary, should the inviscid limit fail.


\section{$L^\infty$ description: Prandtl's theory}


Prandtl in 1904 theorized that a slightly viscous Navier-Stokes flows can be decomposed into an Euler flow, plus a boundary layer corrector. Let us detail this point. 
For the sake of presentation, we consider the Navier Stokes equations in the half-space
$$
\Omega = \mathbb{R}^{n-1} \times \mathbb{R}_+
$$
 with the flat boundary 
 $
 \{ y=0\}.
 $
  We then write 
  $$
  \vec x = (x, y)\in  \mathbb{R}^{n-1} \times \mathbb{R}_+, \qquad   \vec v = (u,v)\in \mathbb{R}^{n-1}\times \mathbb{R}.
  $$ 
The  Navier Stokes equations written in the component-wise read 
\begin{equation}\label{NS-vector} 
\left \{ \begin{aligned}
u_t + (u\cdot \nabla_{x} + v \partial_y) u + \nabla_x p & = \nu \Delta u 
\\
v_t + ( u \cdot \nabla_x + v \partial_y) v + \partial_y p &= \nu \Delta v 
\\
\nabla_x \cdot u + \partial_y v &= 0
\\
(u,v)_{\vert_{y=0}} &= 0.
\end{aligned}\right. \end{equation} 
Let $\vec v^E = (u^E, v^E)$ be the solution to the corresponding Euler equation 
(that is, the equation \eqref{NS-vector} above, with $\nu =0$ and with zero boundary condition on $v^E$). We stress that there is no boundary condition on the tangential component $u^E$ is needed. In particular, $u^E$ needs not to vanish on the boundary $y=0$,
 and hence a boundary layer is needed to correct this discrepancy of the zero boundary conditions for Navier-Stokes. 

\subsection{Prandtl's Ansatz}
In 1904, Prandtl \cite{Prandtl04} postulated the following solution Ansatz: 
\begin{equation}\label{Ansatz} \begin{pmatrix} 
u\\v \end{pmatrix} = \begin{pmatrix} u^E\\v^E \end{pmatrix} (t,x,y)+ \begin{pmatrix} u^P \\  v^P \end{pmatrix}(t,x,\frac y{\sqrt\nu}) + o(1)_{L^\infty}
\end{equation}
for a boundary layer corrector $(u^P,v^P)$ with size of order $\sqrt \nu$, up to a remainder that tends to zero in the inviscid limit. We shall now derive equations for $(u^P,v^P)$. Denote the fast variable by 
$$
z= \frac{y}{\sqrt\nu}
$$
The boundary conditions for the boundary corrector $(u^P, v^P)$ are to ensure the no-slip boundary condition for Navier-Stokes solutions \eqref{NSE-B2}; namely, 
$$ 
u^P_{\vert_{z=0}} = - u^E_{\vert_{y=0}}, \qquad v^P_{\vert_{z=0}} = 0, 
\qquad \lim_{z\to \infty} u^P(t,x,z) = 0.
$$
Next, observe that for boundary layers, $ \partial_y= \frac 1{\sqrt\nu}\partial_z$.
Thus, from the divergence-free condition
 $$
 \nabla_x\cdot u^P + \frac 1{\sqrt\nu} \partial_z v^P = 0.
 $$ 
 This balance shows that the vertical velocity of $v^P$ is of order ${\sqrt\nu}$, 
 as compared to the order one of the horizontal component $u^P$. Thus, we rescale $v^P$ by writing 
 $$
 v^P = {\sqrt\nu} \tilde v^P,
 $$
  for some new $\tilde v^P$. The divergence-free condition becomes 
  $$ \nabla_x \cdot u^P + \partial_z \tilde v^P = 0.$$
Therefore, to leading order, the Navier-Stokes equation for the vertical component of velocity simply reads 
$$ \frac 1 {\sqrt\nu} \partial_z p^P = 0,$$ 
if we denote by $p^P$ the pressure for the Prandtl layer. This yields at once that 
$$
p^P = p^P(t,x),
$$
that is, the pressure does not depend on the vertical direction within the boundary layer. Next, to leading order, the equation for the horizontal component of velocity becomes
$$ 
u^P_t + (u^E + u^P)\cdot \nabla_{x} u^P  + u^P \cdot \nabla_x u^E + (\frac{v^E}{{\sqrt\nu}} + \tilde v^P) \partial_z u^P + \nabla_x p^P  =\partial_z^2 u^P 
$$
In which we note that $\nu^{-1/2} v$ can be approximated by $z \partial_y v^E(t,x,0)$, since $v^E(t,x,0) = 0$. 
As the above equation is defined near the boundary, one can replace $u^E$ by its value on the boundary $u^E(t,x,0)$ (to the leading order).
 The above set of equations is called the Prandtl boundary layer equation, a single equation for the horizontal component of velocity $u^P$, 
 with its vertical velocity $\tilde v^P$ being defined through the divergence-free condition. 
 Note that since $u^P$ vanishes away from the boundary, the above equation yields that $p^P$ is simply a constant.
  That is, the pressure is not an unknown in the Prandtl boundary layer equation. It turns out to be convenient to denote the following Prandtl's variables: 
 \begin{equation}\label{def-Pvar} 
 \bar u = u^E (t,x,0)+ u^P(t,x,z), \qquad \bar v = z \partial_y v^E(t,x,0) + \tilde v^P(t,x,z) .
 \end{equation}
The Prandtl boundary layer problem simply becomes
 \begin{equation}\label{def-Peqs}
\left \{ 
\begin{aligned} 
\bar u_t +\bar  u \cdot \nabla_x\bar  u + \bar v \partial_z\bar  u  &=  \partial_z^2\bar  u -  \nabla_x\bar p
\\
\nabla_x \cdot \bar u + \partial_z \bar v &=0
\\
 (\bar u,\bar v)_{\vert_{z=0}}  & =0
 \\
  \lim_{z\to \infty} \bar u(t,x,z) & = u^E(t,x,0)
\end{aligned}
\right. 
\end{equation}
in which the pressure is known: 
$$ 
\bar p= -u_t^E(t,x,0) - \frac{|u^E(t,x,0)|^2}{2} .
$$ 
  One can eliminate the vertical velocity $\bar v$ by writing 
  \beq \label{barv}
  \bar v = - \int_0^z \nabla_x \cdot \bar u \; dz,
  \eeq
   thanks to the divergence-free condition. The above set of equations and boundary conditions is complete. The tangential velocity component $\bar u$ is the only unknown scalar function, and the normal velocity component $\bar v$ is defined through the divergence-free condition \eqref{barv}.  

The existence and uniqueness of solutions to the Prandtl equations 
have been constructed for monotonic data by Oleinik  \cite{Ole} in the sixties. There are also recent reconstructions \cite{Alex, MW} of Oleinik's solutions via a more direct energy method. For data with analytic or Gevrey regularity, the well-posedness of the Prandtl equations is established in \cite{SammartinoCaflisch1, GVMasmoudi15}, among others. In the case of non-monotonic data with Sobolev regularity, the Prandtl boundary layer equations are known to be ill-posed (\cite{GVDormy,GVN1,GN1}). 

\subsection{Sammartino-Caflisch's analytic solutions}

Concerning the validity of Prandtl's Ansatz \eqref{Ansatz}, this was established for data with analytic regularity in the celebrated work of Caflisch and Sammartino \cite{SammartinoCaflisch2}. In particular, it was proven that if a boundary layer Ansatz exists to describe the limiting
behavior of $u^\nu$, then it must be of the Prandtl's form \eqref{Ansatz}. 
A similar result were also obtained by \cite{Mae} for data whose initial vorticity is compactly supported away from the boundary. The stability of shear flows under perturbations with Gevrey regularity is recently proved in \cite{GVM}.

However these positive results hide a strong instability occurring at high spatial frequencies. For some profiles, instabilities
with horizontal wave numbers of order $\nu^{-1/2}$ grow like $\exp(C t / \sqrt{\nu})$. Within an analytic framework, these
instabilities are initially of order $\exp(-D / \sqrt{\nu})$ and grow like $\exp( (Ct - D) /\sqrt{\nu})$. They remain negligible in bounded
time (as long as $t < D/ 2 C$ for instance). 

Within Sobolev spaces, these instabilities are predominant. In the next section, we shall present our recent instability results for data with Sobolev regularity.

\section{Instability results}


In this section, we present our instability results for data with Sobolev regularity. In particular, we prove that the Prandtl's Ansatz \eqref{Ansatz} is false and the boundary layer asymptotic expansions are invalid. Precisely, we consider the Navier-Stokes equations with forcing $f^\nu$:
\beq \label{NS-B1}
\partial_t u^\nu + (u^\nu \cdot \nabla) u^\nu - \nu \Delta u^\nu + \nabla p^\nu = f^\nu,
\eeq
\beq \label{NS-B2}
\nabla \cdot u^\nu = 0
\eeq
and Dirichlet boundary condition
\beq \label{NS-B3} 
u^\nu = 0 \qquad \hbox{on} \qquad y = 0 .
\eeq
The data we construct are near a shear layer profile which is a solution of the form
$$
U^\nu(t,x,y) = \begin{pmatrix}U(t,y/ \sqrt{\nu}) \\ 0 \end{pmatrix} 
$$
that is a solution of both Prandtl and Navier Stokes equations. Here $U(t,y)$ is a smooth function with $U(t,0) = 0$ such that $U(t,y)$ converges when $y \to + \infty$
to a constant Euler flow $U_\infty$. We consider two cases of shear layers:

\begin{itemize} 

\item time dependent boundary layers: $U^\nu$ solves
$$
\partial_t U^\nu - \partial_{yy} U^\nu = 0,
$$
namely $U^\nu$ is a solution of the classical heat equation

\item time independent boundary layers: in this case, we add  a time-independent forcing term which compensates for
the viscosity. Precisely, we take 
\begin{equation}\label{def-source}
F^\nu = \begin{pmatrix} - U''(0,y / \sqrt{\nu}) \\0\end{pmatrix} .
\end{equation}

\end{itemize}

~\\
There are two kinds of instability we establish: 

\begin{itemize}

{\em \item Inviscid instability:} An instability with horizontal wave numbers of order $O(\nu^{-1/2})$, growing like $\exp(C t / \sqrt{\nu})$.
This instability occurs when the profile $U$ is linearly unstable with respect to the inviscid Euler equations, or, very roughly speaking, 
when the profile $U$ has a "strong" inflexion point (Rayleigh's instability criterium). For such profiles 
E. Grenier proved in \cite{Grenier00CPAM} that Prandtl's asymptotic expansion is false,  up to a remainder of order $\nu^{1/4}$ in $L^\infty$ norm. 
In the next subsection, we are able to replace the lower bound $O(\nu^{1/4})$ by $O(1)$, and prove that the difference between
the genuine solution and the Prandtl's expansion may be of order one in supremum norm, and that this difference does not
vanish as $\nu$ goes to $0$. The Prandtl's Ansatz \eqref{Ansatz} is thus false. These instabilities are driven by inviscid instabilities occurring within the boundary layer.

{\em \item Secondary instability:} An instability with horizontal wave numbers of order $O(\nu^{-1/8})$, growing like $\exp(C t / \nu^{1/4})$. 
These instabilities are much more subtle, since the profile $U$ may be linearly stable with respect to the Euler equations. The growth is much slower, and the instability is driven by the so called "critical layer", which
is at a distance $O(\nu^{5/8})$ from the boundary. We prove the invalidity of boundary asymptotic expansions, an equivalence of the instability result established in \cite{Grenier00CPAM}, for monotone shear profiles.  

\end{itemize}

\subsection{Inviscid instability}

Our first instability result is as follows. 

\begin{theo}[The failure of Prandtl's Ansatz] \label{maintheo-Linfty}
There exists a smooth, analytic function $U(0,Y)$, such that the corresponding  sequence of  time dependent shear layers 
\beq \label{shear}
U^\nu(t,x,y) = \begin{pmatrix} U(t,y / \sqrt{\nu}) \\0 \end{pmatrix},
\eeq
which are smooth solutions
of Navier Stokes, Prandtl, and heat equations satisfies the following assertion. For any $N$ and $s$ arbitrarily 
large, there exist $\sigma_0 > 0$, $C_0 > 0$ 
and a sequence of solutions $u^\nu$ of Navier Stokes equations \eqref{NS-B1}-\eqref{NS-B3} with forcing terms $f^\nu$,
 on some interval $[0,T^\nu]$, such that
$$
\| u^\nu(0) - U^\nu(0) \|_{H^s} \le \nu^N,
$$
$$
\| f^\nu \|_{L^\infty([0,T^\nu],H^s)} \le \nu^N,
$$
but
$$
\| u^\nu(T^\nu) - U^\nu(T^\nu) \|_{L^\infty} \ge \sigma_0
$$
and
$$
T^\nu = O ( \sqrt{\nu} \log \nu^{-1} ).
$$
\end{theo}

Theorem \ref{maintheo-Linfty} is a direct consequence of the isotropic scaling 
$$(T,X,Z) = \frac{1}{\sqrt{\nu}} (t,x,z)$$
and Theorem \ref{theoinstable} on the instability of shear flows. Theorem \ref{maintheo-Linfty} proves that Prandtl Ansatz is false in $L^\infty$ in very small times, of order $\sqrt\nu \log \nu^{-1}$. The same theorem holds true for the time independent boundary layer with a forcing term (\ref{def-source}). In particular, it is proved that the convergence of Navier-Stokes solutions to Euler solutions, plus a boundary layer, fails in $L^\infty$ in the inviscid limit. We remark that this however does not prevent the convergence to hold in $L^p$ for $p<\infty$. In addition, the Navier-Stokes solutions obtained in Theorem \ref{maintheo-Linfty} involve not only the Prandt's layer of size $\sqrt\nu$, but also a viscous sublayer of size $\nu^{3/4}$. It is also important to note that the instability occurs in the vanishing time $T^\nu$ of order $\sqrt\nu \log\nu^{-1}$.

\subsection{Secondary instability}

In this section, we establish the nonlinear instability of the Ansatz \eqref{Ansatz} near monotonic profiles. 
Roughly speaking, given an arbitrary stable boundary layer, the two main results are 
\begin{itemize} 

\item in the case of time-dependent boundary layers, we construct Navier-Stokes solutions, with arbitrarily small forcing, of order $\mathcal{O}(\nu^P)$, with $P$ as large as we want,
so that the Ansatz \eqref{Ansatz} is false near the boundary layer, 
up to a remainder of order $\nu^{1/4+\epsilon}$ in $L^\infty$ norm, $\epsilon$ being arbitrarily small. 

\item in the case of stationary boundary layers, we construct Navier-Stokes solutions, without  forcing term, 
so that the Ansatz \eqref{Ansatz} is false, up to a remainder of order $\nu^{5/8}$ in $L^\infty$ norm. 

\end{itemize}

These results prove that there exist no asymptotic expansion of Prandtl's type, even in the case of monotonic profiles. For such 
profiles adding viscosity destabilizes the flow, which is counter intuitive.
Even in Prandtl boundary layer equation is well posed, it does not describe the limiting behavior as the viscosity goes to $0$.
 

Our main results are as follows. 
 
\begin{theorem}\label{theo-approxBL} 
Let $U_\bl(t,z)$ be a time-dependent stable boundary layer profile. Assume as in Theorem \ref{theo-approx}.  
Then, for arbitrarily large $s,N$ and arbitrarily small positive $\epsilon$, there exists
a sequence of functions $u^\nu$ that solves the Navier-Stokes equations \eqref{NS-B1}-\eqref{NS-B3}, with some forcing $f^\nu$, so that 
$$
\| u^\nu(0) - U_\bl(0) \|_{H^s} + \sup_{t\in [0,T^\nu]} \| f^\nu (t)\|_{H^s} \le \nu^N,
$$
but 
$$
\| u^\nu(T^\nu) - U_\bl(T^\nu) \|_{L^\infty} \ge \nu^{\frac14+\epsilon},
$$
$$\| \omega^\nu(T^\nu) - \omega_\bl(T^\nu) \|_{L^\infty} \to \infty,$$
for time sequences $T^\nu \to 0$, as $\nu \to 0$. Here, $\omega^\nu = \nabla\times u^\nu$ denotes the vorticity of fluids. 
\end{theorem}

This Theorem proves that the Ansatz \eqref{Ansatz} is false, even near stable boundary layers, for data with Sobolev regularity. 

\begin{theorem}[Instability result for stable profiles] \label{firstinstability-stable}
Let $U_\bl(t,z)$ be a time-dependent stable boundary layer profile. Assume as in Theorem \ref{theo-approx}.  
Then, for any $s,N$ arbitrarily large, there exists a sequence of solutions $u^\nu$ to the Navier-Stokes equations \eqref{NS-B1}-\eqref{NS-B3}, 
with forcing $f^\nu = f^P$ (boundary layer forcing), so that $u^\nu$ satisfy 
$$
\| u^\nu(0) - U_\bl \|_{H^s} \le \nu^N,
$$
but 
$$\| u^\nu (T^\nu) - U_\bl\|_{L^\infty} \gtrsim \nu^{5/8},$$
$$\| \omega^\nu (T^\nu) - \omega_\bl \|_{L^\infty} \gtrsim 1,$$
for some time sequences $T^\nu \to 0$, as $\nu \to 0$. 
\end{theorem}

The spectral instability for stable profiles gives rise to sublayers (or critical layers) whose thickness is of order $\nu^{5/8}$. 
The velocity gradient in this sublayer grows like $\nu^{-5/8} e^{t/\nu^{1/4} }$, and becomes large when $t$ is of order $T^\nu$.
As a consequence, they may in turn become unstable after the instability time $T^\nu$ obtained in the above theorem. 
Thus, in order to improve the $\nu^{5/8}$ instability, one needs to further examine the stability properties 
of this sublayer itself. Again, these instability results are a direct consequence of the instability results for shear profiles. 

\chapter{Boundary layer theory}\label{chapter-Prandtl}

 \def\Airy{\mathrm{Airy}}
 \def\OS{\mathrm{OS}}
 \def\Ray{\mathrm{Ray}}
 \def\Iter{\mathrm{Iter}}

With simple integrations by parts, we were able to see in the last chapter essentially the current ``state of art" of the $L^2$ 
convergence of Navier-Stokes to Euler. The inviscid limit problem is widely open as discussed. It is noted that the $L^2$ 
energy norm is quite weak, 
and does not see in the inviscid limit the appearance of thin layers that might (and indeed will) occur near the boundary (for instance, 
the $L^2$ norm of Kato's layer is of order $\sqrt \nu$). We will have to work with a different, stronger norm. 
Regarding the significance of viscosity despite being arbitrarily small (e.g., viscosity of air at zero temperature is about $10^{-4}$, 
which seems to be neglectable), d'Alembert in the 18th century has already argued out that ideal flows can't explain well many of the physics,
 and that the viscosity plays a crucial role near the boundary; for instance, one of his conclusions, known as d'Alembert's paradox, 
 asserts that solid body emerged in stationary ideal flows feels no drag acting on it (in the layman words, birds can't fly!). 
 Not until the beginning of the 20th century, Prandtl then postulated a solution Ansatz that revolutionized the previous understanding
  of slightly viscous flows near a boundary, later known as Prandtl boundary layer theory.  
  The theory gave birth to the field of aerodynamics, and is  regarded as one of the greatest achievements
   in fluid dynamics in the last century. Below, we shall derive the Prandtl boundary layers.


\section{Regular versus singular perturbations}


Let us first discuss the notation of regular versus singular perturbation.
Regular perturbations are perturbations that preserve the nature of the equations. For instance, consider the problem
$$ 
A v - \epsilon B v = f,\qquad \epsilon \to 0.
$$
Assume that $A,B$ are two bounded linear operators from a space $X$ to another space $Y$, such that the domain of $A$ is  contained 
in the domain of $B$. For simplicity, assume that $A$ is invertible.
The solution for the perturbed problem can be constructed in term of the series in $\epsilon$:
\begin{equation}\label{series}  
v= (A-\epsilon B)^{-1}f  
\end{equation}
$$
= A^{-1}(1-\epsilon BA^{-1})^{-1}f = A^{-1} f + A^{-1} \sum_{k = 1}^\infty (\epsilon B A^{-1})^{k} f .
$$
When $\epsilon \to 0$, the solution of the perturbed problem simply converges to the solution of the unperturbed problem. 
 The quantitative nature of the problem is unchanged 
when $\epsilon \to 0$ and no new behavior arises in the limit.

\medskip

A perturbation is singular when $A,B$ are qualitatively not of the same type. For instance, let us study  the following simple algebraic problem
$$ 
a x - \epsilon  x^2 = f,
$$
where $a$ is a nonzero constant. Beside the regular solution of the form 
$x = a^{-1} f + \mathcal{O}(\epsilon)$, 
there is a { new solution} due to the singular perturbation, which is of the form: 
 $$ 
 x = \frac{a}{\epsilon}  - \frac fa + \mathcal{O}(1) .
 $$
(note that the first term in the expansion comes from balancing $ax - \epsilon x^2 = 0$). 
Clearly, this solution does not converge to the solution of the unperturbed problem as $\epsilon \to 0$. 

Boundary layers arise in singular perturbations when the leading operator vanishes in the limit. Let us now give a very simple example
of boundary layer.


\section{A simple example of boundary layer}


Consider the following problem on the half-line $x\ge 0$:
$$\begin{aligned}
  v + \epsilon \partial_x v &= 1,
 \\
 v_{\vert_{x=0}} &= v_0.
 \end{aligned}$$
 When $\epsilon =0$, the constant function $v^E  = 1$ is the only solution to the limiting equation, 
 but it does not solve the boundary (or rather, initial) condition at $x=0$, except if $v_0$ identically vanishes. 
 
 When $\epsilon>0$, the ODE problem has a unique smooth solution,
  which is equal to the prescribed boundary condition $v_0$ and is well approximated by $v^E =1$ away from the boundary. 
  As $v_0$ might be different from $v^E=1$, a boundary layer appears near the boundary to correct the discrepancy. 
  We have to deal with a singular perturbation.
  
  To find the equation for this boundary layer, we postulate a solution Ansatz: 
 $$
 v  = v^E(x) + v^\mathrm{bl} (\frac x\delta)
 $$ 
 for some $\delta$ which goes to $0$ as $\epsilon \to 0$. Plugging this Ansatz into the equation, one gets the ``boundary layer equation'':
 $$
 v^\mathrm{bl} + \frac{\epsilon}{\delta} \partial_z v^\mathrm{bl} (z) = 0, \qquad v^\mathrm{bl}_{\vert_{z=0}} 
 = v_0 - v^E_{\vert_{x=0}}, \qquad \lim_{z\to \infty} v^\mathrm{bl} (z) =0,
 $$
  where $z$ denotes the fast variable $z = x/\delta$. To balance the above equation,
  we choose the ``boundary layer thickness'' $\delta$ to be
  $$
  \delta  = \epsilon.
  $$
   We get the solution 
   $$
   v^\mathrm{bl} = (v_0 - 1) e^{-x/\epsilon},
   $$
    which is called a boundary layer. By construction, the solution ansatz 
    $$
    v^{app} =   v^E(x) + v^\mathrm{bl} (\frac x\epsilon)
    $$ 
    satisfies the boundary condition and approximately solves the ODE equation (in this example, as
    the equation is linear and as
     $\partial_x v^E =0$, the ansatz in fact  exactly solves the equation). 
     We remark that we would miss out this boundary layer, if we were doing $L^2$ estimates. More precisely
     $$
     (v - v_E)   + \eps \partial_x (v - v_E) = 0 .
     $$
     Hence
 $$ \int_0^\infty | v - v^E|^2\; dx =  - \epsilon \int_0^\infty \partial_x (v-v^E) (v - v^E)\; dx = \frac \epsilon 2 |v_0 - v^E|^2 $$ 
which converges to zero as $\epsilon \to 0$, for arbitrary initial data $v_0$.  

Let us give another example of a singular perturbation, which will later come up in our study 
(for instance, when studying the resolvent equation of some linearized operators around a given profile solution). 
Consider the ODE problem on the half-line $x\ge 0$:
$$\begin{aligned}  v + a \partial_x v - \epsilon \partial^2_x v &= f(x)
\\
v_{\vert_{x=0}} &=0.
\end{aligned}$$    
Assume that the given source $f(x)$ is sufficiently smooth and integrable. Here, $a$ is some nonzero constant, 
and $\epsilon$ is positive and sufficiently small. When $\epsilon =0$, the unique bounded solution of the ODE equation is defined by 
$$ v^E(x) = \left \{ \begin{aligned} \frac 1a \int_0^x e^{-(x-y)/a} f(y)\; dy , &\qquad a>0 \\
- \frac 1a \int_x^\infty e^{-(x-y)/a} f(y)\; dy , &\qquad a<0,
\end{aligned}\right. $$
Consider the case $a>0$. In this case, the inviscid solution $v^E$ satisfies the zero boundary condition and there is no need
 to add a boundary layer to correct the boundary condition. The solution of the perturbed ODE can be constructed via a formal series 
 of the form \eqref{series}, yielding the estimate $v = v^E(x) + \mathcal{O}(\epsilon)$ in the $L^\infty$ norm. 

Next, consider the case $a<0$. The inviscid solution $v^E$ defined as above does not in general satisfy the zero boundary condition,
 and so a boundary layer is needed. Again, let us denote the boundary layer solution by $v^\mathrm{bl}(\delta^{-1} x)$.
 Then, the boundary layer thickness is $\delta = \epsilon$ and $v^\mathrm{bl}(z)$ solves the boundary layer problem 
$$ a \partial_z v^{\mathrm{bl}}  - \partial_z^2 v^\mathrm{bl} = 0, \qquad   v^\mathrm{bl}_{\vert_{z=0}} 
= - v^E_{\vert_{x=0}}, \qquad \lim_{z\to \infty} v^\mathrm{bl} (z) =0 ,$$
which simply gives $v^\mathrm{bl} = -v^E(0) e^{a x / \epsilon}$.
It can be proved that the exact solution of the perturbed ODE problem satisfies 
$$
v = v^E (x)+ v^\mathrm{bl}(\frac x \epsilon) + \mathcal{O}(\epsilon)
$$ 
in the $L^\infty$ norm. 

\medskip

In this section we presented two simple examples of boundary layers, together with an usual method to investigate it, namely
to construct approximate solutions and to match asymptotic expansions. This approach appears to be awkward in the case
of Navier Stokes equations. Instead of constructing approximate solutions and matched expansions, we prefer an operator approach
that we will introduce in the next section


\section{Operator approach to boundary layers}\label{sec-opBL}


In this section we introduce our strategy on a very simple example.
Namely we
consider the toy problem 
$$
L\phi: = (1 - \partial_z - \epsilon \partial_z^2) \phi = f, \qquad \forall ~z\ge 0,
$$
with the zero boundary condition on $z=0$. Our iterative construction starts with the inviscid solution. Indeed, let $\phi_0$ be the (bounded) inviscid solution $\phi_0$
of $$
\phi_0 - \partial_z \phi_0 = f .
$$
 Due to the outward characteristic to the boundary $\{z=0\}$, bounded inviscid solutions
  need not satisfy the zero boundary condition and the viscous equation. 
  Boundary Layers are needed, solving 
$$ -(\partial_z + \epsilon \partial_z^2) \phi_{bl} = g,$$
for some source $g$, which comes from the error introduced by the inviscid solution. 
Precisely, 
$$ L\phi_0 = f  - \epsilon \partial_z^2 \phi_0  = f  - \epsilon \partial_z^2 (1-\partial_z)^{-1} f ,$$
leaving the error term: 
$\epsilon \partial_z^2 (1-\partial_z)^{-1} f $, 
which might not be well-defined if $f$ is not smooth enough 
(and so not contractive for the iteration to work!).
We then correct this by introducing the iteration operator: 
\begin{equation}\label{Iter-toy}
\Iter: = -  \underbrace{(\epsilon \partial_z^2  + \partial_z )^{-1}}_{\mbox{boundary layer}} \circ \quad \underbrace{\epsilon \partial_z^2}_{\mbox{error}} \quad \circ \quad \underbrace{(1-\partial_z)^{-1}}_{\mbox{inviscid}}.
\end{equation}
Note that formally $Iter$ is a zeroth order operator.
The Boundary Layer operator is added to regularize the possible singularity coming from the loss of derivatives in the error term
$\partial_z^2 (1-\partial_z)^{-1}f$. 
Note that
$$
L (\Iter(f)) =  -  (\eps \partial_z^2 + \partial_z) \Iter(f) + \Iter(f)
$$
$$
=  \eps \partial_z^2 (1 - \partial_z)^{-1} f + \Iter(f).
$$
We can check that the new approximate solution 
$$
\phi_1: = \phi_0 + \Iter(f)
$$ 
solves 
$$ 
L\phi_1 = f + \Iter(f) .
$$
Inductively, let $E_n = \Iter(f)^n$ be the error at the $n^{th}$ step and define 
$$
\phi_{n+1} = \phi_{n} + (1-\partial_z)^{-1} E_{n} + \Iter(E_n) 
$$ to be the next iterative solution. There holds 
$$ 
L\phi_{n+1} = f + \Iter(f)^{n+1}.
$$ 
It suffices to show that $\Iter(\cdot)$ is well-defined and contractive in some function space, 
and so $E_n \to 0$ and $\phi_n \to \phi_\infty$ as $n\to \infty$, giving an exact solution 
$\phi_\infty$ to the problem $L \phi =f$. Let us work with the function space $X = L^{\infty}$. The iteration solves 
$$
(\epsilon \partial_z^2  + \partial_z ) \Iter(f) = \epsilon \partial^2_z g,
$$
with $g = (1-\partial_z)^{-1}f$, which gives 
$$ 
\Iter(f)(z) = e^{-z/\epsilon} \Iter(f)(0) + \int_0^z e^{-(z-y) /\epsilon} \partial_y g \; dy.
$$  
The boundary value $\Iter(f)(0)$ can be taken to be the minus of that of the inviscid solution. For simplicity, we take $\Iter(f)(0)=0$. The iteration is contractive in $L^\infty$, since 
$$ 
|\Iter(f)(z)| = \Big| \int_0^z e^{-(z-y) /\epsilon} \partial_y g \; dy  \Big| \le \| \partial_z g \|_\infty   
\int_0^\infty  e^{-|z-y|/\epsilon}  \; dy =  \epsilon \| \partial_z g \|_\infty 
$$
and 
$$
\partial_z g  = \partial_z (1-\partial_z)^{-1} f = - f + (1-\partial_z)^{-1} f ,
$$ 
which is  bounded by $C \|f\|_{L^\infty}$. This proves that the Iter operator is contractive for sufficiently small $\epsilon$.
Pecisely there holds 
$$
\|\Iter(f)\|_{L^\infty} \le C \epsilon \|f\|_{L^\infty}
$$ 
for all bounded functions $f$. For small $\eps$, $Iter$ is contractive. Therefore the series $\sum \phi_n$ converges, towards a solution
of our problem. The asymptotic expansion $\sum \phi_n$ gives a complete description of the solution since $\phi_n$ is of order
$O(\eps^n)$.

\chapter{Nonlinear instability: the stable case}\label{chapter-nonlinear-stable}

\def\app{\mathrm{app}}


\section{Introduction}


In this chapter, we prove the nonlinear instability of monotonic shear profiles. 
Roughly speaking, the two main results are
\begin{itemize} 

\item in the case of time-dependent shear flows, we construct Navier-Stokes solutions, with arbitrarily small forcing, of order $\mathcal{O}(\nu^P)$, with $P$ as large as we want,
so that the shear flows are nonlinear unstable up to order $\nu^{1/2+\epsilon}$ in $L^\infty$ and $L^2$, with $\epsilon$ being arbitrarily small. 

\item in the case of stationary shear flows, we construct Navier-Stokes solutions, without  forcing term, 
so that the shear flows are nonlinear unstable up to order $\nu^{5/4}$ in $L^\infty$ and $L^2$. 

\end{itemize}

%
%
%
%


\subsection{Stable shear profiles}\label{sec-BL}



%

We consider {\em stable shear profiles}; namely, those that are spectrally stable to the Euler equations. This includes, for instance, boundary layer profiles without an inflection point 
by view of the classical Rayleigh's inflection point theorem. We also assume that $U(z)$ is strictly monotonic, real analytic, that $U(0) =0$ 
and that $U(z)$ converges exponentially fast at infinity to a finite constant $U_+$. 
By a slight abuse of language, such profiles will be referred to
as stable profiles throughout this book.

In order to study the instability of such shear layers, we first analyze the spectrum of the corresponding linearized problem 
around initial profiles $U(z)$. This leads to the following linearized problem for vorticity $\omega = \partial_z v_1 - \partial_x v_2$, which reads 
\begin{equation}\label{EE-vort} 
(\partial_t - L)\omega = 0, \qquad L\omega : = \nu \Delta \omega - U\partial_x \omega - v_2 U'' ,
\end{equation}
together with $v = \nabla^\perp \phi$ and $\Delta \phi = \omega$, satisfying the no-slip boundary conditions 
$\phi = \partial_z \phi =0$ on $\{z=0\}$. 

We then take the the Fourier transform in the $x$ variable only, denoting by $\alpha$ the corresponding wavenumber,
which leads to
\begin{equation}\label{EE-vort-a}
 (\partial_t - L_\alpha)\omega = 0, \qquad L_\alpha\omega : = \nu \Delta_\alpha \omega - i\alpha U \omega - i\alpha \phi U'' ,
 \end{equation}
where
$$
\omega_\alpha = \Delta_\alpha \phi_\alpha,
$$
together with the zero boundary conditions 
$\phi_\alpha = \phi'_\alpha =0$ on $z=0$. Here, 
$$
\Delta_\alpha = \partial_z^2 - \alpha^2. 
$$
Together with Y. Guo, we proved in \cite{GGN1,GGN3} that, even for profiles $U$ which are stable as $\nu = 0$, 
there are unstable eigenvalues to the Navier-Stokes problem \eqref{EE-vort-a} for sufficiently small viscosity $\nu$ 
and for a range of wavenumber $\alpha \in [\alpha_1,\alpha_2]$, with $\alpha_1\sim \nu^{1/4}$ and $\alpha_2\sim \nu^{1/6}$. 
The unstable eigenvalues $\lambda_*$ of $L_\alpha$, found in \cite{GGN3}, satisfy 
\begin{equation}\label{lambda-GGN} 
\Re \lambda_* \sim \nu^{1/2} .
\end{equation}
Such an instability was first observed by Heisenberg \cite{Hei,HeiICM}, then Tollmien and C. C. Lin \cite{Lin0,LinBook}; 
see also Drazin and Reid \cite{Reid,Schlichting} for a complete account of the physical literature on the subject. 
See also Theorem \ref{theo-spectralinstablity} below for precise details. In coherence with the physical literature \cite{Reid},
we believe that, $\alpha$ being fixed,
 this eigenvalue is the most unstable one. However, this point is an open question from the mathematical point of view.
 
Next, we observe that $L_\alpha$ is a compact perturbation of the Laplacian $ \nu \Delta_\alpha$, and hence its unstable spectrum in the usual $L^2$ space is discrete. Thus, for each $\alpha,\nu$, we can define the maximal unstable eigenvalue $\lambda_{\alpha,\nu}$ so that $\Re \lambda_{\alpha,\nu}$ is maximum. We set $\lambda_{\alpha,\nu} =0$, if no unstable eigenvalues exist. 

In this paper, we assume that the unstable eigenvalues found in the spectral instability result, Theorem \ref{theo-spectralinstablity}, 
are maximal eigenvalues. Precisely, we introduce 
\begin{equation}\label{def-ga0max0}
\gamma_0 : = \lim_{\nu \to 0} \sup_{\alpha \in \RR}  \nu^{-1/2}\Re \lambda_{\alpha,\nu}  .
\end{equation}
The existence of unstable eigenvalues in Theorem \ref{theo-spectralinstablity} implies that $\gamma_0$ is positive. 
Our spectral assumption is that $\gamma_0$ is finite (that is, the eigenvalues in Theorem \ref{theo-spectralinstablity} are maximal).


\subsection{Main results}


We are ready to state the main instability results for stable shear layers.

\begin{theorem}\label{theo-approx} 
Let $U_\bl(t,z)$ be a time-dependent stable shear layer profile as described in Section \ref{sec-BL}. 
Then, for arbitrarily large $s,N$ and arbitrarily small positive $\epsilon$, there exists
a sequence of functions $u^\nu$ that solves the Navier-Stokes equations, with some forcing $f^\nu$, so that 
$$
\| u^\nu(0) - U_\bl(0) \|_{H^s} + \sup_{t\in [0,T^\nu]} \| f^\nu (t)\|_{H^s} \le \nu^N,
$$
but 
$$
\| u^\nu(T^\nu) - U_\bl(T^\nu) \|_{L^\infty} \gtrsim \nu^{\frac12+\epsilon},
$$
$$
\| u^\nu(T^\nu) - U_\bl(T^\nu) \|_{L^2} \gtrsim \nu^{\frac12+\epsilon},
$$
$$\| \omega^\nu(T^\nu) - \omega_\bl(T^\nu) \|_{L^\infty} \to \infty,$$
for time sequences $T^\nu$ of order $\nu^{-1/2}\log \nu^{-1}$. Here, $\omega^\nu = \nabla\times u^\nu$ denotes the vorticity of fluids. 
\end{theorem}


\begin{theorem}[Instability result for stable profiles] \label{firstinstability-stable}
Let $U_\bl = U(z)$ be a stable stationary shear layer profile as described in Section \ref{sec-BL}. 
Then, for any $s,N$ arbitrarily large, there exists a sequence of solutions $u^\nu$ to the Navier-Stokes equations, 
with forcing $f^\nu = f^P$ (boundary layer forcing), so that $u^\nu$ satisfy 
$$
\| u^\nu(0) - U_\bl \|_{H^s} \le \nu^N,
$$
but 
$$\| u^\nu (T^\nu) - U_\bl\|_{L^\infty} \gtrsim \nu^{5/4},$$
$$\| u^\nu (T^\nu) - U_\bl\|_{L^2} \gtrsim \nu^{5/4},$$
$$\| \omega^\nu (T^\nu) - \omega_\bl \|_{L^\infty} \gtrsim 1,$$
for some time sequences  $T^\nu$ of order $\nu^{-1/2}\log \nu^{-1}$. 
\end{theorem}

The spectral instability for stable profiles gives rise to sublayers (or critical layers) whose thickness is of order $\nu^{3/4}$. 
The velocity gradient in this sublayer grows like $\nu^{-3/4} e^{t/\nu^{1/2} }$, and becomes large when $t$ is of order $T^\nu$.
As a consequence, they may in turn become unstable after the instability time $T^\nu$ obtained in the above theorem. 
Thus, in order to improve the $\nu^{5/4}$ instability, one needs to further examine the stability properties 
of this sublayer itself (see \cite{GrN4}).


\subsection{Boundary layer norms}


We end the introduction by introducing the boundary layer norms to be used throughout the chapter. Precisely, for each vorticity function $\omega_\alpha = \omega_\alpha(z)$, we introduce 
the following boundary layer norms
\begin{equation}\label{assmp-wbl-stable} 
\| \omega_\alpha\|_{ \beta, \gamma, 1} : = \sup_{z\ge 0} 
\Big [ \Bigl( 1 +  \delta^{-1} \phi_{p} (\delta^{-1} z)  \Bigr)^{-1} e^{\beta z} |\omega_\alpha (z)| \Big],
\end{equation}
where $\beta > 0$, $p$ is a large, fixed number, 
$$
\phi_p(z) = {1 \over 1 + z^p},
$$
and with the boundary layer thickness 
$$
\delta = \gamma \nu^{1/4}
$$
for some $\gamma>0$. We introduce the boundary layer space ${\cal B}^{\beta,\gamma,1}$ 
  to consist of functions whose $\|\cdot \|_{ \beta, \gamma, 1} $ norm is finite. 
  We also denote by $L^\infty_\beta$ the function spaces equipped with the finite norm 
  $$
  \|\omega \|_\beta = \sup_{z\ge 0} e^{\beta z} |\omega(z)|.
  $$ 
  When there is no weight $e^{\beta z}$, we simply write $L^\infty$ for the usual bounded function spaces. 
  Clearly, 
  $$
  L^\infty_\beta \subset {\cal B}^{\beta,\gamma,1}
.  $$ 
 In addition, it is straightforward to check that 
\begin{equation}\label{al-norm}
\| f g \|_{\beta,\gamma, 1} \le \| f \|_{L^\infty} \| g \|_{ \beta,\gamma,1}.
\end{equation}
Finally, for functions $\omega(x,z)$, we introduce 
$$\| \omega\|_{ \sigma,\beta, \gamma, 1} : = \sup_{\alpha \in \RR} (1+|\alpha|)^{\sigma}\| \omega_\alpha\|_{ \beta, \gamma, 1} ,$$
for $\sigma> 1$, in which $\omega_\alpha$ is the Fourier transform of $\omega$ in the variable $x$. Combining with \eqref{al-norm}, we have 
\begin{equation}\label{al-norm1}
\| f g \|_{\sigma,\beta,\gamma, 1} \le \| f \|_{\sigma,0} \| g \|_{\sigma,\beta,\gamma,1},
\end{equation}
where $\| f \|_{\sigma,0} = \sup_{\alpha \in \RR} (1+|\alpha|)^{\sigma}\| f_\alpha\|_{L^\infty}$. 


\section{Linear instability}\label{sec-linear}


In this section, we shall recall the spectral instability of stable boundary layer profiles \cite{GGN3} and the semigroup estimates 
on the corresponding linearized Navier-Stokes equation \cite{GrN2,GrN3}.
 

\subsection{Spectral instability}\label{sec-grmode}


The following theorem, proved in \cite{GGN3}, provides an unstable eigenvalue of $L$ for generic shear flows.  

\begin{theorem}[Spectral instability; \cite{GGN3}]\label{theo-spectralinstablity}
Let $U(z)$ be an arbitrary shear profile with $U(0)=0$ and $U'(0) > 0$ and satisfy 
$$
\sup_{z \ge 0} | \partial^k_z (U(z) - U_+) e^{\eta_0 z} | < + \infty, \qquad k=0,\cdots ,4,
$$ 
for some constants $U_+$ and $\eta_0 > 0$. Let $R = \nu^{-1}$ be the Reynolds number, 
and set $\alpha_\mathrm{low}(R)\sim R^{-1/4}$ and  $ \alpha_\mathrm{up}(R)\sim R^{-1/6}$ be the lower and upper stability branches. 

Then, there is a critical Reynolds number $R_c$ so that for all $R\ge R_c$ and all $\alpha \in (\alpha_\mathrm{low}(R), 
\alpha_\mathrm{up}(R))$, there exist
a nontrivial triple $c(R), \hat v(z; R), \hat p(z;R)$, with $\mathrm{Im} ~c(R) >0$, 
such that $v_R: = e^{i\alpha(x-ct) }\hat v(z;R)$ and $p_R: = e^{i\alpha(x-ct)} \hat p(z;R)$ solve the linearized Navier-Stokes problem 
\eqref{EE-vort}. Moreover there holds the following estimate for the growth rate of the unstable solutions:
$$ 
\alpha \Im c(R) \quad \approx\quad  R^{-1/2}
$$
as $R \to \infty$. 
\end{theorem}

The proof of the previous Theorem, which can be found in \cite{GGN3}, gives a detailed description of the unstable mode.
In this paper we focus on the lower branch of instability.  In this case
$$
\alpha_\nu \approx R^{-1/4} = \nu^{1/4}, \qquad \Re \lambda_\nu\approx R^{-1/2} = \nu^{1/2},
$$
The vorticity of the unstable mode is of the form 
\begin{equation}\label{w0-gr-stable} 
\omega_0 = e^{\lambda_\nu t} \Delta ( e^{i \alpha_\nu x} \phi_0(z) ) \quad + \quad \mbox{complex conjugate} 
\end{equation}
The stream function $\phi_0$ is constructed through asymptotic expansions, and is of the form 
$$
 \phi_0: =  \phi_{in,0}(z) +\delta_\bl \phi_{bl,0} ( \delta_\bl^{-1} z) 
  $$ 
for some boundary layer function $\phi_{bl,0}$, where $\delta_\bl = \nu^{1/4}$.


By construction,  derivatives of $\phi_{\bl,0}$ satisfy 
$$ 
|\partial_z^k \phi_{\bl,0}(\delta_\bl^{-1} z)| \le C_k \delta_\bl^{-k} e^{-\eta_0 z / \delta_\bl} .
$$
%
%
In addition, it is clear that each $x$-derivative of $\omega_0$ gains a small factor of $\alpha_\nu \approx \nu^{1/4}$. 
We therefore have an accurate description of the linear unstable mode.


\subsection{Linear estimates}


The corresponding semigroup $e^{Lt}$ of the linear problem \eqref{EE-vort} is constructed through the path integral 
\beq \label{int2}
e^{L t} \omega = \int_\RR e^{i\alpha x} e^{L_\alpha t} \omega_\alpha \; d\alpha 
\eeq
in which $\omega_\alpha$ is the Fourier transform of $\omega$ in tangential variables and
$L_\alpha$, defined as in \eqref{EE-vort-a}, is the Fourier transform of $L$. 
One of the main results proved in \cite{GrN3} is the following theorem.

\begin{theorem} \cite{GrN3}\label{theo-eLt-stable} 
Let $\alpha \lesssim \nu^{1/4}$.
Let $\omega_\alpha \in {\cal B}^{\beta,\gamma,1}$ for some positive $\beta$, $\gamma_0$ be defined as in \eqref{def-ga0max0}. 
Assume that $\gamma_0$ is finite. 
Then, for any $\gamma_1>\gamma_0$, there is some positive constant $C_\gamma$ so that 
$$\begin{aligned}
\| e^{L_\alpha t}\omega_\alpha\|_{ \beta, \gamma, 1} &\le C_\gamma  
e^{\gamma_1 \sqrt{\nu} t }  e^{- \frac14 \alpha^2 \nu t}  \| \omega_\alpha\|_{ \beta, \gamma, 1},
\\
\| \partial_ze^{L_\alpha t}\omega_\alpha\|_{ \beta, \gamma, 1} &\le C_\gamma \Big( \nu^{-1/4}+
 (\nu t)^{-1/2} \Big) e^{\gamma_1 \sqrt{\nu} t }  e^{- \frac14 \alpha^2 \nu t}  \| \omega_\alpha\|_{ \beta, \gamma, 1}. 
\end{aligned}$$
\end{theorem}


\section{Approximate solutions}\label{sec-approx}



Theorem \ref{theo-approx} follows at once from the following theorem. 

\begin{theorem}\label{theo-approx-stable} 
Let $U(z)$ be a stable shear layer profile, and let $U_\bl(\nu t,z)$ be the corresponding time-dependent shear profiles. 
Then, there exist an approximate solution $\widetilde u_\app$ that approximately solves Navier-Stokes equations in the following sense: 
for arbitrarily large numbers $p,M$ and for any $\epsilon>0$, the functions $\widetilde u_\app$ 
solve
\begin{equation}\label{NS-Uapp}
\begin{aligned}
\partial_t \widetilde u_\app + (\widetilde u_\app \cdot \nabla) \widetilde u_\app + \nabla \widetilde p_\app   &=
\nu \Delta \widetilde u_\app + {\cal E}_\app,
\\
\nabla \cdot \widetilde u_\app &= 0,
\end{aligned}\end{equation}
for some remainder $ {\cal E}_\app$ and for time $t \le T_\nu$, with $T_\nu$ being defined through
$$
\nu^p e^{\Re \lambda_0 T_\nu} = \nu^{\frac12 + \epsilon}.
$$
In addition, for all $t \in [0,T_\nu ]$, there hold
$$
\begin{aligned}\| \mathrm{curl} (\widetilde u_\app - U_\bl(\nu t,z) )\|_{\beta,\gamma,1} &\lesssim \nu^{\frac12 + \epsilon},
\\
\|\mathrm{curl} \mathcal{E}_\app (t)\|_{\beta,\gamma,1} 
&\lesssim   \nu^{M} .
\end{aligned}$$
Furthermore, there are positive constants $\theta_0,\theta_1,\theta_2$ independent of $\nu$ so that there holds 
$$  
\theta_1 \nu^p e^{\Re \lambda_0 t} \le  \| (\widetilde u_\app - U_\bl) (t)\|_{L^\infty} \le \theta_2 \nu^p e^{\Re \lambda_0 t} 
$$
for all $t\in [0,T_\nu]$. In particular, 
$$
\| (\widetilde u_\app - U_\bl) (T_\nu)\|_{L^\infty}\gtrsim \nu^{\frac12 + \epsilon}.
$$ 
\end{theorem}


\subsection{Formal construction}


The construction is classical. Indeed, the approximate solutions are constructed in the following form
\beq \label{def-Uapp}
\widetilde u_\app(t,x,z) = U_\bl( \nu t,z) + \nu^p \sum_{j = 0}^M \nu^{j/4} u_j(t,x,z).
\eeq
For convenience, let us set $v = u -U_\bl $, where $u$ denotes the genuine solution to the Navier-Stokes equations. 
Then, the vorticity $\omega = \nabla \times v$ solves 
$$
 \partial_t \omega + (U_\bl  (\nu t, y)  + v)\cdot \nabla \omega + v_2 \partial_y^2U_s(\nu t,y) - \nu \Delta \omega =0 
$$
in which $v = \nabla^\perp \Delta^{-1} \omega$ and $v_2$ denotes the vertical component of velocity. Here and in what follows, 
$\Delta^{-1}$ is computed with the zero Dirichlet boundary condition. As $U_\bl $ depends slowly on time, 
we can rewrite the vorticity equation as follows: 
\begin{equation}\label{vort-Phi}
 (\partial_t - L) \omega  + \nu^{1/4} S \omega + Q(\omega, \omega) =0.
\end{equation}
In \eqref{vort-Phi}, $L$ denotes the linearized Navier-Stokes operator around the stationary boundary layer $U  = U_s(0,z)$:  
$$ L \omega : = \nu \Delta \omega  - U \partial_x \omega - u_2 U'' ,$$
$Q(\omega, \tilde \omega)$ denotes the quadratic nonlinear term $u \cdot \nabla \tilde \omega$, 
with $v = \nabla^\perp \Delta^{-1} \omega$, and $S$ denotes the perturbed operator defined by 
$$
\begin{aligned}
 S\omega: &= \nu^{-1/4}[U_s ( \nu t, z)  - U (z)] \partial_x \omega + \nu^{-1/4} u_2 [\partial_y^2 U_s ( \nu t, z)  - U''(z)] 
.\end{aligned}$$
Recalling that $U_s$ solves the heat equation with initial data $U(z)$, we have 
$$| U_s ( \nu t, z)  - U (z) | \le C \| U''\|_{L^\infty} e^{-\eta_0 z} \nu t$$ 
and 
$$ | \partial_z^2 U_s ( \nu t, z)  - U''(z)| \le C \| U''\|_{W^{2,\infty}} e^{-\eta_0 z} \nu t.$$
Hence, 
\begin{equation}\label{def-Sop}
\begin{aligned}
 S\omega  = \nu^{-1/4}\cO( \nu t e^{-\eta_0 z})\Big[  |\partial_x \omega |+ |\partial_x \Delta^{-1}\omega| \Big]
\end{aligned}\end{equation}
in which $\Delta^{-1} \omega$ satisfies the zero boundary condition on $z=0$. 
The approximate solutions are then constructed via the asymptotic expansion: 
\begin{equation}\label{def-omegaAAA} \omega_\app = \nu^p \sum_{j=0}^M \nu^{j/4} \omega_j,\end{equation}
in which $p$ is an arbitrarily large integer. Plugging this Ansatz into \eqref{vort-Phi} and matching order in $\nu$,
 we are led to solve 
 \begin{itemize}

\item for $j =0$: 
$$ (\partial_t - L) \omega_0 = 0$$
with zero boundary conditions on $v_0= \nabla^\perp (\Delta)^{-1}\omega_0$ on $z=0$; 

\item for $0<j\le M$: 
\begin{equation}\label{eqs-omegaj} (\partial_t - L) \omega_j  = R_j, \qquad {\omega_j}_{\vert_{t=0}}=0,\end{equation}
with zero boundary condition on $v_j = \nabla^\perp (\Delta)^{-1}\omega_j$ on $z=0$. Here, the remainders $R_j$ are defined by 
$$ 
R_j = S \omega_{j-1} + \sum_{k + \ell + 4p = j}Q(\omega_k, \omega_\ell).
$$

\end{itemize}
As a consequence, the approximate vorticity $\omega_\app$ solves \eqref{vort-Phi}, leaving the error $R_\app$ defined by
\begin{equation}
\label{error-app}
\begin{aligned}
R_\app
&= \nu^{p+\frac{M+1}{4}} S \omega_{M} +  \sum_{k+ \ell> M+1 -4 p; 1\le k,\ell \le M} \nu^{2p+ \frac{k+\ell}{4}}Q(\omega_k, \omega_\ell) 
\end{aligned}\end{equation}
which formally is of order $ \nu^{p+\frac{M+1}{4}}$, for arbitrary $p$ and $M$.

 
\subsection{Estimates}


We start the construction with $\omega_0$ being the maximal growing mode, constructed in Section \ref{sec-grmode}. We recall 
\begin{equation}\label{def-omega000}
\omega_0 = e^ {\lambda_\nu t} e^{i\alpha_\nu x} \Delta_{\alpha_\nu} \Big( \phi_{in,0}(z) 
+\nu^{1/4} \phi_{\bl,0} ( \nu^{-1/4} z) \Big)  \quad+\quad \mbox{c.c.}
\end{equation}
with $\alpha_\nu \sim \nu^{1/4}$ and $\Re \lambda_\nu \sim \sqrt \nu$. In what follows, $\alpha_\nu$ and $\lambda_\nu$ are fixed. 
We obtain the following lemma.  

\begin{lemma}\label{lem-omegaj} 
Let $\omega_0$ be the maximal growing mode \eqref{def-omega000}, and let $\omega_j$ 
be inductively constructed by \eqref{eqs-omegaj}. Then, there hold the following uniform bounds:
\begin{equation}\label{induction-jn}
\begin{aligned}
\| \partial_x^a\partial_z^b\omega_j \|_{\sigma,\beta,\gamma,1} & \le C_0 \nu^{a/4}\nu^{-b/4} \nu^{-\frac12[\frac{j}{4p}]}
e^{\gamma_0 (1+\frac{j}{4p})\nu^{1/2}t } 
\end{aligned} \end{equation}
for all $a,j\ge 0$ and for $b=0,1$. 
In addition, the approximate solution $\omega_\app$ defined as in \eqref{def-omegaAAA} satisfies 
\begin{equation}\label{est-omegaAAA}
\| \partial_x^a \partial_z^b \omega_\app\|_{\sigma,\beta,\gamma,1} 
\lesssim \nu^{a/4} \nu^{-b/4} \sum_{j=0}^M \nu^{-\frac12[\frac{j}{4p}]}  \Big( \nu^p e^{\gamma_0 \nu^{1/2} t}\Big)^{1+ \frac{j}{4p}} , 
\end{equation}
for $a\ge 0$ and $b=0,1$. 
 Here, $[k]$ denotes the largest integer so that $[k]\le k$. 
\end{lemma}

\begin{proof} 
For $j\ge 1$, we construct $\omega_j$ having the form 
$$ \omega_j = \sum_{n \in \ZZ} e^{i n \alpha_\nu x} \omega_{j,n}$$
It follows that $\omega_{j,n}$ solves  
$$
(\partial_t - L_{\alpha_n}) \omega_{j,n}  = R_{j,n}, \qquad {\omega_{j,n}}_{\vert_{t=0}}=0 
$$
with $\alpha_n = n \alpha_\nu$ and $R_{j,n}$ the Fourier transform of $R_j$ 
evaluated at the Fourier frequency $\alpha_n$. Precisely, we have 
$$ 
R_{j,n} = S_{\alpha_n} \omega_{j-1,n} + \sum_{k + \ell + 8p = j}\sum_{n_1+n_2 = n}Q_{\alpha_n}(\omega_{k,n_1}, \omega_{\ell,n_2}),
$$
in which $S_{\alpha_n}$ and $Q_{\alpha_n}$ denote the corresponding operator $S$ and $Q$ in the Fourier space. 
The Duhamel's integral reads 
\begin{equation}\label{Duh-omegajn}
\omega_{j,n}(t) = \int_0^t e^{L_{\alpha_n} (t-s)} R_{j,n}(s)\; ds \end{equation}
for all $j\ge 1$ and $n \in \ZZ$. 

It follows directly from an inductive argument and the quadratic nonlinearity of $Q(\cdot,\cdot)$ that for all $0\le j\le M$, 
$\omega_{j,n} = 0$ for all $|n|\ge 2^{j+1}$. 
This proves that $|\alpha_n| \le 2^{M+1} \alpha_\nu \lesssim \nu^{1/4}$, for all $|n| \le 2^{M+1}$.  
Since $\alpha_n \lesssim \nu^{1/4}$, the semigroup bounds from Theorem \ref{theo-eLt-stable} read 
\begin{equation}\label{eLt-jn}\begin{aligned}
\| e^{L_\alpha t}\omega_\alpha\|_{ \beta, \gamma, 1} &\lesssim   e^{\gamma_1 \nu^{1/2} t }  
e^{- \frac14 \alpha^2 \nu t}  \| \omega_\alpha\|_{ \beta, \gamma, 1},
\\\| \partial_ze^{L_\alpha t}\omega_\alpha\|_{ \beta, \gamma, 1} &\lesssim \Big( \nu^{-1/4}
+ ( \nu t)^{-1/2} \Big) e^{\gamma_1 \nu^{1/2} t }  e^{- \frac14 \alpha^2 \nu t}  \| \omega_\alpha\|_{ \beta, \gamma, 1}. 
\end{aligned}\end{equation}
In addition, since $\alpha_n\lesssim \nu^{1/4}$, from \eqref{def-Sop}, 
we compute 
$$
\begin{aligned}
 S_{\alpha_n}\omega_{j-1,n}  = \cO( \nu t e^{-\eta_0 z})\Big[  |\omega_{j-1,n}|+ |\Delta_{\alpha_n}^{-1}\omega_{j-1,n}| \Big]
\end{aligned}
$$
and hence by induction we obtain 
\begin{equation}\label{bd-Sjn}
\begin{aligned}
\| S_{\alpha_n}\omega_{j-1,n}\|_{\beta,\gamma,1} 
&\lesssim  \nu t \Big[ \|\omega_{j-1,n}\|_{\beta,\gamma,1}+ \| e^{-\eta_0 z}\Delta_{\alpha_n}^{-1}\omega_{j-1,n}\|_{\beta,\gamma,1} \Big]
\\
&\lesssim  \nu t \Big[ \|\omega_{j-1,n}\|_{\beta,\gamma,1}+ \| \Delta_{\alpha_n}^{-1}\omega_{j-1,n}\|_{L^\infty} \Big]
\\
&\lesssim  \nu t \nu^{-\frac12[\frac{j-1}{4p}]} e^{\gamma_0 (1+\frac{j-1}{4p})\nu^{1/2}t } ,
\end{aligned}
\end{equation}
 where we used $\|e^{-\eta_0 z} \cdot \|_{\beta,\gamma,1} \le \|\cdot \|_{L^\infty}$ for $\beta < \eta_0$, and 
 $$
 \| \Delta_{\alpha}^{-1} \omega\|_{L^\infty} 
 \le C \| \omega \|_{\beta,\gamma,1},
 $$
uniformly in small $\alpha$; we shall prove this inequality in the Appendix. Let us first consider the case when $1\le j\le 4p-1$, for which $R_{j,n} = S_{\alpha_n} \omega_{j-1,n} $. 
That is, there is no nonlinearity in the remainder. 
Using the above estimate on $S_{\alpha_n}$ and the semigroup estimate \eqref{eLt-jn} into \eqref{Duh-omegajn}, 
we obtain, for $1\le j\le 4p-1$,
$$\begin{aligned} \|\omega_{j,n}(t)\|_{\beta,\gamma,1} 
&\le \int_0^t \| e^{L_{\alpha_n} (t-s)} S_{\alpha_n} \omega_{j-1,n} (s) \|_{\beta,\gamma,1}\; ds 
\\
&\le C\int_0^t e^{\gamma_1 \nu^{1/2} (t-s) }  \| S_{\alpha_n} \omega_{j-1,n} (s) \|_{\beta,\gamma,1}\; ds 
\\
&\le C\int_0^t e^{\gamma_1 \nu^{1/2} (t-s) }  \nu s e^{\gamma_0 (1+\frac{j-1}{4p})\nu^{1/2}s } \; ds .
\end{aligned}$$
We choose
$$
\gamma_1 = \gamma_0 (1+ \frac{j-1}{4p} + \frac{1}{8p})
$$
in (\ref{eLt-jn}) and use the inequality 
$$
\nu^{1/2} t \le C e^{\frac{\gamma_0}{8p}\nu^{1/2}t}.
$$
and obtain
\begin{equation}\label{est-ojn1}\begin{aligned} \|\omega_{j,n}(t)\|_{\beta,\gamma,1} 
&\le C\int_0^t e^{\gamma_1 \nu^{1/2} (t-s) } \nu^{1/2} e^{\gamma_0 (1+\frac{j-1}{4p} + \frac{1}{8p})\nu^{1/2}s } \; ds 
\\
&\le C\nu^{1/2} e^{\gamma_0 (1+\frac{j-1}{4p} + \frac{1}{8p})\nu^{1/2} t} \int_0^t  \; ds 
\\
&\le C\nu^{1/2} t e^{\gamma_0 (1+\frac{j-1}{4p} + \frac{1}{8p})\nu^{1/2} t} 
\\& \le C e^{\gamma_0 (1+\frac{j}{4p})\nu^{1/2}t } .
\end{aligned}\end{equation}
Similarly, as for derivatives, we obtain 
$$\begin{aligned} 
&
\|\partial_z\omega_{j,n}(t)\|_{\beta,\gamma,1} 
\\&\le \int_0^t \| e^{L_{\alpha_n} (t-s)} S_{\alpha_n} \omega_{j-1,n} (s) \|_{\beta,\gamma,1}\; ds 
\\
&\le C \int_0^t \Big( \nu^{-1/4}+ ( \nu (t-s))^{-1/2} \Big) e^{\gamma_1 \nu^{1/2} (t-s) }  \| S_{\alpha_n} \omega_{j-1,n} (s) \|_{\beta,\gamma,1}\; ds 
\\
&\le C 
\int_0^t \Big( \nu^{-1/4}+ ( \nu (t-s))^{-1/2} \Big) e^{\gamma_1 \nu^{1/2} (t-s) }  \nu s e^{\gamma_0 (1+\frac{j-1}{4p})\nu^{1/2} s } \; ds ,
\end{aligned}$$
in which the integral involving $\nu^{-1/4}$ is already treated in \eqref{est-ojn1} and bounded by $C\nu^{-1/4}e^{\gamma_0 (1+\frac{j}{4p})\nu^{1/2}t }.$ As for the second integral, we estimate 
\begin{equation}\label{t12-bd} 
\begin{aligned}
\int_0^t &( \nu (t-s))^{-1/2} e^{\gamma_1 \nu^{1/2} (t-s) }  \nu s e^{\gamma_0 (1+\frac{j-1}{4p})\nu^{1/2} s } \; ds
\\ 
&\le  \int_0^t ( \nu (t-s))^{-1/2} e^{\gamma_1 \nu^{1/2} (t-s) }  \nu^{1/2} e^{\gamma_0 (1+\frac{j-1}{4p} + \frac{1}{8p})\nu^{1/2} s } \; ds
\\&\le \nu^{1/2} e^{\gamma_0 (1+\frac{j-1}{4p} + \frac{1}{8p})\nu^{1/2} t }  \int_0^t ( \nu (t-s))^{-1/2} \; ds 
\\&\le C \sqrt t e^{\gamma_0 (1+\frac{j-1}{4p} + \frac{1}{8p})\nu^{1/2} t } 
\\&\le C \nu^{-1/4} e^{\gamma_0 (1+\frac{j}{4p} )\nu^{1/2} t } .
\end{aligned}\end{equation}
Thus, 
$$\|\partial_z\omega_{j,n}(t)\|_{\beta,\gamma,1} \le C \nu^{-1/4} e^{\gamma_0 (1+\frac{j}{4p} )\nu^{1/2} t }.$$ 
This and \eqref{est-ojn1} prove the inductive bound \eqref{induction-jn} for $j\le 4p-1$. 

For $j \ge 4p$, the quadratic nonlinearity starts to play a role. For $k+\ell = j-4p$, we compute 
\begin{equation}\label{est-Q2w}
Q_{\alpha_n}(\omega_{k,n_1}, \omega_{\ell,n_2}) = i \alpha_\nu \Big( n_2 \partial_z \Delta_{\alpha_n}^{-1} \omega_{k,n_1} 
\omega_{\ell,n_2}  - n_1 \Delta_{\alpha_n}^{-1} \omega_{k,n_1} \partial_z \omega_{\ell,n_2}\Big).
\end{equation}
Using the algebra structure of the boundary layer norm (see \eqref{al-norm}), we have 
$$\begin{aligned}
\alpha_\nu\| \partial_z \Delta_{\alpha_n}^{-1} \omega_{k,n_1} \omega_{\ell,n_2} \|_{\beta,\gamma,1} 
&\lesssim 
\nu^{1/4}\| \partial_z \Delta_{\alpha_n}^{-1} \omega_{k,n_1}\|_{L^\infty} \|\omega_{\ell,n_2} \|_{\beta,\gamma,1}
\\&\lesssim 
\nu^{1/4}\| \omega_{k,n_1}\|_{\beta,\gamma,1} \|\omega_{\ell,n_2} \|_{\beta,\gamma,1}
\\&\lesssim \nu^{1/4}\nu^{-\frac12[\frac{k}{4p}]}  \nu^{-\frac12[\frac{\ell}{4p}]}  e^{\gamma_0 (2+\frac{k+\ell}{4p})\nu^{1/2}t }
 \end{aligned}$$
 where we used
 $$
 \| \partial_z \Delta_{\alpha_n}^{-1} \omega_{k,n_1} \|_{L^\infty} 
 \le C \| \omega_{k,n_1}\|_{\beta,\gamma,1},
 $$
 an inequality which is proven in the Appendix.
Moreover,
$$\begin{aligned}
\alpha_\nu\| \Delta_{\alpha_n}^{-1} \omega_{k,n_1} \partial_z \omega_{\ell,n_2} \|_{\beta,\gamma,1} 
&\lesssim \nu^{1/4}
\|  \Delta_{\alpha_n}^{-1} \omega_{k,n_1}\|_{L^\infty} \| \partial_z\omega_{\ell,n_2} \|_{\beta,\gamma,1}
\\&\lesssim 
\nu^{1/4}\| \omega_{k,n_1}\|_{\beta,\gamma,1} \|\partial_z\omega_{\ell,n_2} \|_{\beta,\gamma,1}
\\&\lesssim \nu^{-\frac12[\frac{k}{4p}]}  \nu^{-\frac12[\frac{\ell}{4p}]}  e^{\gamma_0 (2+\frac{k+\ell}{4p})\nu^{1/2}t },
 \end{aligned}$$
 in which the derivative estimate \eqref{induction-jn} was used. 
We note that 
$$
[\frac{k}{4p}]+ [\frac{\ell}{4p}] \le [\frac{k+\ell}{4p}] = [\frac{j}{4p}] - 1.
$$
 This proves 
$$\begin{aligned}
\|& Q_{\alpha_n}(\omega_{k,n_1}, \omega_{\ell,n_2}) \|_{\beta,\gamma,1} 
\lesssim \nu^{1/2} \nu^{-\frac12[\frac{j}{4p}]} e^{\gamma_0 (1+\frac{j}{4p})\nu^{1/2}t }
\end{aligned}$$ 
for all $k+\ell = j-4p$. This, together with the estimate \eqref{bd-Sjn} on $S_{\alpha_n}$, yields 
$$
\begin{aligned}
\| R_{j,n} (t)\|_{\beta,\gamma,1} 
&\lesssim \nu t \nu^{-\frac12[\frac{j-1}{4p}]} e^{\gamma_0 (1+\frac{j-1}{4p})\nu^{1/2}t } + \nu^{1/2} \nu^{-\frac12[\frac{j}{4p}]} e^{\gamma_0 (1+\frac{j}{4p})\nu^{1/2}t }
\\
&\lesssim \nu^{1/2} \nu^{-\frac12[\frac{j}{4p}]} e^{\gamma_0 (1+\frac{j}{4p})\nu^{1/2}t },
\end{aligned}
$$
for all $j\ge 4p$ and $n\in \ZZ$, in which we used $\nu^{1/2} t \le e^{\gamma_0 t/4p}$. 

Putting these estimates into the Duhamel's integral formula \eqref{Duh-omegajn}, we obtain, for $j\ge 4p$,
$$
\begin{aligned}
\|\omega_{j,n} (t)\|_{\beta,\gamma,1}  
&\le C \int_0^t  e^{\gamma_1 \nu^{1/2} (t-s) } \| R_{j,n}(s)\|_{\beta,\gamma,1} \; ds
\\
&\le C \int_0^t  e^{\gamma_1 \nu^{1/2} (t-s) } \nu^{1/2} \nu^{-\frac12[\frac{j}{4p}]} e^{\gamma_0 (1+\frac{j}{4p})\nu^{1/2}s } \; ds  
\\
&\lesssim \nu^{-\frac12[\frac{j}{4p}]}  e^{\gamma_0 (1+\frac{j}{4p})\nu^{1/2}s } 
 \end{aligned}$$
and 
$$\begin{aligned} 
&
\|\partial_z\omega_{j,n}(t)\|_{\beta,\gamma,1} 
\\&\le C \int_0^t \Big( \nu^{-1/4}+ (\nu (t-s))^{-1/2} \Big) e^{\gamma_1 \nu^{1/2} (t-s) }  \| R_{j,n}(s)\|_{\beta,\gamma,1} \; ds
\\
&\le C 
\int_0^t \Big( \nu^{-1/4}+ (\nu (t-s))^{-1/2} \Big) e^{\gamma_1 \nu^{1/2} (t-s) } \nu^{1/2} \nu^{-\frac12[\frac{j}{4p}]} e^{\gamma_0 (1+\frac{j}{4p})\nu^{1/2}s } \; ds .
\end{aligned}$$
Using \eqref{t12-bd}, we obtain 
$$
\begin{aligned}
\|\partial_z \omega_{j,n} (t)\|_{\beta,\gamma,1}  
\lesssim \nu^{-1/4}\nu^{-\frac12[\frac{j}{4p}]}  e^{\gamma_0 (1+\frac{j}{4p})\nu^{1/2}s } ,
 \end{aligned}$$
which completes the proof of \eqref{induction-jn}. The lemma follows. 
\end{proof}

 
 \subsection{The remainder}
 

We recall that the approximate vorticity $\omega_\app$, constructed as in \eqref{def-omegaAAA}, approximately solves \eqref{vort-Phi}, 
leaving the error $R_\app$ defined by
$$
\begin{aligned}
R_\app
&= \nu^{p+\frac{M+1}{4}} S \omega_{M} +  \sum_{k+ \ell> M+1 -4 p; 1\le k,\ell \le M} \nu^{2p+ \frac{k+\ell}{4}}Q(\omega_k, \omega_\ell) .
\end{aligned}$$
Using the estimates in Lemma \ref{lem-omegaj}, we obtain  
$$\begin{aligned}
\|S \omega_{M} \|_{\sigma,\beta,\gamma,1} &\lesssim  \nu^{1/2} \nu^{-\frac12[\frac{M+1}{4p}]} 
e^{\gamma_0 (1+\frac{M+1}{4p})\nu^{1/2}t }
 \\
 \|Q(\omega_k, \omega_\ell) \|_{\sigma,\beta,\gamma,1} &\lesssim  \nu^{1/2} 
 \nu^{-\frac12[\frac{k+\ell}{4p}]} e^{\gamma_0 (2+\frac{k+\ell}{4p})\nu^{1/2}t }.\end{aligned}$$
This yields 
\begin{equation}\label{est-RAA}
\begin{aligned}
\|\ R_\app\|_{\sigma,\beta,\gamma,1} 
&\lesssim  \nu^{1/2} \sum_{j=M+1}^{2M} \nu^{-\frac12[\frac{j}{4p}]}  \Big( \nu^p e^{\gamma_0 \nu^{1/2} t}\Big)^{1+ \frac{j}{4p}} .
\end{aligned}\end{equation}


\subsection{Proof of Theorem \ref{theo-approx-stable}}


The proof of the Theorem now  straightforwardly follows from the estimates from Lemma \ref{lem-omegaj} 
and the estimate \eqref{est-RAA} on the remainder. Indeed, we choose the time $T_*$ so that 
\begin{equation}\label{def-Tstar} \nu^p e^{\gamma_0 \nu^{1/4} T_*} = \nu^\tau\end{equation}
for some fixed $\tau>\frac12$. It then follows that for all $t\le T_*$ and $j\ge 0$, there holds 
$$
\begin{aligned}
 \nu^{-\frac12[\frac{j}{4p}]}  \Big( \nu^p e^{\gamma_0 \nu^{1/2} t}\Big)^{1+ \frac{j}{4p}}  \lesssim \nu^\tau  \nu^{(\tau - \frac12) \frac{j}{4p}}.
\end{aligned}
$$
Using this into the estimates \eqref{est-omegaAAA} and \eqref{est-RAA}, respectively, we obtain 
\begin{equation}\label{est-wRapp}
\begin{aligned}
\| \partial_z^b\omega_\app(t)\|_{\sigma,\beta,\gamma,1} &\lesssim \nu^{p-b/4} e^{\gamma_0 \nu^{1/2} t}  \lesssim \nu^{\tau- b/4},
\\
\|R_\app (t)\|_{\sigma,\beta,\gamma,1} 
&\lesssim  \nu^{1/2} \nu^{-\frac12[\frac{M}{4p}]}  \Big( \nu^p e^{\gamma_0 \nu^{1/2} t}\Big)^{1+ \frac{M}{4p}} 
\\&\lesssim \nu^{\tau+1/2}  \nu^{(\tau - \frac12) \frac{M}{4p}},
\end{aligned}\end{equation}
for all $t \le T_*$. Since $\tau>\frac12$ and $M$ is arbitrarily large (and fixed), the remainder is of order $\nu^P$ 
for arbitrarily large number $P$. The theorem is proved.


\section{Nonlinear instability}\label{sec-nonlinear}


We are now ready to give the proof of Theorem \ref{firstinstability-stable}. Let $\widetilde u_\app$ 
be the approximate solution constructed in Theorem \ref{theo-approx-stable} and let  
$$v = u - \widetilde u_\app,$$
with $u$ being the genuine solution to the nonlinear Navier-Stokes equations. 
The corresponding vorticity $\omega = \nabla \times v$ solves 
$$ 
\partial_t   \omega + (\widetilde u_\mathrm{app} + v) \cdot \nabla \omega + v \cdot \nabla \widetilde \omega_\mathrm{app} 
= \Delta \omega +  R_\app
$$
for the remainder $R_\app  = \mathrm{curl } \,  \, {\cal E}_\app$ satisfying the estimate \eqref{est-RAA}. Let us write 
$$
u_\app = \widetilde u_\app - U_\bl .
$$
To make use of the semigroup bound for the linearized operator $\partial_t - L$, we rewrite the vorticity equation as  
$$ 
(\partial_t - L) \omega + (u_\app + v) \cdot \nabla \omega + v \cdot \nabla \omega_\app  = R_\app
$$
with $\omega_{\vert_{t=0}} = 0$. We note that since the boundary layer profile is stationary, 
the perturbative operator $S$ defined as in \eqref{def-Sop} is in fact zero. The Duhamel's principle then yields 
\begin{equation}\label{Duh-omega} 
\omega (t) = \int_0^t e^{L(t-s)} \Big( R_\app - (u_\app + v) \cdot \nabla \omega - v \cdot \nabla \omega_\app  \Big) \; ds.
\end{equation}
Using the representation \eqref{Duh-omega}, we shall prove the existence and give estimates on $\omega$. 
We shall work with the following norm 
\begin{equation}\label{def-123norm} 
||| \omega(t) |||: = \| \omega(t) \|_{\sigma,\beta,\gamma,1}+ \nu^{1/4}\| \partial_x \omega(t) \|_{\sigma,\beta,\gamma,1} 
+ \nu^{1/4}\| \partial_z \omega (t)\|_{\sigma,\beta,\gamma,1} 
\end{equation}
in which the factor $\nu^{1/4}$ added in the norm is to overcome the loss of $\nu^{-1/4}$ for derivatives
(see \eqref{semi-bounds} for more details).

Let $p$ be an arbitrary large number. We introduce the maximal time $T_\nu$ of existence, defined by 
\begin{equation}\label{claim-www}
\begin{aligned}
T_\nu: = \max\Big \{ t\in [0,T_*]~:~\sup_{0\le s\le t} ||| \omega(s) |||  \le \nu^p e^{\gamma_0 \nu^{1/2} t}\Big\}
 \end{aligned}\end{equation}
in which $T_*$ is defined as in \eqref{def-Tstar}. By the short time existence theory, with zero initial data, $T_\nu$ exists and is positive. 
It remains to give a lower bound estimate on $T_\nu$.  First, we obtain the following lemmas. 
\begin{lemma} For $t \in [0,T_*]$, there hold
$$
\begin{aligned}
\| \partial_x^a \partial_z^b \omega_\app(t)\|_{\sigma,\beta,\gamma,1} &\lesssim \nu^{a/4-b/4} \Big( \nu^p e^{\gamma_0 \nu^{1/2} t}\Big)
\\
\|\partial_x^a \partial_z^bR_\app (t)\|_{\sigma,\beta,\gamma,1} 
&\lesssim  \nu^{1/2+a/4-b/4} \nu^{-\frac12[\frac{M}{4p}]}  \Big( \nu^p e^{\gamma_0 \nu^{1/2} t}\Big)^{1+ \frac{M}{4p}} .
\end{aligned}$$
\end{lemma}
\begin{proof} This follows directly from Lemma \ref{lem-omegaj} and the estimate \eqref{est-RAA} on the remainder $R_\app$, 
upon noting the fact that for $t\in [0,T_*]$, $\nu^p e^{\gamma_0 \nu^{1/2} t}$ remains sufficiently small.
\end{proof}

\begin{lemma} There holds 
$$
\Big\|( u_\app + v) \cdot \nabla \omega + v \cdot \nabla  \omega_\app \Big\|_{\sigma,\beta,\gamma,1} 
\lesssim \nu^{-\frac14}\Big( \nu^p e^{\gamma_0 \nu^{1/2} t}\Big)^2.
$$
for $t\in [0,T_\nu]$. 
\end{lemma}
\begin{proof}
We first recall the elliptic estimate 
$$
\| u\|_{\sigma,0} \lesssim \|\omega\|_{\sigma,\beta,\gamma,1}
$$ 
which is proven in the Appendix \ref{sec-elliptic}), and the following uniform bounds (see \eqref{al-norm1})
$$
\begin{aligned}
\| u \cdot \nabla \tilde \omega \|_{\sigma,\beta,\gamma,1}
&\le \| u \|_{\sigma,0} \| \nabla \tilde \omega \|_{\sigma,\beta,\gamma,1}
\\&\le \| \omega \|_{\sigma,\beta,\gamma,1} \| \nabla\tilde \omega \|_{\sigma,\beta,\gamma,1} .
 \end{aligned} $$
Using this and the bounds on $\omega_\app$, we obtain 
$$
\begin{aligned}
 \| v \cdot \nabla \omega_\app \|_{\sigma,\beta,\gamma,1} &\lesssim 
 \nu^{-1/4}\Big( \nu^p e^{\gamma_0 \nu^{1/2} t}\Big) \| \omega \|_{\sigma,\beta,\gamma,1} 
 \lesssim\nu^{-\frac14} \Big( \nu^p e^{\gamma_0 \nu^{1/2} t}\Big)^2
\end{aligned}
$$
and 
$$
\begin{aligned}
 \| ( u_{\app} + v)\cdot \nabla \omega \|_{\sigma,\beta,\gamma,1}  &
 \lesssim \Big( \nu^p e^{\gamma_0 \nu^{1/4} t} + \| \omega\|_{\sigma,\beta,\gamma,1}\Big) \| \nabla \omega \|_{\sigma,\beta,\gamma,1} 
\\&\lesssim \nu^{-\frac14}\Big( \nu^p e^{\gamma_0 \nu^{1/4} t}\Big)^2
 .\end{aligned}
 $$
This proves the lemma. \end{proof}

Next, using Theorem \ref{theo-eLt-stable} and noting that $\alpha e^{-\alpha^2 \nu t} \lesssim 1+ (\nu t)^{-1/2}$, we obtain the following uniform semigroup bounds:
\begin{equation}\label{semi-bounds}\begin{aligned}
\| e^{L t}\omega\|_{ \sigma,\beta, \gamma, 1} &\le C_0 \nu^{-1/2} e^{\gamma_1 \nu^{1/2} t }  \| \omega\|_{\sigma, \beta, \gamma, 1} 
\\
\| \partial_x e^{L t}\omega\|_{ \sigma,\beta, \gamma, 1} &\le
 C_0 \nu^{-1/2}\Big( 1+ (\nu t)^{-1/2} \Big) e^{\gamma_1 \nu^{1/2} t }  \| \omega\|_{\sigma, \beta, \gamma, 1} 
\\
\| \partial_z e^{L t}\omega\|_{ \sigma,\beta, \gamma, 1} &\le
 C_0 \nu^{-1/2} \Big( \nu^{-1/4}+ (\nu t)^{-1/2} \Big) e^{\gamma_1 \nu^{1/2} t }  \| \omega\|_{\sigma, \beta, \gamma, 1} .
\end{aligned}\end{equation}
We are now ready to apply the above estimates into the Duhamel's integral formula \eqref{Duh-omega}. 
We obtain 
$$
\begin{aligned}
 \|  \omega (t)\|_{\sigma,\beta,\gamma,1} 
 &\lesssim \nu^{-1/2}\int_0^t e^{\gamma_1 \nu^{1/2} (t-s)} \nu^{-\frac14}\Big (\nu^{p} e^{\gamma_0 \nu^{1/2} s}\Big)^2\; ds
 \\&\quad + \nu^{-1/2}\int_0^t e^{\gamma_1 \nu^{1/2} (t-s)}  
 \nu^{1/2} \nu^{-\frac12[\frac{M}{4p}]}  \Big( \nu^p e^{\gamma_0 \nu^{1/2} s}\Big)^{1+ \frac{M}{4p}}  
 \; ds 
\\
 &\lesssim \nu^{-5/4} \Big (\nu^{p} e^{\gamma_0 \nu^{1/2} t}\Big)^2 + \nu^P\Big( \nu^p e^{\gamma_0 \nu^{1/2} t} \Big),
 \end{aligned}$$
upon taking $\gamma_1$ sufficiently close to $\gamma_0$. Set $T_1$ so that 
\begin{equation}\label{def-T11} \nu^p e^{\gamma_0 \nu^{1/2} T_1} =\theta_0 \nu^{\frac54},\end{equation}
for some sufficiently small and positive constant $\theta_0$. 
Then, for all $t\le T_1$, there holds 
$$
 \begin{aligned}\|  \omega (t)\|_{\sigma,\beta,\gamma,1}  
 &\lesssim \nu^{p} e^{\gamma_0 \nu^{1/2}t} \Big[ \theta_0
 + \nu^{P}\Big]
\end{aligned} $$
Similarly, we estimate the derivatives of $\omega$. The Duhamel integral and the semigroup bounds yield 
$$
\begin{aligned}
\| \nabla \omega (t)\|_{\sigma,\beta,\gamma,1} 
 &\lesssim \nu^{-1/2}\int_0^t e^{\gamma_1 \nu^{1/2} (t-s)}  \Big( \nu^{-1/4}+ (\nu (t-s))^{-1/2} \Big) 
 \\&\quad \times  
 \Big[ 
 \nu^{-\frac14}\Big (\nu^{p} e^{\gamma_0 \nu^{1/2} s}\Big)^2 
+  \nu^{-\frac12[\frac{M}{4p}]}  \Big( \nu^p e^{\gamma_0 \nu^{1/2} s}\Big)^{1+ \frac{M}{4p}} 
\Big] \; ds 
\\&\lesssim \nu^{-5/4} \Big[ 
 \nu^{-\frac14}\Big (\nu^{p} e^{\gamma_0 \nu^{1/2} s}\Big)^2 
+  \nu^{-\frac12[\frac{M}{4p}]}  \Big( \nu^p e^{\gamma_0 \nu^{1/2} s}\Big)^{1+ \frac{M}{4p}} 
\Big] 
 \end{aligned}$$
By view of \eqref{def-T11} and the estimate \eqref{t12-bd}, 
the above yields 
$$
 \begin{aligned}
 \| \nabla \omega (t)\|_{\sigma,\beta,\gamma,1}  &\lesssim \nu^{p-\frac14} e^{\gamma_0 \nu^{1/2}t} \Big[ \theta_0 
+ \nu^P\Big] .
\end{aligned} 
$$
To summarize, for $t\le \min\{T_*, T_1,T_\nu\}$, with the times defined as in \eqref{def-Tstar},
 \eqref{claim-www}, and \eqref{def-T11}, we obtain 
$$ ||| w(t)||| \lesssim  \nu^{p} e^{\gamma_0 \nu^{1/2}t} \Big[ \theta_0+ \nu^P\Big].$$
Taking $\theta_0$ sufficiently small, we obtain 
$$
 ||| w(t)||| \ll\nu^{p} e^{\gamma_0 \nu^{1/2}t} 
$$
for all time $t\le \min\{T_*, T_1,T_\nu\}$. 
In particular, this proves that the maximal time of existence $T_\nu$ is greater than $ T_1$, defined as in \eqref{def-T11}.  
This proves that at the time $t=T_*$, the approximate solution grows to order of $\nu^{5/8}$ in the $L^\infty$ norm. 
Theorem \ref{firstinstability-stable} is proved.



\section{Elliptic estimates}\label{sec-elliptic}


In this section, for sake of completeness, we recall the elliptic estimates with respect to the boundary layer norms. 
These estimates are proven in \cite[Section 3]{GrN2}. 

First, we consider the classical one-dimensional Laplace equation
\beq \label{Lap1}
\Delta_\alpha \phi = \partial_z^2 \phi - \alpha^2 \phi = f
\eeq
on the half line $z \ge 0$, with the Dirichlet boundary condition $ \phi(0) = 0$. We recall the function space $L^\infty_\beta$ defined by the finite norm $\| f\|_\beta  =\sup_{z \ge 0}|f(z) | e^{\beta z}$. We will prove
\begin{prop}
If $f \in L^\infty_\beta$ for some $\beta >0$, then $\phi \in L^\infty$. In addition, there holds
\beq \label{Lap3}
(1+\alpha^2) \| \phi \|_{L^\infty} + (1+ | \alpha |) \,  \| \partial_z \phi \|_{L^\infty}
+ \| \partial_z^2 \phi \|_{L^\infty}  \le C \| f \|_{\beta},
\eeq
where the constant $C$ is independent of $\alpha \in \RR$.
\end{prop}
\begin{proof} The solution $\phi$ of (\ref{Lap1}) is explicitly given by
\begin{equation}\label{laplacephi1}
\phi(z) =\int_0^\infty G_\alpha (x,z)f(x) dx 
\end{equation}
where $G_\alpha(x,z) =  - {1 \over 2 \alpha} \Bigl( e^{- \alpha | x-z | }  - e^{-\alpha | x + z |} \Bigr) $. A direct bound leads to
$$
\| \phi \|_{L^\infty} \le {C \over \alpha^2} \| f \|_{\beta}
$$
in which the extra factor of $\alpha^{-1}$ is due to the $x$-integration. 
Differentiating the integral formula, we get
$$
 \|\partial_z \phi \|_{L^\infty} \le {C \over \alpha} \| f \|_{\beta} .
$$
The estimate for $\partial_z^2 \phi$ follows by using directly the equation $\partial_z^2 \phi = \alpha^2 \phi + f$. This yields the lemma for the case when $\alpha$ is bounded away from zero. 

As for small $\alpha$, we note that $G_\alpha(0,z) = 0$ and $|\partial_x G_\alpha(x,z)|\le 1$. Hence, $|G_\alpha(x,z)|\le |x|$ and so    
$$ |\phi(z)| \le \int_0^\infty |G_\alpha (x,z)f(x)| dx  \le \|f\|_\beta \int_0^\infty |x| e^{-\beta x} \; dx \le C \| f\|_\beta.$$
Similarly, since $|\partial_z G_\alpha(x,z)|\le 1$, we get 
$$ |\partial_z\phi(z)| \le \int_0^\infty |\partial_zG_\alpha (x,z)f(x)| dx  \le \|f\|_\beta \int_0^\infty e^{-\beta x} \; dx \le C \| f\|_\beta.$$
The lemma follows. 
\end{proof}
We now establish a similar property for ${\cal B}^{\beta,\gamma,1}$ norms:
\begin{prop} \label{proplaplace3}
If $f \in {\cal B}^{\beta,\gamma,1}$ for some $\beta>0$, then $\phi \in L^\infty$. In addition, there holds
\beq \label{Lap4}
(1+| \alpha |) \, \| \phi \|_{L^\infty} 
+  \| \partial_z \phi \|_{L^\infty}   \le C \| f \|_{\beta,\gamma,1},
\eeq
where the constant $C$ is independent of $\alpha \in \RR$. 
\end{prop}
\begin{proof}
We will only consider the case $\alpha > 0$, the opposite case being similar. As above, since $G_\alpha(x,z)$ is bounded by $\alpha^{-1}$, using (\ref{laplacephi1}), we have 
$$
| \phi (z) | \le \alpha^{-1} \| f \|_{\beta,\gamma,1} \int_0^\infty
e^{- \alpha |   z -   x |} e^{-\beta   x} 
\Bigl( 1 + \delta^{-1} \phi_P(\delta^{-1} x) \Bigr) d  x
$$
$$
\le  \alpha^{-1} \| f \|_{\beta,\gamma,1} 
\Bigl( \alpha^{-1} + \delta^{-1} \int_0^\infty \phi_P(\delta^{-1} x) d   x \Bigr) 
$$
which yields the claimed bound for $\phi$ since $P >1$. A similar proof applies for $\partial_z \phi$. 
\end{proof}

%

Next, let us now turn to the two dimensional Laplace operator.
\begin{prop} \label{inverseLaplace}
Let $\phi$ be the solution of
$$
- \Delta \phi = \omega
$$
with the zero Dirichlet boundary condition, and let
$$
v = \nabla^\perp \phi. 
$$
If $\omega \in {\cal B}^{\sigma,\beta,\gamma,1}$, then $\phi  \in {\cal B}^{\sigma,0}$ and
$v  = (v_1,v_2) \in {\cal B}^{\sigma,0}$. Moreover, there hold the following elliptic estimates 
\beq \label{Lap4}
\| \phi \|_{\sigma,0} + \|  v_1 \|_{\sigma,0} + \|  v_2 \|_{\sigma,0}   \le C \| \omega \|_{\sigma,\beta,\gamma,1},
\eeq
\end{prop}
\begin{proof}
The proof follows directly from taking the Fourier transform in the $x$ variable, with dual integer Fourier component $\alpha$, and using Proposition \ref{proplaplace3}. 
\end{proof}

%
%
%
%
%
%
%
%
%

\bibliographystyle{plain}

\def\cprime{$'$} \def\cprime{$'$}

\end{document}